\documentclass[11pt,a4paper]{report}

\usepackage[english]{babel}
\usepackage[latin1]{inputenc}
\usepackage[T1]{fontenc}
\usepackage{graphicx}
\usepackage{amsmath}
\usepackage{amssymb}
\usepackage{latexsym}
\usepackage{amsthm}
\usepackage[all]{xy}
\usepackage{stmaryrd}
\usepackage{subfigure}
\usepackage{MnSymbol}
\usepackage{comment} 
\usepackage{hyperref}

\def\be{\begin{equation}}
\def\ee{\end{equation}}
\def\beq{\begin{eqnarray*}}
\def\eeq{\end{eqnarray*}}
\def\Z{\mathbb{Z}}

\newcommand{\pic}[3]{\parbox[c]{#1cm}{\includegraphics[scale=#2]{#3}}}
\newcommand{\picw}[2]{\parbox[c]{#1cm}{\includegraphics[width = #1cm]{#2}}}

\newcommand{\cerchio}{\includegraphics[width = .4 cm]{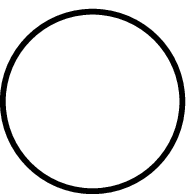}}
\newcommand{\teta}{\includegraphics[width = .4 cm]{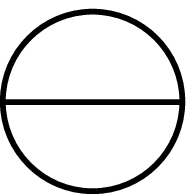}}
\newcommand{\tetra}{\includegraphics[width = .4 cm]{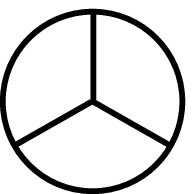}}
\newcommand{\gl}{{\rm gl}}
\newcommand{\ord}{{\rm ord}}

\newtheorem{theo}{Theorem}[section]
\newtheorem{cor}[theo]{Corollary}
\newtheorem{lem}[theo]{Lemma}
\newtheorem{prop}[theo]{Proposition}
\newtheorem{conj}[theo]{Conjecture}

\theoremstyle{definition}
\newtheorem{defn}[theo]{Definition}
\newtheorem{ex}[theo]{Example}
\newtheorem{quest}[theo]{Question}
\newtheorem{rem}[theo]{Remark}

\begin{document}

\begin{titlepage}

\begin{center}

\large{\textbf{UNIVERSIT\`A DEGLI STUDI DI PISA}}\\
\large{\textbf{TESI DI DOTTORATO}}\\
\large{\textbf{IN MATEMATICA}}\\
\vspace{2cm}
\Large{\textbf{ALESSIO CARREGA}}\\
\vspace{1.5cm}
\Large{\textbf{SHADOWS AND QUANTUM INVARIANTS}}\\
\vspace{2cm}
\centerline{\includegraphics[width=3cm]{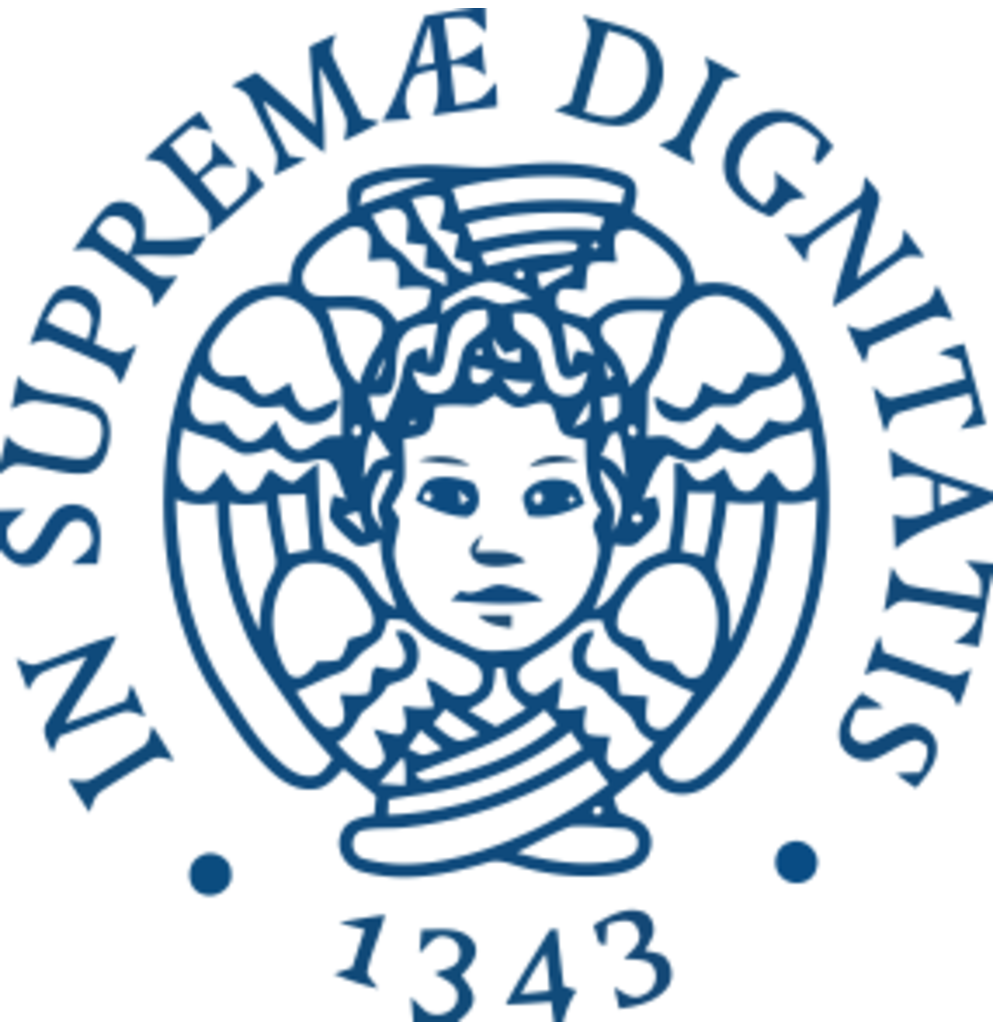}}
\vspace{2cm}
\large{\textbf{RELATORE}}\\
\large{Prof. Bruno Martelli}\\
\vspace{2cm}
\large{$28^\circ$ ciclo di dottorato}\\
\large{2016}

\end{center}

\end{titlepage}

\pagenumbering{roman}

\tableofcontents

\chapter*{Introduction}

\addcontentsline{toc}{chapter}{Introduction}



In 1984 \cite{Jones} V. Jones defined the famous \emph{Jones polynomial}. His discovery unveiled some unexpected connections between algebra, topology, and theoretical physics. Following Jones' initial discovery a variety of knots and 3-manifolds invariants came out. These are called \emph{quantum invariants} and are quite complicated and mysterious. It is a fundamental aim in modern knot theory to ``understand'' the Jones polynomial, that means finding relations between quantum invariants (in particular the Jones polynomial) and geometric and topological proprieties of manifolds (\emph{e.g.} hyperbolic volume, embedded surfaces, decompositions).

Most quantum invariants arise from representations of \emph{quantum groups}, which are deformations $U_q(\mathfrak{g})$ of the universal enveloping algebras $U(\mathfrak{g})$ of a semi-simple complex Lie algebras $\mathfrak{g}$ (see for instance \cite{Chari-Pressley, Kassel}). The most simple and studied quantum group is $U_q(\mathfrak{sl}_2)$. Although it is the simplest case, it is rather general and complicated. 

From this framework of representations of $U_q(\mathfrak{sl}_2)$ we can get the 3-manifold invariants called $SU(2)$-\emph{Reshetikhin-Turaev-Witten invariants} \cite{Turaev} that form the first invariants of 3-manifolds and were given by Witten \cite{Witten} as a part of his quantum-fields-theoretic explanation of the origin of the Jones polynomial. They were then rigorously constructed by Reshetikhin and Turaev \cite{Reshetikhin-Turaev}. Other important quantum invariants coming from the representations of $U_q(\mathfrak{sl}_2)$ are the \emph{colored Jones polynomials}. In particular the $n^{\rm th}$ colored Jones polynomial arises from the $n^{\rm th}$ \emph{indecomposable} representation (or the $(n+1)$-dimensional representation). The first one is the classical Jones polynomial.

Turaev's \emph{shadows} are 2-dimensional polyhedral objects related to smooth 4-manifolds. These are the 4-dimensional analogue of \emph{spines} of 3-manifolds. They were defined by Turaev \cite{Turaev:preprint, Turaev} and then considered by various authors, see for instance \cite{Burri, Carrega_RTW, Carrega_Taitg, Carrega-Martelli, Costantino1, Costantino2, Costantino-Thurston, Costantino-Thurston:preprint, Goussarov, IK, Martelli, Shu, Thurston, Turaev1}. Shadows represent a large class of compact 4-manifolds and encode all the pairs $(M,G)$ where $M$ is an oriented 3-manifold and $G$ a knotted, framed, trivalent graph in $M$ (\emph{e.g.} a framed link).

In \cite{Turaev} Turaev defined shadows and showed how to get the quantum invariants from them through a formula that works in a general context: for any \emph{ribbon category}. Representations of quantum groups form non trivial examples of ribbon categories. These formulas are called \emph{shadow formulas}. These look like Euler characteristics: they are composed by elementary bricks associated to the maximal connected pieces of dimension $0$ (vertices), $1$ (edges) and $2$ (regions) of the shadow, and they are combined together with a ``sign'' depending on the parity of the dimension.

An alternative approach to representation theory for quantum invariants is provided by \emph{skein theory}. The word ``skein'' and the notion were introduced by Conway in 1970 for his model of the Alexander polynomial. This idea became really useful after the work of Kauffman \cite{Kauffman} which redefined the Jones polynomial in a very simple and combinatorial way passing through the \emph{Kauffman bracket}. These combinatorial techniques allow us to reproduce all quantum invariants arising from the representations of $U_q(\mathfrak{sl}_2)$ without any reference to representation theory. This also leads to many interesting and quite easy computations. This skein method was used by Lickorish \cite{Lickorish1, Lickorish2, Lickorish3, Lickorish4}, Blanchet, Habegger, Masbaum and Vogel \cite{BHMV}, and Kauffman and Lins \cite{Kauffman-Lins}, to re-interpret and extend some of the methods of representation theory. We are interested just in these quantum invariants that can be obtained via skein theory, in particular the colored Jones polynomial, the Kauffman bracket, the $SU(2)$-Reshetikhin-Turaev-Witten invariants and the \emph{Turev-Viro invariants}.

The first notion in skein theory is the one of ``\emph{skein space}'' (or \emph{skein module}). These are vector spaces (or modules over a ring) associated to oriented 3-manifolds. These were introduced independently in 1988 by Turaev \cite{Turaev0} and in 1991 by Hoste and Przytycki \cite{HP0}. The framed links in a oriented 3-manifold $M$ can be seen as elements of the skein space of $M$. In fact these generate the skein space. There are many interesting open questions about skein spaces. We can get an important application of quantum invariants already from skein spaces. In fact the evaluation in $A=-1$ of the $\mathbb{C}[A,A^{-1}]$-skein module is an algebra and almost coincides with the ring of the $SL_2(\mathbb{C})$-\emph{character variety} of the 3-manifold \cite{Bullock2}. Moreover they are useful to generalize the Kauffman bracket to manifolds other than $S^3$ and this is the aspect we are mostly interested in. Thanks to result of Hoste-Przytycki \cite{HP2, Pr2} and (with different techniques) to Costantino \cite{Costantino2}, now we can define the Kauffman bracket also in the connected sum $\#_g(S^1\times S^2)$ of $g\geq 0$ copies of $S^1\times S^2$.

Only for few manifolds the skein space is known. A natural open question about skein
spaces is whether the skein vector space of every closed 3-manifold is
finitely generated. In \cite{Carrega_3-torus} we proved that the skein vector space of the 3-torus is finitely generated, in particular we showed $9$ generators. In \cite{Gilmer} it has been proved that that set of generators is actually a basis.

Thirty years after its discovery, we know only a few topological applications of the Jones polynomial. Several topological applications of quantum invariants concern their behavior near a fixed complex point. Some notable applications (or conjectures) are: 
\begin{itemize}
\item{the \emph{Bullock's theorem} about the \emph{character variety} \cite{Bullock2};}
\item{the \emph{volume conjecture} \cite{Murakami};}
\item{the \emph{Chen-Yang's volume conjecture} \cite{Chen-Yang};}
\item{the \emph{slope conjecture} \cite{Garoufalidis};}
\item{the \emph{AJ-conjecture} \cite{ThangLe-Tran, Marche2};}
\item{the \emph{Tait conjecture} \cite{Lickorish, Thistlethwaite, Kauffman_Tait, Murasugi1, Murasugi2};}
\item{the \emph{Eisermann's theorem} \cite{Eisermann}.}
\end{itemize}

The volume conjecture and the Chen-Yang's volume conjecture are about a limit of evaluations respectively of the colored Jones polynomial, and the Turaev-Viro invariants and the Reshetikhin-Turaev-Witten invariants, where the evaluation points converge to $1$. 

The slope conjecture relates the degree of the colored Jones polynomial of a knot in $S^3$ with the slope of the incompressible surfaces of the complement. 

The AJ-conjecture concerns some more complex algebraic properties of the colored Jones polynomial, like generators of principal ideals related to it. This relates the colored Jones polynomial to the $A$-\emph{polynomial}.

The Tait conjecture regards the \emph{breadth} of the Jones polynomial that is something like the degree, it concerns both the behavior near $\infty$ and near $0$. This is a proved theorem about the \emph{crossing number} of \emph{alternating links}.

Eisermann's theorem concerns the behavior of the Kauffman bracket in the imaginary unit $q=A^2= i$. This connects the Jones polynomial to 4-dimensional smooth topology, in particular to \emph{ribbon surfaces}. 

The Tait conjecture (as a result, not just as a conjecture) and Eisermann's theorem have been extended by the author and B. Martelli \cite{Carrega_Tait1, Carrega_Taitg, Carrega-Martelli} in several directions by using the technology of Turaev's shadows and the shadow formula for the Kauffman bracket.

In the $19^{\rm  th}$ century, during his attempt to tabulate all knots in $S^3$, P.G. Tait \cite{Tait} stated three conjectures about crossing number, alternating links and \emph{writhe number}. By ``the Tait conjecture'' we mean the one stating that \emph{alternating} \emph{reduced} diagrams of links in $S^3$ have the minimal number of crossings. As said before, the conjecture has been proved in 1987 by Thistlethwaite-Kauffman-Murasugi studying the Jones polynomial. In \cite{Carrega_Tait1} we proved the analogous result for alternating links in $S^1\times S^2$ giving a complete answer to this problem. In \cite{Carrega_Taitg} we extended the result to alternating links in the connected sum $\#_g(S^1\times S^2)$ of $g\geq 0$ copies of $S^1\times S^2$. In $S^1\times S^2$ and $\#_2(S^1\times S^2)$ the appropriate version of the conjecture is true for $\Z_2$-homologically trivial links, and the proof also uses the Jones polynomial. Unfortunately in the general case the method provides just a partial result and we are not able to say if the appropriate statement is true. For $\Z_2$-homologically non trivial links the appropriate version of the Tait conjecture is false.

Eisermann showed that the Jones polynomial of a $n$-component ribbon link $L\subset S^3$ is divided by the Jones polynomial of the trivial $n$-component link. The theorem has been improved by the author and Martelli \cite{Carrega-Martelli} extending its range of application from links in $S^3$ to colored knotted trivalent graphs in $\#_g(S^1\times S^2)$. The result is based on the \emph{order} at $q=A^2=i$ of the Kauffman bracket. This is an extension of the multiplicity of the Kauffman bracket in $q=A^2=i$ as a zero. In particular we showed that if the Kauffman bracket of a knot in $\#_g(S^1\times S^2)$ has a pole in $q=A^2=i$ of order $n$, the \emph{ribbon genus} of the knot is at least $\frac {n+1}2$. The result could be a tool to show that a slice link in $\#_g(S^1\times S^2)$ is not ribbon, namely that the (extended) \emph{slice-ribbon conjecture} is false.

\subsection*{Structure of the thesis}

$1.$ In the first chapter we talk about general proprieties of quantum invariants, in particular the Jones polynomial, skein spaces, the Kauffman bracket in $\#_g(S^1\times S^2)$ and the $SU(2)$-Reshetikhin-Turaev-Witten invariants, moreover we investigate the skein space of the 3-torus. We start giving some basic notions about knot theory, then we talk about the Jones polynomial for links in $S^3$ in the Kauffman version and we give a brief general overview on quantum invariants. After that, we start with skein theory and we give a brief survey about skein spaces (and skein modules). Then we talk about the skein vector space of the 3-torus showing a basis of $9$ elements. As said before, with skein theory we can define the Kauffman bracket in the connected sum $\#_g(S^1\times S^2)$ of $g$ copies of $S^1\times S^2$, and we dedicate a section to proprieties and examples of the Kauffman bracket in this general setting. We conclude the chapter introducing the $SU(2)$-Reshetikhin-Turaev-Witten invariants via skein theory.

$2.$ The second chapter is devoted to Turaev's shadows. We introduce them, we list some general theorems and some examples. There are moves that relates shadows
representing the same object and we talk also about them. Then we introduce the shadow formula for the Kauffman bracket in $\#_g(S^1\times S^2)$ and the $SU(2)$-Reshetikhin-Turaev-Witten invariants and we provide proofs of these formulas that are based on skein theory, and hence might be easier to understand than the ones presented by Turaev in the extremely general case he was interested in.

$3.$ The third chapter is devoted to topological and geometric applications of the quantum invariants. This is a survey where we present some famous applications (or conjectures) and we describe some little new ones. In particular we focus on: the Bullock's theorem about the character variety, the volume conjecture, the Chen-Yang's volume conjecture, the Tait conjecture, the Eisermann's theorem, the classification of \emph{rational 2-tangles} and a criterion for \emph{non sliceness} of \emph{Montesinos links}. In order to understand some applications we also dedicate a section to the notions of ``ribbon surface'', ``ribbon link'', ``slice link'' \ldots Although the aim of the chapter is to present topological applications of quantum invariants, in the last section we talk about something a little different. We present an application of Turaev's shadows that goes in favor of the slice-ribbon conjecture.

$4.$ In the fourth chapter we state and prove the Tait conjecture in $\#_g(S^1\times S^2)$ as extended by the author. We discuss all the hypothesis of the main theorems. Note that the case $g=1$ needs just some basic notions of skein theory while the general case needs more complicated tools like shadows. We also ask some open questions related to the problem.

$5.$ In the fifth chapter we talk about the extension of Eisermann's theorem. We state and prove the result and we list some examples and corollaries. We give other lower and upper bounds of the order at $q=A^2=i$ and we ask some open questions.

$6.$ In the latest chapter we classify the non \emph{H-split} links in $S^1\times S^2$ that are not contained in a 3-ball (Definition~\ref{defn:split_homotopic_genus} and Proposition~\ref{prop:split}, the really interesting links in $S^1\times S^2$) and have crossing number at most $3$ (Definition~\ref{defn:alt_cr_num}
), and we compute some invariants. In the list links are seen up to reflections with respect to a Heegaard torus and reflections with respect to the $S^1$ factor. Although we just consider links with crossing number at most $3$, interesting examples come out.

\subsection*{Original contributions}

Section~\ref{sec:3-torus}, Section~\ref{sec:Kauf_g}, Subsection~\ref{subsec:proof_sh_for_RTW}, Section~\ref{sec:Montesinos}, Section~\ref{sec:sh_rib_handle}, Chapter~\ref{chapter:Tait}, Section~\ref{sec:lower_bounds}, Section~\ref{sec:upper_bounds} and Chapter~\ref{chapter:table} are original contributions of the author. Subsection~\ref{subsec:proof_sh_for_Kauf_br} and most of Chapter~\ref{chapter:Eisermann} are made in strict collaboration with Bruno Martelli. 

\subsection*{Acknowledgments}

The author is warmly grateful to Bruno Martelli for his constant support and encouragement.

\chapter{Quantum invariants and skein theory}\label{chapter:quantum_inv}

\pagenumbering{arabic}

\section{Preliminaries}

We will work in the smooth category. In particular for ``manifold'' we mean ``smooth manifold'', and for ``embedding we mean ``smooth embedding''. In dimension at most 4 the PL-theory is equivalent to the smooth one and we will always be in this situation. 

It is natural to have an ambient object and to study its sub-objects. Our main objects are low-dimensional manifolds (manifolds of dimension at most $4$), and the sub-objects that we study are mainly smooth sub-manifolds. In topology, it is customary to consider objects up to some transformations: the right transformations to consider here are the
\emph{isotopies}:
\begin{defn}
An \emph{isotopy} is a smooth map $N\times [-1,1] \rightarrow M$ from the product of a manifold $N$ and the interval $[-1,1]$ to the ambient manifold $M$ such that at each time $t$ the restriction $N \times \{t\} \rightarrow M$ is an embedding. An \emph{isotopy} between two embeddings $e_-,e_+:N \rightarrow M$ is an isotopy $N\times [-1,1] \rightarrow M$ such that the restriction $ N \cong N\times \{-1\}\rightarrow M$ to the time $-1$, is the first embedding $e_-$, and the restriction $N \cong N\times\{1\} \rightarrow M$ to the time $1$, is the second one $e_+$. The embeddings $e_-$ and $e_+$ are said to be \emph{isotopic}.

An \emph{ambient isotopy} of a manifold $M$ is a map $M\times [-1,1] \rightarrow M$ that is an isotopy where the restriction to $M \times \{-1\}\cong M$ is the identity of $M$. 
\end{defn}

\begin{defn}
A manifold is said to be \emph{closed} if it is compact and without boundary. 
\end{defn}

\begin{theo}[Thom]
Let $e_-,e_+:N\rightarrow M$ be two embeddings. If $N$ is compact, $e_-$ and $e_+$ are isotopic if and only if there is an ambient isotopy $\Phi : M\times [-1,1] \rightarrow M$ of $M$ that sends $e_-$ to $e_+$, that is $\Phi|_{N\times \{-1\}} =e_-$ and $\Phi|_{N\times \{1\}} =e_+$.
\end{theo}

\begin{defn}
Let $M$ be a compact smooth oriented manifold of dimension 3. A \emph{link} $L$ of $M$ is a closed 1-sub-manifold of $M$ considered up to isotopies. If the link has only one component, it is said to be a \emph{knot} (see for instance Fig.~\ref{figure:ex_links}). The \emph{unknot} is the only knot that bounds an embedded 2-disk, while the $k$-\emph{component unlink} is the only link that bounds $k$ disjoint 2-disks.
\end{defn}

\begin{figure}[htbp]
$$
\includegraphics[scale=0.65]{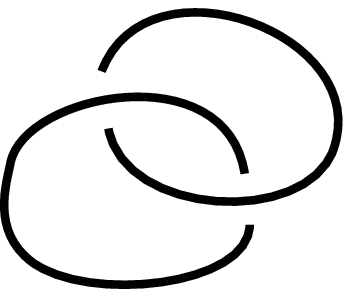} \hspace{2cm} \includegraphics[scale=0.65]{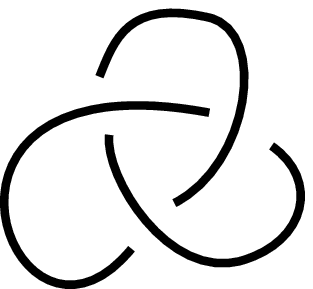}
$$
\caption{Two links in $S^3$: the Hopf link (left) and the trefoil knot (right).}
\label{figure:ex_links}
\end{figure}

\section{Jones polynomial and quantum invariants}

\subsection{Jones polynomial}\label{subsec:Jones_pol}

The classic ambient manifold in knot theory is the 3-sphere $S^3$. One can switch from $S^3$ to $\mathbb{R}^3$ and back by removing/adding a point, so the two ambient spaces make no essential difference.

A 4-valent graph is a graph such that every vertex locally has four edges adjacent to it, counted with multiplicities. Every link in $S^3$ can be represented by a 4-valent graph in the plane where each vertex is equipped with the information of which strand overpasses and which underpasses. Such a graph is called a \emph{diagram} of the link. 

There are three important moves that modify a diagram without altering the represented link: the Reidemeister moves (see Fig.~\ref{figure:Reid}).

\begin{figure}[htbp]
\begin{center}
\begin{tabular}{c|c|c}
 & & \\
\pic{1.5}{0.4}{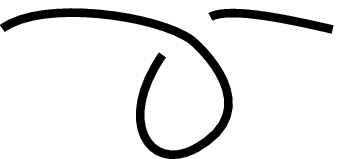} $\stackrel{I}{\leftrightarrow}$  \pic{1.5}{0.4}{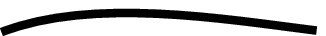} & \pic{1.5}{0.4}{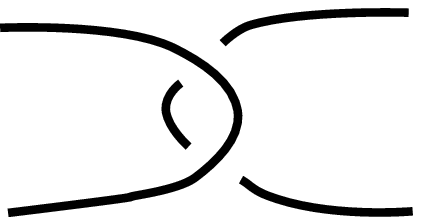} $\stackrel{II}{\leftrightarrow}$ \pic{1.5}{0.4}{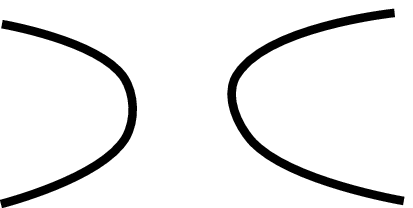} & \pic{1.5}{0.4}{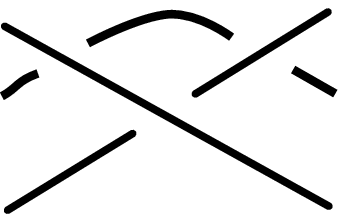} $\stackrel{III}{\leftrightarrow}$ \pic{1.5}{0.4}{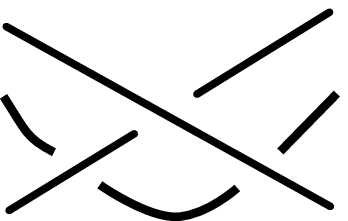} \\
 & & \\
First type & Second type & Third type
\end{tabular}
\end{center}
\caption{The Reidemeister moves.}
\label{figure:Reid}
\end{figure}

\begin{theo}[Reidemeister]
Two diagrams represent the same link in $S^3$ if and only if they are related by a sequence of planar isotopies and Reidemeister moves.
\end{theo}

A common way to construct invariants for links in $S^3$ consists in defining a fuction on the set of link diagrams that is invariant under application of planar isotopies and Reidemeister moves.

Links can be equipped with two further structures: \emph{orientation} and \emph{framing}. The first one is very natural and the second one allows to represent closed 3-manifolds with links via Dehn surgery (see Theorem~\ref{theorem:Lickorish-Wallace}).

An \emph{oriented link} in a 3-manifold is a closed oriented 1-sub-manifold considered up to isotopies that respect the orientation. There is a natural version of the Reidemeister theorem for oriented links in $S^3$. 

A \emph{framed link} is a closed 1-sub-manifold with a \emph{framing} considered up to isotopies that respect this structure. A \emph{framing} can be defined as a finite collection of annuli in the ambient manifold such that one of the two boundary components of each annulus is a component of the link and each component of the link touches only one annulus. Equivalently we can say that the core $S^1\times \{\frac 1 2\}$ of each annulus component is a component of the link. By ``\emph{annulus}'' we mean the product $S^1\times [0,1]$ of a circle and an interval. 

Every diagram represents a framed link in $S^3$ via the \emph{black-board framing}, and every framed link in $S^3$ can be represented so. The framing is encoded by the diagram as follows: give an orientation to the link, draw a parallel line to the diagram on the left (or the right) following the orientation getting a new link ``parallel'' to the first one that is the union of all the boundary components of the framing, then remove the orientation. The application of a Reidemester move of the first type changes the framing by adding or removing a twist (see Fig.~\ref{figure:framing_change}).
\begin{figure}[htbp]
$$
\pic{1.8}{0.4}{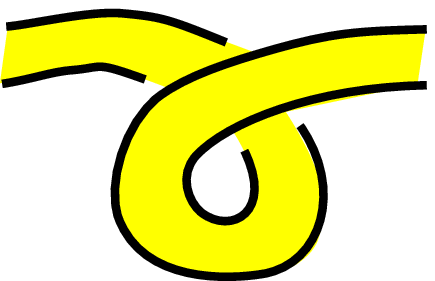} \leftrightarrow \pic{1.8}{0.4}{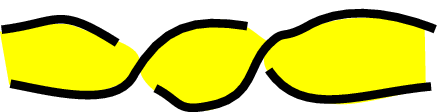} \leftrightarrow \pic{1.8}{0.4}{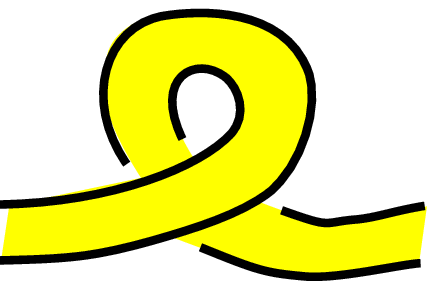} 
$$
\caption{A negative twist. The Reidemeister moves of the first type change the framing adding a twist.}
\label{figure:framing_change}
\end{figure}

If the ambient 3-manifold $M$ is oriented, the operations of adding a positive or a negative twist
to a component of a framed link are well-defined. Every framing can be obtained from an initial one by repeating these moves, namely by adding $n\in \Z$ positive twists, where if $n<0$ we mean adding $-n$ negative twists. These moves correspond to adding curls to a representative diagram. The addiction of a negative twist is described in Fig.~\ref{figure:framing_change}, while adding a positive twist is the opposite of that figure. 

\begin{theo}\label{theorem:framed_Reid}
Two diagrams represent the same framed link (via the black-board framing) in $S^3$ if and only if they are related by a sequence of planar isotopies, Reidemeister moves of the second and third type, and the following changing on curls:
\beq
\pic{1.8}{0.4}{ricciolopos.eps} & \leftrightarrow & \pic{1.8}{0.4}{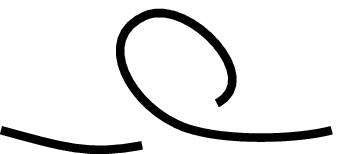} \\
\pic{1.8}{0.4}{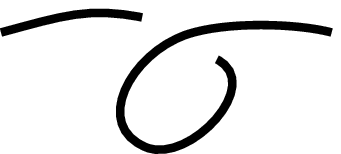} & \leftrightarrow & \pic{1.8}{0.4}{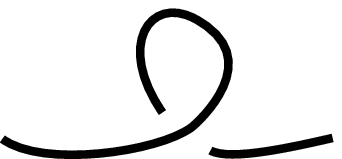} .
\eeq
\end{theo}

Unless specified otherwise, we will always tacitly represent framed links using the black-board framing.

A famous invariant of oriented links and unoriented knots is the \emph{Jones polynomial} \cite{Jones}. An easy way to define uses the notion of ``\emph{Kauffman bracket}'' that is an invariant for framed links \cite{Kauffman}.

\begin{defn}
A \emph{Laurent polynomial} $f\in R[X,X^{-1}]$ is a ``polynomial'' that may have negative exponents.

The \emph{Kauffman bracket}, $\langle L \rangle \in \Z[A,A^{-1}]$, of a framed link $L\subset S^3$ is the Laurent polynomial with integer coefficients and variable $A$ that is completely defined by the \emph{skein relations}:
\beq
\left\langle \pic{1.2}{0.3}{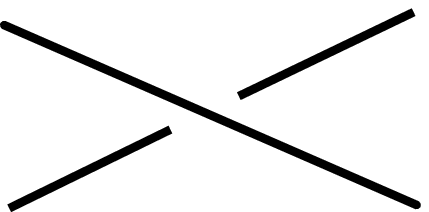} \right\rangle & = & A\left\langle \pic{1.2}{0.3}{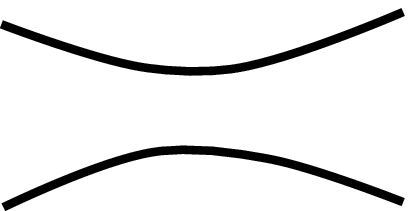} \right\rangle + A^{-1} \left\langle \pic{1.2}{0.3}{Bcanalep.eps} \right\rangle \\
\left\langle D \sqcup \pic{0.8}{0.3}{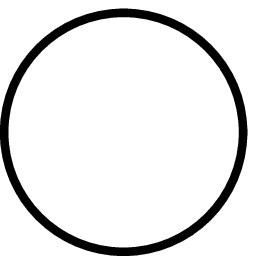} \right\rangle & = & (-A^2 - A^{-2}) \langle D \rangle \\
\left\langle \pic{0.8}{0.3}{banp.eps} \right\rangle & = & 1 
\eeq
where the diagrams in the first equation differ only in the portion drawn, in the second equation the diagrams differ by the addiction or the removal of a disjoint circle. Although this is the original definition, we are going to use a different normalization by imposing
$$
\left\langle \pic{0.8}{0.3}{banp.eps} \right\rangle = -A^2-A^{-2} .
$$
Clearly to get this normalization from the previous one it suffices to multiply $\langle L \rangle$ by $-A^2-A^{-2}$.
\end{defn}
The variables $q=A^2$ or $t=A^{-4}$ are often used instead of $A$. The Jones polynomial with the variable $q$ is still a Laurent polynomial, while if we use the variable $t$ we allow also half-integer exponents $\Z[t^{\frac{1}{2}},t^{-\frac{1}{2}}]$, but the Jones polynomial of links with an odd number of components is always a Laurent polynomial.

It is easy to check the behavior of the Kauffman bracket under the application of a Reidemester move: it is unchanged under Reidemeister moves of the second and third type while it changes by the multiplication of $-A^3$ or $-A^{-3}$ if we apply a move of the first type
\beq
\left\langle \pic{1.2}{0.3}{ricciolopos.eps} \right\rangle & = & -A^3 \left\langle \pic{1.2}{0.3}{riga.eps} \right\rangle \\
\left\langle \pic{1.2}{0.3}{riccioloneg.eps} \right\rangle & = & -A^{-3} \left\langle \pic{1.2}{0.3}{riga.eps} \right\rangle .
\eeq

If the link $L$ is oriented, each crossing of a diagram of $L$ has a sign as shown in Fig.~\ref{figure:crossingsign}.

\begin{figure}[htbp]
\begin{center}
\includegraphics[scale=0.8]{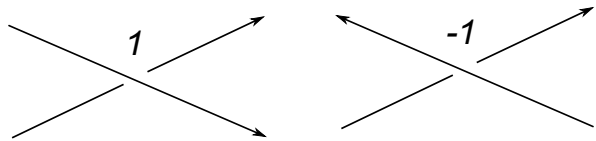}
\end{center}
\caption{A positive and a negative crossing.}
\label{figure:crossingsign}
\end{figure}
If the represented link has only one component (it is a knot) its crossing signs do not depend on the orientation. 
\begin{defn}\label{defn:writhe}
The \emph{writhe number} $w(D)$ of a diagram $D$ of an oriented link is the sum of the crossing signs. 
\end{defn}

The writhe number is an invariant for oriented framed links and unoriented framed knots.

\begin{defn}
The \emph{Jones polynomial} in the Kauffman version $f(L)$ of an oriented link $L\subset S^3$ is
$$
f(L):= (-A^3)^{-w(D)} \langle D\rangle ,
$$
where $D$ is any diagram of $L$.
\end{defn}
The proof of the invariance follows easily by studying the behavior of the Kauffman bracket under the Reidemeister moves.

Every compact surface with boundary in $S^3$ determines a framing on its boundary link just by taking a collar.
\begin{theo}
All the oriented surfaces in $S^3$ bounded by an oriented link $L$ such that the orientation of $L$ is the induced one to the boundary, determine the same framing on $L$. This is called the \emph{Seifert framing} or the 0-\emph{framing}. Furthermore the Jones polynomial of $L$ is equal to the Kauffman bracket of $L$ equipped with the Seifert framing.
\end{theo}

We use the standard orientation of $S^3$ to get a bijection between the set of framing over a knot in $S^3$ and the integers: the framing $n\in \Z$ is obtained by adding $n$ positive twists to the Seifert framing (if $n$ is negative we add $-n$ negative twists). This number coincides with the writhe number of a diagram representing the framed knot via the black-board framing.

\subsection{A brief overview about quantum invariants}

There are several ways to define the Jones polynomial. Jones defined it in 1984 \cite{Jones} by using von Newman algebras. His
discovery unveiled some unexpected connections between algebra,
topology, and theoretical physics. Following Jones' initial discovery a variety of knots and 3-manifolds invariants came out. These are called \emph{quantum invariants}. It is a fundamental target in modern knot theory to ``understand'' the Jones polynomial, that means finding relations between quantum invariants (in particular the Jones polynomial) and geometric and topological proprieties of manifolds (\emph{e.g.} hyperbolic volume, embedded surfaces, decompositions). There are several types of invariants that generalize, extend or deform the Jones polynomial: complex numbers, Laurent polynomials, rational functions, vector spaces (or $R$-modules, \emph{e.g.} \emph{skein spaces} or \emph{categorifications} of the previous ones like \emph{Khovanov homology}). All of them can be called quantum invariants, but this expression is typically employed just for numbers, polynomials, or more general rational functions, especially the ones coming from \emph{quantum groups}. It is recognized that quantum invariants have a lot of connections with several areas of mathematics and theoretical physics.

Quantum invariants arise from representations of braid groups. The images of the generators are examples of $R$-matrices, which play an important role in solving statistical mechanical models and quantum integrable systems in two dimensions. By the end of the 80's to discover new $R$-matrices Jimbo, Drinfel and others, developed the formalisms of \emph{quantum groups} (or \emph{quantum universal enveloping algebras}), which are deformations $U_q(\mathfrak{g})$ of the universal enveloping algebras $U(\mathfrak{g})$ of a semi-simple complex Lie algebras $\mathfrak{g}$ (see for instance \cite{Chari-Pressley, Kassel}). This theory is a part of a more general and categorical theory, where representations of quantum groups form non trivial examples of \emph{ribbon categories} (see for instance \cite{Turaev}).

The most simple and studied quantum group is $U_q(\mathfrak{sl}_2)$. Although it is the simplest case, it is rather general and complicated. The \emph{indecomposable} representations correspond to the irreducible representations of a field. As for $U(\mathfrak{g})$, the finite-dimensional indecomposable representations of $U_q(\mathfrak{sl}_2)$ are in bijection with the positive integers. The $n^{\rm th}$ indecomposable representation (or the $(n+1)$-dimensional representation) defines the $n^{\rm th}$ \emph{colored Jones polynomial}. The first one is the classical Jones polynomial (up to normalizations and changes of variable).

From this framework of quantum groups, in particular from the fundamental representation of $SU(2)$, one can get the 3-manifold invariants called $SU(2)$-\emph{Reshetikhin-Turaev-Witten invariants} (see \cite{Turaev}). These form the first suggestion of a 3-manifold invariant and were given by Witten \cite{Witten} as a part of his quantum-fields-theoretic explanation of the origin of the Jones polynomial. They were then rigorously constructed by Reshetikhin and Turaev \cite{Reshetikhin-Turaev} via surgery presentations of 3-manifolds and quantum groups. 

\section{Skein theory}\label{sec:skein_theory}

An alternative approach for quantum invariants is provided by \emph{skein theory}. The word ``skein'' and the notion were introduced by Conway in 1970 for his model of the Alexander polynomial. This idea became really useful after the work of Kauffman which, as we saw, redefined the Jones polynomial in a very simple and combinatorial way. These simple combinatorial techniques allow us to reproduce all quantum invariants arising from the representations of $U_q(\mathfrak{sl}_2)$ without any reference to representation theory. This also leads to many interesting and quite easy computations. This skein method was used by Lickorish \cite{Lickorish1, Lickorish2, Lickorish3, Lickorish4}, Blanchet, Habegger, Masbaum and Vogel \cite{BHMV}, and Kauffman and Lins \cite{Kauffman-Lins}, to re-interpret and extend some of the methods above. In this section we give some notions of this theory. One can use as a reference \cite{Lickorish} or \cite{Kauffman-Lins}.

\subsection{Skein spaces}\label{subsec:sk_sp}
\emph{Skein spaces} (or \emph{skein modules}) are vector spaces (or modules over a ring) associated to low-dimensional oriented manifolds, in particular to 3-manifolds. These were introduced independently in 1988 by Turaev \cite{Turaev0} and in 1991 by Hoste-Przytycki \cite{HP0}. We can think of them as an attempt to get an algebraic topology for knots: they can be seen as homology spaces obtained using isotopy classes instead of homotopy or homology classes. In fact they are defined taking a vector space (a free module) generated by sub-objects (framed links) and then quotienting them by some relations. In this framework, the following questions arise naturally and are still open in
general:
\begin{quest}
$\ $
\begin{itemize}
\item{Are skein spaces (modules) computable?}
\item{How powerful are them to distinguish 3-manifolds and links?}
\item{Do the spaces (modules) reflect the topology/geometry of the 3-manifolds (\emph{e.g.} surfaces, geometric decomposition)?}
\item{Does this theory have a functorial aspect? Can it be extended to a functor from a category of cobordisms to the category of vector spaces (modules) and linear maps?}
\end{itemize}
\end{quest}

Skein spaces can also be seen as deformations of the ring of the $SL_2(\mathbb{C})$-character variety of the 3-manifold (see Section~\ref{sec:Bul_char_var}). Moreover they are useful to generalize the Kauffman bracket to manifolds other than $S^3$ and this is the aspect we are mostly interested in. We refer to \cite{Pr:survey} for this theory.

Let $M$ be an oriented 3-manifold, $R$ a commutative ring with unit and $A\in R$ an invertible element of $R$. Let $V$ be the abstract free $R$-module generated by all framed links in $M$ (considered up to isotopies) including the empty set $\varnothing$. 

\begin{defn}\label{defn:sk_space}
The $(R,A)$-\emph{Kauffman bracket skein module} of $M$, or the $R$-\emph{skein module}, or simply the \emph{KBSM}, is sometimes indicated with $KM(M;R,A)$, or $S_{2,\infty}(M;R,A)$, and is the quotient of $V$ by all the possible \emph{skein relations}:
\beq
 \pic{1.2}{0.3}{incrociop.eps}  & = & A \pic{1.2}{0.3}{Acanalep.eps}  + A^{-1}  \pic{1.2}{0.3}{Bcanalep.eps}  \\
 L \sqcup \pic{0.8}{0.3}{banp.eps}  & = & (-A^2 - A^{-2})  D  \\
\pic{0.8}{0.3}{banp.eps}  & = & (-A^2-A^{-2}) \varnothing 
\eeq
These are local relations where the framed links in an equation differ just in the pictured 3-ball that is equipped with a positive trivialization. An element of $KM(M;R,A)$ is called a \emph{skein} or a \emph{skein element}. If $M$ is the oriented $I$-bundle over a surface $S$ (this is $M=S\times [-1,1]$ if $S$ is oriented) we simply write $KM(S;R,A)$ and call it the \emph{skein module} of $S$. 

If the base ring $R$ is the ring $\Z[A,A^{-1}]$ of all the Laurent polynomials with integer coefficients and abstract variable $A$, we set
$$
KM(M) := KM(M;\Z[A,A^{-1},A]) .
$$

If the base ring $R$ is the field $\mathbb{Q}(A)$ of all rational functions with rational coefficients with abstract variable $A$ we call it the \emph{skein vector space} of $M$, or simply the \emph{skein space} of $M$, and we set
$$
K(M) := KM(M;\mathbb{Q}(A),A) .
$$
\end{defn}

We will concretely see just the base rings $\Z[A,A^{-1}]$, $\mathbb{C}$ with a root of unity $A$, and $\mathbb{Q}(A)$. We are more interested in this latest case $R=\mathbb{Q}(A)$.

\begin{rem}
In order to perform the first skein relation we need that the ambient manifold $M$ is oriented (and not just orientable). 
\end{rem}

\begin{rem}
We could modify our construction of skein modules taking a different set of generators (\emph{e.g.} oriented links), a different base ring, and different skein relations, for instance the ones giving the Conway polynomial or the HOMFLY polynomial. The result is another kind of skein module studied in literature \cite{Pr:survey} that we will not investigate here. 
\end{rem}

Here we list the exact structure of the skein module of some manifolds:

\begin{theo}[Bullock, Carrega, Dabkowski, Gilmer, Hoste, March\'e, Mroczkowski, Przytycki, Sikora, Le]\label{theorem:known_sk_mod}
$\ $
\begin{enumerate}

\item{
$$
KM(S^3;R,A) \cong R
$$ 
and it is generated by the empty set $\varnothing$.}

\item{Let $S$ be a surface, then $KM(S;R,A)$ is the free $R$-module generated by all the multicurves of $S$, namely the embedded closed 1-sub-manifolds (up to isotopies) without homotopically trivial components, including the empty set $\varnothing$.}

\item{Let $\mathbb{R}\mathbb{P}^k$ be the $k$-dimensional real projective space. Then
$$
KM(\mathbb{R}\mathbb{P}^3 ; R,A) = KM(\mathbb{R}\mathbb{P}^2 ; R,A) \cong R \oplus R .
$$}

\item{Let $L(p,q)$ be the $(p,q)$-lens space. If $p>1$, $KM(L(p,q);R,A)$ is the free $R$-module of rank $\left\lfloor \frac p q \right\rfloor + 1$, where $\lfloor x\rfloor$ is the integer part of $x$. A surgery presentation in $S^3$ of $L(p,q)$ is the unknot with surgery coefficient $-p/q$. The set $\{ \varnothing, x_1, \ldots , x_{\left\lfloor \frac p q \right\rfloor }\}$ is a basis of $KM(L(p,q))$, where $x_j$ is the link obtained taking $j$ meridians of the tubular neighborhood of the unknot (the boundary of $j$ disjoint 2-disks properly embedded in the solid torus with core the unknot).}

\item{
$$
KM(S^1\times S^2) \cong \Z[A,A^{-1}] \oplus \bigoplus_{k=1}^\infty \Z[A,A^{-1}] / (1-A^{2k+4})
$$
and the empty set $\varnothing$ generates the factor without torsion.}

\item{The $R$-skein module of the complement of the $(k,2)$-torus knot is free.}

\item{Let $M$ be the classical Whitehead manifold, then $KM(M)$ is infinitely generated, torsion free, but not free.}

\item{The $R$-skein module of the complement of a rational 2-bridge knot is free and a basis is described.}

\item{Let $K\subset S^3$ be a torus knot. Then
\begin{itemize}
\item{the $\mathbb{C}[A,A^{-1}]$-module $KM(S^3\setminus K;\mathbb{C}[A,A^{-1}], A)$ is free;}

\item{$\mathbb{C}[A,A^{-1}] \otimes_\mathbb{C} \mathcal{R}(S^3 \setminus K)$ (see Definition~\ref{defn:ring_char}) and $KM(M;\mathbb{C}[A,A^{-1}], A)$ are isomorphic $\mathbb{C}[A,A^{-1}]$-modules, where $\mathcal{R}(S^3 \setminus K)$ is the \emph{ring of characters} of  $\pi_1(S^3\setminus K)$ (see Definition~\ref{defn:ring_char});}
\item{$\mathcal{R}(S^3 \setminus K)$ and $KM(S^3\setminus K; \mathbb{C},-1)$ are isomorphic $\mathbb{C}$-algebras without nilpotent elements (see Theorem~\ref{theorem:Bullock}).}
\end{itemize}}

\item{Let $M$ be an oriented compact 3-manifold with commutative fundamental group (\emph{e.g.} $S^1\times S^2$, $S^1\times S^1\times S^1$, a lens space $L(p,q)$) then $\mathcal{R}(M)$ and $KM(M; \mathbb{C},-1)$ are isomorphic $\mathbb{C}$-algebras without nilpotent elements (see Theorem~\ref{theorem:Bullock}).}

\item{Let $F$ be an oriented compact surface (maybe with non empty boundary), then
\begin{itemize}
\item{$KM(F;R,A)$ is a central $R$-algebra over a ring of polynomials induced by the boundary $\partial F$;}
\item{$KM(F;R,A)$ has no non null zero-divisors;}
\item{$\mathcal{R}(F\times [-1,1])$ and $KM(F; \mathbb{C},-1)$ are isomorphic $\mathbb{C}$-algebras (see Theorem~\ref{theorem:Bullock}).}
\end{itemize}}

\item{Let $S_{(2)}$ be the 2-disk with two holes. Then $KM(S_{(2)}\times S^1; R,A)$ is a free $R$-module with infinite rank and a basis is described.}

\item{$$
KM(\mathbb{R}\mathbb{P}^3 \# \mathbb{R}\mathbb{P}^3) \cong \Z[A,A^{-1}] \oplus \Z[A,A^{-1}] \oplus \frac{\Z[A,A^{-1}][t]}{J}
$$
where $J$ is the sub-module of $\Z[A,A^{-1}][t]$ generated by the elements
\beq
& (A^{n+1} -A^{-n-1})(Q_n -1)- 2(A+A^{-1})\sum_{k=1}^{\frac n 2} A^{n+2-4k} \ \  n\in 2\Z, \ n>0 & \\
& (A^{n+1} -A^{-n-1})(Q_n -t)- 2t(A+A^{-1})\sum_{k=1}^{\frac{n-1}{2} } A^{n+1-4k} \ \ n\in 2\Z+1, \ n>1 &
\eeq
where $Q_0=1$, $Q_1=t$ and $Q_{n+2} = t Q_{n+1} -Q_n$.}

\item{Let $M$ be a prism manifold. Then $KM(M;R,A)$ is a finitely generated free $R$-module with rank bigger than $1$ and a basis is described.}

\item{The skein vector space $K(T^3)$ of the 3-torus $T^3=S^1\times S^1\times S^1$ has dimension $9$ and a basis is described.}

\end{enumerate}
\begin{proof}
See \cite{Pr:survey, Pr1} for $1.$, $2.$ and $3.$, otherwise see below for $3.$. See \cite{HP1} for $4.$. See \cite{HP2} for $5.$. Se \cite{Bullock0} for $6.$. See \cite{HP3} for $7.$. See \cite{ThangLe} for $8.$. See \cite{Marche1} for $9.$. See \cite{PS1, PS2} for $10.$. See \cite{Charles-Marche} for $11.$. See \cite{Mroczkowski-Dabkowski} for $12.$. See \cite{Mroczkowski1} for $13.$. See \cite{Mroczkowski2} for $14.$. See \cite{Carrega_3-torus, Gilmer} and Section~\ref{sec:3-torus} for $15.$.
\end{proof}
\end{theo}

Theorem~\ref{theorem:known_sk_mod}-(1.) can be associated to the work of Kauffman since it comes out just from the skein relations and Theorem~\ref{theorem:framed_Reid}.

The following are useful and quite easy (except the last one) proprieties of skein modules:

\begin{theo}\label{theorem:sk_mod_prop}
$\ $
\begin{enumerate}

\item{Let $M$ and $N$ be two oriented 3-manifolds. Every orientation preserving embedding $\iota: M \rightarrow N$ induces a linear map of $R$-modules $\iota_* : KM(M;R,A) \rightarrow KM(N;R,A)$.}

\item{If $N$ is obtained attaching a 3-handle to $M$ and $\iota:M \hookrightarrow N$ is the inclusion map, $\iota_*: KM(M;R,A) \rightarrow KM(N;R,A)$ is an isomorphism.}

\item{If $N$ is obtained attaching a 2-handle to $M$ and $\iota:M \hookrightarrow N$ is the inclusion map, $\iota_*: KM(M;R,A) \rightarrow KM(N;R,A)$ is surjective.}

\item{The skein module of the disjoint union is the tensor product of the skein modules
$$
KM(M\sqcup N;R,A) \cong KM(M;R,A) \otimes_R KM(N;R,A) .
$$}

\item{(Universal Coefficient Property) Let $f:\Z[A,A^{-1}] \rightarrow R$ be a homomorphism of rings (commutative with unit). The map $f$ gives to $R$ a structure of $\Z[A,A^{-1}]$ module, and the identity of the set of framed links induces an isomorphism of $\Z[A,A^{-1}]$-modules
$$
KM(M) \otimes_{\Z[A,A^{-1}]} R  \rightarrow KM(M;R, f(A)) .
$$}

\item{Let $N$ be an oriented 3-manifold obtained attaching a 2-handle along a curve $\gamma \subset \partial M$ in the boundary of the 3-manifold $M$. Then
$$
KM(N;R,A) = \frac{KM(M;R,A)}{J} ,
$$
where $J$ is the sub-module of $KM(M;R,A)$ generated by all the skeins of the form $L -{\rm sl}_\gamma(L)$ where $L$ is any framed link of $M$ and ${\rm sl}_\gamma (L)$ is a framed link of $M$ obtained from $L$ with a slide along $\gamma$ (along the 2-handle attached along $\gamma$) (of course there are many ways to slide along a curve).}

\item{The skein vector space of the connected sum is the tensor product of the vector spaces 
$$
K(M\#N) \cong K(M)\otimes_{\mathbb{Q}(A)} K(N) .
$$}
\end{enumerate}
\begin{proof}
See \cite[Proposition {\rm IX}.6.2]{Pr:survey} and \cite{Pr2}.
\end{proof}
\end{theo}

\begin{rem}
Every compact 3-manifold $M$ is obtained attaching 2- and 3-handles to an orientable handlebody $H$ (a manifold with a handle-decomposition with just 0- and 1-handles). The handlebody is the thickening of the 2-disk with $g$ holes $S_{(g)}$. Therefore by Theorem~\ref{theorem:sk_mod_prop}-(2.), Theorem~\ref{theorem:sk_mod_prop}-(6.) and Theorem~\ref{theorem:known_sk_mod}-(2.) we have that the skein module $KM(M;R,A)$ of $M$ is generated by the multicurves of the punctured disk $S_{(g)}$ (including the empty set) and is defined by the relations in $KM(S_{(g)};R,A)$ ($=KM(H;R,A)$) $L - {\rm sl}_\gamma (L)$, where $L$ runs over all the framed links of $H$ (or the multicurves of $S_{(g)}$), ${\rm sl}_\gamma (L)$ is a framed link obtained from $L$ with a slide along $\gamma$, and $\gamma$ runs over all the curves that define the 2-handles.
\end{rem}

\begin{proof}[Proof of Theorem~\ref{theorem:known_sk_mod}-(3.)]
The complement of an open 3-ball of $\mathbb{R}\mathbb{P}^3$ is the orientable $I$-bundle over $\mathbb{R}\mathbb{P}^2$. Hence by Theorem~\ref{theorem:sk_mod_prop}-(2.) $KM(\mathbb{R}\mathbb{P}^3; R,A)=KM(\mathbb{R}\mathbb{P}^2;R,A)$. The multicurves of $\mathbb{R}\mathbb{P}^2$ are just the empty set and the orientation reversing curve, therefore by Theorem~\ref{theorem:known_sk_mod}-(2.) $KM(\mathbb{R}\mathbb{P}^3;R,A) = KM(\mathbb{R}\mathbb{P}^2;R,A)$ is free and is generated by those two elements.
\end{proof}

\begin{rem}\label{rem:dec_sk_sp}
The skein relations transform a framed link $L$ into a combination of framed links representing the same $\Z_2$-homology class of $L$ ($\Z_2= \Z / 2\Z$). Therefore the skein module has a decomposition in direct sum where the index $h$ varies over the first homology group $H_1(M;\Z_2)$ with coefficients in $\Z_2$:
$$
KM(M;R,A) = \bigoplus_{h\in H_1(M; \Z_2)} KM_h(M;R,A) ,
$$
where $KM_h(M;R,A)$ is the sub-module generated by all the framed links with $\Z_2$-homology class equal to $h$.
\end{rem}

\begin{rem}\label{rem:sk_sp_S1xS2}
By the universal coefficient property (Theorem~\ref{theorem:sk_mod_prop}-(5.)) with $f:\Z[A,A^{-1}] \rightarrow \mathbb{Q}(A)$ the inclusion map, we have
$$
K(M) := KM(M;\mathbb{Q}(A) , A) \cong KM(M) \otimes_{\Z[A,A^{-1}]} \mathbb{Q}(A) .
$$
Hence if $KM(M)$ is the direct sum of a torsion part and a free part, $K(M)$ is equal to the free part tensored with $\mathbb{Q}(A)$. In particular in that case the rank of $KM(M)$ is equal to the dimension of $K(M)$. Therefore by Theorem~\ref{theorem:known_sk_mod}-(5.) we have that the skein space of $S^1\times S^2$ is isomorphic to the base field $\mathbb{Q}(A)$ and is generated by the empty set. Adding Theorem~\ref{theorem:sk_mod_prop}-(7.) we get the same result for the connected sum $\#_g(S^1\times S^2)$ of $g\geq 0$ copies of $S^1\times S^2$ ($g=0$ means $S^3$ and $g=1$ means $S^1\times S^2$)
$$
K(\#_g(S^1\times S^2)) \cong \mathbb{Q}(A) . 
$$
By using the fact that $K(S^3)=\mathbb{Q}(A)$ is generated by the empty set and Remark~\ref{rem:graph_and_sphere} on a non separating set of $g$ 2-spheres in $\#_g(S^1\times S^2)$ we can easily prove that every skein in $K(\#_g(S^1\times S^2))$ is a multiple of the empty set $\varnothing$ \cite[Proposition 1]{FK1}. Unfortunately the fact that $\varnothing$ is non zero is non trivial.
\end{rem}

Thanks to result of Hoste-Przytycki (Remark~\ref{rem:sk_sp_S1xS2}) and (with different techniques) to Costantino \cite{Costantino2}, now we can define the Kauffman bracket also in connected sums of copies of $S^1\times S^2$:

\begin{defn}\label{defn:Kauf}
Let $M$ be a 3-manifold with skein space isomorphic to $\mathbb{Q}(A)$ and generated by the empty set $\varnothing$. The \emph{Kauffman bracket} $\langle S \rangle \in \mathbb{Q}(A)$ of a skein element $S\in K(M)$ is the unique coefficient that multiply the empty set to get $S$, $S= \langle S \rangle \cdot \varnothing$.
\end{defn}

\begin{quest}
Are there 3-manifolds different from $\#_g(S^1\times S^2)$ where we can define the Kauffman bracket, namely with skein space isomorphic to $\mathbb{Q}(A)$ and generated by the empty set? 
\end{quest}

\begin{prop}\label{prop:diffeo}
Let $\varphi: M \rightarrow N$ be a diffeomorphism of the oriented 3-manifolds $M$ and $N$. Then
\begin{enumerate}
\item[1.]{if $\varphi$ is orientation preserving it induces a natural isomorphism of $\Z[A,A^{-1}]$-modules $\varphi_* : KM(M)\rightarrow KM(N)$ and if we have a relation $\sum_i \lambda_i L_i =0$ in $KM(M)$ ($\lambda_i \in \Z[A,A^{-1}]$ and $L_i$ is a framed link), the relation $\sum_i \lambda_i \varphi(L_i) =0$ holds in $KM(N)$;}
\item[2.]{if $\varphi$ is orientation reversing it induces a natural \emph{anti-linear} isomorphism of $\Z[A,A^{-1}]$-modules $\varphi_* : KM(M)\rightarrow KM(N)$ and if we have a relation $\sum_i \lambda_i L_i =0$ in $KM(M)$ ($\lambda_i \in \Z[A,A^{-1}]$ and $L_i$ is a framed link), the relation $\sum_i \lambda_i|_{A^{-1}} \varphi(L_i) =0$ holds in $KM(N)$.}
\end{enumerate}
By ``\emph{anti-linear isomorphism}'' we mean a bijective map such that for every $S,S'\in KM(M)$ and $\lambda \in \Z[A,A^{-1}]$, $\varphi_*(S + S') = \varphi_*(S) + \varphi_*(S')$ and $\varphi_*(\lambda S) = \bar \lambda \varphi_*(S)$, where $\bar \lambda$ is the image of $\lambda$ under the involution of $\Z[A,A^{-1}]$ that is the isomorphism of $\Z$-modules defined by $A \mapsto A^{-1}$.

The statements holds also for skein vector spaces. Moreover if $K(M)\cong \mathbb{Q}(A)$ and is generated just by the empty set we get 
\begin{enumerate}
\item[3.]{if $\varphi$ is orientation preserving the Kauffman bracket is preserved $\langle \varphi(L) \rangle = \langle L \rangle$;} 
\item[4.]{if $\varphi$ is orientation reversing we have that $\langle \varphi (L) \rangle = \langle L \rangle|_{A^{-1}}$, where $\langle L \rangle|_{A^{-1}}$ is the composition of the Kauffman bracket $\langle L \rangle$ and $A \mapsto A^{-1}$.}
\end{enumerate}
\begin{proof}
The map $\varphi$ defines a linear isomorphism $V(M) \rightarrow V(N)$ between the free modules $V(M)$ and $V(N)$ generated by the framed links in $M$ and $N$. Select a 3-ball $B$ in $M$ with a positive oriented parametrization. Consider a skein relation that modifies $L$ in $B$. If a link $L\subset M$ is in such position, its image $\varphi(L)\subset N$ is either in the same position with respect to $\varphi(B)\subset N$, or in the \emph{mirror image} of that position. By ``\emph{mirror image}'' we mean the image of the map $(x,y,z) \mapsto (x,y,-z)$ in the trivialization of the 3-ball. These cases happen respectively when $\varphi$ is orientation preserving and orientation reversing. Therefore if $\varphi$ is orientation preserving the skein relations are preserved and we get the first statement.

If $\varphi$ is orientation reversing we have
$$
\begin{array}{rcccl}
 \varphi \left( \pic{1.2}{0.3}{incrociop.eps} \right) & = & \pic{1.2}{0.3}{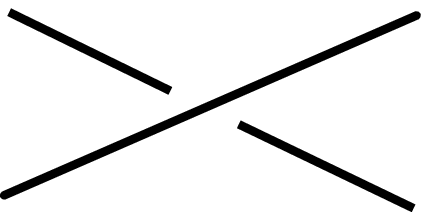}  & = & A^{-1} \pic{1.2}{0.3}{Acanalep.eps}  + A  \pic{1.2}{0.3}{Bcanalep.eps}  \\
\varphi\left( D \sqcup \pic{0.8}{0.3}{banp.eps} \right) & = & \varphi(D) \sqcup \pic{0.8}{0.3}{banp.eps}  & = & (-A^2 - A^{-2})  \varphi(D)  \\
 \varphi\left( \pic{0.8}{0.3}{banp.eps} \right) &=& \pic{0.8}{0.3}{banp.eps} & = & (-A^2-A^{-2}) \varnothing
\end{array} .
$$
Therefore we obtain the skein relations with $A$ replaced by $A^{-1}$. Therefore we get the second statement.

Clearly all these considerations work also for skein vector spaces. The empty set is sent to it self $\varphi_*(\varnothing) = \varnothing$. Therefore we get the last two statements.
\end{proof}
\end{prop}

\begin{rem}
Let $M$ be a 3-manifold with skein space $K(M)\cong \mathbb{Q}(A)$ generated by the empty set. There might be framed links $L, L' \subset M$ that are related by an orientation preserving diffeomorphism of $M$, $L'= \varphi(L)$ but are not isotopic $L\neq L'$ (this can not happen if $M=S^3$). By Proposition~\ref{prop:diffeo}-(3.) we can not distinguish them with the Kauffman bracket.
\end{rem}

\begin{rem}\label{rem:tensor}
There is an obvious canonical linear map $K(M) \to K(M\# N)$ defined by considering a skein in $M$ inside $M\# N$. The linear map $K(\#_g(S^2\times S^1)) \to K(\#_{g+1}(S^2\times S^1))$ sends $\varnothing$ to $\varnothing$ and hence preserves the bracket $\langle S \rangle$ of a skein $S \in K(\#_g(S^2\times S^1))$.

This shows in particular that if $S$ is contained in a ball, the bracket $\langle S \rangle$ is the same we would obtain by considering $S$ inside $S^3$.
\end{rem}

\begin{quest}
Is the skein of the empty set $\varnothing \in KM(M)$ always non null?
\end{quest}

\begin{quest}
Is the skein vector space $K(M)$ of every closed oriented 3-manifold $M$ finite dimensional?
\end{quest}

\begin{quest}
Is the skein module of the complement of any non trivial knot $K \subset S^3$, $K\neq \pic{0.8}{0.3}{banp.eps}$, a finitely generated free module?
\end{quest}

\begin{conj}[Conjecture 4.3 \cite{Othsuki}]
Let $M$ be a compact oriented 3-manifold with non empty boundary. If every closed incombressible surface in $M$ is parallel to the boundary of $M$ then the skein module $KM(M)$ is without torsion.
\end{conj}

\subsection{Temperley-Lieb algebras}

For a non negative integer $n$, the $n^{\rm th}$ \emph{Temperley-Lieb algebra} $TL_n$ is a quite famous finitely generated $\mathbb{Q}(A)$-algebra. We can think of it as a relative version of the skein space of the 3-cube. The cube has $n$ fixed points on the top side and $n$ on the bottom one and the generators are \emph{framed tangles}.

\begin{defn}\label{defn:tangle}
A $n$-\emph{tangle} $T$ is the equivalence class of $n$ disjoint properly embedded arcs and some circles in the closed 3-cube $D^3$ under the relation of isotopy pointwise fixing the boundary $\partial D^3$.

Tangles are represented by diagrams in a square with $n$ points on the top side and $n$ points on the bottom side.

There is a trivial way to get a link in $S^3$ from a tangle: identify the top side of the cube with the bottom one getting a link in the solid torus, then embed the solid torus in $S^3$ in the standard way. The resulting link in $S^3$ is called the \emph{closure} of the tangle. We get a diagram of the closure of a tangle by identifying the top and the bottom of the ambient square and embedding the resulting annulus in the plane.
\end{defn}

\begin{defn}\label{defn:framed_tangle}
There is a notion of ``\emph{framing}'' on tangles, hence a notion of ``\emph{framed tangle}''. These are isotopy classes of disjoint embedded annuli and strips in the cube that collapse onto the components of the tangle, intersect the cube only in the top and the bottom side and the intersection coincides with the union of the top and bottom side of the strips. The admissible framings of the tangles are exactly the ones obtained by the diagrams via the black-board framing. 

There is also a natural way to get a framed link in $S^3$ from a framed tangle that is still called \emph{closure}. We get a diagram of the closure of a framed tangle just by identifying the top and bottom side of the square of a diagram of the framed tangle (that represents the framed link by the black-board framing), and embedding this annulus in the plane.
\end{defn}

\begin{defn}\label{defn:mult_tangle}
Given two (framed) $n$-tangles, $T_1$ and $T_2$, we can define their product $T_1\cdot T_2$ by gluing the bottom side of the cube with $T_1$ to the top side of the one with $T_2$. We get a diagram of $T_1 \cdot T_2$ by gluing the bottom side of the square with a diagram $D_1$ of $T_1$ with the top side of the square with a diagram $D_2$ of $T_2$.
\end{defn}

\begin{defn}
Let $n\geq 0$ and let $V$ be the $\mathbb{Q}(A)$-vector space generated by the framed $n$-tangles in the 3-cube. Quotient $V$ modulo the sub-space generated by the \emph{skein relations}: 
$$
\begin{array}{rcl}
 \pic{1.2}{0.3}{incrociop.eps}  & = & A \pic{1.2}{0.3}{Acanalep.eps}  + A^{-1}  \pic{1.2}{0.3}{Bcanalep.eps}  \\
 D \sqcup \pic{0.8}{0.3}{banp.eps}  & = & (-A^2 - A^{-2})  D  
\end{array} .
$$
The $n^{\rm th}$ \emph{Temperley-Lieb algebra} is the resulting vector space equipped with the multiplication defined by linear extension on the multiplication with $n$-tangles.
\end{defn}

There is a natural injective map $TL_n \rightarrow TL_{n+1}$, obtained by adding a straight line connecting the $(n+1)$'s points. We can identify $TL_n$ with its image. We also have a natural map $TL_n \rightarrow \mathbb{Q}(A)$ called \emph{closure} or \emph{trace}. To get this map we extend by linearity the map defined on tangles: take a $n$-framed tangle, identify the starting points with the corresponding end points, embed the obtained solid torus in $S^3$ in the standard way, we get a framed link in $S^3$ and take its Kauffman bracket. 

More generally, if we have a framed link $L$ in a 3-manifold $M$ and a 3-ball containing $n$ parallel strands of the link, we can change the skein defined by $L$ by putting a fixed element of $TL_n$ into the 3-ball replacing the parallel strands. This defines a linear map $TL_n \rightarrow K(M)$. The closure of an element of $TL_n$ corresponds to the case $L$ is the unlink in $S^3$.

The Temperley-Lieb algebra $TL_n$ is generated as a $\mathbb{Q}(A)$-algebra by the elements $1,e_1,\ldots , e_{n-1}$ shown in Fig.~\ref{figure:TLgen}.

\begin{figure}[htbp]
$$
1 = \pic{1.5}{0.5}{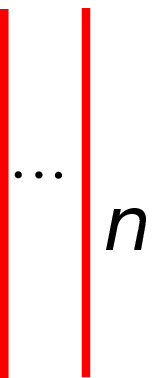} ,\ \ \ e_i = \pic{1.5}{0.5}{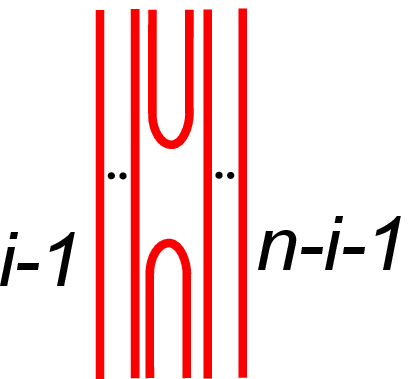}
$$
\caption{Standard generators for the algebra $TL_n$.}
\label{figure:TLgen}
\end{figure}

The generators $e_i$ satisfy the following algebraic relations:
\begin{itemize}
\item{$e_i^2 = (-A^2-A^{-2}) \cdot e_i$;}
\item{$e_i \cdot e_{i\pm1} \cdot e_i = e_i$;}
\item{$e_i \cdot e_j = e_j \cdot e_i$ for $|i-j|>1$.}
\end{itemize}

\subsection{Colors and trivalent graphs}
There is a standard way to color the components of a framed link $L\subset M$ with a natural number and get a skein element of $K(M)$. These colorings are based on particular elements of the Temperley-Lieb algebra called \emph{Jones-Wenzl projectors}. Coloring a component with $0$ is equivalent to remove it, while coloring with $1$ corresponds to consider the standard skein.

\begin{defn}
The \emph{Jones-Wenzl projectors} $f^{(n)}\in TL_n \subset TL_{n+1}$, $n\geq 0$, are particular elements of the Temperley-Lieb algebras. Usually the $n^{\rm th}$ projector is denoted with one or $n$ parallel straight lines covered by a white square box with a $n$ inside. They are defined by recursion:
\beq
f^{(0)} & := & \varnothing \\
f^{(1)} & := & 1_{TL_1} = \pic{1}{0.5}{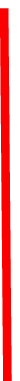} \\
\pic{1.2}{0.5}{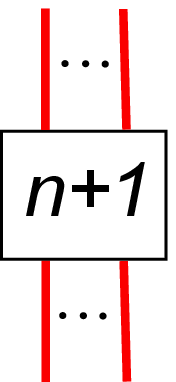} & := & \pic{1.2}{0.5}{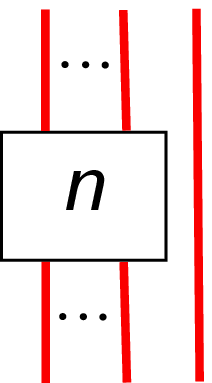} - \frac{\cerchio_{n-1}}{\cerchio_n} \cdot \pic{1.2}{0.5}{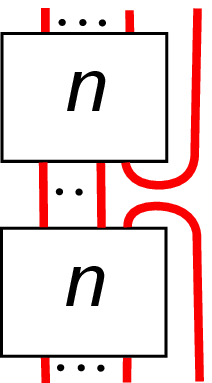}
\eeq
where $\cerchio_n = \Delta_n \in \mathbb{Q}(A)$ is the closure of $f^{(n)}$ and we have
$$
\cerchio_n = (-1)^n \frac{A^{2(n+1)} -A^{-2(n+1)}}{A^2-A^{-2}} .
$$
\end{defn}

The Jones-Wenzl projectors are completely determined by the following important properties:
\begin{itemize}
\item{$f^{(n)} \cdot e_i = e_i \cdot f^{(n)}=0$ for all $1\leq i \leq n-1$;}
\item{$(f^{(n)} -1)$ belongs on the sub-algebra generated by $e_1, \ldots , e_{n-1}$;}
\item{$f^{(n)} \cdot f^{(m)} = f^{(m)}$ for all $n\leq m$.}
\end{itemize}

Sometimes it is better to use the \emph{quantum integers} instead of the symbols $\cerchio_n$ or $\Delta_n$.
\begin{defn}
The $n^{\rm th}$ \emph{quantum integer} $[n]$, with $n>0$, is the rational function so defined:
$$
[n]: \ q \mapsto \frac{q^n -q^{-n}}{q-q^{-1}} = q^{-n+1} + q^{-n+3} + \ldots + q^{n-3} + q^{n-1}.
$$
\end{defn}
Note that the limit for $q\rightarrow 1$ of $[n]$ is the integer $n$.

We have that
$$
\cerchio_n = (-1)^n[n+1]|_{q=A^2} .
$$

\begin{defn}\label{defn:framed_graph}\label{defn:admissible}
A \emph{framed knotted trivalent graph} $G\subset M$ is a knotted trivalent graph equipped with a \emph{framing}, namely an orientable surface thickening the graph considered up to isotopies (see Fig.~\ref{figure:knotted_graph}). We admit also closed edges, namely knot components of $G$. These objects are often called \emph{ribbon graphs}, but we do not use this terminology here to avoid confusion with ribbon surfaces. 

A triple $(a,b,c)$ of natural numbers is \emph{admissible}
if the numbers satisfy the triangle inequalities $a\leq b+c$, $b\leq a+c$, $c\leq a+b$, and their sum $a+b+c$ is even. An \emph{admissible coloring} of $G$ is the assignment of a natural number (a \emph{color}) to each edge of $G$ such that the three numbers $a$, $b$, $c$ coloring the three edges incident to each vertex form an admissible triple. A \emph{colored trivalent graph} is a framed knotted trivalent graph with an admissible coloring. 
\end{defn}
\begin{figure}[htbp]
$$
\includegraphics[scale= 0.45]{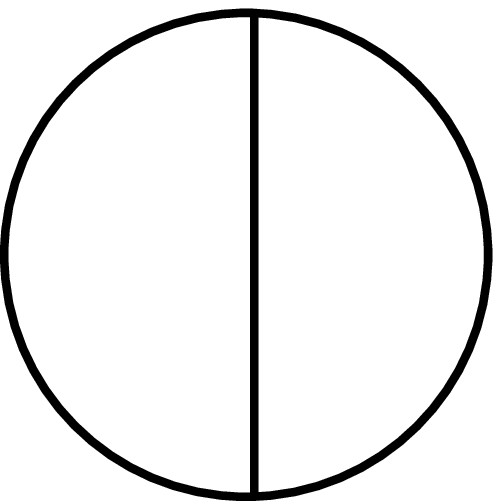} 
\hspace{1.5cm} 
\includegraphics[scale= 0.38]{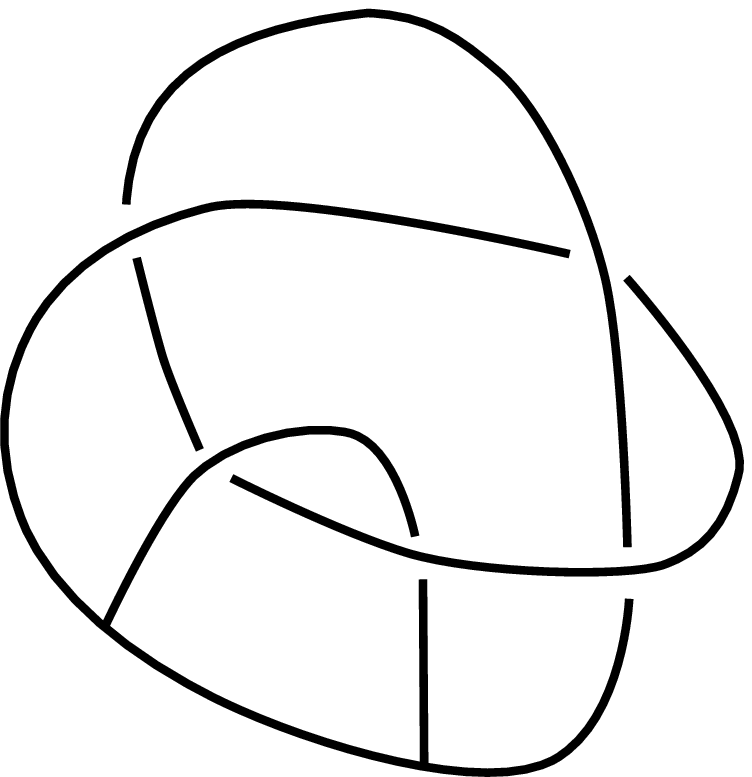}
$$
\caption{Two knotted trivalent graphs in $S^3$: a planar one (left) and one not holding in a plane (right).}
\label{figure:knotted_graph}
\end{figure}
Clearly a framed knotted trivalent graph without vertices is a framed link.

There is a standard way to define a skein element associated to a colored framed knotted trivalent graph. The admissibility requirements on colors allow to associate uniquely to the colored graph a linear combination of framed links by putting the $k^{\rm th}$ Jones-Wenzl projector at each edge colored with $k$ and by substituting vertices with bands as shown in Fig.~\ref{figure:trivalent_vertex}.

\begin{figure}[htbp]
\begin{center}
\includegraphics[width = 7 cm]{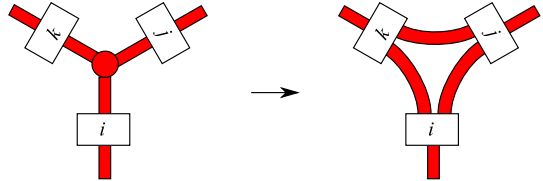}
\end{center}
\caption{A colored framed knotted trivalent graph determines a linear combination of framed links: replace every edge with a Jones-Wenzl projector, and connect them at every vertex via non intersecting strands contained in the depicted bands. For instance there are exactly $\frac{i+j-k}2$ bands connecting the projectors $i$ and $j$.}
\label{figure:trivalent_vertex}
\end{figure}

The identity in Fig.~\ref{figure:Cheb} holds in a solid torus and comes from the recursion formula for the Jones-Wenzl projectors. This is a global relation in the solid torus and works only for parallel (colored) copies of the core with trivial framing (the framing given by a fixed properly embedded annulus in the solid torus). This identity shows that we can define the skein of colored links in the skein module constructed with any base ring, in particular with $\Z[A,A^{-1}]$. 

\begin{figure}
\begin{center}
\includegraphics[width = 5.3 cm]{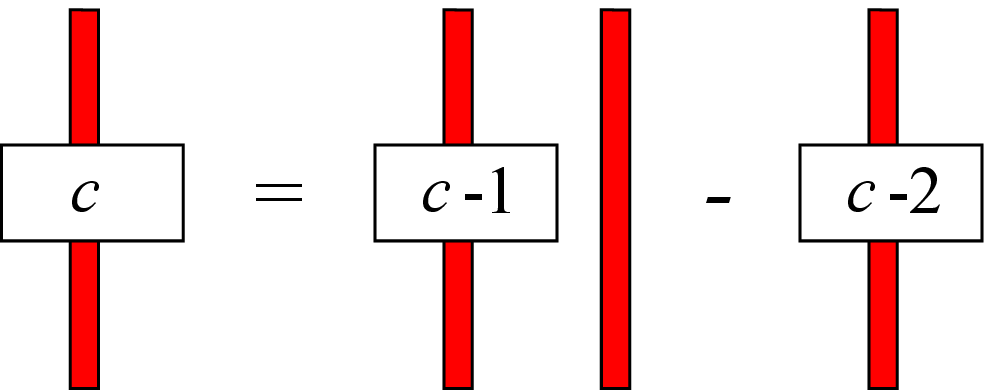}
\caption{A move in the skein space of the solid torus ($c \geq 2$). Only a portion of the solid torus is pictured, the portion is diffeomorphic to $[-1,1]\times D^2$.}
\label{figure:Cheb}
\end{center}
\end{figure}

We get infinitely many new invariants for framed links just by fixing a color, giving it to the link components, and taking the Kauffman bracket of that colored graph. This is (a renormalization of) the \emph{colored Jones polynomial} (see Section~\ref{sec:vol_conj}).

\subsection{Three basic colored trivalent graphs}\label{subsec:three_graphs}

Three basic planar colored framed trivalent graphs $\cerchio$, $\teta$, and $\tetra$ in $S^3$ are shown in Fig.~\ref{figure:graphs}. Their Kauffman brackets are some rational functions in $q=A^2$ that we now describe.

\begin{figure}[htbp]
\begin{center}
\includegraphics[width = 9.5 cm]{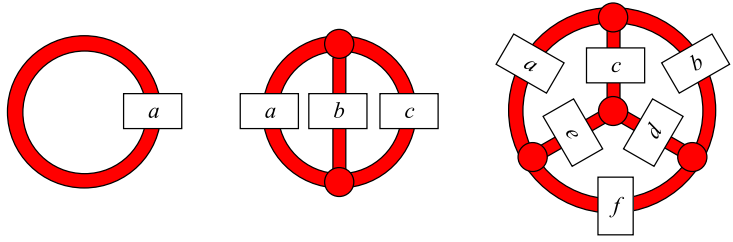}
 \end{center}
\caption{Three important planar colored framed trivalent graphs in $S^3$.}
\label{figure:graphs}
\end{figure}

We use the usual \emph{factorial notation}:
$$
[n]! := [2] \cdot [3]\cdot \ldots \cdot [n] ,
$$
with the convention $[0]! := 1$. Similarly we define the \emph{generalized multinomials}:
$$
\begin{bmatrix}
m_1, \ldots, m_h \\
n_1, \ldots n_k
\end{bmatrix} := \frac{[m_1]!\cdot \ldots \cdot [m_h]!}{[n_1]!\cdot \ldots \cdot [n_k]!} .
$$
When using these generalized multinomials we always suppose that
$$
m_1 + \ldots + m_h = n_1+ \ldots + n_k.
$$
The evaluations of $\cerchio$, $\teta$ and $\tetra$ are:
\beq
\cerchio_a  & = & (-1)^{a} [a+1], \\
\teta_{a,b,c} & = & (-1)^{\frac{a+b+c}2}
\begin{bmatrix} \frac{a+b+c}2+1, \frac{a+b-c}2, \frac{b+c-a}2, \frac{c+a-b}2 \\
a, b, c, 1 \end{bmatrix},  \\
\pic{2}{0.8}{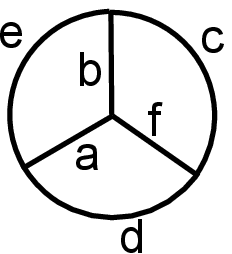} & = &
\begin{bmatrix} \Box_i-\triangle_j \\
a, b, c, d, e, f \end{bmatrix} \cdot \sum_{z = \max \triangle_j }^{\min \Box_i}\!\!\! (-1)^z
\begin{bmatrix} z+1 \\
z-\triangle_j, \Box_i-z, 1 \end{bmatrix}.
\eeq

Triangles and squares in the latter equality are defined as follows:
\beq
& \triangle_1 = \frac{a+b+c}{2},\ \triangle_2 = \frac{a+e+f}{2},\ \triangle_3 =\frac{ d+b+f}{2},\ \triangle_4 = \frac{d+e+c}{2}, & \\
& \Box_1 = \frac{a+b+d+e}{2},\ \Box_2 = \frac{a+c+d+f}{2},\ \Box_3 = \frac{b+c+e+f}{2}. &
\eeq

The indices in the formula vary as $1\leq i \leq 3$ and $1\leq j \leq 4$, so the term $\Box_i - \triangle_j$ indicates $3\cdot 4 = 12$ numbers.
The formula for $\tetra$ was first proved by Masbaum and Vogel \cite{MaVo}. These formulas are all rational functions in $q=A^2$.

\subsection{Important identities}\label{subsec:identities}

Here we introduce some important identities in skein spaces that we need. We can find their proofs in \cite{Lickorish}. The first identity is the following

\begin{equation}\label{eqn:first_id}
\pic{2}{0.7}{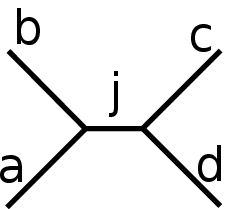} = \sum_i \left\{\begin{matrix} a & b & i \\ c & d & j \end{matrix}\right\} \pic{2}{0.7}{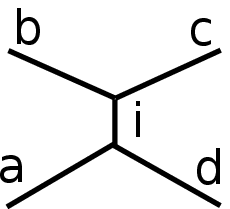}    \quad \quad \quad \text{\cite[Page 155]{Lickorish}},
\end{equation}
where the sum is taken over all $i$ such that the shown colored graphs are admissible. The coefficients between brackets are called $6j$-\emph{symbols}. We have
\begin{equation}\label{eqn:second_id}
\left\{\begin{matrix} a & b & i \\ c & d & j \end{matrix}\right\} = \frac{ \cerchio_i}{\teta_{a,d,i} \teta_{c,b,i}} \pic{2}{0.7}{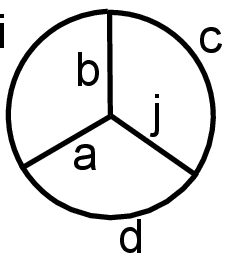} .
\end{equation}

The effect of a positive full twist is shown in Fig.~\ref{figure:framingchange} (see \cite[Lemma 14.1]{Lickorish} or Proposition~\ref{prop:half-framingchange}).
\begin{figure}[htbp]
$$
\pic{2}{0.9}{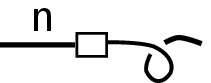} = (-1)^n A^{n^2 + 2n} \ \pic{2}{0.9}{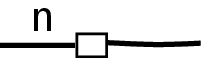}
$$
\caption{A positive full twist.}
\label{figure:framingchange}
\end{figure}

The \emph{fusion rule} is shown in Fig.~\ref{figure:fusion}, it takes place inside a 3-ball. This comes out from (\ref{eqn:first_id}) and (\ref{eqn:second_id}) with $j=0$ since the skein of a graph $G$ is equal to the skein of a graph obtained by adding to $G$ a strand colored with $0$. 

\begin{figure}
\begin{center}
\includegraphics[width = 8 cm]{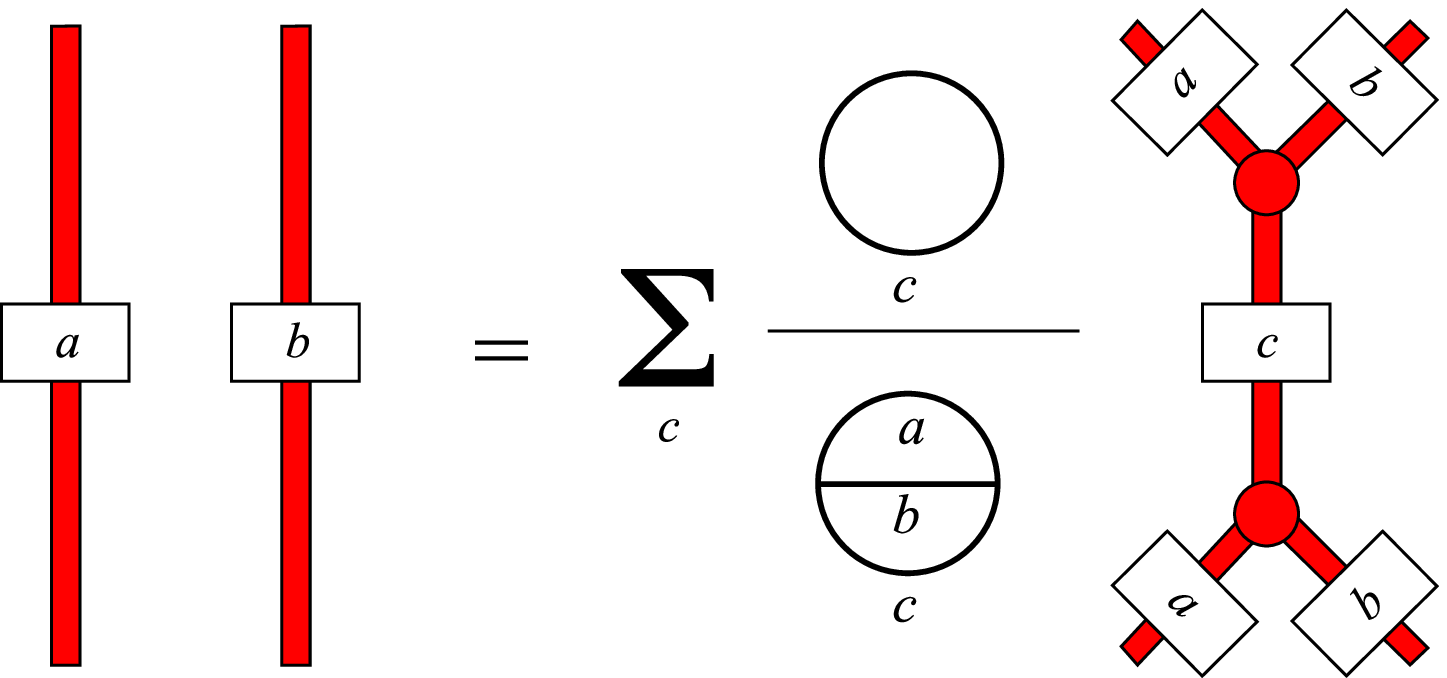}
\caption{The fusion rule. Recall that all framings are orientable. We suppose here that the two bands in the left are oriented coherently, so that the right knotted trivalent graph is also orientable.}
\label{figure:fusion}
\end{center}
\end{figure}

Let $G$ be a framed trivalent graph that intersects a 2-sphere only in a point contained in a framed strand $T\subset G$. We can apply an isotopy (a slide along the 2-sphere) and discover that $G$ is isotopic to the framed graph obtained adding two full twists to $T$. Since $A\neq 1$, using the equality in Fig.~\ref{figure:framingchange} we get that the skein of a framed colored graph that intersects once a 2-sphere with a strand with a non null color is $0$. This is the identity in Fig.~\ref{figure:sphere}-(left). 

Let $G$ be a colored framed trivalent graph that intersects a 2-sphere $S$ exactly in two points contained in two strands colored with $a$ and $b$. We apply the fusion rule of Fig.~\ref{figure:fusion}. We get a linear combination of framed colored graphs that intersect $S$ once. By the previous identity all the summands are null except the one whose strand that intersects $S$ is colored with $0$ (we can remove that strand). We get that color if and only if $a=b$. This is the identity in Fig.~\ref{figure:sphere}-(right).

\begin{figure}
\begin{center}
\includegraphics[width = 10 cm]{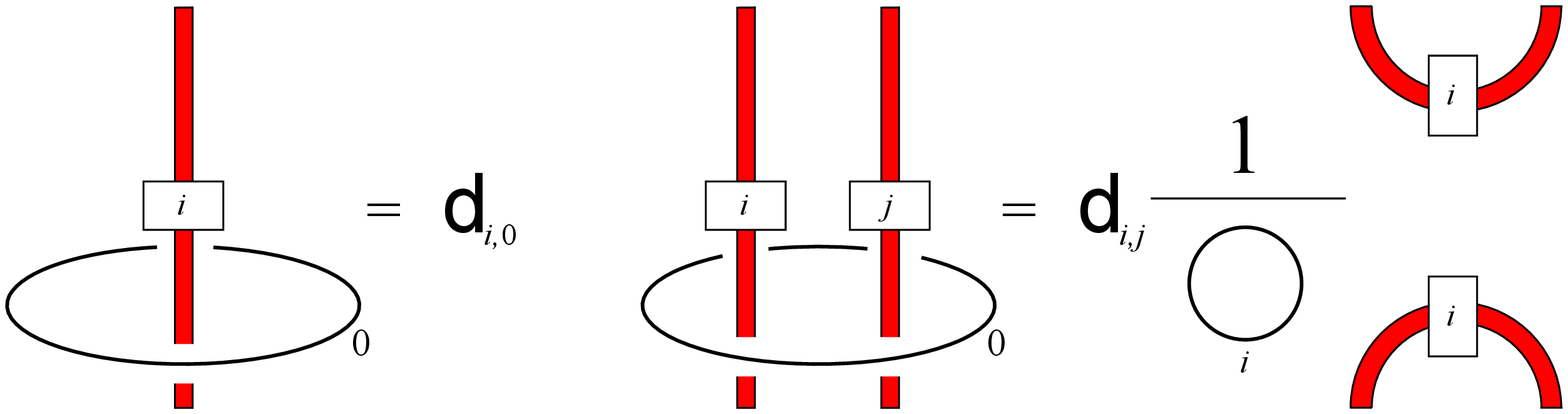}
\caption{Sphere intersection. The skein of a framed colored trivalent graph intersecting once or twice a 2-sphere. The symbol $d_{i,j}$ is the Kronecker delta $d_{i,i}=1$, $d_{i,j}=0$ if $i \neq j$.}
\label{figure:sphere}
\end{center}
\end{figure}

\begin{rem}\label{rem:graph_and_sphere}
Let $G$ be a colored framed trivalent graph and $S$ be an embedded 2-sphere that is transverse to $G$ ($G$ and $S$ intersect in a finite number of points). Using the fusion rule several times we get that the skein of $G$ is a linear combination of colored graphs intersecting $S$ once. Therefore the skein of $G$ is a linear combination of skeins of graphs not intersecting $S$. Furthermore the colors of those strands have the parity of the sum of the colors of the strands of $G$ that intersect $S$ counted with the multiplicity, namely if an edge intersects $n$ times $S$ its color must be counted $n$ times. Therefore if that sum is odd the skein of $G$ is $0$.
\end{rem}

By (\ref{eqn:first_id}) and the one in Fig.~\ref{figure:sphere}-(left) we get the following:
$$
\pic{1.4}{0.8}{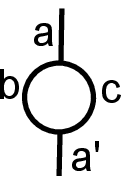} = \left\{\begin{array}{cl}
\frac{\teta_{a,b,c}}{\cerchio_a} \ \pic{1.4}{0.8}{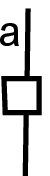} & \text{if } a=a' \\
0 & \text{if } a\neq a'
\end{array}\right. .
$$

\subsection{Another basis for the skein space of a handlebody}

Identify the solid torus $V$ with the product $V=A\times [-1,1]$ of an annulus $A\cong S^1\times [-1,1]$ and an interval. We can define a multiplication on framed links of $V$ by overlapping two framed links according to the factor $[-1,1]$. Let $\alpha \subset V$ be the core $S^1\times \{0\} \times \{0\} \subset V$ of a solid torus $V$ with trivial framing (the one given by the annulus $A$). By Theorem~\ref{theorem:known_sk_mod}-(2.) the set of powers $\alpha^n$ with $n\geq 0$, of $\alpha$, is a basis for the free module $KM(V)$ ($\alpha^n=\alpha \cdot \ldots \cdot \alpha$ is the link consisting of $n$ disjoint copies of $\alpha$ and $\alpha^0=\varnothing$), moreover with this multiplication $KM(V)$ is isomorphic to the algebra of polynomials with coefficients in $\Z[A,A^{-1}]$ where $\alpha$ corresponds to the variable of the polynomials. Of course, the analogous facts hold also for any base ring, in particular for $\mathbb{Q}(A)$.

\begin{prop}
For every integral domain $R$, the set $(y_n)_{n\geq 0}$, $y_n\in R[x]$, of \emph{Chebyshev polynomials} (of the second kind) is a basis for the $R$-module $R[x]$. These are defined as follows: 
\beq
y_0 & := & 1 \\
y_1 & := & x \\
y_{n+1} & := & x\cdot y_n - y_{n-1} .
\eeq
\begin{proof}
The polinomial $y_n$ is of the form $y_n = x^n + f_n$ where $f_n\in R[x]$ is a polynomial whose degree is smaller than $n$. Hence by induction we get $x^n$ as a linear combination of $y_0, y_1, \ldots , y_n$. Let $\lambda_0 y_0 + \lambda_1y_1 + \ldots +\lambda_n y_n =0$ be a null linear combination ($\lambda_i\in R$ for $i=0,\ldots n$), then $\lambda_n y_n =  - \sum_{i=0}^{n-1} \lambda_i y_i$. The degree of $\sum_{i=0}^{n-1} \lambda_i y_i$ is at most $n-1$, hence $\lambda_n y_n =0$. Since $R$ is a domain and $y_n\neq 0$, we get $\lambda_n=0$. Therefore the Chebyshev polynomials are linearly independent.
\end{proof}
\end{prop}

Let $\alpha_{(n)}\subset V$ be the core $\alpha$ equipped with color $n$. Fig.~\ref{figure:Cheb} shows that in the correspondence between $KM(V)$ and the algebra of polynomials, $\alpha_{(n)}$ corresponds to the $n^{\rm th}$ Chebyshev polynomial $y_n$. 

Let $H_g$ be the 3-dimensional handlebody of genus $g$ (the orientable compact 3-manifold with a handle-decomposition with just $k$ 0-handles and $k+g-1$ 1-handles) and let $\Gamma \subset H_g$ be a framed trivalent graph onto which $H_g$ collapses. 
\begin{prop}
The set consisting of $\Gamma$ equipped with all the admissible colorings is a basis for the skein space $K(H_g)$.
\begin{proof}
See \cite[Proposition 6.3]{Costantino-Martelli}.
\end{proof}
\end{prop}

\section{The skein space of the 3-torus}\label{sec:3-torus}

In this section we talk about the skein vector space $K(T^3)$ of the 3-torus $T^3=S^1\times S^1\times S^1$. We prove that the vector space is finitely generated and we show a set of $9$ generators. Then we cite \cite{Gilmer} to get that this set of generators is actually a basis. We note that each result works also for $KM(T^3; \mathbb{C}, A)$ where $A^n\neq 1$ for every $n>0$. Our main tool is the algebraic work of Frohman-Gelca (Theorem~\ref{theorem:F-G}). We follow \cite{Carrega_3-torus}.

\subsection{The skein algebra of the 2-torus}

\begin{defn}
Let $S$ be an orientable surface. The skein module $KM(S;R,A)$ has a natural structure of $R$-algebra. This structure is given by the linear extension of a multiplication defined on framed links of $S\times [-1,1]$. Given two framed links $L_1,L_2 \subset S\times [-1,1]$, the product $L_1\cdot L_2 \subset S\times [-1,1]$ is obtained by putting $L_1$ above $L_2$, $L_1\cdot L_2 \cap S\times [0,1] = L_1$ and $L_1\cdot L_2 \cap S\times [-1,0] = L_2$.
\end{defn}

Look at the 2-torus $T^2$ as the quotient of $\mathbb{R}^2$ modulo the standard lattice of translations generated by $(1,0)$ and $(0,1)$, hence for any non null pair $(p,q)$ of integers we have the notion of $(p,q)$-\emph{curve}: the simple closed curve in the 2-torus that is the quotient of the line passing trough $(0,0)$ and $(p,q)$.

\begin{defn}
Let $p$ and $q$ be two co-prime integers, hence $(p,q)\neq (0,0)$. We denote by $(p,q)_T$ the $(p,q)$-curve in the 2-torus $T^2$ equipped with the black-board framing. Given a framed knot $\gamma$ in an oriented 3-manifold $M$ and an integer $n\geq 0$, we denote by $T_n(\gamma)$ the skein element defined by induction as follows:
\beq
T_0(\gamma) & := & 2\cdot \varnothing  \\
T_1(\gamma) & := & \gamma \\
T_{n+1}(\gamma) & := & \gamma \cdot T_n(\gamma) - T_{n-1}(\gamma)
\eeq
where $\gamma \cdot T_n(\gamma)$ is the skein element obtained adding a copy of $\gamma$ to all the framed links that compose the skein $T_n(\gamma)$. For $p,q\in \Z$ such that $(p,q)\neq (0,0)$, we denote by $(p,q)_T$ the skein element 
$$
(p,q)_T := T_{{\rm MCD}(p,q)} \left( \left( \frac{p}{ {\rm MCD}(p,q) } , \frac{q}{ {\rm MCD}(p,q) } \right)_T \right) ,
$$
where ${\rm MCD}(p,q)$ is the maximum common divisor of $p$ and $q$. Finally we set 
$$
(0,0)_T := 2 \cdot \varnothing .
$$
\end{defn}

It is easy to show that the set of all the skein elements $(p,q)_T$ with $p,q\in\Z$ generates $KM(T^2;R,A)$ as $R$-module.

This is not the standard way to color framed links in a skein module. The colorings $JW_n(\gamma)$, $n\geq 0$, with the Jones-Wenzl projectors are defined in the same way as $T_n(\gamma)$ but at the $0$-level we have $JW_0(\gamma)=\varnothing$.

\begin{theo}[Frohman-Gelca]\label{theorem:F-G}
For any $p,q,r,s\in \Z$ the following holds in the skein module $KM(T^2;R,A)$ of the 2-torus $T^2$:
$$
(p,q)_T \cdot (r,s)_T = A^{\left| \begin{matrix} p & q \\ r & s \end{matrix} \right|} (p+r,q+s)_T + A^{-\left| \begin{matrix} p & q \\ r & s \end{matrix} \right|} (p-r,q-s)_T ,
$$
where $\left| \begin{matrix} p & q \\ r & s \end{matrix} \right|$ is the determinant $ps - qr$.
\begin{proof}
See \cite{Frohman-Gelca}.
\end{proof}
\end{theo}

\subsection{The abelianization}

\begin{defn}
Let $B$ be a $R$-algebra for a commutative ring with unity $R$. We denote by $C(B)$ the $R$-module defined as the following quotient: 
$$
C(B) := \frac{B}{ [ B , B ] }  
$$
where $[B , B ]$ is the sub-module of $B$ generated by all the elements of the form $ab- ba$ for $a,b\in B$. We call $C(B)$ the \emph{abelianization} of $B$.
\end{defn}

\begin{rem}
Usually in non-commutative algebra the \emph{abelianization} is the $R$-algebra defined as the quotient of $B$ modulo the sub-algebra (sub-module and ideal) generated by all the elements of the form $ab-ba$. In our definition the denominator is just a sub-module and we only get a $R$-module. We use the word ``abelianization'' anyway. 
\end{rem}

Now we work with $C(K(T^2))$ and we still use $(p,q)_T$ and $(p,q)_T \cdot (r,s)_T$ to denote the class of $(p,q)_T\in K(T^2)$ and $(p,q)_T\cdot (r,s)_T\in K(T^2)$ in $C(K(T^2))$.

\begin{lem}[\cite{Carrega_3-torus}]\label{lem:ab_alg_2-tor}
Let $(p,q)$ be a pair of integers different from $(0,0)$. Then in the abelianization $C(K(T^2))$ of the skein algebra $K(T^2)$ of the 2-torus $T^2$ we have
$$
(p,q)_T = \begin{cases}
(1,0)_T & \text{if } p\in 2\Z+1 , \ q \in 2\Z \\
(0,1)_T & \text{if } p\in 2\Z , \ q \in 2\Z+1 \\
(1,1)_T & \text{if } p,q\in 2\Z+1 \\
(2,0)_T & \text{if } p,q \in 2\Z
\end{cases} .
$$
Hence $C(K(T^2))$ is generated as a $\mathbb{Q}(A)$-vector space by the empty set $\varnothing$, the framed knots $(1,0)_T$, $(0,1)_T$, $(1,1)_T$, and a framed link consisting of two parallel copies of $(1,0)_T$.
\begin{proof}
By Theorem~\ref{theorem:F-G} for every $p,q\in \Z$ we have
\beq
A^{-q} (p+2,q)_T + A^q (p,q)_T & = & (p+1,q)_T \cdot (1,0)_T \\
 & = & (1,0)_T \cdot (p+1,q)_T \\
 & = & A^q (p+2,q)_T + A^{-q} (-p,-q)_T .
\eeq
Since $(p,q)_T=(-p,-q)_T$ we have $ (A^q - A^{-q}) (p,q)_T = (A^q -A^{-q}) (p+2,q)_T $. Hence if $q\neq 0$ we get $(p,q)_T = (p+2,q)_T $ (here we use the fact that the base ring is a field and $A^n \neq 1$ for every $n>0$). Thus
$$
(p,q)_T = \begin{cases}
(0,q)_T & \text{if } p\in 2\Z, \ \ q\neq 0 \\ 
(1,q)_T & \text{if } p\in 2\Z +1, \ \ q\neq 0
\end{cases} .
$$
Analogously by using $(0,1)_T$ instead of $(1,0)_T$ for $p\neq 0$ we get
$$
(p,q)_T = \begin{cases}
(p,0)_T & \text{if } q\in 2\Z, \ \ q\neq 0 \\ 
(p,1)_T & \text{if } q\in 2\Z +1, \ \ q\neq 0
\end{cases} .
$$
Therefore if $p,q\in 2\Z+1$, $(p,q)_T= (1,1)_T$. If $p\neq 0$ we get
$$
(p,0)_T=(p,2)_T = \begin{cases}
(0,2)_T & \text{if } p\in 2\Z \\
(1,2)_T & \text{if } p\in 2\Z +1
\end{cases} 
= \begin{cases}
(0,2)_T & \text{if } p\in 2\Z \\
(1,0)_T & \text{if } p\in 2\Z +1
\end{cases} .
$$
In the same way for $q\neq 0$ we get
$$
(0,q)_T = (2,q)_T = \begin{cases}
(2,0)_T & \text{if } p\in 2\Z \\
(2,1)_T & \text{if } p\in 2\Z +1
\end{cases} 
= \begin{cases}
(2,0)_T & \text{if } p\in 2\Z \\
(0,1)_T & \text{if } p\in 2\Z +1
\end{cases} .
$$
In particular we have
$$
(2,0)_T = (2,2)_T = (2,-2)_T = (0,2)_T = (p,q)_T  \text{ for } (p,q)\neq (0,0), \ p,q \in 2\Z .
$$
\end{proof}
\end{lem}

\subsection{The $(p,q,r)$-type curves}

As for the 2-torus $T^2$, we look at the 3-torus $T^3$ as the quotient of $\mathbb{R}^3$ modulo the standard lattice of translations generated by $(1,0,0)$, $(0,1,0)$ and $(0,0,1)$.

\begin{defn}
Let $(p,q,r)$ be a triple of co-prime integers, that means ${\rm MCD}(p,q,r)=1$, where ${\rm MCD}(p,q,r)$ is the maximum common divisor of $p$, $q$ and $r$, in particular we have $(p,q,r)\neq (0,0,0)$. The $(p,q,r)$-\emph{curve} is the simple closed curve in the 3-torus that is the quotient (under the standard lattice) of the line passing through $(0,0,0)$ and $(p,q,r)$. We denote by $[p,q,r]$ the $(p,q,r)$-curve equipped with the framing that is the collar of the curve in the quotient of any plane containing $(0,0,0)$ and $(p,q,r)$. The framing does not depend on the choice of the plane.
\end{defn}

\begin{defn}
An embedding $e:T^2\rightarrow T^3$ of the 2-torus in the 3-torus is \emph{standard} if it is the quotient (under the standard lattice) of a plane in $\mathbb{R}^3$ that is the image of the plane generated by $(1,0,0)$ and $(0,1,0)$ under a linear map defined by a matrix of $SL_3(\Z)$ (a $3\times 3$ matrix with integer entries and determinant $1$).
\end{defn}

\begin{rem}\label{rem:std_emb}
There are infinitely many standard embeddings even up to isotopies. A standard embedding of $T^2$ in $T^3$ is the quotient under the standard lattice of the plane generated by two columns of a matrix of $SL_3(\Z)$.
\end{rem}

\begin{lem}[\cite{Carrega_3-torus}]\label{lem:pqr-curve}
Let $(p,q,r)$ be a triple of co-prime integers. Then the skein element $[p,q,r]\in K(T^3)$ is equal to $[x,y,z]$, where $x,y,z\in \{0,1\}$ and have respectively the same parity of $p$, $q$ and $r$.
\begin{proof}
Every embedding $e:T^2 \rightarrow T^3$ of the 2-torus $T^2$ in $T^3$ defines a linear map between the skein spaces
$$
e_* : K(T^2) \longrightarrow K(T^3) .
$$
The map $e_*$ factorizes with the quotient map $K(T^2) \rightarrow C(K(T^2))$. In fact we can slide the framed links in $e(T^2\times [-1,1])$ from above to below getting $e_*(L_1\cdot L_2) = e_*(L_2\cdot L_1)$ for every two framed links, $L_1$ and $L_2$, in $T^2\times [-1,1]$. As said in Remark~\ref{rem:std_emb}, a standard embedding $e:T^2 \rightarrow T^3$ corresponds to the plane generated by two columns $(p_1,q_1,r_1) , (p_2,q_2,r_2) \in \Z^3$ of a matrix in $SL_3(\Z)$. In this correspondence $e_*(( a,b)_T)= [ap_1 + bp_2, aq_1 +b q_2, ar_1 + br_2]$ for every co-prime $a,b \in \Z$. Therefore by Lemma~\ref{lem:ab_alg_2-tor} we get 
\beq
[a'p_1 + b'p_2, a'q_1 +b' q_2, a'r_1 + b'r_2] & = & e_*(( a',b')_T) \\
& = & e_*(( a,b)_T) \\
& = & [ap_1 + bp_2, aq_1 +b q_2, ar_1 + br_2]
\eeq
for every two pairs $(a,b),(a',b')\in \Z^2$ of co-prime integers such that $a+a', b+b'\in 2\Z$.

Let $(p,q,r)$ be a tripe of co-prime integers. By permuting $p,q,r$ we get either $(p,q,r)=(1,0,0)$ or $p,q\neq 0$. Consider the case $p,q\neq 0$. Let $d$ be the maximum common divisor of $p$ and $q$, and let $\lambda,\mu \in \Z$ such that $\lambda p + \mu q = d$. The following matrix belongs in $SL_3(\Z)$:
$$
M_1:= \left( \begin{matrix}
\frac p d & -\mu & 0 \\
\frac q d & \lambda & 0 \\
0 & 0 & 1
\end{matrix} \right) .
$$
Let $v^{(1)}_1$ and $v^{(1)}_3$ be the first and the third columns of $M_1$. We have $(p,q,r)= dv^{(1)}_1 + r v^{(1)}_3$. Hence
$$
[p,q,r] = \begin{cases}
[\frac p d , \frac q d, 0]  & \text{if } d\in 2\Z+1 ,\ r\in 2\Z \\
[0 , 0 , 1]  & \text{if } d \in 2\Z, \  r\in 2\Z+1 \\
[\frac p d , \frac q d, 1] & \text{if } d,r\in 2\Z+1 
\end{cases} .
$$
The integers $p,q,r$ can not be all even because they are co-prime, hence $d$ and $r$ can not be both even. Therefore we just need to study the cases where $r\in \{0,1\}$. 

If $r=0$ we consider the trivial embedding of $T^2$ in $T^3$. The corresponding matrix of $SL_3(\Z)$ is the identity. We have $(p/d,q/d,0)=p/d(1,0,0) + q/d(0,1,0)$, hence
$$
[p,q,0] = [\frac p d , \frac q d ,0] = \begin{cases}
[1,0,0] & \text{ if } \frac p d \in 2\Z+1, \ \frac q d \in 2\Z \\
[0,1,0] & \text{ if } \frac p d \in 2\Z, \ \frac q d \in 2\Z+1 \\
[1,1,0] & \text{ if } \frac p d , \frac q d \in 2\Z+1 
\end{cases} .
$$

If $r=1$ we take the following matrix of $SL_3(\Z)$:
$$
M_2:= \left( \begin{matrix}
0 & 0 & 1 \\
q & -1 & 0 \\
1 & 0 & 0
\end{matrix} \right) .
$$
Let $v^{(2)}_1$ and $v^{(2)}_3$ be the first and the third columns of $M_2$. We have $(p,q,1)= pv^{(2)}_3 + v^{(2)}_1$, hence
$$
[p,q,1] = \begin{cases}
[1,q,1] & \text{if } p\in 2\Z+1 \\
[0,q,1] & \text{if } p\in 2\Z 
\end{cases} .
$$

By permuting $p,q,r$ we reduce the case $(p,q,r)=(0,q,1)$ to the case $p,q\neq 0$, $r=0$ that we studied before.
 
It remains to consider the case $p=r=1$. We consider the following matrix of $SL_3(\Z)$:
$$
M_3:= \left( \begin{matrix}
1 & 0 & 0 \\
0 & 1 & 0 \\
1 & 0 & 1
\end{matrix} \right) .
$$
Let $v^{(3)}_1$ and $v^{(3)}_2$ be the first and the second columns of $M_3$. We have $(1,q,1)= v^{(3)}_1 + qv^{(3)}_2$. Hence
$$
[1,q,1] = \begin{cases}
[1,0,1] & \text{if } q\in 2\Z \\
[1,1,1] & \text{if } q\in 2\Z+1 
\end{cases} .
$$
\end{proof}
\end{lem}

\begin{lem}[\cite{Carrega_3-torus}]\label{lem:intersect_2-tor}
The intersection of any two different standard embedded 2-tori in $T^3$ contains a $(p,q,r)$-type curve.
\begin{proof}
Let $T_1$ and $T_2$ be two standard embedded tori in the 3-torus, and let $\pi_1$ and $\pi_2$ be two planes in $\mathbb{R}^3$ whose projection under the standard lattice is respectively $T_1$ and $T_2$. The intersection $T_1\cap T_2$ contains the projection of $\pi_1\cap \pi_2$. We just need to prove that in $\pi_1\cap \pi_2$ there is a point $(p,q,r)\neq (0,0,0)$ with integer coordinates $p,q,r\in \Z$. Every plane defining a standard embedded torus is generated by two vectors with integer coordinates, and hence it is described by an equation $ax+by+cz=0$ with integer coefficients $a,b,c\in \Z$. Applying a linear map described by a matrix of $SL_3(\Z)$ we can suppose that $\pi_1$ is the trivial plane $\{ z=0 \}$. Let $a,b,c\in \Z$ such that $\pi_2= \{ax+by+cz=0 \}$. The vector $(-b,a,0)$ is non null and lies on $\pi_1\cap \pi_2$.
\end{proof}
\end{lem}

\subsection{Diagrams}

Framed links in $T^3$ can be represented by diagrams in the 2-torus $T^2$. These diagrams are like the usual link diagrams but with further oriented signs on the edges (see Fig.~\ref{figure:diag_3-tor}-(left)). Fix a standard embedded 2-torus $T$ in $T^3$. After a cut along a parallel copy $T'$ of $T$, the 3-torus becomes diffeomorphic to $T\times [-1,1]$ and framed links in $T^3$ correspond to framed tangles of $T\times [-1,1]$. These diagrams are generic projections on $T$ of the framed tangles in $T\times [-1,1]$ via the natural projection $(x,t)\mapsto x$. The further signs on the diagrams represent the intersection of the framed links with the boundary $T\times \{-1,1\}$, in other words they represent the passages of the links along the $(p,q,r)$-type curve that in the Euclidean metric is orthogonal to $T$ (see Fig.~\ref{figure:diag_3-tor}-(right)). If $T$ is the trivial torus $S^1\times S^1 \times \{x\}$, the further signs represent the passages through the third $S^1$-factor. We use the proper notion of black-board framing.

\begin{figure}[htbp]
\begin{center}
\end{center}
\includegraphics[width = 9cm]{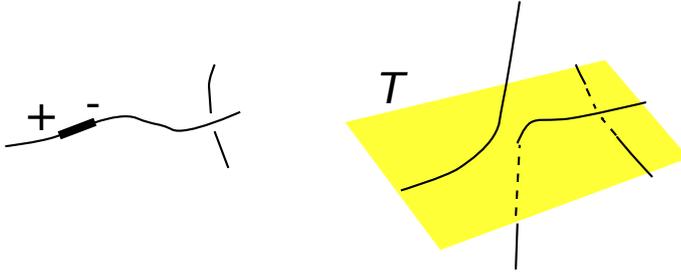}
\caption{Diagrams of framed links in $T^3$. The yellow plane is a part of the standard embedded torus $T\subset T^3$ where the links project. If we look at the framed links in $T^3$ as framed tangles in $T\times [-1,1]$, the two strands that get out vertically from the yellow plane end in the boundary points $(x,1)$ and $(x,-1)$ for some $x\in T$.}
\label{figure:diag_3-tor}
\end{figure}

\subsection{Generators for the 3-torus}

In the following theorem we use all the previous lemmas to get a set of $9$ generators of $K(T^3)$.

\begin{theo}[\cite{Carrega_3-torus}]\label{theorem:sk_3-tor}
The skein space $K(T^3)$ of the 3-torus $T^3$ is generated by the empty set $\varnothing$, $[1,0,0]$, $[0,1,0]$, $[0,0,1]$, $[1,1,0]$, $[1,0,1]$, $[0,1,1]$, $[1,1,1]$ and a skein $\alpha$ that is equal to the framed link consisting of two parallel copies of any $(p,q,r)$-type curve.
\begin{proof}
Let $T$ be the trivial embedded 2-torus: the one containing the $(p,q,r)$-type curves with $r=0$. Use $T$ to project the framed links and make diagrams. By using the first skein relation on these diagrams we can see that $K(T^3)$ is generated by the framed links described by diagrams on $T$ without crossings. These diagrams are union of simple closed curves on $T$ equipped with some signs as the one with $+$ and $-$ in Fig.~\ref{figure:diag_3-tor}. These simple closed curves are either parallel to a $(p,q)$-curve or homotopically trivial. The framed links described by these diagrams lie in the standard embedded tori that are the projection (under the standard lattice) of the planes generated by $(0,0,1)$ and $(p,q,0)$ for some $p$ and $q$. Therefore $K(T^3)$ is generated by the images of $K(T^2)$ under the linear maps induced by the standard embeddings of $T^2$ in $T^3$.

As said in the proof of Lemma~\ref{lem:pqr-curve}, the linear map $e_*$ induced by any standard embedding $e:T^2\rightarrow T^3$ factorizes with the quotient map $K(T^2) \rightarrow C( K(T^2) )$. Lemma~\ref{lem:ab_alg_2-tor} applied on the standard embedding $e$ shows that the image $e_*(K(T^2))$ is generated by $\varnothing$, three $(p,q,r)$-type curves lying on $e(T^2)$, and the skein $\alpha_e$ that is equal to the framed link consisting of two parallel copies of any $(p,q,r)$-type curve lying on $e(T^2)$.

Let $e_1,e_2:T^2 \rightarrow T^3$ be two standard embeddings. By Lemma~\ref{lem:intersect_2-tor} $e_1(T^2)\cap e_2(T^2)$ contains a $(p,q,r)$-type curve $\gamma$, hence $\alpha_{e_1}$ and $\alpha_{e_2}$ coincide with the framed link that is two parallel copies of $\gamma$. Therefore the skein element $\alpha_e$ does not depend on the embedding $e$.

We conclude by using Lemma~\ref{lem:pqr-curve} that says that the skein of any $(p,q,r)$-type curve is equal to the one of a standard representative of a non null element of the first homology group $H_1(T^3;\Z_2)$ with coefficient in $\Z_2$, namely a $(p,q,r)$-type curve with $p,q,r\in \{0,1\}$.
\end{proof}
\end{theo}

\subsection{Linear independence}

Here we talk about the linear independence of generators of $K(T^2)$ we have shown. Lemma~\ref{lem:dec_dir_sum} shows a decomposition in direct sum of $K(T^3)$, while Lemma~\ref{lem:Gilmer} says that the shown generators of the summands of the decomposition actually form a basis.

\begin{lem}[\cite{Carrega_3-torus}]\label{lem:dec_dir_sum}
The skein space $K(T^3)$ is the direct sum of $8$ sub-spaces
$$
K(T^3) = V_0 \oplus V_1 \oplus \ldots \oplus V_7 
$$
such that:
\begin{enumerate}
\item{$V_0$ is generated by the empty set $\varnothing$ and the skein $\alpha$ (see Theorem~\ref{theorem:sk_3-tor});}
\item{every $(p,q,r)$-type curve generates a $V_j$ with $j>0$ and
every $V_j$ with $j>0$ is generated by one such curve.}
\end{enumerate}
\begin{proof}
The skein relations relates framed links holding in the same $\Z_2$-homology class. Hence for every oriented 3-manifold $M$ we have a decomposition in direct sum 
$$
KM(M;R,A) = \bigoplus_{h \in H_1(M;\Z_2)} V_h ,
$$
where $V_h$ is generated by the framed links whose $\Z_2$-homology class is $h$. The statement follow by this observation and the fact that if $[p,q,r]$ and $[p',q',r']$ represent the same $\Z_2$-homology class, then $[p,q,r]=[p',q',r']\in K(T^3)$.
\end{proof}
\end{lem}

\begin{rem}
Given a triple of integers $(x,y,z)\neq (0,0,0)$ such that $x,y,z\in \{0,1\}$, we can easily find an orientation preserving diffeomorphism of the 3-torus $T^3$ sending $[x,y,z]$ to $[1,0,0]$. Hence if the skein of one such curve $[x,y,z]$ is null then also all the other skein elements of such curves are null. Therefore by Lemma~\ref{lem:dec_dir_sum} the possible dimensions of the skein space $K(T^3)$ are $0$, $1$, $2$, $7$, $8$ and $9$.
\end{rem}

\begin{lem}[Gilmer]\label{lem:Gilmer}
$\ $
\begin{enumerate}
\item{The skein element $[1,0,0] \in K(T^3)$ is non null.}
\item{The empty set and the skein $\alpha$ (see Theorem~\ref{theorem:sk_3-tor}) are linear independent in $K(T^3)$.}
\end{enumerate}
\begin{proof}
See \cite{Gilmer}. We just sketch the main ideas. 

Consider the $SO(3)$-\emph{Reshetikhin-Turaev-Witten} invariants. These are invariants of pairs $(M,L)$ where $M$ is a closed oriented 3-manifold and $L$ is a framed link of $M$. These invariants are constructed with skein theory, associates to each pair a complex number, and are based on the choice of a root of unity $A\in \mathbb{C}$ and a surgery presentation of the manifold $M$ in $S^3$ (see Subsection~\ref{subsec:surgery}). The construction is very similar to the one of the $SU(2)$-\emph{Reshetikhin-Turaev-Witten} invariants (see Section~\ref{sec:RTW}). 

A surgery presentation of $T^3$ are the 0-framed \emph{Borromean rings}. Set $\Gamma:= \{ e^{\pm \frac {\pi i}{2d+1}} \ | \ d\in \Z , \ d>0 \}$. Then construct a $\mathbb{Q}(A)$-linear map $I: K(T^3) \rightarrow \mathbb{C}^\Gamma $, where $\mathbb{C}^\Gamma$ is the space of the functions with values in $\mathbb{C}$ and are defined in all but a finite number of elements of $\Gamma$ and two such functions are considered equal if they agree in all but a finite number of elements of $\Gamma$. The image $I(L)$ of a framed link $L\subset T^3$ is the $SO(3)$-Reshetikhin-Turaev-Witten invariant of $(T^3,L)$ built with the elements of $\Gamma$. Clearly if the image $I(S)$ of a skein element $S\in K(T^3)$ is not zero, the skein is not null. 

To show that $\varnothing$ and $\alpha$ are linear independent, suppose $\alpha = \lambda \cdot \varnothing$ for some $\lambda \in \mathbb{Q}(A)$. Get $\lambda(\gamma)$ for each $\gamma \in \Gamma$ as a consequence of the computation of $I(\varnothing)$ and $I(\alpha)$. Also get $\lambda(1)= \lim_{d\rightarrow \infty} \lambda (e^{\pm \frac{\pi i}{2d+1} }) =1$. Show a holomorphic function $f: U \rightarrow \mathbb{C}$ defined on an open set $U\subset \mathbb{C}$ that is not a rational function and such that $\Gamma \subset U$, $f(e^{\frac{\pi i}{2d+1} }) = \lambda(e^{\frac{\pi i}{2d+1} })$, $f(1)=\lambda(1)= 1$ and $f(e^{- \frac{\pi i}{2d+1} }) \neq \lambda (e^{ - \frac{\pi i}{2d+1} })$ for every $d>0$. The functions $f$ and $\lambda$ are both holomorphic functions defined in open sets $U,V \subset \mathbb{C}$ and coincide on an infinite set of points with a limit point. Hence $f=\lambda$ on $U\cap V$. But $f$ and $\lambda$ also disagree in an infinite set of points. Therefore once such $\lambda$ can not exists.
\end{proof}
\end{lem}

\begin{theo}
The $9$ generators showed in Theorem~\ref{theorem:sk_3-tor} form a basis of the skein vector space $K(T^3)$ of the 3-torus.
\begin{proof}
It follows from Lemma~\ref{lem:dec_dir_sum} and Lemma~\ref{lem:Gilmer}.
\end{proof}
\end{theo}

\begin{rem}
All the results in this section work for every base pair $(R,A)$ such that $A^n-1$ is an invertible element of $R$ for any $n>0$. In particular they work for $(\mathbb{C}, A)$, where $A^n\neq 1$ for any $n>0$.
\end{rem}

\section{The Kauffman bracket in $\#_g(S^1\times S^2)$}\label{sec:Kauf_g}

In Remark~\ref{rem:sk_sp_S1xS2} we saw that the Kauffman bracket is defined in the connected sum $\#_g(S^1\times S^2)$ of $g\geq 0$ copies of $S^1\times S^2$ ($g=0$ means $S^3$, and $g=1$ means $S^1\times S^2$) (see Definition~\ref{defn:Kauf}). In this section we talk about the Kauffman bracket in $\#_g(S^1\times S^2)$, we introduce some tools to compute it and we express some proprieties that we are going to need in other chapters or are just interesting, for instance there are phenomena that do not happen in $S^3$. We list some examples of links in $\#_g(S^1\times S^2)$ together with their Kauffman bracket. We can find more examples of links in $S^1\times S^2$ in Chapter~\ref{chapter:table}.

\subsection{Diagrams and moves}\label{subsec:diag}

The manifold $\#_g(S^1\times S^2)$ is the double of the 3-dimensional handlebody $H_g$ of genus $g$ (the compact orientable 3-manifold with a handle-decomposition with just $k$ 0-handles and $k+g-1$ 1-handles). We call one such decomposition a \emph{H-decomposition}.

By a theorem of Thom we know that every two embeddings of a 3-disk in a fixed manifold are isotopic. Hence up to isotopies there is a unique Heegaard decomposition of $S^3$ that splits it into two 3-balls.

\begin{theo}[Waldhausen-Carvalho]\label{theorem:Heegaard_split_S1xS2}
Every two embeddings of the closed surface $\partial H_g$ of genus $g$ in $\#_g(S^1\times S^2)$ that split it into two copies of the the handlebody $H_g$ are isotopic.
\begin{proof}
In \cite[Remark 4.1]{Schultens} it is showed that if we glue two copies of $H_g$ along the boundary to get $\#_g(S^1\times S^2)$, the gluing map must be isotopic to the identity (\cite{Schultens} is an updated and illustrated translation of \cite{Waldhausen}). In \cite[Theorem 1.4]{Carvalho} is shown that two embeddings of $\partial H_g$ that split $\#_g(S^1\times S^2)$ in two copies of $H_g$ and define the identity map $\partial H_g$ (once identified the two embedded handlebodies with $H_g$), are isotopic. For the case $g=1$ we can also see the proof of \cite[Theorem 2.5]{Hatcher}.
\end{proof}
\end{theo}

Since $H_g$ collapses onto a graph, every link in $\#_g(S^1\times S^2)$ can be isotoped in a fixed handlebody of the H-decomposition. The handlebody is the natural 3-dimensional thickening of the disk with $g$ holes $S_{(g)}$. 

\begin{defn}\label{defn:e-shadow}
We call \emph{e-shadow} a proper embedding of the disk with $g$ holes $S_g$ (it is a compact surface) in a handlebody $H_g$ of the H-decomposition of $\#_g(S^1\times S^2)$ such that $H_g$ collapses on it.
\end{defn}

There are many e-shadows even up to isotopies. If $g=1$, $S_{(g)}$ is an annulus and two e-shadows differ by twists. Once an e-shadow is fixed, every link in $\#_g(S^1\times S^2)$ can be represented by a link diagram in $S_{(g)}$. Of course one such diagram is a generic projection of the link in the embedded disk with $g$ holes.

Clearly Reidemeister moves still do not change the represented link, but they are not sufficient to connect all the diagrams representing the same link.

Now we describe a new move (see Fig.~\ref{figure:new_move} for the case $g=1$) that will be used later. The move is essentially the second Kirby move. Given a diagram $D$ in the punctured disk $S_{(g)}$ and given an e-shadow, we get a position (an embedding not up to isotopies) of the link $L\subset H_g \subset \#_g(S^1\times S^2)$ described by $D$. We embed the handlebody in $\mathbb{R}^3$ in the standard way so that the image of the embedded punctured disk $S_{(g)}$ lies on $\mathbb{R}^2\subset \mathbb{R}^3$. Then we add a \emph{system of 0-framed meridians} of the handlebody, where a \emph{system of meridians} of a handlebody $H_g$ is the boundary of a non separating sub-manifold $N$ (namely $H_g\setminus N$ is connected) consisting  of $g$ disjoint properly embedded disks. We have obtained a surgery presentation of the pair $(L,\#_g(S^1\times S^2))$ in $\mathbb{R}^3\subset S^3$ (Definition~\ref{defn:surgery_pres}) that is in regular position with respect to $\mathbb{R}^2\subset \mathbb{R}^3$. We apply the second Kirby move to a component of $L$ along one of the 0-framed meridians along an obvious band. This gives another surgery presentation of $(L,\#_g(S^1\times S^2))$ consisting of a link $L'$ in the handlebody encircled by the 0-framed meridians. The link $L'$ is again in regular position and hence gives another diagram of $L\subset \#_g(S^1\times S^2)$ in the punctured disk.

\begin{figure}[htbp]
$$
\picw{4.8}{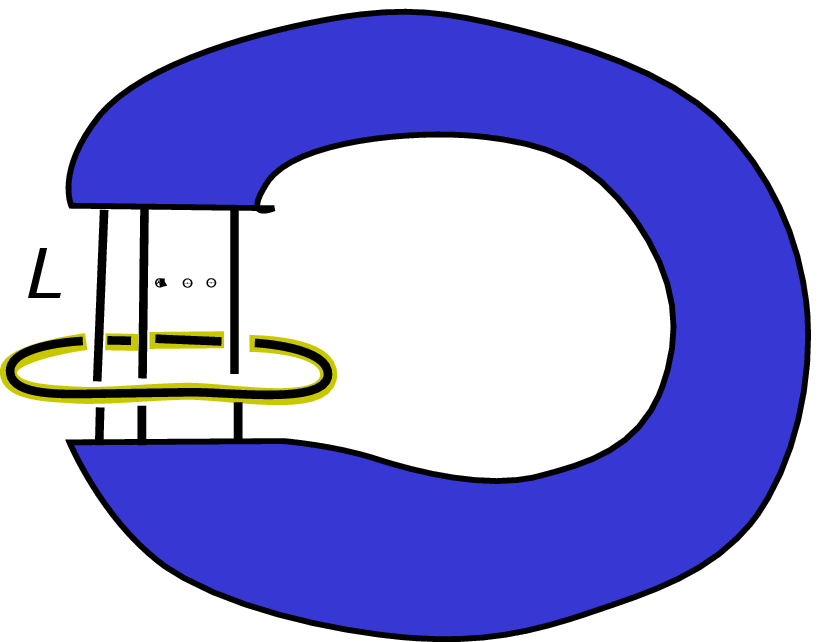} \leftrightarrow \picw{4.8}{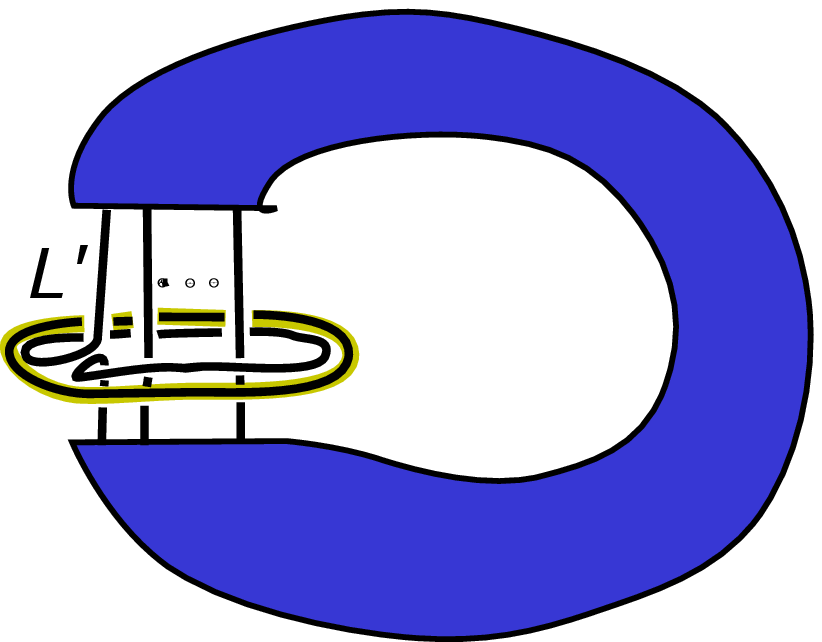}
$$
\caption{A new move on diagrams in the annulus ($g=1$). The links $L$ and $L'$ represent two different links in the solid torus but they represent the same in $S^1\times S^2$. The highlighted circle is the 0-framed meridian needed to get a surgery presentation of $(L,S^1\times S^2)$. The highlighted circle is not contained in the solid torus, hence the diagrams in the annulus do not contain its projection.}
\label{figure:new_move}
\end{figure}

\begin{prop}
Once an e-shadow is fixed, Reidemeister moves together with the move described above are sufficient to connect all the diagrams in the disk with $g$ holes representing the same link in $\#_g(S^1\times S^2)$.
\begin{proof}
Let $H$ be the 3-dimensional handlebody that is the thickening of the e-shadow, let $\Gamma $ be a core graph of the other handlebody $H'$ ($H\cup H' = \#_g(S^1\times S^2)$). Consider any isotopy $\varphi : G\times[-1,1] \rightarrow \#_g(S^1\times S^2)$ of a graph $G \subset H$. We can suppose that for each $t$, $\varphi(G\times \{t\})$ intersects $\Gamma$ transversely and in finitely many point, and $\varphi(G\times \{\pm 1\}) \subset H$. Most of the time, the isotopy will be standard within $H$, except at distinct times when it intersects $\Gamma$. At those moments, one strand of $G$ will perform exactly the described encircling move.
\end{proof}
\end{prop}

All the considerations above work also for framed trivalent colored graphs in $\#_g(S^1\times S^2)$, where the diagrams are graphs in $S_{(g)}$ whose vertices are either 3-valent or 4-valent: the 3-valent vertices correspond to the vertices of the embedded graph and the 4-valent ones are equipped with the further information of the over/underpass.

For $g=1$, $S_{(g)}$ is an annulus. Once an e-shadow is fixed and given a diagram $D\subset S^1\times [-1,1] = S_{(1)}$ of a link $L\subset S^1\times S^2$, we can get a diagram $D'\subset S^1\times [-1,1]$ that represents $L$ with the embedding of the annulus obtained from the previous one by adding a twist following the move described in Fig.~\ref{figure:twist_diagr}.

\begin{figure}[htbp]
\beq
\picw{4.8}{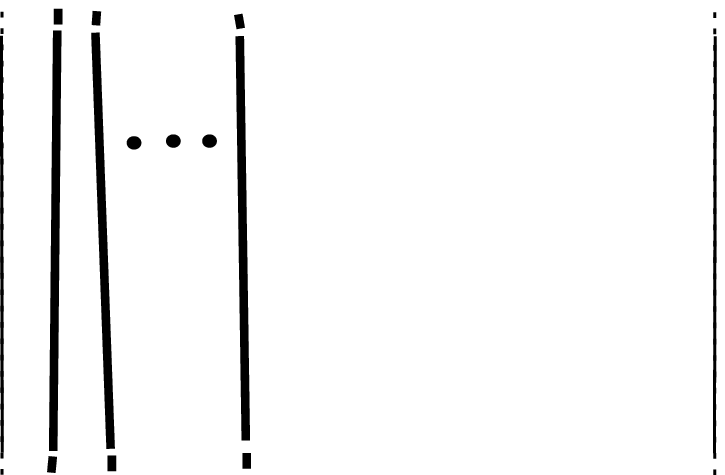} & \longrightarrow & \picw{4.8}{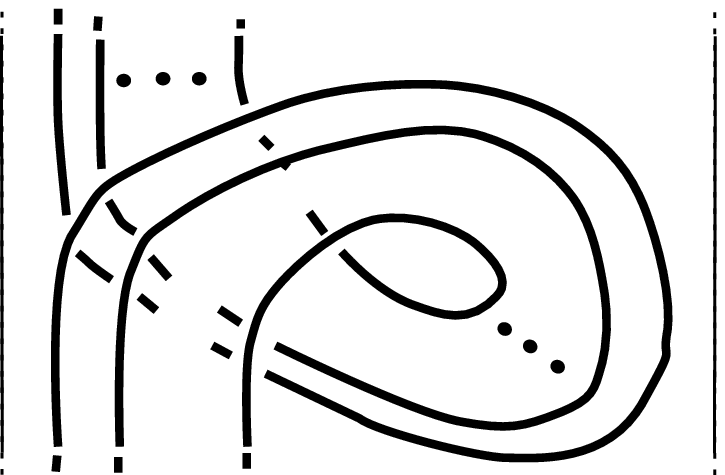} \\
D & & D' 
\eeq
\caption{Two diagrams of the same link in $S^1\times S^2$, the embedding of the annulus $S_{(1)}$ for $D'$ differs from the one of $D$ by the application of a positive twist. The diagrams differ just in the pictured portion that is diffeomorphic to $[-1,1]\times (-1,1)$.}
\label{figure:twist_diagr}
\end{figure}

\begin{defn}\label{defn:alt_cr_num}
A link diagram $D\subset S_{(g)}$ is \emph{alternating} if the parametrization of its components $S^1\rightarrow D\subset S_{(g)}$ meets overpasses and underpasses alternately.

Let $L$ be a link in $\#_g(S^1\times S^2)$. The link $L$ is \emph{alternating} if there is an alternating diagram $D\subset S_{(g)}$ that represents $L$ for some e-shadow. 

The \emph{crossing number} of $L\subset \#_g(S^1\times S^2)$ is the minimal number of crossings that a link diagram $D\subset S_{(g)}$ must have to represent $L$ for some e-shadow.
\end{defn}

In Chapter~\ref{chapter:Tait} we will get some criteria to detect if a link in $\#_g(S^1\times S^2)$ is alternating (Corollary~\ref{cor:conj_Tait_Jones_g}) and we will give some examples of non alternating links and knots (Example~\ref{ex:no_alt_g}).

\begin{rem}
Let $\varphi: \#_g(S^1\times S^2) \rightarrow \#_g(S^1\times S^2)$ be a diffeomorphism and let $L\subset \#_g(S^1\times S^2)$ be a link with a fixed position ($L$ is just a sub-manifold, it is not up to isotopies). Suppose that $L$ is in regular position for a properly embedded disk with $g$ holes $S\subset H_{(1)} \subset \#_g(S^1\times S^2)$ in a handlebody of the H-decomposition $\#_g(S^1\times S^2)= H_{(1)}\cup H_{(2)}$, $H_{(1)} \cong H_{(2)} \cong H_g$, such that $H_{(1)}$ is a thickening of $S$. Hence the pair $(L,S)$ defines a link diagram $D\subset S_{(g)}$. Then the link $\varphi(L)$ is in regular position for the punctured disk $\varphi(S)$ that is properly embedded in $\varphi(H_{(1)})$. By Theorem~\ref{theorem:Heegaard_split_S1xS2} $\varphi(H_{(1)}) = H_{(j)}$ (up to isotopies) for some $j=1,2$. The pair $(\varphi(L), \varphi(S))$ defines a diagram $D_\varphi \subset S_{(g)}$ that is obtained from $D$ by the application of a diffeomorpihsm of $S_{(g)}$. Therefore the crossing number and the condition of being alternating are invariant under diffeomorphisms of $\#_g(S^1\times S^2)$.
\end{rem}

\subsection{Kauffman states}

Diagrams form an extremely useful tool to study links, and in particular to compute the Kauffman bracket.

\begin{defn}
We just use $S^1\times [-1,1]$ instead of $S_{(1)}$ if we focus on $g=1$. Let $D\subset S_{(g)}$ be a link diagram in the disk with $g$ holes (it is a compact surface). A \emph{Kauffman state}, or just a \emph{state}, of $D$ is a function $s$ from the set of crossings of $D$ to $\{1,-1\}$. The assignment of $\pm 1$ to a crossing determines a unique way to remove that crossing as described in Fig.~\ref{figure:splitting}. Hence a state removes all the crossings producing a finite collection of non intersecting circles in the surface. This collection  of circles is called the \emph{splitting}, or the \emph{resolution}, of $D$ with $s$. We denote by $sD$ the number of homotopically trivial circles of the splitting of $D$ with $s$, with $D_s$ the diagram without crossings obtained removing all the homotopically trivial components from the splitting of $D$ with $s$, and with $\sum_i s(i)$ the sum of all the signs associated to the crossings by $s$.
\end{defn}

\begin{figure}[htbp]
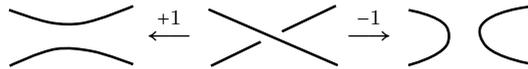

$$
\pic{1.6}{0.4}{Acanalep.eps} \ \stackrel{+1}{\longleftarrow} \ \pic{1.6}{0.4}{incrociop.eps} \ \stackrel{-1}{\longrightarrow} \ \pic{1.6}{0.4}{Bcanalep.eps}
$$
\caption{The splitting of a crossing.}
\label{figure:splitting}
\end{figure}

Proceeding by induction, splitting the crossings and using the skein relations we get the following:
\begin{prop}\label{prop:state_sum}
Fix an e-shadow of $\#_g(S^1\times S^2)$. Let $L$ be a framed link in $\#_g(S^1\times S^2)$ and $D\subset S_{(g)}$ be a diagram of $L$. Then
$$
\langle L \rangle = \sum_s \langle D \ |\ s \rangle  ,
$$
where the sum is taken over all the Kauffman states of $D$ and
$$
\langle D \ | \ s \rangle := A^{\sum_i s(i)} (-A^2 -A^{-2})^{sD} \langle D_s\rangle .
$$
\end{prop}
Therefore the computation of the Kauffman bracket is reduced to the one of diagrams without crossings and without homotopically trivial components. If the splitting of the diagram $D$ with the state $s$ has only homotopically trivial components, $D_s$ is empty and $\langle D_s \rangle =1$. An easy way to conclude the computation is given by the shadow formula (see Remark~\ref{rem:sh_for_br}). 

\subsection{Regularity and triviality}

In this subsection we provide some results and examples about the form of the Kauffman bracket in $\#_g(S^1\times S^2)$.

In the case of links in $S^3$ ($g=0$), the diagrams $D_s$ are all empty. Hence, as we already knew, the Kauffman bracket of a framed link in $S^3$ is an integer Laurent polynomial, $\langle L \rangle \in \Z[A,A^{-1}]$. In $\#_g(S^1\times S^2)$ the Kauffman bracket is a rational function and it may not be a Laurent polynomial (see Example~\ref{ex:Kauff}).

\begin{ex}\label{ex:Kauff}
We show in the table below some links in $\#_g(S^1\times S^2)$ together with their Kauffman bracket. They are all $\Z_2$-homologically trivial (Definition~\ref{defn:Z_2_tr}). In the list there are knots and links with a varying number of components. There are \emph{alternating} and \emph{non alternating} knots and links (Definition~\ref{defn:alt_cr_num}), Corollary~\ref{cor:conj_Tait_Jones_g} ensures us that example $(5)$, example $(6)$ and example $(7)$ are actually non alternating, unfortunately we are not able to say if example $(11)$ is alternating (see Example~\ref{ex:no_alt_g}). Some of them are \emph{H-split} (Definition~\ref{defn:split_homotopic_genus}) and some are not.

In example $(2)$, $\binom g k$ is the binomial coefficient, and the Kauffman bracket can be written as $f/h$, where $f$ and $h$ are the following Laurent polynomials: $f=\sum_{k=0}^g \binom{g}{k} (-A^2-A^{-2})^k$, $h=(-A^2-A^{-2})^{1-g}$. We have $f|_{A^2=i} = 1$, $h|_{A^2=i} =0$, hence $\langle D \rangle$ can not be a Laurent polynomial. 

In all the examples except $(8)$, $(10)$ and $(11)$, the Kauffman bracket is of the form $f/\cerchio_1^n$ for some $n\geq 0$ and $f\in \Z[A,A^{-1}]$ ($\cerchio_1= -A^2 -A^{-2}$). Examples $(8)$, $(10)$ and $(11)$ are not of that form. Example $(8)$ and $(11)$ are of the form $f/\cerchio_2$, while $(10)$ is of the form $f/(\cerchio_2\cerchio_3)$ for some $f\in \Z[A,A^{-1}]$ not divided by $\cerchio_2$ or $\cerchio_3$, where $\cerchio_n$ is the Kauffman bracket of the unknot colored with $n$. Note that $\cerchio_n\in \Z[A,A^{-1}]$ and that the roots of $\cerchio_n$ are all roots of unity. Hence these Kauffman brackets have poles in roots of unity different from $q=A^2=i$.

$$
\begin{array}{|c|c|c|}
\hline
 & \text{Diagram} & \langle D \rangle \\
\hline
\hline
(1) & \picw{6.3}{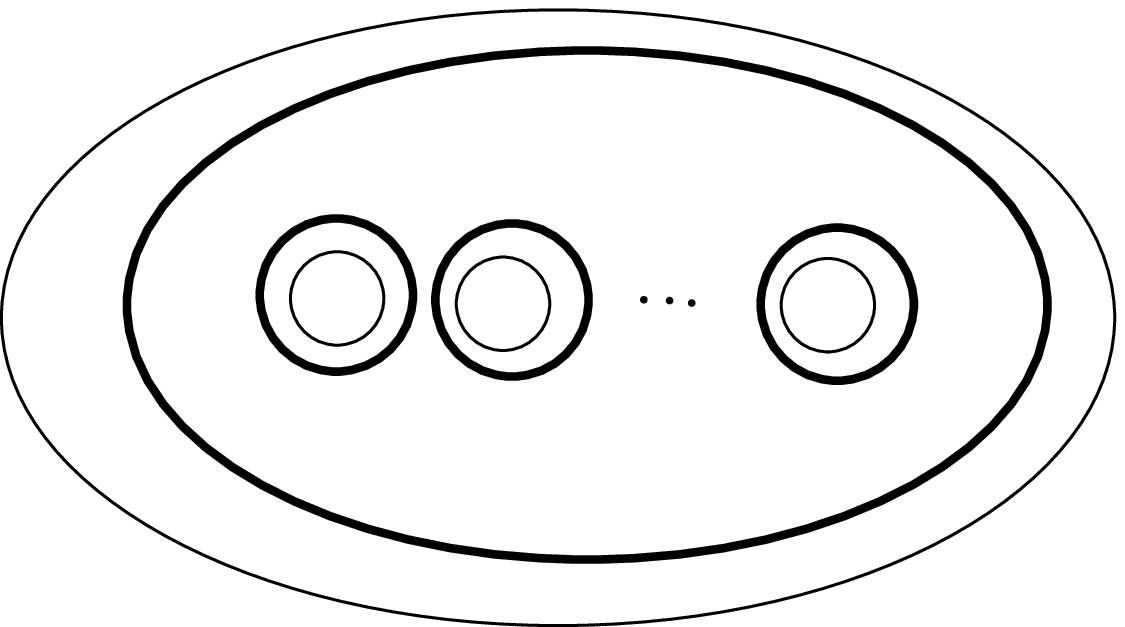} & \parbox[c]{4.5cm}{ 
\begin{center}$
(-A^2-A^{-2})^{1-g} 
$\end{center}
}\\
\hline
(2) & \picw{6.3}{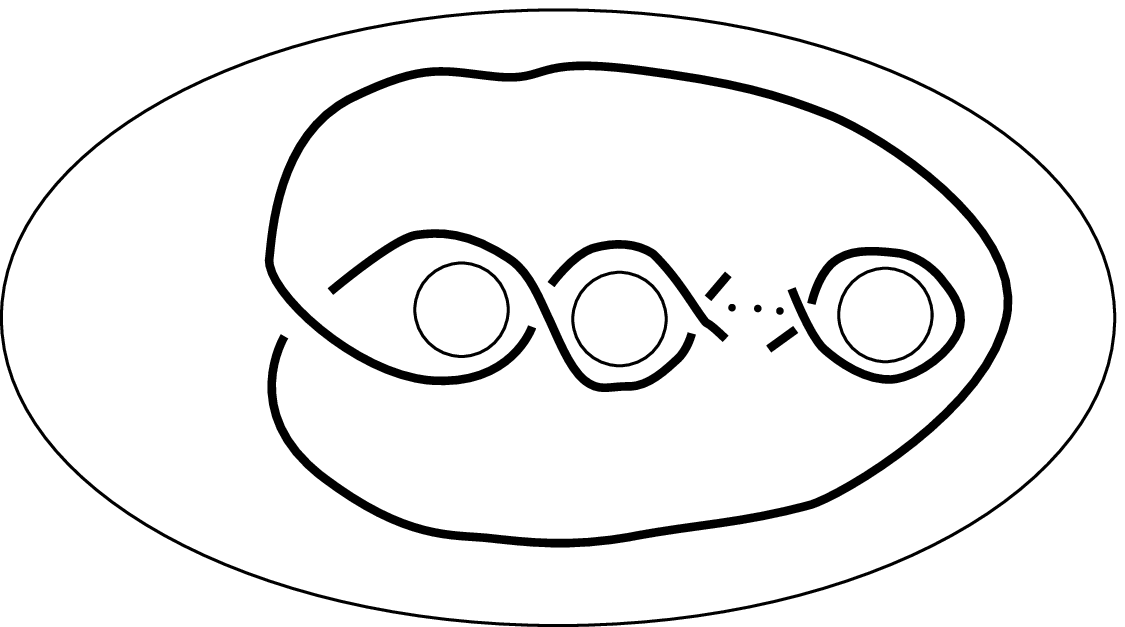} &  \parbox[c]{4.5cm}{
\begin{center}$
\sum_{k=0}^g \binom{g}{k} (-A^2-A^{-2})^{1-g+k} 
$\end{center}
}\\
\hline
(3) & \picw{6.3}{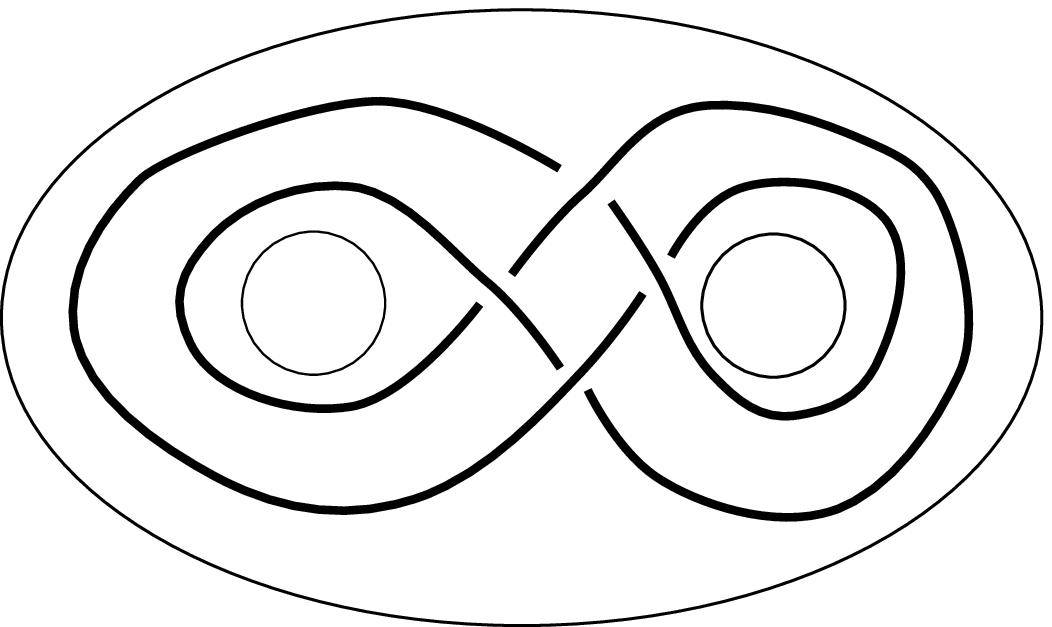} &  \parbox[c]{4.5cm}{
\begin{center}$
\frac{ A^{16} -A^{12} +A^8 +1 }{ A^8 +A^4} 
$\end{center}
}\\
\hline
(4) & \picw{6.3}{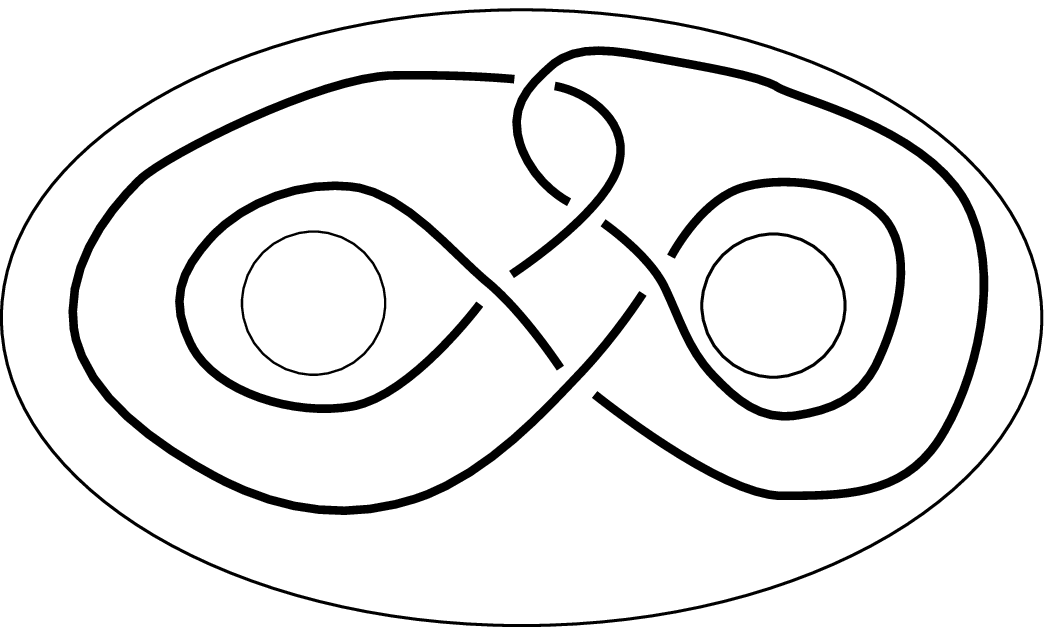} & \parbox[c]{4.5cm}{
\begin{center}$
\frac{ A^{20} -A^{16} +2A^{12} -A^8 +A^4 -1 }{ A^{11} +A^7 }
$\end{center}
}\\
\hline
(5) & \picw{6.3}{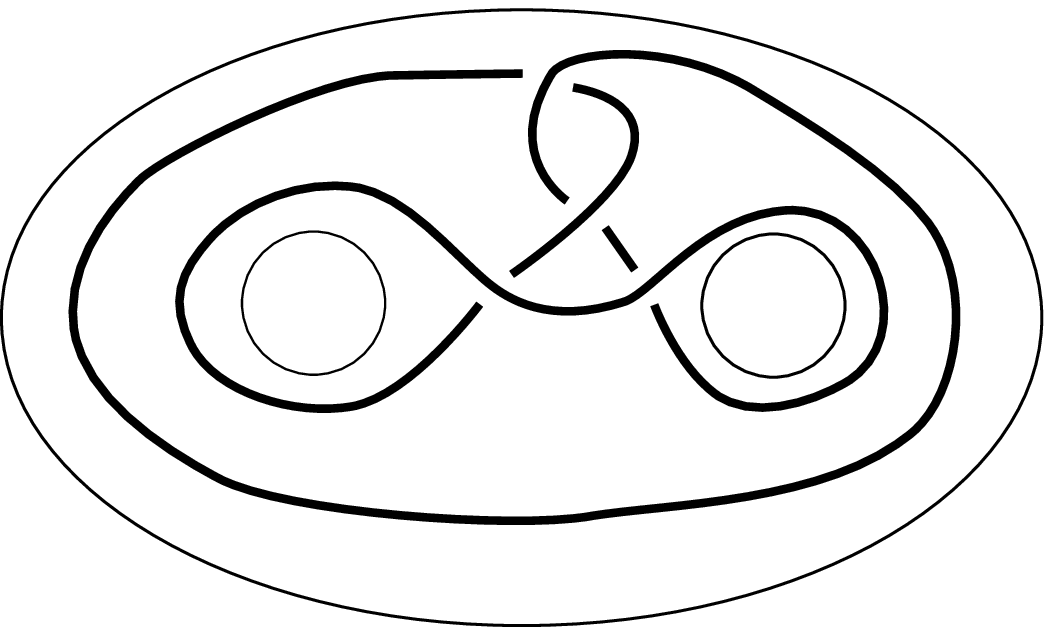} & \parbox[c]{4.5cm}{
\begin{center}$
\frac{ A^{12} -A^{10} +2A^8 +3A^4 +A^2 -1 }{ A^8 +A^4 } 
$\end{center}
}\\
\hline
\end{array}
$$

$$
\begin{array}{|c|c|c|}
\hline
(6) & \picw{6.3}{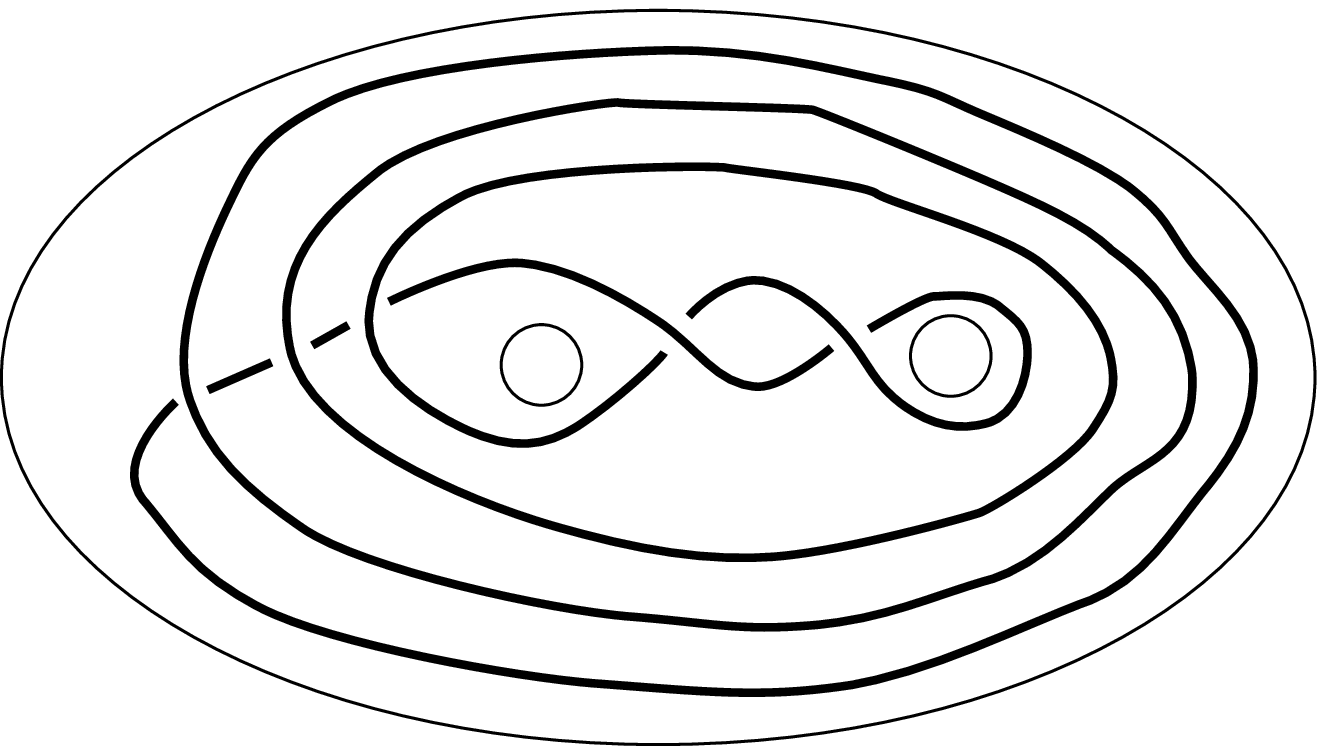} & \parbox[c]{4.5cm}{
\begin{center}$
\frac{-2A^6 -3A^4 -3A^2 -1}{ A^7 +A^3 } 
$\end{center}
}\\
\hline
(7) & \picw{6.3}{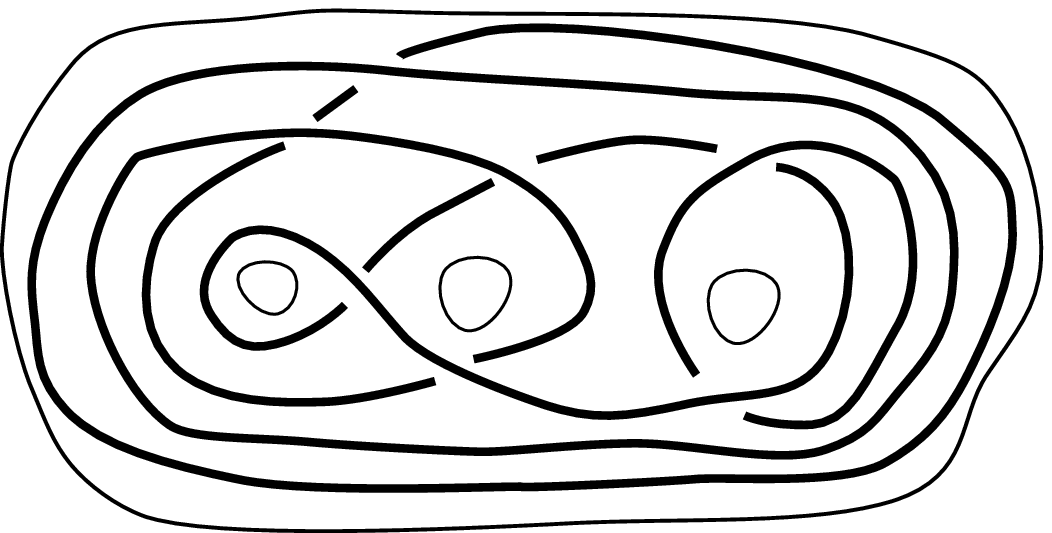} & \parbox[c]{4.5cm}{
\begin{center}$
( -A^{22} -A^{20} -A^{18} +2A^{14} +5A^{12} +3A^{10} +5A^8 +A^6 +4A^4 +1 )/( A^{15} +2A^{11} +A^7) 
$\end{center}
}\\
\hline
(8) & \picw{6.3}{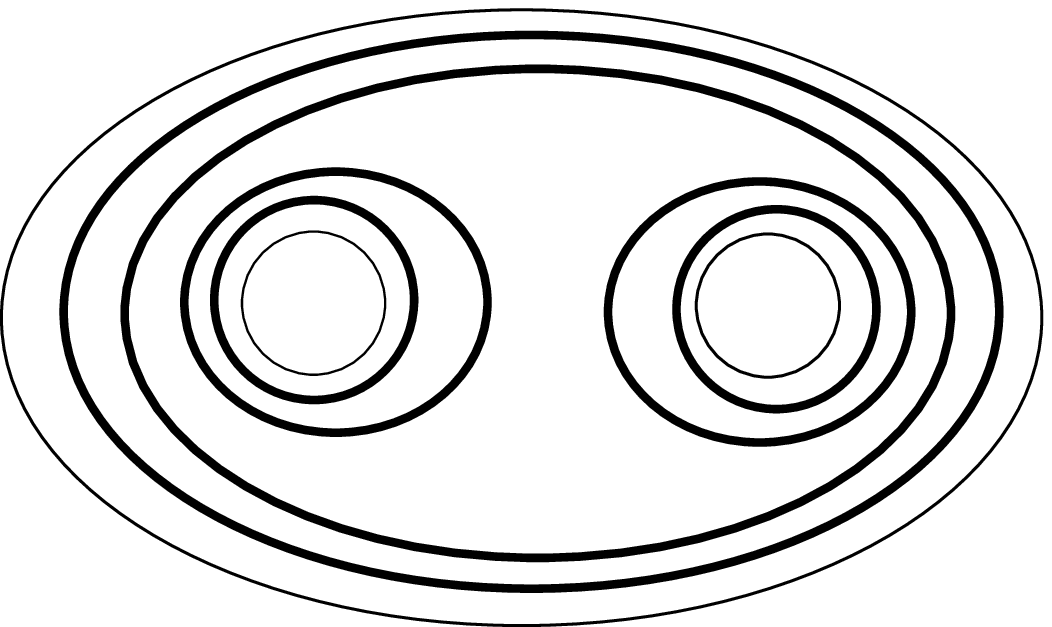} & \parbox[c]{4.5cm}{
\begin{center}$
\frac{ A^8 +2A^4 +1 }{ A^8 +A^4 +1 } 
$\end{center}
}\\
\hline
(9) & \picw{6.3}{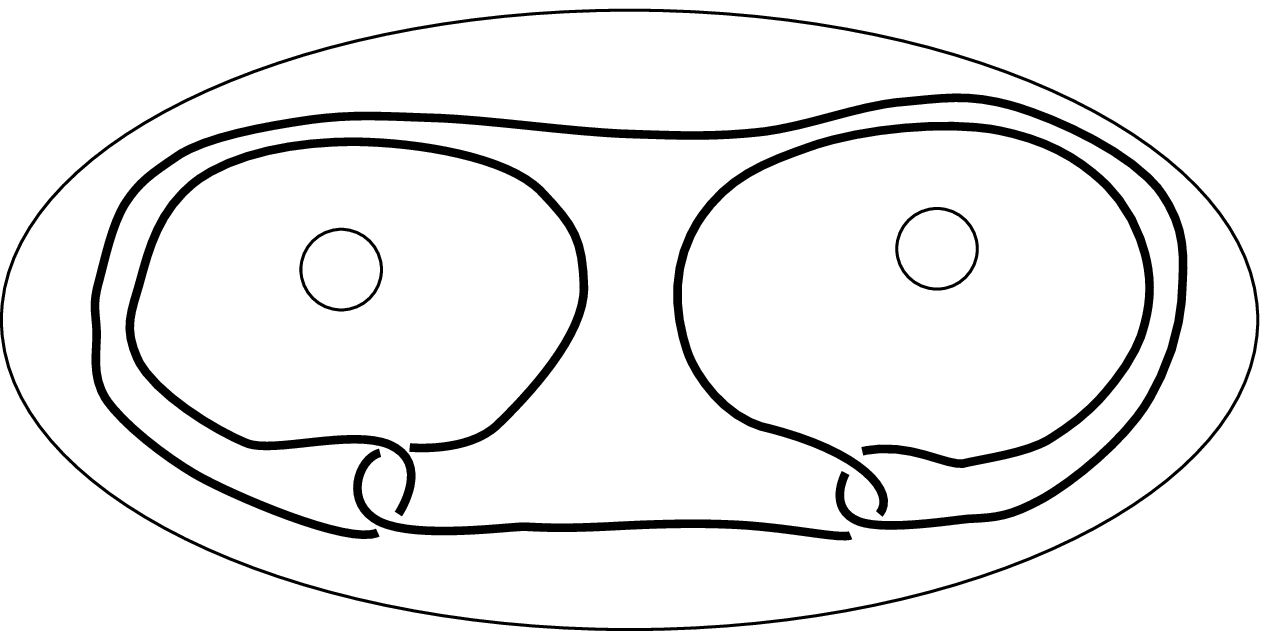} & \parbox[c]{4.5cm}{
\begin{center}$
\frac{-A{16} -2A^8 -1}{ A^{10} +A^6 }
$\end{center}
}\\
\hline
(10) & \picw{6.3}{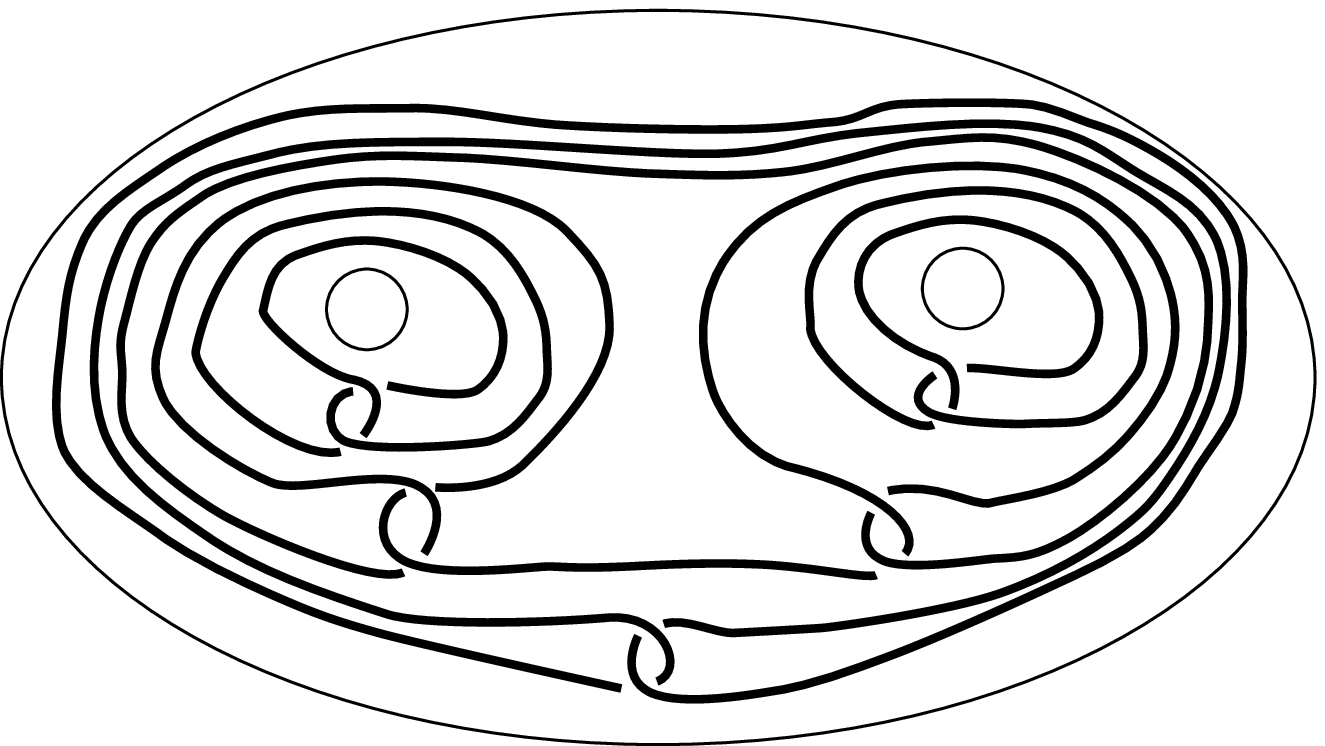} & \parbox[c]{4.5cm}{
\begin{center}$
( 8A^{56} - 15A^{52} + 46A^{48} - 45A^{44} + 106A^{40} -472A^{36} + 97A^{32} - 24A^{28} + 47A^{24} - 14 A^{20} + 22A^{16} - A^{12} + 9A^8 + 2A^4 +1 )/( A^{38} +2A^{34} + 3A^{30} + 3A^{26} + 2A^{22} +A^{18} )
$\end{center}
}\\
\hline
\end{array}
$$

$$
\begin{array}{|c|c|c|}
\hline
(11) & \picw{6.3}{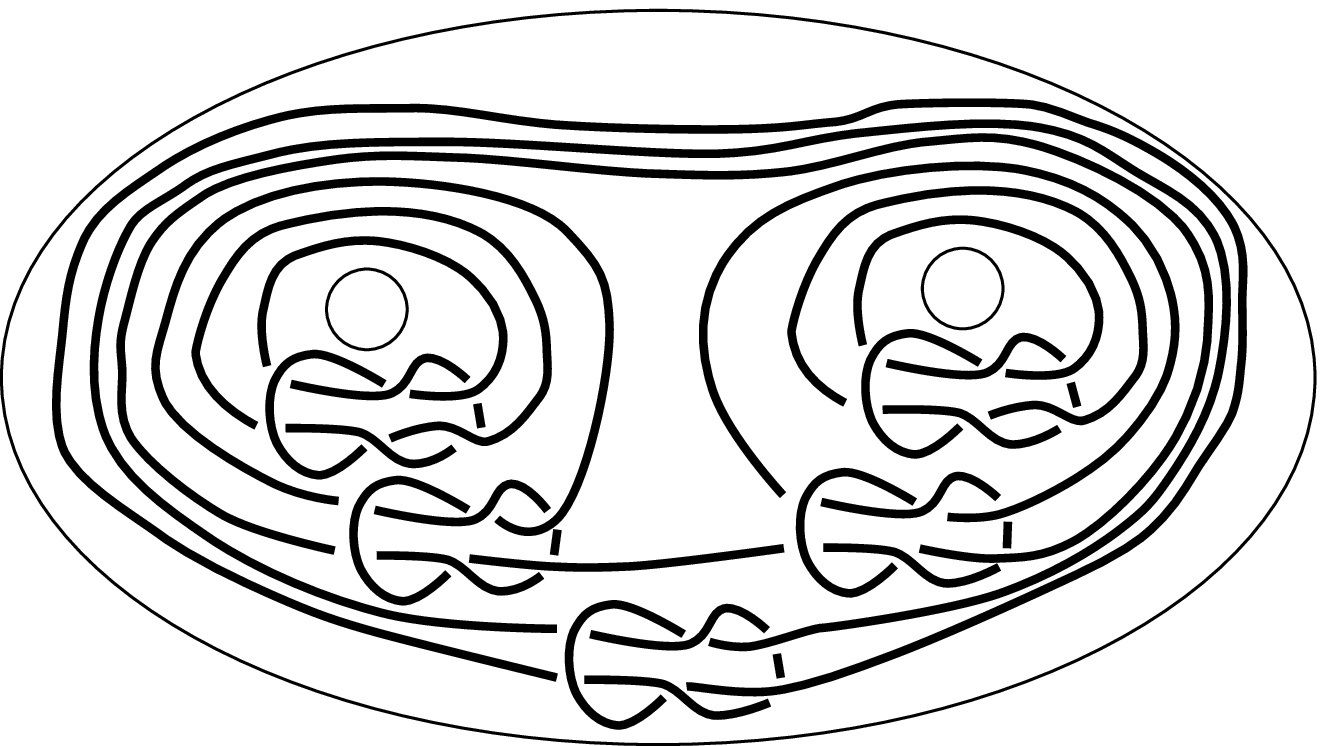} & \parbox[c]{4.5cm}{
\begin{center}$
( -A^{104} + 7A^{100} - 20 A^{96} + 42 A^{92} - 69 A^{88} + 61 A^{84} + 11 A^{80} - 124 A^{76} + 260 A^{72} - 317 A^{68} + 154 A^{64} + 87 A^{60} - 323 A^{56} + 512 A^{52} - 323 A^{48} + 87 A^{44} + 154 A^{40} - 317 A^{36} +260 A^{32} - 124 A^{28} + 112 A^{24} + 61 A^{20} - 69 A^{16} + 42 A^{12} - 20 A^8 + 7A^4 -1 )/( A^{56} + A^{52} + A^{48} )
$\end{center}
}\\
\hline
\end{array}
$$
\end{ex}

\begin{defn}\label{defn:Z_2_tr}
We say that a link $L\subset M$ is $\Z_2$-\emph{homologically trivial} if its homology class with coefficients in $\Z_2=\Z/2\Z$ is null:
$$
0 = [L] \in H_1(M;\Z_2) .
$$
\end{defn}

\begin{prop}\label{prop:Z_2_tr}
Let $D\subset S_{(g)}$ be a diagram of the link $L\subset \#_g(S^1\times S^2)$ for a fixed e-shadow. The following facts are equivalent:
\begin{enumerate}
\item{the link $L$ is $\Z_2$-homologically trivial;}
\item{the link $L$ bounds an embedded (maybe not orientable) surface in $\#_g(S^1\times S^2)$;}
\item{every generic embedded 2-sphere in $\#_g(S^1\times S^2)$ intersects $L$ in an even number of points (maybe $0$);}
\item{every generic properly embedded arc in $S_{(g)}$ intersects $D$ in an even number of points (maybe $0$);}
\item{the splitting of $D$ with any state $s$ bounds an embedded surface in $S_{(g)}$.}
\end{enumerate}
\begin{proof}
($1.\Leftrightarrow 2.$) The equivalence between $1.$ and $2.$ is a well known fact in low-dimensional topology. 

($2.\Rightarrow 3.$) Suppose we have $2.$. Let $S$ be a generic embedded 2-sphere ($S$ intersects transversely $L$ in a finite number of points) and let $S_L$ be an embedded surface bounded by $L$. We can suppose that $S$ and $S_L$ are transverse, hence $S\cap S_L$ is the union of disjoint arcs and circles properly embedded in $S_L$. Hence $S\cap L$ consists of a pair of points for each arc of $S\cap S_L$.

($3.\Rightarrow 4.$) Suppose we have $3.$. A properly embedded arc in $S_{(g)}$ is generic if it intersects transversely $D$ in a finite number of points that are not crossings. Every properly embedded arc in $S_{(g)}$ gives a generic properly embedded disk in the 3-dimensional handlebody of genus $g$, hence gives a generic 2-sphere in $\#_g(S^1\times S^2)$. The intersections of $D$ with the arc correspond to the intersections of $L$ with the sphere, hence this intersection is an even number of points.

($4.\Rightarrow 5.$) Suppose we have $4.$. The splitting of a crossing does not change the homology class with coefficients in $\Z_2$ of the represented link. Hence both $D$ and its splitting with a state $s$ represent a $\Z_2$-homologically trivial link in $\#_g(S^1\times S^2)$. Hence the splitting of $D$ with $s$ is a $\Z_2$-homologically trivial 1-sub-manifold even in the handlebody and in $S_{(g)}$, hence it bounds an embedded surface in $S_{(g)}$.

($5. \Rightarrow 2.$) Suppose we have $5.$. Let $s$ be a state of $D$ and $S_s$ an embedded surface in $S_{(g)}$ whose boundary is the splitting of $D$ with $s$. We get a surface in the handlebody of genus $g$ that is bounded by $L$ simply by attaching a half-twisted band to $S_s$ for each crossing.
\end{proof}
\end{prop}

\begin{rem}\label{rem:p(s)}
Let $D\subset S_{(g)}$ be a link diagram, let $s$ be a state of $D$, and $p(s)$ the number of homotopically non trivial components of the splitting of $D$ with $s$. If $g=1$ the components of $D_s$ are parallel and by Proposition~\ref{prop:Z_2_tr}-(5.) we have that $p(s)\in 2\Z$ if and only if the link is $\Z_2$-homologically trivial. For $g\geq 2$ this is not true.
\end{rem}

The identities in Fig.~\ref{figure:Cheb1} and Fig.~\ref{figure:Cheb} are clearly the same, we present both in order to easily explain the next lemma. The identity is a global relation in a solid torus and works only for parallel (colored) copies of the core with trivial framing. 

\begin{figure}
\begin{center}
\includegraphics[width = 5.3 cm]{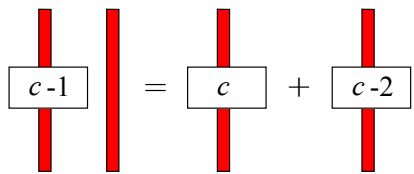}
\caption{A move in the skein space of the solid torus ($c \geq 2$). Only a portion of the solid torus is pictured, the portion is diffeomorphic to $[-1,1]\times D^2$.}
\label{figure:Cheb1}
\end{center}
\end{figure}

\begin{lem}[\cite{Carrega_Tait1}]\label{lem:parallel_cores}
Let $L\subset S^1\times S^2$ be the framed link consisting of $k\geq 0$ parallel copies of the core $S^1\times \{x\}$ with the trivial framing (the one given by a fixed e-shadow). Then $\langle L \rangle =0$ if $k\in 2\Z+1$, otherwise it is a positive integer.
\begin{proof}
This proof is based on the description of a computation. As an example the case $k=3$ is shown in Fig.~\ref{figure:ex_lem}. Let $K$ be the core of $S^1\times S^2$ with the trivial framing. In each step we get a linear combination with positive integers of framed links consisting of colored copies of $K$: in each framed link there is one copy with a non negative color, while the others are colored with $1$. Applying the equality of Fig.~\ref{figure:Cheb1} to each summand we fuse two components, one of them has color $1$ and the other one has the maximal color of that framed link. We apply this equality until we get a linear combination with positive integer coefficients of links consisting just of one colored copy of $K$. The colors of the final summands are all odd if $k\in 2\Z+1$, otherwise they are all even and the coefficient that multiplies the empty set (the copy colored with $0$) is non null. The equality of Fig.~\ref{figure:sphere}-(left) says that all the summands except the one with color $0$ are null.
\end{proof}
\end{lem}

\begin{figure}
\begin{center}
\includegraphics[width = 5.6 cm]{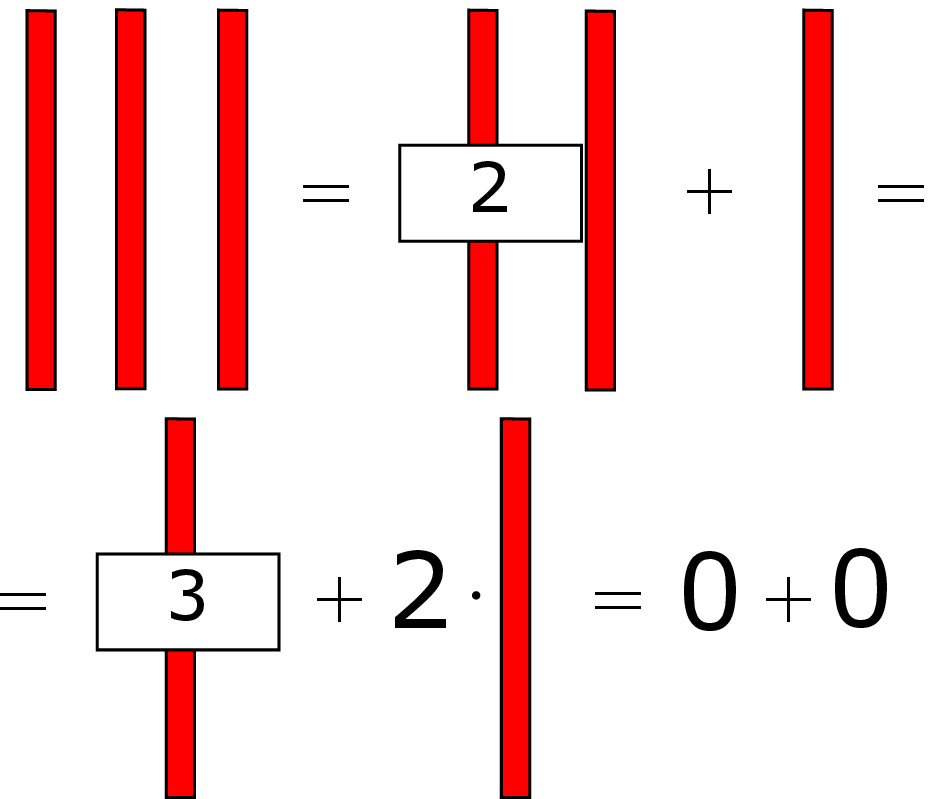}
\caption{The computation of the Kauffman bracket of three parallel copies of the core of $S^1\times S^2$.}
\label{figure:ex_lem}
\end{center}
\end{figure}

\begin{prop}
Let $L\subset S^1\times S^2$ be the framed link consisting of $2n\geq 0$ parallel copies of the core $S^1\times \{x\}$ with the trivial framing (the one given by a fixed e-shadow). Then 
$$
\langle L \rangle = \frac{1}{n+1} \binom{2n}{n} .
$$
\begin{proof}
See \cite[Corollary 5]{HP2}.
\end{proof}
\end{prop}

\begin{prop}[\cite{Carrega_Tait1}]\label{prop:Laurent_pol}
The Kauffman bracket of a link $L$ in $S^1\times S^2$ is a Laurent polynomial
$$
\langle L \rangle \in \Z[A,A^{-1}] .
$$
\begin{proof}
If follows from Proposition~\ref{prop:state_sum} and Lemma~\ref{lem:parallel_cores}.
\end{proof}
\end{prop}

\begin{cor}\label{cor:col_link_int_pol}
The Kauffman bracket of a colored link $L$ in $S^1\times S^2$ (or in $S^3$) is a Laurent polynomial
$$
\langle L \rangle \in \Z[A,A^{-1}] .
$$
\begin{proof}
It follows by induction on the biggest color of the components of $L$ using the identity in Fig.~\ref{figure:Cheb} and Proposition~\ref{prop:Laurent_pol} as base for the induction.
\end{proof}
\end{cor}

Let $D\subset S_{(g)}$ be a link diagram and $s$ a state of $D$. We know that if $g=0$ the diagram $D_s$ is empty. By Lemma~\ref{lem:parallel_cores} $\langle D_s \rangle$ is a positive integer. More in general in Remark~\ref{rem:sh_for_br} we will prove that $\langle D_s \rangle$ is a symmetric function of $A^2$, namely there are two polynomials $f,h\in \Z[q]$ such that
$$
\langle D_s \rangle = \frac{f|_{q=A^2}}{h|_{q=A^2}} , \ \ \langle D_s \rangle |_A = \langle D_s \rangle |_{A^{-1}} .
$$
In particular using Proposition~\ref{prop:state_sum} we get the following:
\begin{prop}[\cite{Carrega_Taitg}]\label{prop:func_of_A^2}
Let $D\subset S_{(g)}$ be a $n$-crossing link diagram. Then there are two polynomials $f,h\in \Z[q]$ such that
$$
\langle D \rangle = \begin{cases}
\frac{f|_{q=A^2}}{h|_{q=A^2}} & \text{if } n \in 2\Z \\
A\cdot \frac{f|_{q=A^2}}{h|_{q=A^2}} & \text{if } n \in 2\Z+1 \\
\end{cases} .
$$
\end{prop}

We are going to get some more information about $\langle D_s \rangle$, and hence about $\langle D\rangle$, in Chapter~\ref{chapter:Tait} (see the study of the quantity $\psi(s)$).

\begin{prop}[\cite{Carrega_Tait1}]\label{prop:0Kauff}
Let $L$ be a framed link in $S^1\times S^2$. Suppose that the $\Z_2$-homology class of $L$ is non trivial
$$
0 \neq [L] \in H_1(S^1\times S^2; \Z_2) .
$$
Then
$$
\langle L \rangle = 0 .
$$
\begin{proof}
Let $D\subset S^1\times [-1,1]$ be a diagram of $L$ for an e-shadow. By Remark~\ref{rem:p(s)} for every state $s$ of $D$ the number $p(s)$ of homotopically non trivial components of the splitting of $D$ with $s$ is odd. Therefore by Lemma~\ref{lem:parallel_cores} all the summands of $\langle L \rangle$ in Proposition~\ref{prop:state_sum} are null.
\end{proof}
\end{prop}

Proposition~\ref{prop:0Kauff} is not true for colored links. In fact the knot in Fig.~\ref{figure:ex_knot_Z2_non_tr} is $\Z_2$-homologically non trivial but if we assign it the color $2$, its bracket becomes $A^4 + 3 + A^{-8} +A^{-10} +A^{-12}$.

\begin{figure}[htbp]
\begin{center}
\includegraphics[scale=0.55]{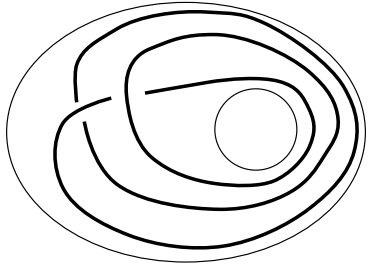}
\end{center}
\caption{A $\Z_2$-homologically non trivial knot in $S^1\times S^2$ such that the Kauffman bracket of it colored with $2$ is $A^4 + 3 + A^{-8} +A^{-10} +A^{-12}$.}
\label{figure:ex_knot_Z2_non_tr}
\end{figure}

In Proposition~\ref{prop:0Kauff} we focused just on links in $S^1\times S^2$ and we used very simple tools. The following is the analogous in $\#_g(S^1\times S^2)$ and considers also colored graphs. We present two different proofs. The first one is based on the manipulation of colored graphs and the identities seen in Subsection~\ref{subsec:identities}. The second one is based on Turaev's \emph{shadows} and the \emph{shadow formula} that we will introduce in Chapter~\ref{chapter:shadows}.
\begin{prop}[\cite{Carrega_Taitg}]\label{prop:0Kauff_gr}
Let $G$ be a knotted framed trivalent colored graph in $\#_g(S^1\times S^2)$ (\emph{e.g.} a framed link) and let $L$ be the sub-link of $G$ obtained joining the odd edges of $G$ (if $G$ is an uncolored framed link $L=G$). Suppose that $L$ is $\Z_2$-homologically non trivial
$$
0 \neq [L] \in H_1(\#_g(S^1\times S^2); \Z_2) .
$$
Then
$$
\langle G \rangle = 0 .
$$
\begin{proof}[Proof 1]
Let $S_1,\ldots, S_g$ be $g$ disjoint embedded 2-spheres that intersect $G$ in a finite number of points and such that the complement $\#_g(S^1\times S^2) \setminus (S_1, \ldots , S_g)$ is a connected contractible manifold. Let $c_j$ be the sum of the colors of the edges of $G$ intersecting $S_j$ counted with the multliplicity (\emph{e.g.} if the edge $e$ intersects $S_j$ twice its color must be counted twice). The parity of $c_j$ is equal to the number of odd edges of $G$ that intersect $S_j$ counted with the multiplicity. The link $L$ is $\Z_2$-homologically non trivial if and only if that number is odd for at least one $j$. Therefore by Remark~\ref{rem:graph_and_sphere} the skein of $G$ is null.
\end{proof}
\begin{proof}[Proof 2]
Let $X$ be a shadow of $(G, \#_g(S^1\times S^2))$ collapsing onto a graph. The 4-dimensional thickening of $X$ is the 4-dimensional handlebody of genus $g$, $W_g$ (the oriented compact 4-manifold with a handle-decomposition with just $k$ 0-handles and $k+g-1$ 1-handles). There is a graph $\Gamma \subset W_g$ such that $W_g\setminus \Gamma$ is a collar of the boundary, namely it is diffeomorphic to $\#_g(S^1\times S^2) \times [0,1) $. The homology class of $L$ in $H_1(W_g ;\Z_2)$ is $0$ if and only if $L$ bounds a surface $S\subset W_g$. By transversality we can suppose that $S$ does not intersect $\Gamma$. Hence $L$ is $\Z_2$-homologically trivial in $W_g$ if and only if it is so in $\#_g(S^1\times S^2) \times [0,1)$, thus in $\#_g(S^1\times S^2)$. 

Given an admissible coloring $\xi$ of $X$ that extends the one of $G$, the regions having
odd colors form a surface $S_\xi \subset W_g$ bounded by $L$. 

By hypothesis $L$ is $\Z_2$-homologically non trivial in $\#_g(S^1\times S^2)$. Hence for what said above a surface like $S_\xi$ can not exists. Therefore there are no admissible colorings of $X$ that extend the one of $G$. Hence by the shadow formula $\langle G \rangle = 0$.
\end{proof}
\end{prop}

\begin{prop}\label{prop:no_0Kauff}
Let $L\subset \#_g(S^1\times S^2)$ be a $k$-component homotopically trivial link. Then the evaluation in $A=-1$ of the Kauffman bracket is
$$
\langle L\rangle|_{A=- 1} = (-2)^k .
$$
Hence $\langle L \rangle \neq 0$. In particular the Kauffman bracket of every link in $S^3$ is non null.
\begin{proof}
Evaluate the Kauffman bracket at $A= -1$. The result is a (possibly infinite) number which does not distinguish the over/underpasses of the crossings:
$$
\left\langle \pic{1.2}{0.3}{incrociop2.eps} \right\rangle|_{A=- 1} =  -\left\langle \pic{1.2}{0.3}{Acanalep.eps} \right\rangle|_{A=- 1} - \left\langle \pic{1.2}{0.3}{Bcanalep.eps} \right\rangle|_{A=- 1}  = \left\langle \pic{1.2}{0.3}{incrociop.eps} \right\rangle|_{A=- 1} .
$$
Furthermore it does not change with a modification of the framing. Since $L$ is homotopically trivial, after a suitable change of the over/underpasses of the crossings, we can modify it to a link composed of $k$ trivial components in a 3-ball. Hence $\langle L \rangle|_{A=- 1} = (-A^3)^\alpha(-A^2-A^{-2})^k|_{A=- 1} = (-2)^k$, for some $\alpha\in \Z$ given by the framing.
\end{proof}
\end{prop}
 
\begin{rem}
By Proposition~\ref{prop:0Kauff_gr} and Proposition~\ref{prop:no_0Kauff} it is natural to ask if the Kauffman bracket of a framed link in $\#_g(S^1\times S^2)$ is null if and only if the link is $\Z_2$-homologically non trivial. The answer to this question is no. A $\Z_2$-homologically trivial knot in $S^1\times S^2$ whose Kauffman bracket is $0$ is shown in Fig.~\ref{figure:ex_0Kauf}. By Corollary~\ref{cor:conj_Tait_Jones_g}, such $\Z_2$-homologically trivial links can not have a connected, simple and alternating diagram in the disk with $g$ holes.
\end{rem}

\begin{figure}
\begin{center}
\includegraphics[scale=0.55]{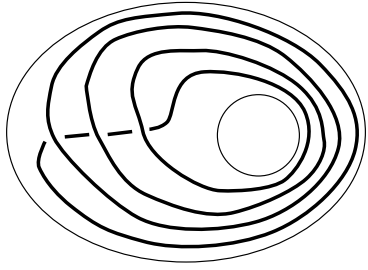} 
\end{center}
\caption{A $\Z_2$-homologically trivial knot in $S^1\times S^2$ whose Kauffman bracket is $0$.}
\label{figure:ex_0Kauf}
\end{figure}

\subsection{Links with the same bracket}

We know that there are knots in $S^3$ with the same Jones polynomial, for instance the ones in Fig.~\ref{figure:same_br_S3}. Hence it is natural to ask if there are links in $S^1\times S^2$ with the same Kauffman bracket that are not contained in a 3-ball.

\begin{figure}[htbp]
\begin{center}
\includegraphics[scale=0.5]{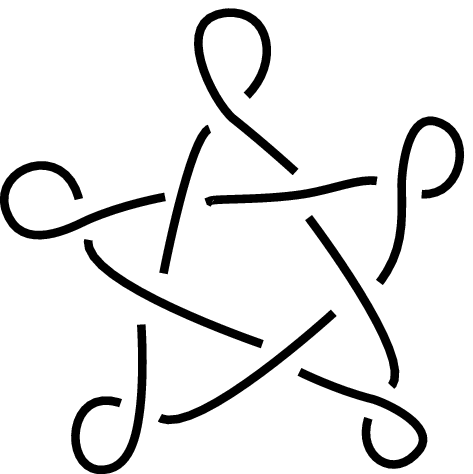}
\hspace{1cm}
\includegraphics[scale=0.5]{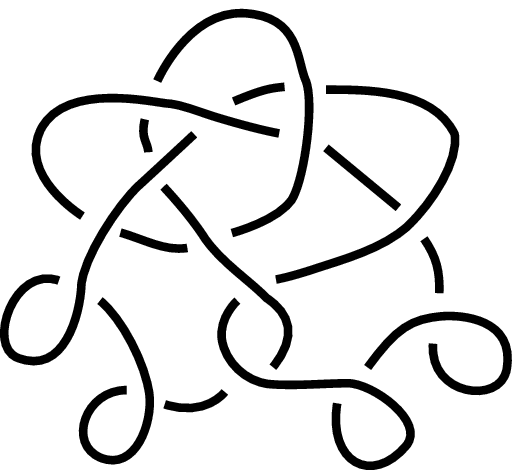}
\end{center}
\caption{Two different framed knots in $S^3$ with the same Kauffman bracket: $-A^{28} + A^{24} - A^{20} + A^{16} + A^8$. They are the knots $5_1$ and $10_{132}$.}
\label{figure:same_br_S3}
\end{figure}

The links $L_1$, $L_2$ and $L_3$ in Fig.~\ref{figure:ex_4cr} have very interesting proprieties: they have the same number of components, they are $\Z_2$-homologically trivial, they have the same Kauffman bracket, the rank of the first homology group of the complement is bigger than the number of components, they have very similar presentations of the fundamental group. The crossing number is $4$ because the \emph{breadth} of the Kauffman bracket is $16$ (Definition~\ref{defn:breadth}) and by Theorem~\ref{theorem:Tait_conj_Jones_g} if they had a lower crossing number their breadth would not be bigger than $12$.

We can distinguish $L_1$ and $L_2$ from $L_3$ noting that the components of $L_1$ and $L_2$ are all contained in a 3-ball, while there is a component of $L_3$ that is not. 

\begin{defn}
Let $K_1$ and $K_2$ be two components of a link in a orientable 3-manifold $M$. Suppose that $K_1$ bounds an orientable surface in $M$. Once an orientation to $M$, $K_1$ and $K_2$ is given, we define the \emph{linking number} $\textrm{lk}(K_1, K_2)$ as follows: take an oriented surface $S \subset M$ bounded by $K_1$ and inducing the proper orientation to the boundary and take the (algebraic) intersection number between $S$ and $K_2$.
\end{defn}

Note that the above notion of linking number is a natural generalization of the one in $S^3$. In particular we have that if $M = \#_g(S^1\times S^2)$ we can compute it by diagrams: give the proper orientation to the diagram, sum the signs of the crossings formed by strands both of projections of $K_1$ and $K_2$, then divide the result by $2$. This implies that if $M = \#_g(S^1\times S^2)$ the linking number is commutative $\textrm{lk}(K_1,K_2) = \textrm{lk}(K_2, K_1)$ if both components bound an orientable surface.

We can distinguish $L_1$ from $L_2$ by looking at the linking number. In fact the linking number of $L_1$ (of its components) is $\pm 1$, while the one of $K_2$ is $0$.

Here are the characteristics of $L_1$, $L_2$ and $L_3$:

\begin{itemize}

\item{$\Z_2$-homologically trivial: yes.}

\item{Homotopically trivial: yes.}

\item{Linking number: 
\begin{center}
\begin{tabular}{ccc}
$L_1$ & $L_2$ & $L_3$ \\
$\pm 1$ & $0$ & $\pm 1$
\end{tabular}.
\end{center}
}

\item{Alternating: 
\begin{center}
\begin{tabular}{ccc}
$L_1$ & $L_2$ & $L_3$ \\
yes & unknown & yes
\end{tabular}.
\end{center}
}

\item{Kauffman bracket: $\langle D \rangle = (-A^4-A^{-4})^2 $.}

\item{$\ord_{q=A^2=i} \langle D \rangle = 0$.}

\item{Hyperbolic: no.}

\item{Homology group: $H_1(S^1\times S^2 \setminus L; \Z) = \Z^3 $.}

\item{Fundamental group, generators: $a,b,c$, relators:
\begin{center}
\begin{tabular}{ccc}
$L_1$ & $L_2$ & $L_3$ \\
$aba^{-1}b^{-1}$ & $aba^{-1}b^{-1}$ & $aba^{-1}b^{-1}$ \\
$ac^{-1}bca^{-1}c^{-1}bc$ & $ac^{-1}bca^{-1}c^{-1}b^{-1}c$ & $ac^{-1}a^{-1}ca^{-1}c^{-1}ac$
\end{tabular}.
\end{center}
}

\end{itemize}

\begin{figure}[htbp]
$$
\begin{matrix}
\includegraphics[scale=0.45]{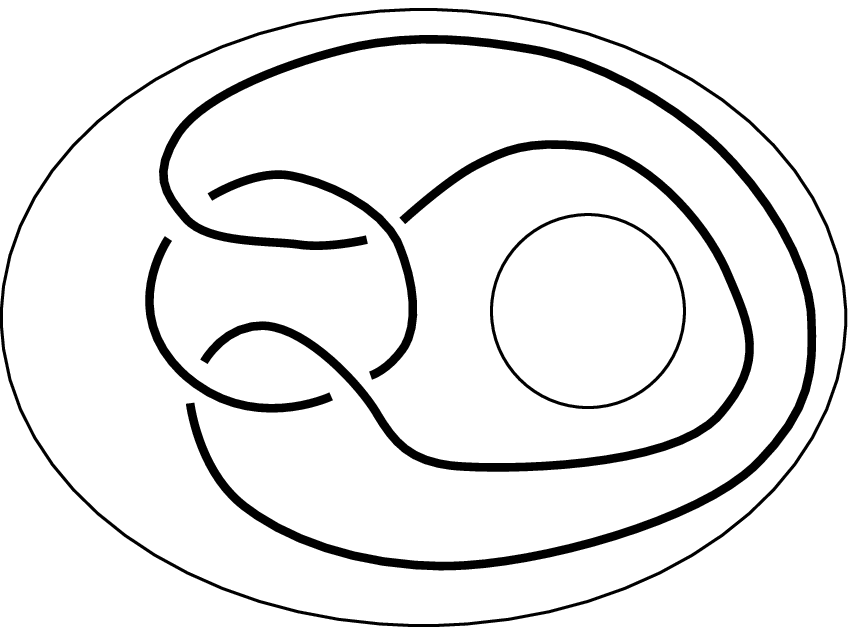}
&
\includegraphics[scale=0.45]{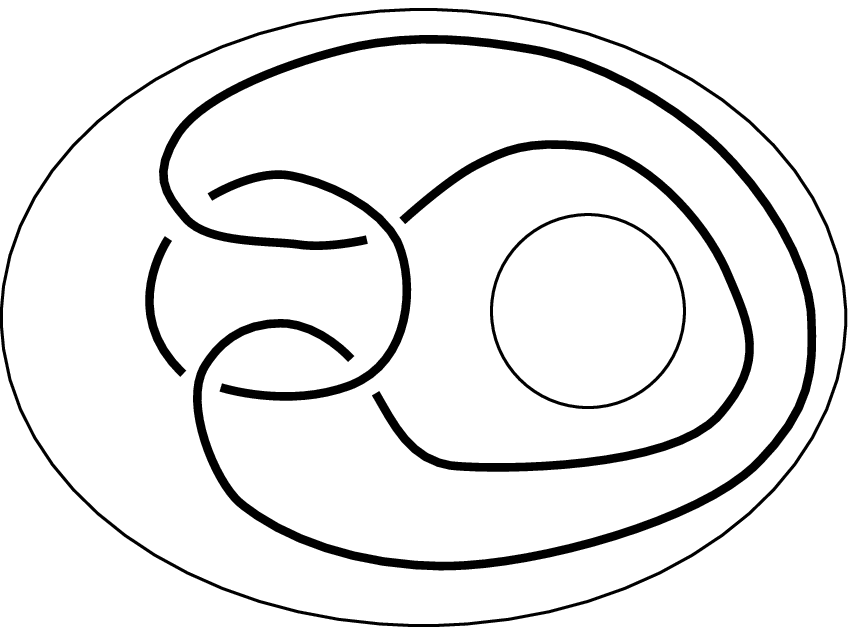}
&
\includegraphics[scale=0.45]{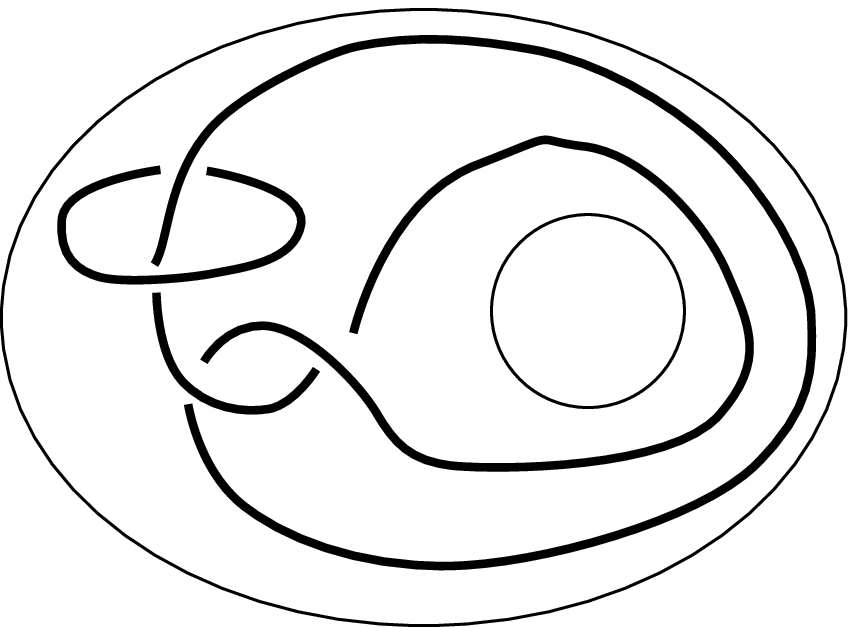}
\\
L_1 & L_2 & L_3
\end{matrix}
$$
\caption{Three links in $S^1\times S^2$ with crossing number $4$, not contained in a 3-ball and with the same Kauffman bracket: $(-A^4-A^{-4})^2$.}
\label{figure:ex_4cr}
\end{figure}

\section{$SU(2)$-Reshetikhin-Turaev-Witten invariants}\label{sec:RTW}

In this section we introduce another family of quantum invariants: the $SU(2)$-\emph{Reshetikhin-Turaev-Witten invariants}. We use again skein theory, this time we need an evaluation of the theory seen in Section~\ref{sec:skein_theory} at a root of unity. The letter $A$ is now a fixed complex number. Again the main references are \cite{Lickorish, Kauffman-Lins}.

\subsection{Skein theory}\label{subsec:quantum_skein_th}

Now we fix an integer $r\geq 3$ and a primitive $(4r)^{\rm th}$ root of unity $A\in \mathbb{C}$. We mean that $A^{4r}=1$ and $A^n\neq 1$ for each $0< n < 4r$, for instance we might take $A= e^{\pi i /2r}$. Actually we do not need that $A^n\neq 1$ for all $0< n <4r$, but just $A^{4r}=1$ and $A^{4n}\neq 1$ for all $0<n<r$. Although the constant $A$ is often omitted when defining quantum invariants, it is important to note that everything we will say depends on the choice of $A$ and not just on $r$. Furthermore we fix a square root of $A$ and of $-1$ that we denote respectively by $\sqrt A$ and $i$ (or $A^{\frac 1 2}$ and $\sqrt{-1}$ or $(-1)^{\frac 1 2}$). These two choices almost do not affect the result, it suffices to remember the initial choice and to be coherent.

For technical reasons in this skein theory we need to extend the notion of ``framed link'' and ``framed graphs'' (Subsection~\ref{subsec:Jones_pol} and Definition~\ref{defn:framed_graph}) considering also non orientable surfaces as framing for the link or the trivalent graph. If the framing is orientable we specify it. 

\begin{defn}\label{defn:skein_space}
Let $M$ be an oriented 3-manifold. Let $V$ be the abstract $\mathbb{C}$-vector space generated by all the (maybe not orientable) framed links in $M$ considered up to isotopies, including the empty set $\varnothing$. The $A$-\emph{skein vector space} $K_A(M)$, or $\mathbb{C}$-\emph{skein vector space}, is the quotient of $V$ by the following \emph{skein relations}:
$$
\begin{array}{rcl}
 \pic{1.2}{0.3}{incrociop.eps}  & = & A \pic{1.2}{0.3}{Acanalep.eps}  + A^{-1}  \pic{1.2}{0.3}{Bcanalep.eps}  \\
 D \sqcup \pic{0.8}{0.3}{banp.eps}  & = & (-A^2 - A^{-2})  D  \\
K & = & i A^{\frac 3 2} K_{-\frac 1 2}
\end{array}
$$
In all relations the links differ only in an oriented 3-ball (we need the orientation of $M$ here). Since we consider non orientable framings too we need to include the third relation that usually is not in the list. In the third relation $K$ is any framed knot and $K_{-\frac 12}$ is $K$ with its framing decreased by $-\frac 12$. The relation thus says that making a positive half-twists on any component of a link has the effect of multiplying everything by $i A^{\frac 3 2}$.

The elements of $K_A(M)$ are called \emph{skeins} or \emph{skein elements}.
\end{defn}

We can easily deduce that
$$
\pic{0.8}{0.3}{banp.eps} = (-A^2 -A^{-2}) \varnothing.
$$

\begin{rem}
By the third skein relation $K_A(M)$ is generated by the orientable framed links and we have
$$
K_A(M) := KM(M;\mathbb{C}, A) ,
$$
Let $f:\Z[A,A^{-1}] \rightarrow \mathbb{C}$ be the homomorphism of commutative rings defined by $f(A)=A\in \mathbb{C}$ (the $A$ in $\Z[A,A^{-1}]$ is the abstract variable of the Laurent polynomials, while if $A\in\mathbb{C}$ it is the fixed root of unity). Therefore by the universal coefficient property (Theorem~\ref{theorem:sk_mod_prop}-(5.)) we have an isomorphism of $\Z[A,A^{-1}]$-modules
$$
K_A(M) \cong KM(M) \otimes_{\Z[A,A^{-1}]} \mathbb{C} .
$$
where $K_A(M)$ has the structure of $\Z[A,A^{-1}]$-module induced by $f$.
\end{rem}

\begin{rem}
The homomorphism $f:\Z[A,A^{-1}] \rightarrow \mathbb{C}$ extends to a homomorphism of $\Z[A,A^{-1}]$-modules
$$
KM(M) \rightarrow K_A(M) .
$$
Therefore we can get a natural version in this theory of all the objects defined in the skein theory with integral Laurent polynomial, for instance the Kauffman bracket of a colored link in $S^3$ or in $S^1\times S^2$ (Corollary~\ref{cor:col_link_int_pol}). There are objects in $K(A)$ ($=KM(M, \mathbb{Q}(A),A)$ with $A$ abstract variable) that must have poles in the root of unity $A\in\mathbb{C}$, hence we can not define an analogue of these objects in $K_A(M)$, for instance the skein of $\teta\subset S^3$ with edges colored with the admissible triple $(r,r+2,2r-4)$ is well defined in $K(S^3)$ ($\teta_{(r,r+2,2r-4)} = (-1)^{2r-3} A^{12-6r} \frac{A^{16r-20} - A^{8r-12} -A^{8r-4} +1}{ A^{4r+4} -A^{4r} -A^4 +1 } $) but not in $K_A(S^3)$. For the elements of $K(M)$ that may have a well defined evaluation in the root of unity $A\in \mathbb{C}$ we have a natural vesion in $K_A(M)$.
\end{rem}

\begin{rem}
As before $\cerchio_n\in K_A(S^3) = \mathbb{C}$ is the closure of the $n^{\rm th}$ Jones-Wenzl projector $f^{(n)}$. Since $A$ is a $4r$-primitive root of unity, we have that $\cerchio_n \neq 0$ for $0\leq n \leq r-2$, and $\cerchio_{r-1}=0$. If $A= e^{\frac{\pi i}{2r}}$ then
$$
\cerchio_n = \frac{\sin\left( \frac{n\pi}{r}\right) }{\sin\left(\frac{\pi}{r}\right)}.
$$
\end{rem}

\begin{defn}\label{defn:q-admissible}
A triple of non negative inters $(a,b,c)$ is $q$-\emph{admissible} if it is admissible (Definition~\ref{defn:admissible}) and $a+b+c\leq 2(r-2)$. In particular it means that each color $a$, $b$, $c$, is at most $r-2$. A coloring of a trivalent graph $G$ is $q$-\emph{admissible} if the three numbers $a$, $b$, $c$ coloring the three edges incident to any vertex form a $q$-admissible triple.

A \emph{colored trivalent graph} is a framed knotted trivalent graph with a $q$-admissible coloring.
\end{defn}

\subsection{Important objects and identities}\label{subsec:q-imp_ob}

Here we introduce some important objects and identities of $A$-skein spaces that we need. We can find their proofs in \cite{Lickorish}. All the identities shown in Subsection~\ref{subsec:identities} work in this theory if we substitute the word ``admissible'' with ``$q$-admissible''. For instance the following holds:
$$
\pic{2}{0.7}{6j-symb1.eps} = \sum_i \left\{\begin{matrix} a & b & i \\ c & d & j \end{matrix}\right\}_q \pic{2}{0.7}{6j-symb2.eps}
$$
where the sum is taken over all $i$ such that the colored graphs shown are $q$-admissible. The coefficients between brackets are called \emph{quantum} $6j$-\emph{symbols} and we have
$$
\left\{\begin{matrix} a & b & i \\ c & d & j \end{matrix}\right\}_q = \frac{ \cerchio_i}{\teta_{a,d,i} \teta_{c,b,i}} \pic{2}{0.7}{tetra_color_ij.eps}  \quad \quad \quad \text{\cite[Page 155]{Lickorish}},
$$

The effect of a full twist is shown in Fig.~\ref{figure:framingchange}. More generally, if we change the framing of a framed trivalent graph by adding $k\in \frac 1 2 \Z$ positive twists on an edge colored with $n$, we get the skein of the previous graph times $(-1)^{nk} A^{nk(n+2)}$. To get this result we do not need hypothesis on $A$:
\begin{prop}[\cite{Carrega_RTW}]\label{prop:half-framingchange}
$$
\pic{2}{0.9}{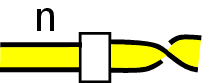} = i^n A^{\frac{n^2 + 2n}{2}} \ \pic{2}{0.9}{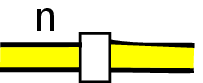}
$$
\begin{proof}
The case for $n=1$ is true for the third skein relation (see Definition~\ref{defn:skein_space}). Then we proceed by induction.
\beq
\pic{2}{0.9}{half-framingchange1.eps} & = & \pic{2.2}{0.9}{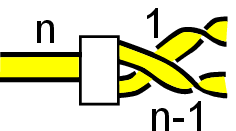} \\
 & = & i A^{\frac 3 2} \ \pic{2.2}{0.9}{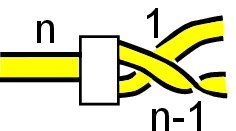} \\
 & = & i A^{n+ \frac 1 2} \ \pic{2.2}{0.9}{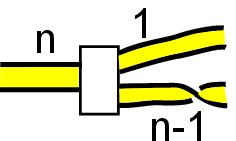} \\
 & = & i A^{n+ \frac 1 2} \ \pic{2}{0.9}{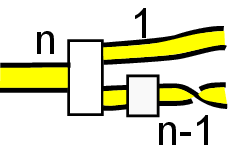} \\
 & = & i^n A^{\frac{n^2 + 2n}{2}} \ \pic{2}{0.9}{half-framingchange2.eps}
\eeq
We get the third equality by using $n-1$ times the first skein relation and the fact that the multiplication of a Jones-Wenzl projector with an element of the standard set of generators of $TL_n$ is $0$: $f^{(n)} \cdot e_i = 0$ for $1\leq i \leq n-1$. In fact this property allows us not to consider all the terms multiplied by $A^{-1}$ coming from the application of the first skein relation.
\end{proof}
\end{prop}

\begin{defn}
Consider the $A$-skein vector space $K_A(S^1 \times D^2)$ of the solid torus. Let $\phi_n\in K_A(S^1\times D^2)$ be the core of the solid torus, with trivial framing and color $n$. We now construct a particular element of $K_A(S^1\times D^2)$:
$$
\Omega := \eta\sum_{n=0}^{r-2} \cerchio_n \phi_n ,
$$
where
$$
\eta := \left( \sum_{n=0}^{r-2} \cerchio_n^2 \right)^{-\frac{1}{2}} = \frac{A^2-A^{-2}}{\sqrt{-2r}}  \quad \quad \quad \quad \text{\cite[Page 141]{Lickorish}}.
$$
If $A= e^{\frac{\pi i}{2r}}$ then
$$
\eta = \sqrt{\frac{2}{r}} \sin\left(\frac{\pi}{r} \right).
$$

If $K$ is a framed knot, we denote by $\Omega K$ the skein obtained by substituting $K$ with $\Omega$. Let $U$, $U_+$, and $U_-$ be the unknot in $S^3$ with framing respectively $0$, $1$, and $-1$. We have

$$
\Omega U = \eta^{-1} \quad \quad \quad \quad \text{\cite[Page 141]{Lickorish}}
$$
and we define
$$
\kappa := \Omega U_+ = \frac{\sum_{n=1}^{4r} A^{n^2}}{2r\sqrt{-2} A^{3+r^2}} \quad \quad \quad \quad \text{\cite[Lemma 14.3]{Lickorish}}.
$$
The latest equality holds if $A$ is a primitive $4r^{\rm th}$ root of unity. We do not know if it holds for all $A$ such that $A^{4r}=1$, $A^{4n}\neq 1$ for all $n<r$. If $A= e^{\frac{\pi i}{2r}}$ then we get
$$
\sum_{n=1}^{4r} A^{n^2} = 2\sqrt{2r} e^{\frac{\pi i}{4} } , \ \kappa = \frac{- i}{\sqrt{r}} e^{\frac{-\pi i}{4r}(2r^2-r +6) }  \quad \quad \text{\cite[Page 148]{Lickorish}}.
$$

Moreover it turns out that
$$
\kappa^{-1} = \Omega U_- \quad \quad \quad \quad \text{\cite[Lemma 13.7]{Lickorish}}.
$$
\end{defn}

There is a fundamental relation about $\Omega$ that is called the \emph{handleslide property} for pairs of $\Omega$'s and is shown in Fig.~\ref{figure:handleslideOmega} \cite[Lemma 13.5]{Lickorish}. This says that a Kirby move of the second type does not change the skein, provided that all components are colored with $\Omega$. This move substitutes a component with its bend sum with another component.

\begin{figure}[htbp]
$$
\pic{3}{0.8}{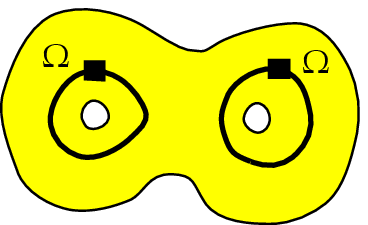} \ = \ \pic{3}{0.8}{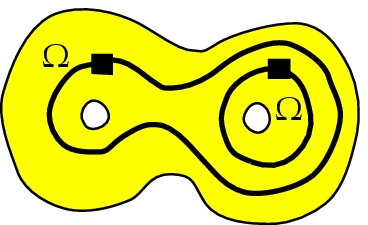}
$$
\caption{Handleslide property for pairs of $\Omega$'s.}
\label{figure:handleslideOmega}
\end{figure}

Two more properties are the 2- and 3-\emph{strand fusion identity} shown in Fig.~\ref{figure:2fusion} and Fig.~\ref{figure:3fusion} \cite[Page 159]{Lickorish}. In the left-hand side of the identities we have 2 or 3 strands colored with the $a^{\rm th}$, $b^{\rm th}$ and $c^{\rm th}$ projector. These strands are encircled by a 0-framed unknot colored with $\Omega$. They are a version with 2 and 3 strands of the identities of Fig.~\ref{figure:sphere}. Like Fig.~\ref{figure:sphere}, we get these identities using the fusion rule (Fig.~\ref{figure:fusion}) and the fact that the skein of a colored graph with a strand with a non null color that is encircled by a 0-framed unknot is 0 (Fig.~\ref{figure:sphere}-(left)). We can see this latest fact using the second Kirby move on the strand over the unknot (Subsection~\ref{subsec:surgery}) and the handleslide property, we get that the same strand with two more curls gives the same skein of the previous one, and then using the identity in Fig.~\ref{figure:framingchange}.

The 2-strand identity says that if $a=b$ then the left skein is equivalent to $\eta^{-1}\cerchio_a^{-1}$ times the skein obtained by removing the circle, breaking the strands and connecting them in the other way. Otherwise it is equivalent to $0$. The 3-strand identity says that if the triple $(a,b,c)$ is $q$-admissible then the left skein is equivalent to $\eta^{-1}\teta_{a,b,c}^{-1}$ times the skein obtained by removing the circle, breaking the strands and connecting the 3 parts on the same side in a vertex. Otherwise it is equivalent to $0$.

\begin{figure}[htbp]
$$
\pic{1.9}{0.8}{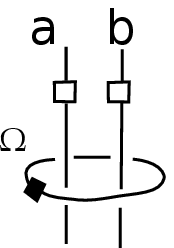} = \left\{\begin{array}{cl}
\frac{\eta^{-1}}{\cerchio_a} \ \pic{1.9}{0.8}{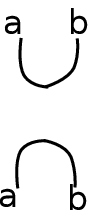} & \text{if }a=b \\
 0 & \text{if }a\neq b
\end{array}\right.
$$
\caption{The 2-strand fusion identity.}
\label{figure:2fusion}
\end{figure}

\begin{figure}[htbp]
$$
\pic{1.9}{0.8}{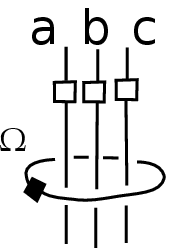} = \left\{\begin{array}{cl}
\frac{\eta^{-1}}{\teta_{a,b,c}} \ \pic{1.9}{0.8}{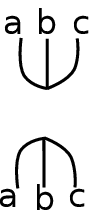} & \text{if }(a,b,c)\text{ is $q$-admissible} \\
 0 & \text{if }(a,b,c)\text{ is not $q$-admissible}
\end{array}\right.
$$
\caption{The 3-strand fusion identity.}
\label{figure:3fusion}
\end{figure}

\subsection{Surgery presentations}\label{subsec:surgery}

Given an orientable framed link $L$ in a closed oriented 3-manifold $M$ we can construct another closed oriented 3-manifold called the \emph{Dehn surgery} on $L$, in the following way. For each component $L_i$ of $L$ we remove the interior of a closed tubular neighborhood $N_i \cong D^2\times S^1$ such that the framing of $L_i$ corresponds to $\{(1,0)\} \times S^1$, and then we glue a solid torus $V_i\cong S^1\times D^2$ to the boundary of $N_i$ via a diffeomorphism of the boundaries that sends a meridian  $ \{y\} \times S^1$ of $V_i$, to the framing of $L_i$.

\begin{defn}\label{defn:surgery_pres}
Let $M$ and $N$ be two closed 3-manifolds. A \emph{surgery presentation} of $M$ in $N$ is an orientable framed link $L \subset N$ such that $M$ is obtained from $N$ by Dehn surgery on $L$.
\end{defn}

\begin{rem}
Once two simple closed curves in the torus $T$ that generate $\pi_1(T)$ are fixed, the isotopy classes of homotopically non trivial simple closed curves in the torus are in bijection with the extended rational numbers $\mathbb{Q} \cup \{\infty\}$. The boundary of a tubular neighborhood of a component of a link in $S^3$ has two natural generators of the fundamental group: the boundary of a properly embedded disk in the solid torus that is not homotopically trivial, and the Seifert framing. The one in Definition~\ref{defn:surgery_pres} is the notion of ``\emph{surgery presentation with integer coefficients}''. In fact if we change the framing of a component adding $n$ full twists, we pass from the curve $p/q$ to the curve $p/q +n$. We could give the definition of surgery presentation for any rational number extending the one we have seen. Only the surgery presentation with integer coefficients has a 4-dimensional intepretation.
\end{rem}

\begin{rem}
Each surgery presentation in $S^3$ has a 4-dimensional interpretation. In fact with a Dehn surgery we can build not only a 3-manifold $M_L$, but also a 4-manifold $W_L$ whose boundary is $M_L = \partial W_L$. It suffices to see $S^3$ as the boundary of the 4-ball, and then attach a 4-dimensional 2-handle $B_i \cong D^2\times D^2$ along the boundary of a tubular neighborhood $N_i \cong D^2\times S^1$ of each component $L_i$ of the link in the way described above. In fact the boundary of $B_i$ is the union of $N_i$ and $V_i$.
\end{rem}

\begin{theo}[Lickorish, Wallace]\label{theorem:Lickorish-Wallace}
Every orientable closed 3-manifold has a (integer) surgery presentation in $S^3$.
\end{theo}

There are two important moves on framed links called  \emph{Kirby moves}. The first one consists in adding a new separated component $U_{\pm}$, that is an unknot with framing $\pm 1$. This corresponds to the connected sum with $S^3$ or (in the 4-dimensional interpretation) to the connected sum with $\mathbb{CP}^2$. The second one is exactly the handleslide that we described above in Fig.~\ref{figure:handleslideOmega}. In the 4-dimensional interpretation, it corresponds to sliding a 2-handle over another one. Both Kirby moves do not change the presented 3-manifold.

\begin{theo}[Kirby]
Two (integer) surgery presentations $L$ and $L'$ in $S^3$ of the same 3-manifold $M$, are related by isotopies and Kirby moves.
\end{theo}

\subsection{Definition}

\begin{defn}\label{defn:linking_matrix}
Let $L\subset S^3$ be an oriented framed link. Let $K_1$ and $K_2$ be two different components of $L$. The \emph{linking number} of $K_1$ and $K_2$, ${\rm lk}(K_1,K_2)$, can be defined in several ways, one of them consists in counting the signs of their intersections in a diagram:
$$
{\rm lk}(K_1,K_2) := \frac 1 2 \sum_{x \text{ in } K_1\cap K_2} {\rm sgn}(x),
$$
where $x$ runs over the the crossings of a diagram $D$ of $L$ that are composed both of strands of $K_1$ and of strands of $K_2$, and ${\rm sgn}(x)$ is the sign of that crossing (Fig.~\ref{figure:crossingsign}). The \emph{writhe number} of $K_1$, $w(K_1)$, is
$$
w(K_1) := \sum_{x \text{ in } K_1} {\rm sgn}(x) ,
$$
where $x$ runs over all the crossings of the diagram $D$ whose strands are all of $K_1$.

Let $h$ be the number of components of $L$ and let $K_1,\ldots , K_h$ be the components of $L$. The \emph{linking matrix} $M(L)$ is the $h\times h$ matrix whose $(i,j)$ entry is
$$
(M(L))_{i,j} := \begin{cases}
{\rm lk}(K_i,K_j) & \text{if } i\neq j \\
w(K_i) & \text{if } i= j
\end{cases} .
$$ 
The linking matrix is an invariant of oriented and framed links in $S^3$ and its signature $\sigma(L)$ does not depend neither on the orientation nor on the framing.
\end{defn}

\begin{defn}\label{defn:RTW}
Let $L$ be a surgery presentation in $S^3$ of the closed orientable 3-manifold $M$. The \emph{Reshetikhin-Turaev-Witten} invariant of $M$ (as defined by Lickorish \cite{Lickorish}) is:
$$
I_r(M) := \eta \kappa^{-\sigma(L)} \Omega L,
$$
where $\Omega L$ is the skein element obtained by attaching $\Omega$ to each component of $L$, and $\sigma(L)$ is the signature of the linking matrix of $L$.
\end{defn}

The quantity $\sigma(L)$ is equal to the signature of the 4-manifold obtained attaching to $D^4$ a 2-handle along each component of $L$. The complex number $I_r(M)$ is a topological invariant. In fact it is clearly invariant by isotopies of links, from the handleslide property we get the invariance under the second Kirby move, and the factor $\kappa^{-\sigma(L)}$ ensures the invariance under the first Kirby move.

\begin{ex}
The connected sum $\#_g(S^1\times S^2)$ of $g$ copies of $S^1\times S^2$ is presented in $S^3$ by the unlink with $g$ components, each one with the 0-framing. Hence $\sigma(L)=0$ and $\Omega L = \eta^{-g}$. Therefore
$$
I_r(S^3)= \eta , \ I_r(S^1\times S^2)= 1 , \ I_r(\#_g(S^1\times S^2)) = \eta^{1-g} .
$$
\end{ex}
\begin{ex}
The unknot with framing $n\in\Z$, $U_n$,  presents the lens space $L(n,1)$. The linking matrix is the $1\times1$ matrix $(n)$. Hence $\sigma(U_n)={\rm sgn}(n)$. Using equality in Fig.~\ref{figure:framingchange} $n$ times we get
$$
I_r(L(n,1)) = \eta^2 \kappa^{-{\rm sgn}(n)} \sum_{a=0}^{r-2} \cerchio_a^2 (-1)^{an} A^{an(a+2)} .
$$
\end{ex}

\begin{prop}\label{prop:RTW_conn_sum}
Let $M$ and $N$ be two closed oriented 3-manifolds. Then the invariant of the connected sum is
$$
I_r(M \# N) = \eta^{-1} I_r(M) I_r(N) .
$$
\begin{proof}
Let $L_M \subset S^3$ and $L_N \subset S^3$ be surgery presentations of $M$ and $N$. The link $L_M \sqcup L_N$ that is obtained gluing the boundaries of two 3-balls, one containing $L_M$ and the other one containing $L_N$, is a surgery presentation of the connected sum $M \# N$ in $S^3$. This link has a diagram in the plane that is the disjoint union of a diagram of $L_M$ and a diagram of $L_N$. Hence the $\Omega (L_M\sqcup L_N) = \Omega L_M \Omega L_N$ and $\sigma(L_M \sqcup L_N) = \sigma(L_M) + \sigma(L_N) $. By definition follows that $I_r(M \# N)= \eta^{-1} I_r(M) I_r(N)$. 
\end{proof}
\end{prop}
Later we will prove that the invariant of the 3-manifold obtained changing the orientation is the conjugate of the previous number (Proposition~\ref{prop:RTW_conj}).

\begin{defn}\label{defn:RTW_G}
Let $G$ be a $q$-admissible colored framed trivalent graph in the closed oriented 3-manifold $M$, and let $L$ be a surgery presentation of $M$ in $S^3$. Let $G'$ be the colored framed trivalent graph of $S^3$ that corresponds to $G$ by the surgery on $L$. If $r$ is bigger equal than the biggest color of the edges of $G$, we define
$$
I_r(M,G) := \eta \kappa^{-\sigma(L)} \Omega (L,G') ,
$$
where $\Omega(L,G')$ is the skein element obtained by coloring the components of $L$ with $\Omega$ and the edges of $G'$ with the projectors corresponding to the colors.
\end{defn}

\begin{rem}
Clearly $I_r$ is an invariant for $q$-admissible colored framed trivalent graphs and if $G\subset M$ is empty $I_r(M,G)= I_r(M)$. If $G$ has no vertices it is a colored framed link. We get a family of invariants of non colored framed links just by varying $r$ and coloring all the link components with $r-2$. If the surgery link $L$ is empty, we have $M=S^3$ and we can copy Kauffman's construction of the Jones polynomial to obtain also a family of invariants for oriented links. It suffices to multiply our invariants for framed links in $S^3$ by $((-1)^{r-2}A^{(r-2)^2+2(r-2)})^{-w(L')}$, where $w(L')$ is the sum of the signs of a diagrammatic representation of $L'$ ($= G'$) (the \emph{writhe number}). In this case, $L=\varnothing$, $M=S^3$, we get an evaluation of the Kauffman bracket or the colored Jones polynomial for each $r$.
\end{rem}

\begin{rem}\label{rem:Turaev-Viro}
Another famous family of quantum invariants for 3-manifolds consists of the \emph{Turaev-Viro invariants} $TV_r(M)$ \cite{Turaev-Viro} (Definition~\ref{defn:Turaev-Viro}). They are still based on the choice of a $4r^{th}$ primitive root of unity $A$. If the 3-manifold $M$ is closed we can compare the two families and get that the Reshetikhin-Turaev-Witten invariants are sharper:
$$
TV_r(M) = |I_r(M)|^2  \ \ \ \ \text{(Theorem~\ref{theorem:Turaev-Viro})} ,
$$
where $|x|^2$ is the square of the absolute value of the complex number $x$.
\end{rem}

\chapter{Shadows}\label{chapter:shadows}

In this section we introduce Turaev's \emph{shadows} and the \emph{shadow formula} for the Kauffman bracket and the $SU(2)$-Reshetikhin-Turaev-Witten invariants. Shadows are 2-dimensional polyhedral objects related to smooth 4-manifolds: these are the 4-dimensional analogue of \emph{spines} of 3-manifolds. They were defined by Turaev \cite{Turaev:preprint, Turaev} and then considered by various authors, see for instance \cite{Burri, Carrega_RTW, Carrega_Taitg, Carrega-Martelli, Costantino1, Costantino2, Costantino-Thurston, Costantino-Thurston:preprint, Goussarov, IK, Martelli, Shu, Thurston, Turaev1}. Shadow formulas are ways to compute quantum invariants via shadows. They look like Euler characteristics: they are composed by elementary bricks associated to the maximal connected pieces of dimension $0$ (vertices), $1$ (edges) and $2$ (regions) of the shadow, and they are combined together with a ``sign'' depending on the parity of the dimension. We follow \cite{Turaev}, \cite{Costantino0}, \cite{Carrega-Martelli} and \cite{Carrega_RTW}. Throughout this chapter we consider also non orientable surfaces as framings for links and knotted trivalent graphs.

In \cite{Turaev} Turaev defined shadows and showed how to get the quantum invariants from them through a formula that works in a general context: for any ribbon category, and hence for any quantum group. In this chapter we focus on the quantum invariants that can be obtained with skein theory: the Kauffman bracket and the $SU(2)$-Reshetikhin-Turaev-Witten invariants (see Chapter~\ref{chapter:quantum_inv}), and we reprove the shadow formula using skein theory.

\section{Generalities}

\begin{defn}
A \emph{polyhedron} $P$ is the topological space obtained by the union of all the simplices of a simplicial complex. The simplicial structure yielding P is called a \emph{triangulation}, and the polyhedron $P$ is also equipped by a
PL-structure, that is some equivalence class of triangulations. A \emph{sub-polyhedron} $Q$ of $P$ is a polyhedron such that there is a triangulation $T_P$ of $P$ and a triangulation $T_Q$ of $Q$ such that $T_Q$ is a sub-complex of $T_P$. 

Let $T$ be a simplicial complex and $\sigma$ a simplex of $T$ that is in the boundary of exactly one simplex $\sigma'$ of $T$. An \emph{elementary collapse} of $T$ is the removal of two simplexes like $\sigma$ and $\sigma'$ from $T$. A polyhedron $P$ \emph{collapses} onto a sub-polyhedron $Q\subset P$ if there is a triangulation $T_P$ of $P$ and a triangulation $T_Q$ of $Q$ such that $T_Q$ is a sub-complex of $T_P$ and after some elementary collapses $T_P$ becomes $T_Q$. Sometimes it is denoted $P \searrow Q$.
\end{defn}

We can switch from the smooth to the PL-category in low
dimensions, since they are equivalent. Namely each compact smooth manifold $M$ with dimension at most $4$ has a unique PL-structure.

\begin{defn}
A \emph{simple polyhedron} $X$ is a connected 2-dimensional compact polyhedron where every point has a neighborhood homeomorphic to one of the five types (1-5) shown in Fig.~\ref{figure:models}. The five types form sub-sets of $X$ whose connected components are called \emph{vertices} (1), \emph{interior edges} (2), \emph{regions} (3), \emph{boundary edges} (4) (or \emph{external edges}), and \emph{boundary vertices} (5) (or \emph{external vertices}). The points (4) and (5) altogether form the \emph{boundary} $\partial X$ of $X$. An edge is either an open segment or a circle. A region is a (possibly non-orientable) connected surface without boundary. The 1-\emph{skeleton} of $X$ is the union of the vertices and the edges. 
\end{defn}

\begin{figure}[htbp]
\begin{center}
\includegraphics[width = 12 cm]{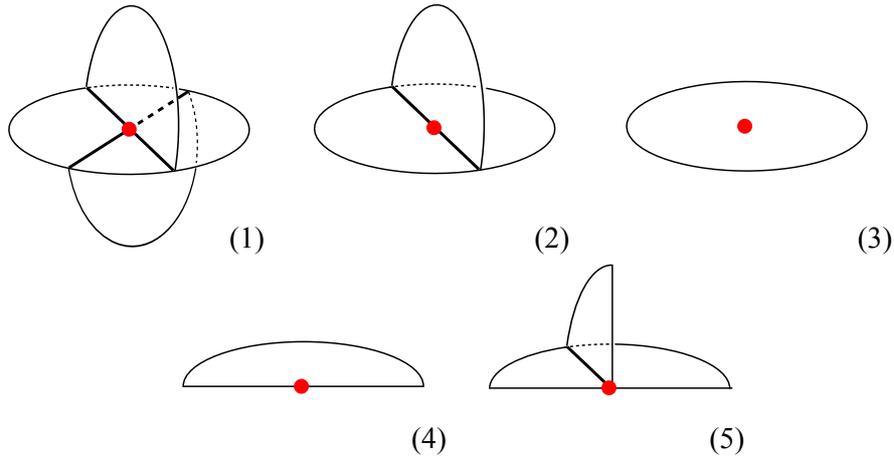}
\caption{Neighborhoods of points in a simple polyhedron.}
\label{figure:models}
\end{center}
\end{figure}

\begin{defn}\label{defn:spine}
Let $M$ be a compact 3-manifold with non empty boundary. A \emph{spine} $X$ of $M$ is a simple polyhedron without boundary that is embedded in the interior of $M$ and such that $M$ collapses onto $X$. If $M$ is closed a \emph{spine} of $M$ is a spine of $M$ minus a 3-ball.
\end{defn}

\begin{theo}
$\ $
\begin{enumerate}
\item{Every compact 3-manifold has a spine.}
\item{A simple polyhedron with only disk regions may be a spine of only one 3-manifold with non empty boundary.}
\end{enumerate}
\begin{proof}
See \cite{Fomenko-Matveev}.
\end{proof}
\end{theo}

\begin{rem}
It is not ture in general that two compact 3-manifolds with non empty boundary have homeomorphic spines if and only if they are homeomorphic.
\end{rem}

\begin{defn}\label{defn:shadow}
Let $W$ be a compact 4-manifold with boundary. A \emph{shadow} of $W$ is a simple polyhedron $X\subset W$ such that the following hold:
\begin{itemize}
\item{$X$ is properly embedded in $W$ ($X\cap \partial W = \partial X$);}
\item{$X$ is \emph{locally flat}, namely every point $p\in X$ has a neighborhood $U$ in $W$ diffeomorphic to $B^3 \times (-1,1)$, or $\{x\in\mathbb{R}^3\ |\ x_3\geq0 \} \times (-1,1)$, such that $U\cap X$ is contained in $B^3 \times 0$, or in $\{x\in\mathbb{R}^3\ |\ x_3\geq0 \} \times \{0\}$, as in Fig.~\ref{figure:models};}
\item{$W$ collapses onto $X$.}
\end{itemize}

Let $M$ be a closed 3-manifold and $G$ be a framed knotted trivalent graph in $M$. The polyhedron $X$ is a \emph{shadow} of $(M,G)$ if it is a shadow of a 4-manifold $W$ bounded by $M$ and its boundary coincide with the graph: $\partial W= M$, $\partial X = G$.

A shadow is said to be \emph{standard} if all the regions that are not adjacent to boundary edges are disks and the boundary edges are adjacent either to annuli or to disks.
\end{defn}

\begin{ex} 
The trivial ribbon disk $D\subset D^4$ is the disk that is properly embedded in the 4-disk $D^4$, is in Morse position with respect to the radiant map $D^4\rightarrow \mathbb{R}$ with one minimum, no saddles and no maxima. The boundary of $D$ is the unknot in $S^3$. The disk $D$ is itself a shadow of $D^4$. However, a non trivial properly embedded disk $D\subset D^4$ is not a shadow: $D^4$ does not collapse onto $D$ (see Proposition~\ref{prop:trivial}).
\end{ex}

\begin{prop} \label{prop:trivial}
A properly embedded disk $D\subset D^4$ is a shadow if and only if $D$ is isotopic to the trivial ribbon disk (and hence $\partial D$ is the unknot).
\begin{proof}
The disk $D$ is a shadow of $D^4$ if and only if $D^4$ collapses onto it. This holds if and only if $D^4$ is a regular neighborhood of $D$. A regular neighborhood of $D$ is a product bundle $D\times D^2$, hence $D^4$ collapses onto $D$ if and only if $D^4 =D\times D^2$. This holds precisely when $D$ is trivial.
\end{proof}
\end{prop}

Let $F$ be an oriented $D^2$-bundle (more precisely $F$ is oriented and is the total space of a $D^2$-bundle) over a closed (maybe not orientable) surface $S$. The oriented bundle $F$ is completely determined by $S$ and an integer number called the \emph{Euler number} of $F$. For every $g\in\Z$ there is a (unique) orientable $D^2$-bundle over $S$ with Euler number $g$. To get the Euler number of $F$ we take a properly embedded copy $S'$ of $S$ in $F$ that is transverse to $S$. The surfaces $S$ and $S'$ intersect in a finite number of points. Each point $p\in S\cap S'$ has a sign $\pm 1$ that is positive if given two positive basis, $(v_1,v_2)$ and $(v'_1,v'_2)$, of the tangent spaces $T_p(S)\subset T_p(F)$ and $T_p(S') \subset T_p(F)$ we have that $(v_1,v_2,v'_1,v'_2)$ is a positive basis of $T_p(W)$, otherwise it is negative. We do not need the hypothesis that the surface is oriented. In fact we can think of $S'$ as a perturbation of $S$ and we can suppose that the points of $S\cap S'$ are fixed points of the perturbation, then given a basis $(v_1,v_2)$ of $T_p(S)\subset T_p(W)$ we have a natural basis $(v'_1,v'_2)$ of $T_p(S')\subset T_p(W)$ and we use them to compute the sign. We can define the Euler number even if the surface $S$ is a compact surface with non empty boundary provided that the boundary $\partial S \subset \partial F$ has a framing. In that case one takes a perturbation $S'$ that is properly embedded and $\partial S'$ lies on the framing of $\partial S$. If we consider just closed surfaces, two different pairs of surface and Euler number determines two different orientation preserving diffeomorphism classes of $D^2$-bundles (the total space). This is not true for surfaces with non empty boundary. 

\begin{defn}
Let $X$ be a shadow of a compact 4-manifold $W$ and let $R$ be a region of $W$. If $R$ is not closed we define the \emph{boundary} $\partial R$ of $R$ as the boundary of a compact surface $\bar R \subset R$ whose interior is homeomorphic to $R$. Often in literature $\bar R$ is confused with $R$.
\end{defn}

The boundary of a region $R$ is an embedded union of circles, the union of the vertices and the edges adjacent to $R$ is just an immersion of $\partial R$.

Suppose that $\partial X \subset \partial W$ has a framing, then the shadow $X$ provides an interval sub-bundle of the normal bundle of $\partial R$ in $W$ as follows:

Let $(\partial R)_i \subset \partial R \subset  R$ be a connected component of $\partial R$. Let $A_i\subset X$ be the annulus that is immersed in $X$, its interior lies in $R$, one of its boundary component coincides with $(\partial R)_i$ and the other one is the union of some edges and vertices adjacent to $R$. Every point $p\in A_i$ has a neighborhood $V$ in $W$ with a chart $\varphi: V \rightarrow \mathbb{R}^4$ such that $\varphi(V \cap X) \subset \{x \in \mathbb{R}^4 \ | \ x_4=0\}$ and $(V\cap A_i, A_i) \cong ( (-1 , 1) \times [-1,1] , S^1\times [-1,1])$. Let $e$ be an edge of $X$ adjacent to $A_i$ and $p \in e$ a point of $e$. Let $V$ be a neighborhood of $p$ like the one described above such that with that trivialization $\varphi(V)\cap \{x_4=0\}$ is as in Fig.~\ref{figure:models}-(2). We treat in the same way the case in which $e$ is an external edge. In fact adding the framing we have a local trivialization of $e$ as in Fig.~\ref{figure:models}-(2) where two of the half-planes adjacent to $e$ are part of the framing of $\partial X$ and the other one is a part of $X$. The same will happen for the vertices. Thus $(X\setminus A_i) \cap V$ is an $I$-bundle over $e\cap V\cong (-1, 1) \times \{-1\}$ and we can naturally transport it to a part of $A_i$, the one corresponding to $\cong (-1, 1) \times \{1\}$. We naturally glue all these interval bundles associated to points of the same edge. If there are no vertices adjacent to $R$ (a boundary component of $A_i$ is a closed edge of $X$) we have constructed the $I$-bundle over $(\partial R)_i$. Let $v$ be a vertex adjacent to $R$ and let $e_1$ and $e_2$ be the two vertices adjacent to $v$ and to $R$ (maybe $e_1=e_2$). Let $V$ be a neighborhood of $v$ as before such that $\varphi(V)\cap \{x_4=0\}$ is as Fig.~\ref{figure:models}-(1), where $\varphi(V \cap A_i)$ is the half-plane lying in $\{x_3\geq 0, x_4=0\}$. Hence in this trivialization the $I$-bundles constructed using $e_1$ and $e_2$ lies in a plane with constant $x_3$ and $x_4=0$, thus they naturally glue together completing the $I$-bundle over $(\partial R)_i$. 

\begin{defn}
Let $W$ be a compact oriented 4-manifold. Let $R$ be a region of a shadow $X\subset W$ such that $\partial X\subset \partial W$ has a framing. The region $R$ is equipped with a half-integer called \emph{gleam}. This generalizes the Euler number of closed surfaces embedded in oriented 4-manifolds. The gleam is defined as follows:

Let $R'$ be a generic small perturbation of $R$ with $\partial R'$ lying in the interval bundle at $\partial R$. The surfaces $R$ and $R'$ intersect only in isolated points, and we count them with signs:
$$
\gl(R) := \frac 1 2 \# (\partial R \cap \partial R') + \#(R\cap R') \in \frac 1 2 \Z
$$
The half-integer $\gl(R)$ is the \emph{gleam} of $R$ and does not depend on the chosen $R'$. Note that the contribution of $\partial R\cap \partial R'$ on a component of $\partial R$ is even or odd depending on whether the interval bundle above it is an annulus or a M\"obius strip.

We say that a (abstract) simple polyhedron $X$ is a \emph{shadow} if $\partial X$ is equipped with a surface that collapses on it that we call \emph{framing}, and each region $R$ of $X$ is equipped with a half-integer $\gl(R)$, called \emph{gleam}, such that $\gl(R) \in \Z$ if and only if the interval bundle over $\partial R$ described by $X$ has an even number of non-orientable components.
\end{defn}

\begin{ex}
A surface $S$ with gleam $g$ is a shadow of the oriented $D^2$-bundle over $S$ with Euler number $g$.
\end{ex}

\begin{ex}
The disk with the 0-framing and gleam $n$ is a shadow of the unknot in $S^3$ with framing $-n$. 
\end{ex}

\begin{ex}
Let $X$ be an annulus $S^1\times [-1,1]$ with a disk $D$ attached along the core $S^1\times \{0\}$. If we give to $D$ gleam $g$ we get a shadow of a link $L\subset S^3$ with two components $K_1,K_2$ that are unknots and have linking number $g$. If $g=\pm 1$ it is the Hopf link.
\end{ex}

\begin{rem}\label{rem:reconstruction}
Given an abstract shadow $X$ we can ask ourselves how many 4-manifolds $W$ there are such that $X$ is a shadow of $W$ and the abstract gleams coincide with the ones induced by the embedding of $X$ in $W$. The reconstruction theorem below (Theorem~\ref{theorem:reconstruction}) ensures the existence of such 4-manifold. Moreover by the following Proposition~\ref{prop:shadow_uniqueness}, one such manifold is unique up to diffeomorphism. The proof of the reconstruction theorem consists of an explicit construction of the 4-manifold that is based on the division of the shadow in regions and 1-skeleton like in this paragraph.
\end{rem}

\begin{prop}\label{prop:shadow_uniqueness}
Let $X$ be an abstract shadow and $W$ and $W'$ two compact 4-manifolds such that $X$ is a shadow of both $W$ and $W'$, and the abstract gleams coincide with the ones induced by the embeddings of $X$ in $W$ and $W'$. Then $W$ and $W'$ are diffeomorphic.
\begin{proof}
Since one such 4-manifold $W$ collapses onto $X$ we have that $W$ is diffeomorphic to the regular neighborhood of $X$ in $W$. The regular neighborhood of the 1-skeleton of $X$ in $W$ collapses onto a graph, hence it is an orientable 4-dimensional handlebody (a manifold with a handle-decomposition with just 0- and 1-handles). The regular neighborhood of a region $R$ of $X$ in $W$ is a $D^2$-bundle with Eluer number the gleam of $R$. Therefore by the classification of $D^2$-bundles the diffeomorphism class of the regular neighborhoods of the various pieces of $X$ is fixed. Moreover there is only one possible way to glue together all these 4-dimensional pieces. Therefore if one such 4-manifold $W$ exists it is unique.
\end{proof}
\end{prop}

\begin{theo}[Reconstruction theorem]\label{theorem:reconstruction}
From a (abstract) shadow $X$ we can construct a compact oriented 4-manifold $W_X$ such that $X$ is a shadow of $W_X$ and the gleams of $X$, as abstract objects, coincide with the ones given by the embedding in $W_X$.
\begin{proof}
We just show how to get $W_X$ from a connected shadow $X$ with orientable regions. We consider the framing of $\partial X$ as a part of $X$ so that each point has a neighborhood homeomorphic to Fig.~\ref{figure:models}-(1), Fig.~\ref{figure:models}-(2) or Fig.~\ref{figure:models}-(3), but we do not add further points to $X$. If $X$ is a surface (closed or not) we take as $W_X$ the oriented $D^2$-bundle over $X$ with the gleam of the region as Euler number. 

Suppose that $X$ is not just a surface. There is a natural 3-dimensional thickening of the regular neighborhood in $X$ of vertices and edges of $X$ (they can be both internal or external, it does not matter). Following the combinatorics of $X$ we find a natural gluing of these 3-dimensional thickenings. Thus we get a possibly non orientable handlebody of dimension 3, $H^{(3)}$, that collapses onto the 1-skeleton of $X$. On the boundary of $H^{(3)}$ there is a a set of disjoint simple closed curves $L$ that is the boundary of the regular neighborhood of the 1-skeleton of $X$ in $X$. As a 4-dimensional thickening of the 1-skeleton of $X$ we take the orientable $I$-bundle $H^{(4)}$ over $H^{(3)}$ (it is unique) and we give it an arbitrary orientation. Hence $H^{(4)}$ is an oriented handlebody of dimension 4 that collapses onto the 1-skeleton of $X$. 

In the boundary of $H^{(4)}$ there is the link $L$ and it has a natural (maybe not orientable) framing $f_1$ that is the one given by the inclusion $L\subset \partial H^{(3)} \subset \partial H^{(4)}$. The components of $L$ are in bijection with the set of all the boundary components of the regions of $X$. Let $R$ be a region of $X$. The framing $f_1$ on the components of $L$ that correspond to the boundary components of $R$, coincides with the interval bundle given by $X$ that we used to define gleams. 

If $f_1$ on the component $(\partial R)_i$ is an annulus (is orientable), $f_1$ determines, up to isotopies, a parametrized solid torus $V_i \subset \partial H^{(4)}$ that is a regular neighborhood of the corresponding component of $L$. If $f_1$ is a M\"obius strip (is non orientable) on the component $(\partial R)_i$ we add a positive half-twist following the induced orientation of the boundary of $H^{(4)}$. In this way we get an orientable $I$-bundle over all the boundary components of $R$. Hence if we do it for each region we get an orientable framing $f_2$ of $L$. 

Let $k(R)$ be the number of components of the boundary of $R$ that had a non orientable $I$-bundle. We define $g(R) := \gl(R) - \frac 1 2 k(R)$, it is an integer number. We fix a decomposition $g(R)= \sum_i g_i(R)$, where $i$ runs over all the boundary components of $R$ and $g_i(R)\in \Z$. For each $i$ we modify the framing given by the $I$-bundle on it, and hence also the parametrization of the solid torus, by adding $-g_i(R)$ positive twists. If we do it for each region we get an orientable framing $f$ of $L$.

We attach $\bar R \times D^2$ to $H^{(4)}$ identifying each component of $\partial R \times D^2$, $(\partial R)_i \times D^2$, with the corresponding solid torus $V_i$ following the chosen parametrization, namely via the map (up to isotopies) that sends the curve $(\partial R)_i \times \{(0,1)\} \subset (\partial R)_i \times  D^2$ to the the $I$-bundle of $(\partial R)_i$ in $V_i$. We repeat this procedure for each region $R$. The construction does not depend on the choice of the decompositions $\gl(R) - k(R) = \sum_i g_i(R)$.
\end{proof}
\end{theo}

\begin{lem}\label{lem:spines}
$\ $
\begin{enumerate}
\item{Let $X$ be a spine (Definition~\ref{defn:spine}) of a compact, with non empty boundary, orientable 3-manifold $M$. Then the shadow obtained by giving to each region of $X$ gleam $0$ is a shadow of $M\times [-1,1]$ and hence it is a shadow of the double of $M$.} 
\item{If $X$ is a spine of a closed orientable 3-manifold $M$, then giving gleam $0$ to all regions of $X$ we get a shadow of the connected sum $M\# \bar M$, where $\bar M$ is $M$ with the opposite orientation.}
\item{A shadow $X$ without boundary and with all gleams $0$ is a shadow of a manifold $M\times [-1,1]$ where $M$ is a compact orientable 3-manifold with boundary $M$ such that $X$ is a spine of $M$.}
\end{enumerate}
\begin{proof}
$1.$ The manifold $M\times [-1,1]$ collapses onto $X\subset M\times\{0\}$. In order to compute the inherited gleams on a region $R$ of $X$ we have to take a perturbation $R'$ of $R$ whose boundary lies in the interval bundles over $\partial R$ given by $X$. Since $X$ is a spine of an oriented 3-manifold each connected component of the interval bundle is orientable. We can easily find one such $R'$ lying in $M\times\{0\}$ that does not intersect $R$. 

$2.$ Since $M$ is closed by $1.$ the shadow $X$ is a shadow of $W=(M\setminus B )\times [-1,1]$, where $B$ is the interior of an embedded compact 3-disk $\bar B \subset M$. We have $\partial W = (M\setminus B) \cup (\bar M \setminus B) \cup \partial B \times [-1,1] $, where $\partial B$ is the boundary of $\bar B$, hence $\partial W=M \#\bar M$.

$3.$ There is a natural 3-dimensional thickening of the regular neighborhood of vertices, edges and regions of $X$. Since the gleam of each region is an integer, we can naturally glue together these 3-dimensional pieces forming a compact orientable 3-manifold with boundary $M$ such that $X$ is a spine of $M$. Then the 4-manifold $M\times [-1,1]$ collapses onto $X$ and computing the gleams as in $1.$ we find the right numbers.
\end{proof}
\end{lem}

In Fig.~\ref{figure:Bing_house} is shown the \emph{Bing's house with two rooms}. This is a spine of the 3-ball, hence by Lemma~\ref{lem:spines} if we give to each region gleam $0$ we get a shadow of the 4-ball. It is composed of $4$ disk regions, $2$ edges and $2$ vertices.

\begin{figure}[htbp]
\begin{center}
\includegraphics[scale=0.6]{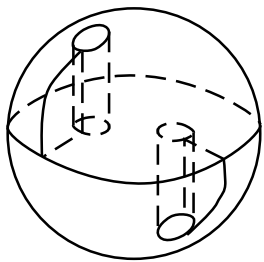}
\end{center}
\caption{The Bing's house with two rooms: a spine of the 3-ball, hence a shadow of the 4-ball. This has $7$ disk regions, $12$ edges and $6$ vertices.}
\label{figure:Bing_house}
\end{figure}

Let $M$ be an orientable compact 3-manifold with non empty boundary, $X$ a spine of $M$ and $G$ a framed knotted trivalent graph in $M$. Orient the regions of $X$ and take a generic projection $D\subset X$ of $G$ in $X$ given by the retraction of $M$ onto $X$ with the further information of the over/under passes on the 4-valent vertices (since both $M$ and the regions of $X$ are oriennted, we can define an oriented $I$-bundle over the region, hence get the over/under passes information). The projection $D$ is a graph whose vertices are 4-valent or trivalent and they lie in regions of $X$. For each component $G_1$ of $G$ attach $G_1\times [0,1]$ to $X$ along $G_1\times \{0\}$ following a parametrization of the projection of that component $G_1\rightarrow D \subset X$. We get a simple polyhedron with non empty boundary $X_G$ that has a framing on the boundary $\partial X_G =G$. Starting from $0$ modify the gleams of the regions in $X_G$ following the local rules in Fig.~\ref{figure:gleam_rule}.

\begin{figure}[htbp]
\begin{center}
\includegraphics[scale=0.6]{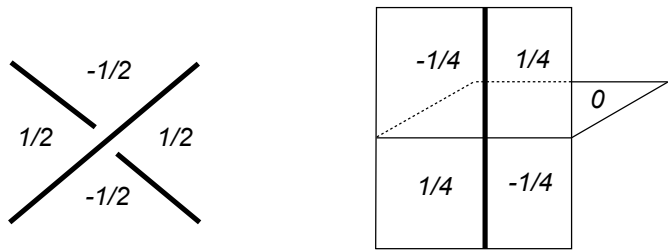}
\end{center}
\caption{The gleam rule. The bold lines are parts of the projection $D$ of the link into the spine.}
\label{figure:gleam_rule}
\end{figure}

\begin{lem}\label{lem:shadow0}
The shadow constructed above is a shadow of the graph $G$ in the double of $M$. If $G$ is a link and we use disks instead of annuli ($G_1\times [0,1]$) we get a shadow of the Dehn surgery along $G$ as a framed link in the double of $M$. If we look at dimension 4 we get a shadow of $M\times [-1,1]$ in the first case, while in the second case we get a shadow of $M\times [-1,1]$ with a 2-handle attached for each component of $G$, where the 2-handles are attached following the framing.
\begin{proof}
See \cite[Section IX.3]{Turaev}.
\end{proof}
\end{lem}
 
\begin{theo}\label{theorem:admit_sh}
An oriented compact 4-manifold admits a shadow (with or without boundary) if and only if it admits a handle-decomposition with just 0-, 1- and 2-handles.
\begin{proof}
Let $X$ be a shadow of a compact 4-manifold $W$. Fix a triangulation $X^t$ of the closure in $X$ of each region $R$ of $X$ (closure as sub-spaces, not $\bar R$). Since $W$ collapses onto $X$, $W$ is diffeomorphic to the regular neighborhood of $X$ in $W$. Hence $X^t$ defines a handle-decomposition of $W$ where there is a 0-handle surrounding each vertex of $X$ or of the triangulation of a region, a 1-hangle surrounding each edge of the triangulations, and a 2-handle surronding each triangle of the triangulations. 

Let $D$ be a handle-decomposition of a compact 4-manifold $W$ with just 0-, 1- and 2-handles. We can suppose that every 1-handle is attached to two or one 0-handles. The union of the 0- and 1-handles is a 4-dimensional oriented handlebody $H^{(4)}$. We have $H^{(4)}=H^{(3)}\times [-1,1]$, where $H^{(3)}$ is a 3-dimensional oriented handlebody decomposed in 0- and 1-handles corresponding to the ones of $D$. Insert a Bing's hose with two rooms (Fig.~\ref{figure:Bing_house}) inside each 0-handle of $H^{(3)}$. Let $h^{(1)}$ be a 1-handle of $D$ attached to the 0-handles $h^{(0)}_1$ and $h^{(0)}_2$ (maybe $h^{(0)}_1 = h^{(0)}_2$). Identify two closed disks embedded in the regions of the Bing's houses, one in the Bing's house in $h^{(0)}_1$ and one in the one in $h^{(0)}_2$. If $h^{(0)}_1 = h^{(0)}_2$ the disks are in the same Bing's house and are disjoint. Repeating this procedure for all 1-handles we get a spine of $H^{(3)}$, hence by Lemma~\ref{lem:spines}-(1.) we get a shadow of $H^{(4)}$. The attaching curves of the 2-handles lie in the boundary of $H^{(4)}$. We apply Lemma~\ref{lem:shadow0} to end. Note that this is a shadow without boundary.
\end{proof}
\end{theo}

\begin{lem}\label{lem:shadow1}
Every closed oriented 3-manifold has a shadow that is homotopically equivalent to a bouquet of 2-spheres (in particular it is simply connected) and has just disk regions.
\begin{proof}
Let $L\subset S^3$ be a surgery presentation of $M$ in $S^3$ (Definition~\ref{defn:surgery_pres} and Theorem~\ref{theorem:Lickorish-Wallace}). Put a Bing's house (Fig.~\ref{figure:Bing_house}) inside the 4-disk $D^4$ ($\partial D^4 =S^3$), then we apply Lemma~\ref{lem:shadow0}. The shadow is homotopically equivalent to the 4-manifold that is homotopically equivalent to a bouquet of $k$ spheres, where $k$ is the number of components of $L$.
\end{proof}
\end{lem}

\begin{prop}\label{prop:shadow1}
Every knotted trivalent graph $G$ in a closed 3-manifold $M$ has a shadow that is homotopically equivalent to a bouquet of 2-spheres.
\begin{proof}
Let $X$ be the shadow of $M$ constructed in Lemma~\ref{lem:shadow1} and $W$ the 4-dimensional thickening of $X$ ($\partial W =M$). To get the simple polyhedron $X_G$ we proceed in the same way as Lemma~\ref{lem:shadow0}, then we compute the gleams using the embedding of $X_G$ in $W$. 
\end{proof}
\end{prop}

In \cite[Theorem 3.14]{Costantino-Thurston} and \cite[Proposition 2.2]{Carrega-Martelli} there is a more constructive proof of Proposition~\ref{prop:shadow1} for the case of graphs in the connected sum $\#_g(S^1\times S^2)$ of $g$ copies of $S^1\times S^2$. Following that method we find a shadow homotopically equivalent to a bouquet of circles:

\begin{prop}[\cite{Carrega-Martelli}]\label{prop:shadow2}
Every knotted trivalent framed graph in $\#_g(S^1\times S^2)$ has a shadow that is homotopically equivalent to a graph of genus $g$.
\begin{proof}
The connected sum of $g$ copies of $S^1\times S^2$ is the boundary of the orientable 4-dimensional handlebody of genus $g$, $W_g$. Pick a diagram $D$ of $G$ in the disk with $g$ holes $S_{(g)}$ (see Subsection~\ref{subsec:diag}). We suppose that there is a smallest closed disk with $g$ holes $S$ that contains the diagram $D$. This is equivalent to ask that the diagram is connected, not contained in a disk with $g'<g$ holes, and no vertex of the diagram disconnects it: these conditions can be easily achieved using Reidemeister moves. We push the punctured disk $S$ entirely inside $W_g$, obviously $W_g$ collapses onto $S$. We enlarge $S$ by adding an immersed cylinder $C_i\cong G_i\times [0,1]$ for each connected component $G_i$ of $G$, such that the part $G_i\times (0,1]$ is properly embedded in $W_g\setminus S$, $C_i \cap \partial W_g = G_i $, and the part $G_i\times \{0\}$ is attached following a parametrization $G_i \rightarrow D \subset S$ of the projection of $G_i$ in $S\subset S_{(g)}$. We do it for every component $G_i$ and we get a simple polyhedron $X_G$ that is properly embedded in $W_g$, $\partial X_G =G$, and $W_g$ collapses onto $X_G$. To get the gleams we proceed as in Lemma~\ref{lem:shadow0}.
\end{proof}
\end{prop}

In Proposition~\ref{prop:standard_sh} we improve Theorem~\ref{theorem:admit_sh}, Lemma~\ref{lem:shadow1}, Proposition~\ref{prop:shadow1} and Proposition~\ref{prop:shadow2} by
showing that we can obtain a standard shadow in all cases.

By Theorem~\ref{theorem:admit_sh} it seems that only a small class of 4-manifolds can be studied by shadows. The following results say that shadows can be used to study all the closed 4-manifolds.

We remind that the connected sum $\#_g(S^1\times S^2)$ of $g$ copies of $S^1\times S^2$ is the boundary of the 4-dimensional orientable handlebody of genus $g$, $W_g$.
\begin{theo}[Laudenbach-Poenaru]\label{theorem:Laudenbach-Poenaru}
Every diffeomorphism of the connected sum $\#_g(S^1\times S^2)$ in itself extends to a diffeomorphism of $W_g$ in itself.
\begin{proof}
See \cite{Laudenbach-Poenaru}.
\end{proof}
\end{theo}

\begin{cor}
Let $W$ be a compact 4-manifold with a handle-decomposition with just 0-,1- and 2-handle. Then up to diffeomorphisms there is at most one closed 4-manifold $W'$ that is obtained adding 3- and 4-handles to $W$. 
\begin{proof}
Let $W_1$ and $W_2$ be two closed 4-manifolds obtained adding 3- and 4-handles to $W$. The union of the 3- and 4-handles of $W_1$ and $W_2$ is an orientable handlebody $W_g$. Therefore $W_1= W \cup_{\varphi_1} W_g$ and $W_2= W \cup_{\varphi_2} W_g$, where $\phi_1,\varphi_2 : \partial W_g \rightarrow \partial W$ are two diffeomorphisms of $\#_g(S^1\times S^2)$ in itself. By Theorem~\ref{theorem:Laudenbach-Poenaru} $\varphi_2^{-1}\circ \varphi_1: \partial W_g \rightarrow \partial W_g$ is the restriction to the boundary of a diffeomorphism $\Phi: W_g \rightarrow W_g$. We can define a diffeomorphism between $W_1$ and $W_2$ as follows:
\beq
&W \cup_{\varphi_1} W_g \longrightarrow  W \cup_{\varphi_2} W_g & \\
&x  \mapsto  x  \ \ \text{for } x\in W & \\
&y  \mapsto  \Phi(y)  \ \ \text{for } y\in W_g  .& 
\eeq
\end{proof}
\end{cor}

\section{Moves}

Usually in a representation theory there are moves relating the representing objects. For instance we can think about Reidemeister moves for link diagrams in the disk or sliding and birth/death of handles for handle-decompositions of manifolds. We have moves for shadows too, and  we introduce them in this section describing their behavior.

\begin{defn}\label{defn:sh_moves}
There are three particular local moves on shadows called $0 \rightarrow 2$ or \emph{lune} (Fig.~\ref{figure:shadowmoves}-(2)), $2 \rightarrow 3$ (Fig.~\ref{figure:shadowmoves}-(1)) and $1 \rightarrow 2$ (Fig.~\ref{figure:shadowmoves}-(4)). They modify a shadow just in a little contractible portion and leave the rest unchanged. They consist of a sliding of a region. In Fig.~\ref{figure:shadowmoves} just the attaching curve of the sliding region is pictured in red. The moves $2\rightarrow 3$ and $1\rightarrow 2$ create a new region and the gleam of this region is specified in the figure. The gleams of the other regions change as specified in figure. We look just at a part of the whole shadow and more than one of the pictured parts of region may belong to the same global region; if it is so, the gleam of such region is the sum of the local gleams. Differently from $0\rightarrow 2$ and $2\rightarrow 3$ we can not completely draw $1\rightarrow 2$ in $\mathbb{R}^3$. The move in Fig.~\ref{figure:shadowmoves}-(3) is interesting and useful and can be obtained as a composition of the other moves in figure and their inverses.

These moves together with their inverses are called \emph{shadow moves}. Shadows related by shadow moves are said to be \emph{equivalent}.
\end{defn}

\begin{figure}[htbp]
\begin{center}
\includegraphics{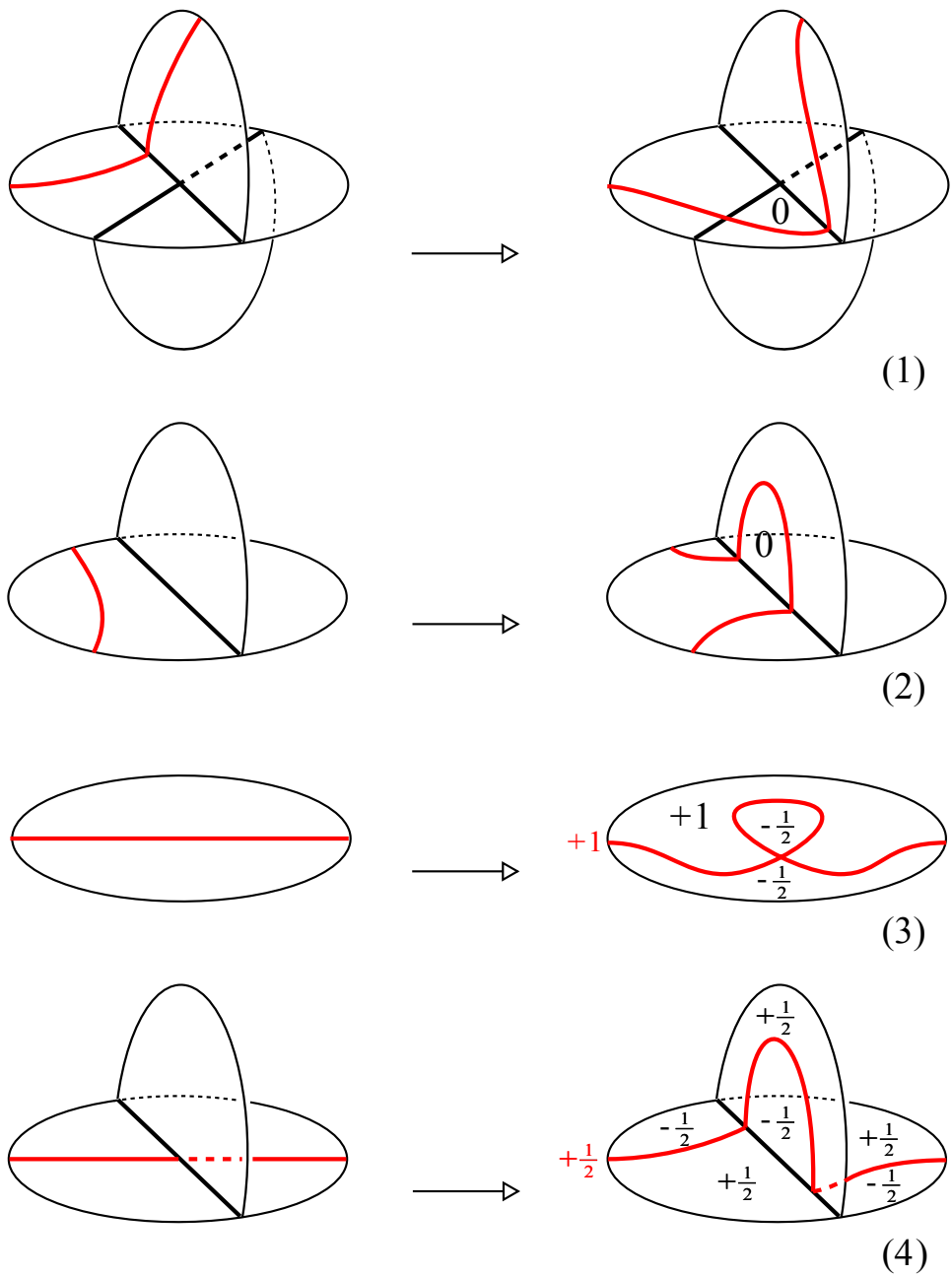}
\end{center}
\caption{Shadow moves.}
\label{figure:shadowmoves}
\end{figure}

It is easy to get the following:
\begin{prop}
Two equivalent shadows give the same 4-manifold.
\end{prop}

We remind that a shadow is standard if all the regions that are not adjacent to boundary edges are disks and the boundary edges are adjacent either to annuli or to disks (Definition~\ref{defn:shadow}).

\begin{prop}\label{prop:standard_sh}
Every shadow without closed regions is equivalent to a standard shadow.
\begin{proof}
We apply the $0 \rightarrow 2$ move to the boundary of a region $R$ of genus bigger than $1$. This creates new disk regions and reduces the genus of $R$. We repeat this procedure for each region of genus bigger than $1$.
\end{proof}
\end{prop}

By Proposition~\ref{prop:standard_sh} we can improve Theorem~\ref{theorem:admit_sh}, Lemma~\ref{lem:shadow1}, Proposition~\ref{prop:shadow1} and Proposition~\ref{prop:shadow2} saying that those shadows are standard too.

Let $X$ and $Y$ be two shadows. Identify two embedded closed disks $D_X$ and $D_Y$ respectively lying in the interior of two regions $R_X\subset X$ and $R_Y\subset Y$. Give to the identification of $D_X$ and $D_Y$ gleam $k\in \Z$. Give to the complement of the disk in $R_X$ and $R_Y$ respectively gleam $\gl(R_X) -k$ and $\gl(R_Y) -k$. Assign to the other regions the gleam that they have in $X$ and $Y$ as separated shadows. Let $Z$ be the resulting shadow.
\begin{prop}\label{prop:intersect_shadows}
$Z$ is a shadow of an oriented 4-manifold $W$ containing $X$ and $Y$ with intersection number $k$. Namely $X,Y\subset W$ intersect transversely in a finite number of points ($R_X \cap R_Y$) and counting them with sign we get $k\in \Z$. 
\begin{proof}
We can suppose that in the construction of the 4-dimensional thickenings $W_X$ and $W_Y$ of $X$ and $Y$, $D_X$ and $D_Y$ are thickened to a $D^2$-bundle $E_X$ and $E_Y$ with Euler number $k$ while $R_X\setminus D_X$ and $R_Y\setminus D_Y$ are thickened to $D^2$-bundle with Euler number $\gl(R_X) - k$ and $\gl(R_Y) -k$. The identification of $D_X$ and $D_Y$ extends to an orientation reversing diffeomorfism $\varphi : E_X \rightarrow E_Y$. We get $W$ by taking the disjoint union of $W_X$ and $W_Y$ and identifying $E_X \subset W_X$ with $E_Y \subset W_Y$ via $\varphi$.
\end{proof}
\end{prop}

\begin{rem}\label{rem:con_sum}
Note that we can transport a region with gleam $0$ whose closure in $X$ is an embedded disk through each region just by the $0\rightarrow 2$ moves and their inverses. Hence we can perform the construction of Theorem~\ref{prop:intersect_shadows} on each pair of regions $R_X$ and $R_Y$ without changing the equivalence class of $Z$, and hence without changing the resulting 4-manifold, provided that the gleam of the attaching disk is $0$.
\end{rem}

\begin{defn}
Given two shadows $X$ and $Y$ we denote by $X\# Y$ the shadow  $Z$ constructed in Theorem~\ref{prop:intersect_shadows} giving to the attaching disk gleam $0$ and we call it the \emph{connected sum} of $X$ and $Y$. By Remark~\ref{rem:con_sum} it is well defined up to $0\rightarrow 2$ moves and their inverses.
\end{defn}

\begin{cor}\label{cor:connected_sum}
Let $X$ and $Y$ be two shadows with 4-dimensional thickening respectively $W_X$ and $W_Y$. Then $X\#Y$ is a shadow of the boundary connected sum of $W_X$ and $W_Y$
$$
W_{X\#Y} = W_X \#_\partial W_Y .
$$
\begin{proof}
Note that the boundary connected sum of $W_X$ and $W_Y$ contains $X$ and $Y$ in a way that they do not intersect. The boundary connected sum is equal to attach a 1-handle connecting the components $W_X$ and $W_Y$. We can suppose that the attaching 3-balls of the 1-handle lie in the intersection with the boundary of the 4-dimensional thickening of $R_X$ and $R_Y$. The boundary connected sum of $W_X$ and $W_Y$ collapses onto the simple polyhedron obtained gluing $D_X$ and $D_Y$, where the connecting 1-handle collapses onto the boundary of such disk. The natural inclusions of $X$ and $Y$ into that simple polyhedron are exactly the embeddings of $X$ and $Y$ in the 4-manifold. With a little perturbation we can see that $X$ and $Y$ are disjoint. Since the 4-manifold differs from the disjoint union of $W_X$ and $W_Y$ only for that 1-handle we have that the induced gleams on the regions different from $R_X \setminus D_X$, $D_X$, $R_Y\setminus D_Y$ and $D_Y$ are unchanged.
\end{proof}
\end{cor}

\begin{defn}
Two more important moves on shadows are the $0$-\emph{bubble} and the $\pm 1$-\emph{bubble} moves, see Fig.~\ref{figure:bubblemoves}. Two shadows related by a sequence of shadow moves, $0$-bubble moves and their inverses are said to be \emph{stably equivalent}. The result of such moves to a shadow $X$ is respectively $X\#S^2_{(0)}$ and $X\# S^2_{(\pm 1)}$ where $S^2_{(k)}$ is the 2-sphere $S^2$ equipped with gleam $k\in \Z$. 
\end{defn}

We note that $S^2_{(0)}$ is a shadow of $S^2\times D^2$, while if we equip the sphere with gleam $\pm 1$ we get $\mathbb{CP}^2$ minus an open 4-ball, $\mathbb{CP}^2$ has the standard orientation in the case of $k = -1$, and the opposite one in the case of $k = 1$. Therefore by Corollary~\ref{cor:connected_sum} we have that applying a $0$-bubble move the 4-manifold changes by the boundary connected sum with $S^2\times D^2$, hence its boundary changes by the connected sum with $S^2\times S^1$. On the other hand if we apply a $\pm 1$-bubble move we get the boundary connected sum with a punctured $\mathbb{CP}^2$ with the proper orientation that is equivalent to the connected sum with $\mathbb{CP}^2$ (or $\overline{\mathbb{CP}^2}$), hence its boundary does not change.

\begin{figure}[htbp]
\begin{center}
\includegraphics[width=11.4cm]{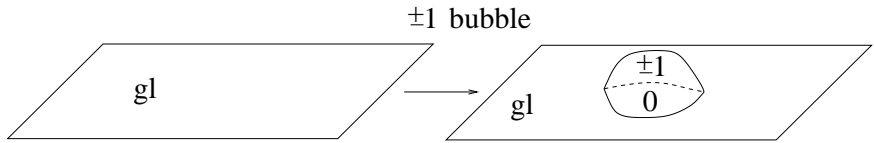}
\end{center}
\caption{The $\pm 1$-bubble moves: the gleam of the region over which the move is applied is unchanged. A new $0$-gleam disk is created and a disk with gleam $\pm 1$ is attached to its boundary. The $0$-bubble move, creates a $0$ gleam disk in instead of the $\pm 1$-gleam one.}
\label{figure:bubblemoves}
\end{figure}

\begin{theo}
Let $X$ and $Y$ be two shadows without closed regions of the same compact oriented 4-manifold $W$ and with the same boundary $\partial X = \partial Y \subset \partial W$. Then $X$ and $Y$ are stably equivalent.
\begin{proof}
See \cite[Theorem IX.1.7, Section IX.7, Theorem IX.8.1.2]{Turaev}.
\end{proof}
\end{theo}

\begin{quest}\label{quest:eq_shadows}
Is it true that two shadows of the same compact oriented 4-manifold are equivalent?
\end{quest}

Experts think that the answer to Question~\ref{quest:eq_shadows} is no. In fact this problem is related to the \emph{Andrews-Curtis conjecture} that is probably false.

The application of the $0$-bubble move increases by $1$ the rank of the second homology group of the 4-manifold $H_2(W_X)$. Therefore the number of $0$-bubble moves in a sequence relating two stably equivalent shadows representing the same 4-manifold must be equal to the number of their inverses in that sequence.

\begin{conj}\label{conj:2}
Two shadows related by shadow moves, $0$-bubble moves and their inverses represent the same compact oriented 4-manifold if and only if the number of $0$-bubble moves in the sequence is equal to the one of their inverses.
\end{conj}

Costantino proved Conjecture~\ref{conj:2} in the simply connected case:
\begin{theo}[Costantino]
Let $X$ and $Y$ be two shadows related by a sequence of  shadow moves, $0$-bubble moves and their inverses such that the number of $0$-bubble moves is equal to the number of their inverses. If $X$ is simply connected (and hence also $Y$) and $\partial X = \partial Y$ then $W_X=W_Y$.
\begin{proof}
See \cite[Theorem 1.9.1]{CostantinoPhD}.
\end{proof}
\end{theo}

There is one more complicated move on shadows called \emph{trading move} or \emph{surgery along a curve} in the shadow. This represents surgery along a closed curve in the interior of the 4-manifold.

\begin{theo}
Let $X$ and $Y$ be two shadows of the same pair $(M,G)$, where $M$ is an oriented closed 3-manifold and $G$ is a knotted framed trivalent graph in $M$. Then $X$ and $Y$ are related by shadow moves, $\pm 1$-bubble moves, trading moves and their inverses.
\begin{proof}
See \cite[Lemma IX.7.8]{Turaev}.
\end{proof}
\end{theo}

The following result shows that for simply connected shadows the trading move is useless. 

\begin{theo}[Costantino-Thurston]\label{theorem:3-manifods}
Let $X$ and $Y$ be two simply connected shadows of the same pair $(M,G)$, where $M$ is an oriented  closed 3-manifold and $G$ a knotted framed trivalent graph in $M$. Then $X$ and $Y$ are related by shadow moves, $\pm 1$-bubble moves and their inverses.
\begin{proof}
See \cite{Costantino-Thurston:preprint} or \cite[Theorem 2.2.7]{CostantinoPhD}.
\end{proof}
\end{theo}

Theorem~\ref{theorem:3-manifods} could seem quite restrictive but by Lemma~\ref{lem:shadow1} it is not so if we are interested just on colosed 3-manifolds.

\section{Bilinear form and signature}\label{sec:bil_form} 
Here we show how to get the signature and the intersection form of an oriented 4-manifold from a shadow. 

Let $X$ be a shadow without boundary and with all regions orientable. Let $R$ be a region of $X$ and let $h\in H_2(X,\Z)$. We denote by $\langle h | R\rangle \in\Z$ the image of $h$ under the map $H_2(X,\Z) \rightarrow H_2(X/(X\setminus R), \Z) \cong \Z$ induced by the quotient identifying the complement of $R$ to a point. The group $H_2(X/(X\setminus R) , \Z)$ is identified with $\Z$ once an orientation to $R$ is given. The map $\mathcal{Q}_X:H_2(X,\Z) \times H_2(X,\Z) \rightarrow \frac 1 2 \Z$ is the bilinear form so defined:
$$
\mathcal{Q}_X(h_1,h_2):= \sum_R \langle h_1|
 R \rangle \cdot \langle h_2 | R\rangle  \cdot \gl(R) ,
$$
where the sum runs over all the regions of $X$. This does not depend on the choice of an orientation of the regions. We call $\mathcal{Q}_X$ the \emph{bilinear form} of $X$. We denote by $\sigma(X)$ the signature of $\mathbb{R} \otimes_{\frac 1 2 \Z} \mathcal{Q}_X$. We call it the \emph{signature} of $X$.

\begin{theo}\label{theorem:bil_and_sign}
Let $X$ be a shadow of the oriented 4-manifold $W$, and let $f_*:H_2(X,\Z) \rightarrow H_2(W,\Z)$ be the isomorphism induced by the inclusion. Then for any $h_1,h_2\in H_2(X,\Z)$
$$
f_*(h_1) \cdot f_*(h_2) = \mathcal{Q}_X(h_1,h_2) ,
$$
where the product on the left-hand side is the intersection product in $H_2(W,\Z)$. Hence 
$$
\sigma(X) = \sigma(W) .
$$
\begin{proof}
See \cite[Section {\rm IX}.5.1]{Turaev}.
\end{proof}
\end{theo}

\begin{ex}\label{ex:signature}
Let $S$ be the closed orientable surface of genus $g$. Let $M$ be the boundary of the oriented $D^2$-bundle $F$ over $S$ with Euler number $n\in\Z$. If $S$ is the 2-sphere $S^2$, $M$ is the lens space $L(n,1)$. The surface $S$ equipped with gleam $n$ is a shadow of $F$. 
$$
\mathcal{Q}_X: \Z \times \Z \rightarrow \Z, \ \ \mathcal{Q}_X(k,h)= k h n . 
$$
By Theorem~\ref{theorem:bil_and_sign} 
$$
\sigma(F) ={\rm sgn}(n) .
$$
\end{ex}

Now we repeat all for the case with boundary.

Let $X$ be a shadow (maybe with boundary) with only orientable regions. For any $h\in H_2(X,\partial X;\Z)$, $\langle h|R \rangle \in \Z$ is the image of $h$ under the map $H_2(X,\partial X;\Z) \rightarrow H_2(X/(X\setminus R);\Z) \cong \Z$ given by the homomorphism $X/ \partial X \rightarrow X/(X\setminus R)$. To define it we need to fix an orientation of $R$. The map $\tilde{\mathcal{Q}}_X : H_2(X,\partial X;\Z)\times H_2(X,\partial X;\Z) \rightarrow \frac 1 2 \Z$ is the bilinear form so defined: 
$$
\tilde{\mathcal{Q}}_X(h_1,h_2) := \sum_R \langle h_1| R\rangle \cdot \langle h_2| R\rangle \cdot \gl(R) .
$$
Here $R$ runs over all the regions and we do not need to fix an orientation of $R$. The group $H_2(X;\Z)$ is contained in $H_2(X,\partial X;\Z)$ and we call $\mathcal{Q}_X$ the restriction of $\tilde{\mathcal{Q}}_X$ to $H_2(X,\Z)$. This is the \emph{bilinear form} of $X$. The \emph{signature} of $X$, $\sigma(X)$, is defined as the signature of $\mathbb{R} \otimes_{\frac 1 2 \Z} \mathcal{Q}_X$. Theorem~\ref{theorem:bil_and_sign} works also for this case with boundary (always with $\mathcal{Q}_X$).

\section{Shadow formula for the Kauffman bracket}\label{sec:sh_for_br}

In this section we show how to compute the Kauffman bracket of graphs in the connected sum $\#_g(S^1\times S^2)$ of $g$ copies of $S^1\times S^2$ (see Definition~\ref{defn:Kauf} and Section~\ref{sec:Kauf_g}) via shadows collapsing onto a graph (maybe a point). Hence we work in skein theory with coefficients in the field $\mathbb{Q}(A)$ of rational functions with rational coefficients.

\subsection{Statement}

\begin{defn}\label{defn:Eul_char}
Let $R$ and $e$ be respectively a region and an edge of a shadow $X$ ($e$ may be external). The \emph{Euler characteristic}, $\chi(R)$ and $\chi(e)$, of $R$ and $e$ is the Euler characteristic of the closure of $R$ and $e$ in $X$ as sub-spaces. We must not confuse them with the Euler characteristic of $R$ and $e$ as topological spaces. We have $\chi(e)=0$ either if $e$ is a circle or if $e$ is adjacent twice to the same vertex, otherwise $\chi(e)=1$. For a vertex $v$ we can define $\chi(v):=1$ ($v$ may be external).
\end{defn}

Let $X$ be a shadow with colored boundary. An \emph{admissible coloring} $\xi$ of $X$ that extends the coloring of $\partial X$ is the assignment of a color to each region of $X$ (a non negative integer), such that for every edge of $X$ the colors of the three incident regions (maybe two of the three regions are the same) form an admissible triple $(a,b,c)$ (see Definition~\ref{defn:admissible}), and the color of a region $R$ of $X$ touching the boundary $R\cap \partial X = e \partial \neq \varnothing$ is equal to the color of the boundary edge $e_\partial$.

The \emph{evaluation} of the coloring $\xi$ is the following function:
$$
\langle X \rangle_\xi = \frac{\prod_f \cerchio_f^{\chi(f)}A_f \prod_v \tetra_v \prod_{v_\partial} \teta_{v_\partial}}
{\prod_e \teta_e^{\chi(e)} \prod_{e_\partial} \cerchio_{e_\partial}^{\chi(e_\partial)} } .
$$
Here the product is taken over all regions $f$, inner edges $e$, inner vertices $v$, boundary edges $e_\partial$, boundary vertices $v_\partial$. The symbols
$$
\cerchio_f,\ \cerchio_{e \partial},\ \teta_e,\ \teta_{v_\partial}, \  \tetra_v
$$
denote the skein element of these graphs in $K(S^3)=\mathbb{Q}(A)$ (see Subsection~\ref{subsec:three_graphs}), colored respectively as $f$, $e_\partial$, or as the regions incident to $e$, $v_\partial$, or $v$.

The \emph{phase} $A_f$ is the following value:
$$
A_f = (-1)^{gc}A^{-gc(c+2)} ,
$$
where $g$ and $c$ are respectively the gleam and the color of $f$.

A shadow that collapses onto a graph of genus $g$ is a shadow of a 4-dimensional handlebody of genus $g$, $W_g$. We remind that the boundary of $W_g$ is $\#_g(S^1\times S^2)$.

It might be non obvious in general to determine whether a polyhedron collapses onto a graph. Luckily, on simple polyhedra there is a nice criterion:
\begin{prop}[Costantino] \label{prop:sh_collapse}
Let $X$ be a connected simple polyhedron. The following facts are equivalent:
\begin{itemize}
\item{$X$ collapses onto a graph (maybe a point);}
\item{$X$ does not contain a simple polyhedron without boundary;}
\item{every coloring of $\partial X$ extends to finitely many admissible colorings of $X$ (maybe no one).}
\end{itemize}
\begin{proof}
See \cite[Lemma 3.6]{Costantino2}.
\end{proof}
\end{prop}

\begin{cor}\label{cor:sh_collapse}
Let $X$ be a simple polyhedron that collapses onto a graph. Every connected simple  sub-polyhedron $X'\subset X$ also collapses onto a graph.
\begin{proof}
The polyhedron $X$ does not contain any simple sub-polyhedron without boundary, hence $X'$ also does not.
\end{proof}
\end{cor}

\begin{theo}[Shadow formula]\label{theorem:sh_for_br}
Let $X$ be a shadow of the colored knotted trivalent graph $G \subset \#_g(S^1\times S^2)$. If $X$ collapses onto a graph
$$
\langle G \rangle = \sum_\xi \langle X \rangle_\xi ,
$$
where the sum is taken over all the admissible colorings $\xi$ of $X$ that extend the coloring of $\partial X=G$.
\end{theo}

\begin{rem}\label{rem:sh_for_br}
In Section~\ref{sec:Kauf_g} we saw that we can represent every link in $\#_g(S^1\times S^2)$ with a diagram in the disk with $g$ holes $S_{(g)}$. Using this diagrammatic representation and the skein relations we got Proposition~\ref{prop:state_sum} that reduces the computation of the Kauffman bracket of a link to the one of diagrams $D_s\subset S_{(g)}$ without crossings and without homotopically trivial components. Here we explain how to easily end the computation using the shadow formula. In Proposition~\ref{prop:shadow2} we already saw a method to get a shadow of any link in $\#_g(S^1\times S^2)$, but the method we present now is better than compute directly the shadow formula on that shadow. In fact we reduce the computation to using the shadow formula just to shadows with null gleams, without vertices and whose edges are closed circles. This means that in the shadow formula we do not have graphs of the form $\teta$ and $\tetra$ and there are no phases $A_f$.

Let $D\subset S_{(g)}$ be a link diagram without crossings and without homotopically trivial components. Let $X$ be the shadow obtained attaching to $S_{(g)}$ an annulus along each component of $D$ and giving to each region gleam $0$. The boundary $\partial X$ of $X$ consists of the components of $D$ plus the $g+1$ boundary components of $S_{(g)}$. Color the boundary components of $X$ corresponding to the components of $D$ with $1$ (the remaining boundary components of the annuli attached to the components of $D \subset S_{(g)}$), and the other ones with $0$. The polyhedron $X$ is a shadow of a framed link in $\#_g(S^1\times S^2)$ that is the union of the link described by $D$ and a link colored with $0$. Therefore the shadow formula applied to $X$ gives $\langle D\rangle$. Hence
$$
\langle D \rangle = \sum_\xi \prod_R \cerchio^{\chi(R)}_{\xi(R)} ,
$$
where $\xi$ runs over all the admissible colorings of $X$ that extend the coloring of the boundary, $R$ runs over all the regions of $X$ (the external regions do not matter because either they are annuli or their color is $0$), $\chi(R)$ is the Euler characteristic of $R$, and $\xi(R)$ is the color of $R$ given by $\xi$. Therefore by Subsection~\ref{subsec:three_graphs} $\langle D \rangle$ is a symmetric function of $A^2$, namely there are two polynomials $f,h\in \Z[q]$ such that 
$$
\langle D_s \rangle = \frac{f|_{q=A^2}}{h|_{q=A^2}} , \ \ \langle D \rangle|_A = \langle D \rangle|_{A^{-1}}  .
$$

If the diagram is one of the diagrams $D_s$'s, the constructed shadow will be denoted $X_s$. 
\end{rem}

\subsection{Proof}\label{subsec:proof_sh_for_Kauf_br}

Now we prove the shadow formula (Theorem~\ref{theorem:sh_for_br}).

We recall the skein equalities in Fig.~\ref{figure:sphere}. They take place in the neighborhood of a 2-sphere $S$, drawn as a $0$-framed circle in the picture. If $G$ intersects $S$ transversely in exactly one point, then Fig.~\ref{figure:sphere}-(left) applies, while if $G$ intersects the sphere twice Fig.~\ref{figure:sphere}-(right) applies. Note that after applying the move we can surger along the sphere without affecting $\langle G \rangle$ (see Remark~\ref{rem:tensor}).

\begin{figure}
\begin{center}
\includegraphics[width = 9 cm]{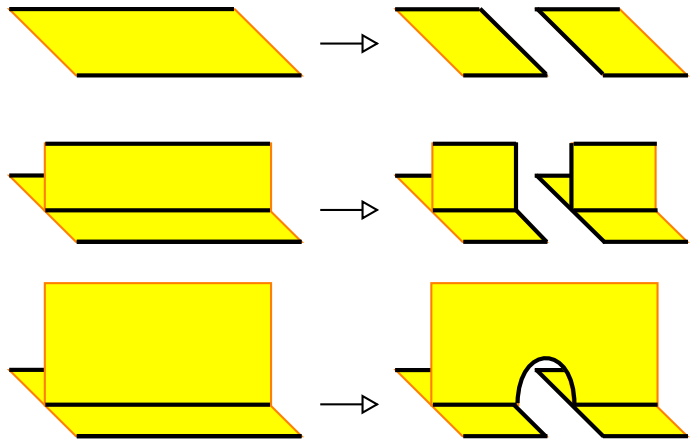}
\caption{A simple polyhedron $X$ that collapses onto a graph reduces to a finite union of atomic polyhedra after finitely many moves of this type. The bold exterior lines are portions of $\partial X$.}
\label{figure:cut}
\end{center}
\end{figure}

\begin{rem}
By Corollary~\ref{cor:sh_collapse} each move in Fig.~\ref{figure:cut} transforms a simple polyhedron that collapses onto a graph into one or two simple polyhedra that collapse onto a graph.
\end{rem}

A simple polyhedron $X$ is \emph{atomic} if it is the cone over $\cerchio, \teta$, or $\tetra$, namely $X$ is as in Fig.~\ref{figure:models}-(3,2,1). We will use the following:
\begin{prop}[\cite{Carrega-Martelli}]
Let $X$ be a simple polyhedron that collapses onto a graph. The polyhedron reduces to a finite union of atomic polyhedra after a finite combination of moves as in Fig.~\ref{figure:cut}.
\begin{proof}
We say that a region of $X$ is \emph{external} if it is incident to $\partial X$, and \emph{internal} otherwise. Suppose $X$ contains some internal regions. There is an edge $e$ that is adjacent to one internal region and to two external regions: if not, the interior regions would form a simple sub-polyhedron contradicting Proposition~\ref{prop:sh_collapse}. The move in Fig.~\ref{figure:cut}-(bottom) applied to $e$ transforms the interior region into an exterior one. After finitely many such moves we kill all the internal regions.

Now we can use Fig.~\ref{figure:cut}-(center) to cut every interior edge in two halves, and then Fig.~\ref{figure:cut}-(top) to cut every region into disks that are incident to $\partial X$ only in one arc or circle. We are left with atomic pieces.
\end{proof}
\end{prop}

If we apply one of the moves of Fig.~\ref{figure:cut} to a shadow $X$ of some graph $G\subset M$, we get a new simple polyhedron $X'$ that can be interpreted as a shadow of some graph $G'$ in some manifold $M'$. We show this fact for each move.

We start by examining Fig.~\ref{figure:cut}-(top). The yellow strip thickens to a $D^3 \times [-1,1]$, with boundary $S^2\times [-1,1]$. The 2-sphere $S = S^2\times \{0\}$ intersects $G$ transversely in two points. Topologically, the move corresponds to surgerying $M$ along the 2-sphere $S^2\times \{0\}$ and modifying $G$ as in Fig.~\ref{figure:sphere}-(right). We get a new graph $G'$ inside a new manifold $M'$, with a new shadow $X'$. If $S$ is separating, these objects actually split in two components.

The move in Fig.~\ref{figure:cut}-(center) is analogous, the only difference being that now $S$ intersects $G$ in three points. The move in Fig.~\ref{figure:cut}-(bottom) is the fusion shown in Fig.~\ref{figure:fusion}.

\begin{rem} \label{rem:cut}
In the moves of Fig.~\ref{figure:cut}, some region $R\subset X$ is cut into two regions $R_1,R_2 \subset X'$. The gleams $g_1$ and $g_2$ of these new regions sum to give the gleam $g=g_1+g_2$ of $R$. The gleams of all the other regions of $X$ do not change.
\end{rem}

\begin{proof}[Proof of Theorem~\ref{theorem:sh_for_br}, \cite{Carrega-Martelli}]
The formula holds when $X$ is atomic with zero gleams: there is a single coloring $\xi$ on $X$ extending that of $G$, and we get $\langle X \rangle_\xi = \langle G \rangle$. To prove that, note that the contribution of every non closed boundary edge $e_\partial$ or boundary vertex $v_\partial$ cancels with the contribution of the incident region $f$ or edge $e$. Therefore:
\begin{itemize}
\item{if $G=\cerchio$ we get obviously $\cerchio$;}
\item{if $G=\teta$ everything cancels except $\teta_{v_\partial}^2 / \teta_e = \teta_e$;}
\item{if $G=\tetra$ everything cancels except $\tetra_v$.}
\end{itemize}
Suppose now that $X$ is atomic with arbitrary gleams. We modify the gleams using the following moves:
\begin{enumerate}
\item{add a gleam $\pm 1$ on a region: this corresponds to a $\mp 1$ full twist of the corresponding framed edge of $G$;}
\item{add a gleam $\pm \frac 1 2$ to the three regions incident to an interior edge of $X$: this corresponds to a $\mp \frac 1 2 $ half-twist to each of the three edges of $G$ incident to a vertex of $G$.}
\end{enumerate}
Fig.~\ref{figure:framingchange_vertex} shows that even the second move does not change the orientability of the framing.

\begin{figure}[htbp]
$$
\pic{4}{0.7}{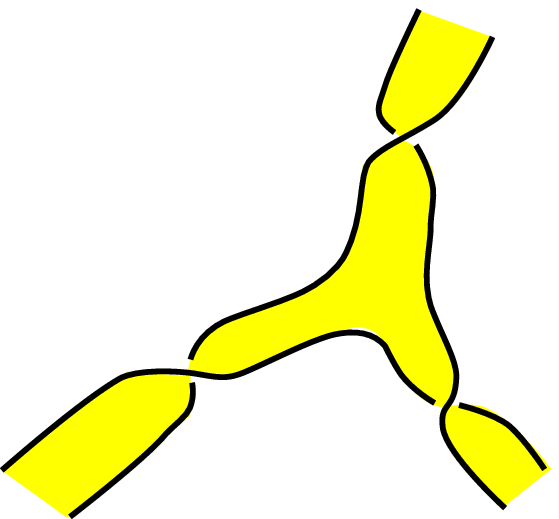} \ \ \ \longleftrightarrow \ \ \ \pic{3.5}{0.7}{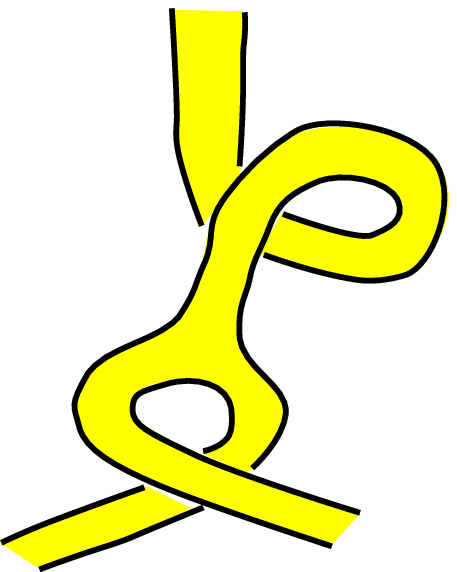}  
$$
\caption{A trivalent framed vertex with a positive half-twist in each strand.}
\label{figure:framingchange_vertex}
\end{figure}

Using finitely many such moves we can reduce all gleams to zero. To show that,
color in green the regions having a half-integer (but non-integer) gleam. Recall that the framing of $G$ is orientable: this implies that every sub-circle $C\subset G$ intersects an even number of green faces, and it is easy to check that with moves $2.$ we can transform all gleams into integers. Then we reduce them to zero using $1.$.

Let $G'$ be obtained from $G$ by $1.$ or $2.$. By Fig.~\ref{figure:framingchange} we have
\beq
\langle G' \rangle & = & (-1)^c A^{\mp\frac c(c+2)} \langle G \rangle\\
\langle G' \rangle & = & (-1)^{\frac {a+b+c}2} A^{\mp\frac a2(a+2) \mp \frac b2(b+2) \mp \frac c2(c+2)} \langle G \rangle
\eeq
corresponding respectively to moves $1.$ and $2.$. In the formula the contribution of the phases $A_f = (-1)^{gc} A^{- 
gc(c+2)}$ changes exactly in the same way: this proves the theorem for any atomic shadow $X$.

A more general $X$ decomposes into atoms via finitely many moves as in Fig.~\ref{figure:cut}. Let $n(X)$ be the number of moves necessary to atomize $X$: we prove the theorem by induction on $n(X)$. 

Pick a move transforming $X$ into a $X'$ with $n(X')<n(X)$. 
The polyhedron $X'$ is a shadow of some graph $G'$ in some manifold $M'$. The objects $X'$ and $M'$ may have two components, but the following arguments work anyway. We suppose by induction that the theorem holds for $X'$ and $G'$, and we prove it for $X$ and $G$.

Consider the move in Fig.~\ref{figure:cut}-(top). The pair $(M',G')$ is obtained from $(M,G)$ via the move shown in Fig.~\ref{figure:sphere}-(right), with $G'$ inheriting the coloring of $G$. Therefore
$$
\langle G \rangle = \frac{1}{\cerchio_f} \langle G' \rangle ,
$$
where $f$ is the yellow colored region that we have cut. There is an obvious correspondence between colorings of $X$ and $X'$, and the formula says that for each coloring $\xi$ we have
$$
\langle X \rangle_\xi  = \frac 1{\cerchio_f} \langle X' \rangle_\xi 
$$
(We use here Remark~\ref{rem:cut} to show that the phases of $X$ and $X'$ with $\xi$ are the same). The theorem holds for the pair $(X', G')$, and hence holds also for $(X,G)$.

The move in Fig.~\ref{figure:cut}-(center) is treated analogously. Using a fusion and Fig.~\ref{figure:sphere} we find easily that
$$
\langle G \rangle = \frac{1}{\teta_e} \langle G' \rangle ,
$$
where $e$ is the colored edge cut in Fig.~\ref{figure:cut}-(center). There is an obvious correspondence between colorings of $X$ and $X'$, and for  each such coloring $\xi$ we have
$$
\langle X \rangle_\xi  = \frac 1{\teta_e} \langle X' \rangle_\xi .
$$

Finally, the move in Fig.~\ref{figure:cut}-(bottom) is a fusion. The fusion formula says
$$
\langle G \rangle = \sum_c \frac{\cerchio_c}{\teta_{a,b,c}} \langle G'_c \rangle ,
$$
where the coloring $G'_c$ on $G'$ varies on the new colored edge $c$. Every coloring of $X$ induces one of $X'$ and we get
$$
\langle X \rangle_\xi  = \frac {\cerchio_c}{\teta_{a,b,c}} \langle X' \rangle_\xi .
$$
This proves the theorem.
\end{proof}

\section{Shadow formula for the $SU(2)$-Reshetikhin-Turaev-Witten invariants}

In this section we show how to compute the $SU(2)$-Reshetikhin-Turaev-Witten invariants (see Definition~\ref{defn:RTW}, Definition~\ref{defn:RTW_G} and Section~\ref{sec:RTW}) via shadows. Hence we fix an integer $r\geq 3$ and a primitive $4r^{\rm th}$ root of unity $A$. We remind that we need just that $A^{4r}=1$ and $A^{4n}\neq 1$ for $0<n<r$.

\subsection{Statement}

We introduced the definition of Euler characteristic of a region, an edge and a vertex of a shadow (Definition~\ref{defn:Eul_char}).

Let $X$ be a shadow with colored boundary. A $q$-\emph{admissible coloring} $\xi$ of $X$ that extends the coloring of $\partial X$ is the assignment of a color to each region of $X$ (a non negative integer), such that for every edge of $X$ the colors of the three incident regions (maybe two of the three regions are the same) form a $q$-admissible triple $(a,b,c)$ (see Definition~\ref{defn:q-admissible} and Definition~\ref{defn:admissible}), and the color of a region $R$ of $X$ touching the boundary $R\cap \partial X = e \partial \neq \varnothing$ is equal to the color of the boundary edge $e_\partial$.

The \emph{evaluation} of the coloring $\xi$ is the following function:
$$
|X|_\xi = \frac{\prod_f \cerchio_f^{\chi(f)}A_f \prod_v \tetra_v \prod_{v_\partial} \teta_{v_\partial}}
{\prod_e \teta_e^{\chi(e)} \prod_{e_\partial} \cerchio_{e_\partial}^{\chi(e_\partial)} } .
$$
Here the product is taken over all regions $f$, inner edges $e$, inner vertices $v$, boundary edges $e_\partial$, boundary vertices $v_\partial$. The symbols
$$
\cerchio_f,\ \cerchio_{e \partial},\ \teta_e,\ \teta_{v_\partial}, \  \tetra_v
$$
denote the skein element of these graphs in $K_A(S^3)=\mathbb{C}$ (see Subsection~\ref{subsec:three_graphs} and Subsection~\ref{subsec:quantum_skein_th}), colored respectively as $f$, $e_\partial$, or as the regions incident to $e$, $v_\partial$, or $v$.

The \emph{phase} $A_f$ is the following value:
$$
A_f = (-1)^{gc}A^{-gc(c+2)} ,
$$
where $g$ and $c$ are respectively the gleam and the color of $f$.
$$
|X|^r := \sum_{\xi} |X|_\xi
$$
where the sum is taken over all the $q$-admissible colorings of $X$. Note that since $X$ is compact and each color of the region is at most $r-2$, the sum is finite.

We denote by $\sigma(W)$ the signature of a oriented 4-manifold $W$ (see Section~\ref{sec:bil_form}) and with $\chi(X)$ the Euler characteristic of $X$. If $X$ is a shadow of $W$ we have
$$
\chi(X) = \sum_v 1 - \sum_e \chi(e) + \sum_f \chi(f) = \chi(W)
$$
where the sum is taken over all the regions $f$, edges $e$ and vertices $v$ of $X$. Moreover we introduced the notion of ``$\kappa$'' (Subsection~\ref{subsec:q-imp_ob}).

\begin{theo}[Shadow formula, Turaev]\label{theorem:sh_for_RTW}
Let $X$ be a shadow of a oriented 4-manifold $W$, $G=\partial X$ and $M=\partial W$. If $G$ is a $q$-admissible colored knotted graph then
$$
I_r(G,M) = \kappa^{-\sigma(W)} \eta^{\chi(X)} |X|^r .
$$
\end{theo}

\begin{ex}
Let $S$ be the closed orientable surface of genus $g$. Let $M$ be the boundary of the oriented $D^2$-bundle $F$ over $S$ with Euler number $n\in\Z$. If $S=S^2$, $M$ is the lens space $L(n,1)$. The surface $S$ equipped with gleam $n$ is a shadow of $F$. By Example~\ref{ex:signature} $\sigma(F) = \sigma(X) = {\rm sgn}(n)$. Therefore by the shadow formula (Theorem~\ref{theorem:sh_for_RTW}) we get
$$
I_r(M) = \kappa^{-{\rm sgn}(n)} \eta^{2-2g} \sum_{k=0}^{r-2} (-1)^{nk} A^{-nk(k+2)} \cerchio_k^{2-2g} .
$$
\end{ex}

\begin{prop}\label{prop:RTW_conj}
Let $M$ be an oriented compact 3-manifold and let $\bar M$ be $M$ with the opposite orientation. Then
$$
I_r(\bar M) = \overline{I_r(M)} ,
$$
where $\overline{x}$ is the conjugate of the complex number $x$.
\begin{proof}
Let $X$ be a shadow of $M$ and let $W$ be the 4-dimensional thickening of $X$. If we change the orientation of $W$ we get an oriented 4-manifold $\bar W$ whose boundary is $\bar M$. If we change the gleam of the regions of $X$ with its opposite we get a shadow $\bar X$ of $\bar W$. The signature of the 4-manifold is the opposite of the previous one, $\sigma(\bar W) = -\sigma(W)$, hence by the shadow formula (Theorem~\ref{theorem:sh_for_RTW}) we get
$$
I_r(\bar M) = \kappa^{\sigma(W)} \eta^{\chi(X)} |\bar X|^r .
$$
Since $A$ is a root of unity its conjugate is $A^{-1}$. The quantum integers valued in $A$ are real numbers, hence $\cerchio_f$, $\cerchio_{e_\partial}$, $\teta_e$, $\teta_{v_\partial}$ and $\tetra_v$ are real numbers too. The phases $A_f$ of $\bar X$ are the inverses of the ones of $X$ and are roots of unity. Hence $|\bar X|^r$ is the conjugate of $|X|^r$. The skein $\kappa= \Omega U_+$ is the unknot with framing $+1$ and color $\Omega$. This framed knot has a diagram that is a circle with a curl. Since the conjugate of $A$ is $A^{-1}$ and we need just the smallest filed containing $\mathbb{Q}$ and $A$, the conjugate of the skein of a colored link is obtained substituting $A$ with $A^{-1}$. If we substitute $A$ with $A^{-1}$ we get the skein of the mirror image of the diagram with the same color, namely the skein of the unknot with framing $-1$ and color $\Omega$ that is $\Omega U_- = \kappa^{-1}$. The same topics apply to $\eta$ but it is based on the 0-framed unknot and substituting $A$ with $A^{-1}$ we still get $\eta$, $\eta= \overline{\eta}$. Therefore we have $I_r(\bar M) = \overline{\kappa}^{-\sigma(W)} \overline{\eta}^{\chi(X)} \overline{| X|^r }= \overline{I_r(M)}$. 
\end{proof}
\end{prop}

The following is another famous quantum invariant \cite{Turaev-Viro}.
\begin{defn}\label{defn:Turaev-Viro}
Let $M$ be a compact oriented 3-manifold and let $X$ be a spine of $M$ minus some 3-balls. With the notations above the \emph{Turaev-Viro invariant} $TV_r(M)$ of $M$ is
$$
TV_r(M) := \eta^{\chi(X)} \sum_{\xi} \frac{\prod_f \cerchio_f^{\chi(f)} \prod_v \tetra_v }
{\prod_e \teta_e^{\chi(e)} } ,
$$
where $\xi$ runs over all the $q$-admissible colorings of $X$.
\end{defn}
If $M$ is closed and $X$ is the dual of a triangulation of $X$ with $n$ vetices $\chi(X)=n$.

The following result has been proved by Walker \cite{Walker:preprint} and Turaev \cite{Turaev:preprint} with a long and complicated method. Later Roberts \cite{RobertsPhD} proved it in an easy way without using shadows. The shadow formula provides a very easy way to get the result.
\begin{theo}[Walker-Turaev]\label{theorem:Turaev-Viro}
Let $M$ be a closed oriented 3-manifold. Then
$$
TV_r(M) = |I_r(M)|^2 ,
$$
where $|x|^2$ is the square of the absolute value of the complex number $x$.
\begin{proof}
Let $X$ be a spine of $M$. Equip every region of $X$ with gleam $0$. By Lemma~\ref{lem:spines}-(2.) $X$ is a shadow of the connected sum $M\# \bar M$, where $\bar M$ is $M$ with the opposite orientation. Therefore by the shadow formula (Theorem~\ref{theorem:sh_for_RTW}) we get
$$
TV_r(M) = \kappa^{\sigma(W)}  I_r(M \# \bar M),
$$
where $W$ is the 4-dimensional thickening of $X$. By Lemma~\ref{lem:spines} $W=(M\setminus B^3 )\times [-1,1]$ where $B$ is an open 3-ball, hence $\sigma(W)=0$. We conclude noting that by Proposition~\ref{prop:RTW_conn_sum} and Proposition~\ref{prop:RTW_conj} $I_r(M \# \bar M) = I_r(M) \overline{I_r (M) }= |I_r(M)|^2$.
\end{proof}
\end{theo}

\subsection{Proof}\label{subsec:proof_sh_for_RTW}

Now we prove the shadow formula (Theorem~\ref{theorem:sh_for_RTW}). First we focus on the case without boundary, then on the general case.

\begin{proof}[Proof of Theorem~\ref{theorem:sh_for_RTW}, case $G=\varnothing$, \cite{Carrega_RTW}]
We follow the following steps:

\begin{enumerate}
\item{from a triangulation of the regions of $X$ we construct a surgery presentation $(L,f)$  of $M=\partial W$ in $S^3$, where $f$ is the framing and $L$ is the underlying link;}
\item{we change the framing of the surgery presentation using the identity in Proposition~\ref{prop:half-framingchange};}
\item{we apply the 2-strand fusion identity (Fig.~\ref{figure:2fusion}) to remove the triangulation of the regions;}
\item{we apply the 3-strand fusion identity (Fig.~\ref{figure:3fusion}) to reduce ourselves to a trivalent graph in a tubular neighborhood of a tree (a union of trees);}
\item{we reduce the trivalent graph to isolated tetrahedra;}
\item{we note that the contributions of all our moves give the shadow formula.}
\end{enumerate}

\subsubsection{Step 1}

We can suppose that $W$, and hence also $X$, is connected. Let $\Gamma \subset S^3$ be an embedding of the 1-skeleton of $X$. Note that we may choose $\Gamma$ so that it has only unknotted edges. 

Select a small 3-ball $B_v$ centered in every vertex $v$ of $\Gamma$. The intersection of one of such 3-balls with $\Gamma$ consists of four unknotted and unlinked strands with one end in $v$ and one in the boundary $\partial B_v$ of the ball. For each vertex $v$ of $X$, put a tetrahedron in $\partial B_v$ whose vertices coincide with $\partial B_v \cap \Gamma$, and give it the framing of the 2-sphere $\partial B_v$. Every non closed edge $e$ of $\Gamma$ connects two different vertices $p_{e,1}$, $p_{e,2}$ of the set of framed tetrahedra (the vertices hold in the same tetrahedron if and only if $\chi(e)=0$). We connect the local framed strands incident to $p_{e,1}$ to the ones of $p_{e,2}$ via three strips running in a regular neighborhood $H_e$ of $e$ as follows. 

\begin{rem}\label{rem:bridge}
Usually a framed graph in $S^3$ with orientable framing, is represented by a diagram in $D^2$ (or $S^2$) giving it the black-board framing. Our pictures are made thinking to that method. Note that the regular neighborhoods $H_e$'s are bridges maybe passing over a part of the previous diagram (Fig.~\ref{figure:bridge1}). We will describe the strips using their projection into a rectangle inside $H_e$. This rectangle is exactly the planar one in Fig.~\ref{figure:bridge1}. In this way we can easily construct a diagram of the final framed link. 
\end{rem}

\begin{figure}[htbp]
\begin{center}
\includegraphics[width = 9 cm]{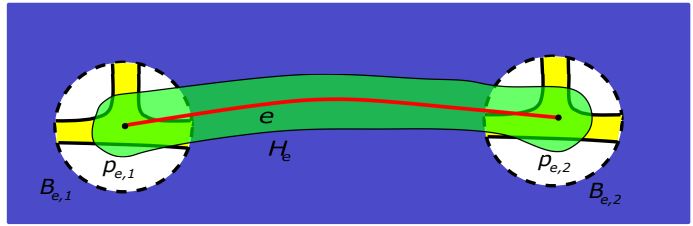}
\caption{The regular neighborhood $H_e$ of a non closed edge $e$ of $\Gamma$. The tube $H_e$ is colored with green, $e$ with red, The framed tetrahedra with yellow, and the blue part covers the rest.}
\label{figure:bridge1}
\end{center}
\end{figure}

The points $p_{e,1}$ and $p_{e,2}$ are trivalent vertices of the set of framed tetrahedra. Let $B_{e,1}$ and $B_{e,2}$ be two small ball neighborhoods of them. We can positively parametrize $B_{e,1}$ and $B_{e,2}$ as in Fig.~\ref{figure:attaching_an_edge1}: the vertex is in the origin of $\mathbb{R}^3$, the framing lies on the plane $\{ (x,y,0) \ | \ x,y\in \mathbb{R} \} $, the first strand lies on the first positive semi-axis ($x\geq 0$, $y=z=0$), the second strand lies on the second positive semi-axis ($y\geq 0$, $x=z=0$), and the third one lies on the first negative semi-axis ($x\leq 0,$ $ y=z=0$).
\begin{figure}[htbp]
\begin{center}
\includegraphics[width = 9 cm]{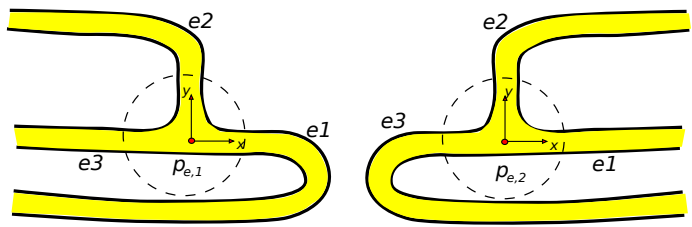}
\caption{Parametrized neighborhood of the vertices.}
\label{figure:attaching_an_edge1}
\end{center}
\end{figure}

The regular neighborhood of $e$ in $X$ describes a bijection between the framed strands incident to $p_{e,1}$ and the ones incident to $p_{e,2}$. Up to enumerations of the incident strands, and up to isotopies, we have just two possible bijections:
\begin{itemize}
\item{the second edge is fixed while the first and the third ones are exchanged; }
\item{all the enumerated edges are fixed.}
\end{itemize}
In the first case we connect those strands with three strips passing through $H_e$ and running around $e$ as described in Fig.~\ref{figure:attaching_an_edge_b}-(left): the strips lie in a rectangle whose intersection with $B_{e,1}$ and $B_{e,2}$ is the plane $\{z=0\}$ (see Remark~\ref{rem:bridge}). In the second case we connect those strands with three strips passing through $H_e$ and running around $e$ as described in Fig.~\ref{figure:attaching_an_edge_b}-(right): the strips can be drawn in that way in a rectangle whose intersection with $B_{e,1}$ and $B_{e,2}$ is the plane $\{z=0\}$ (see Remark~\ref{rem:bridge}). 
\begin{figure}[htbp]
\begin{center}
\includegraphics[width = 4.5 cm]{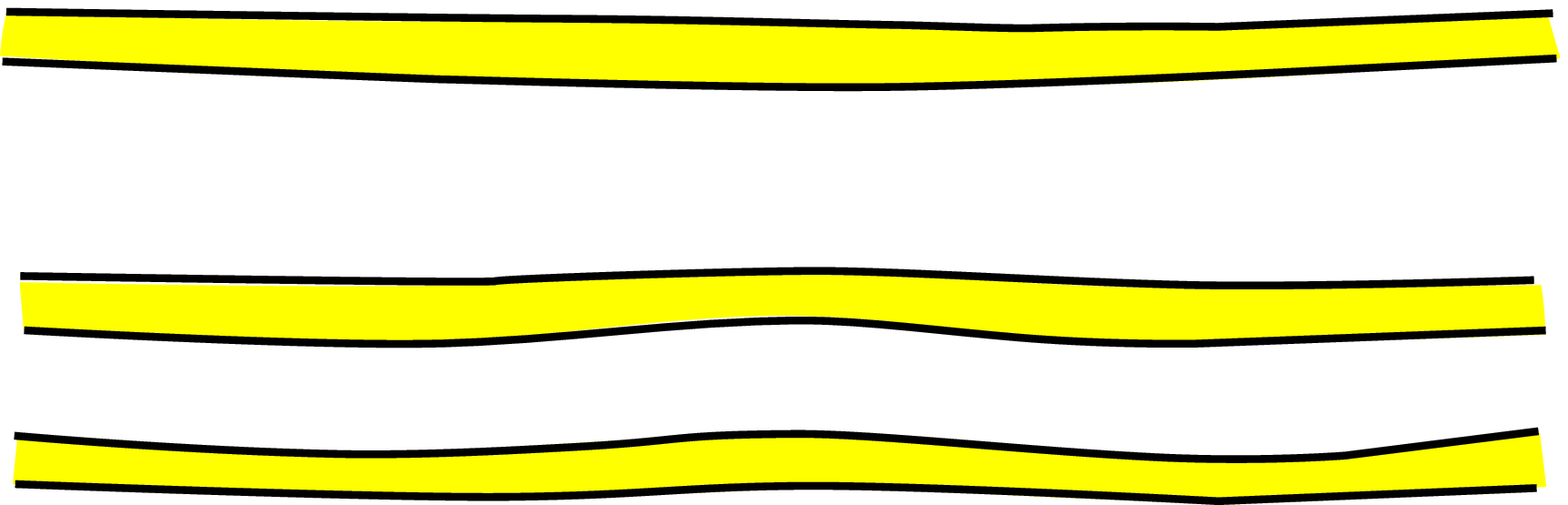} \hspace{1cm} \includegraphics[width = 4.5 cm]{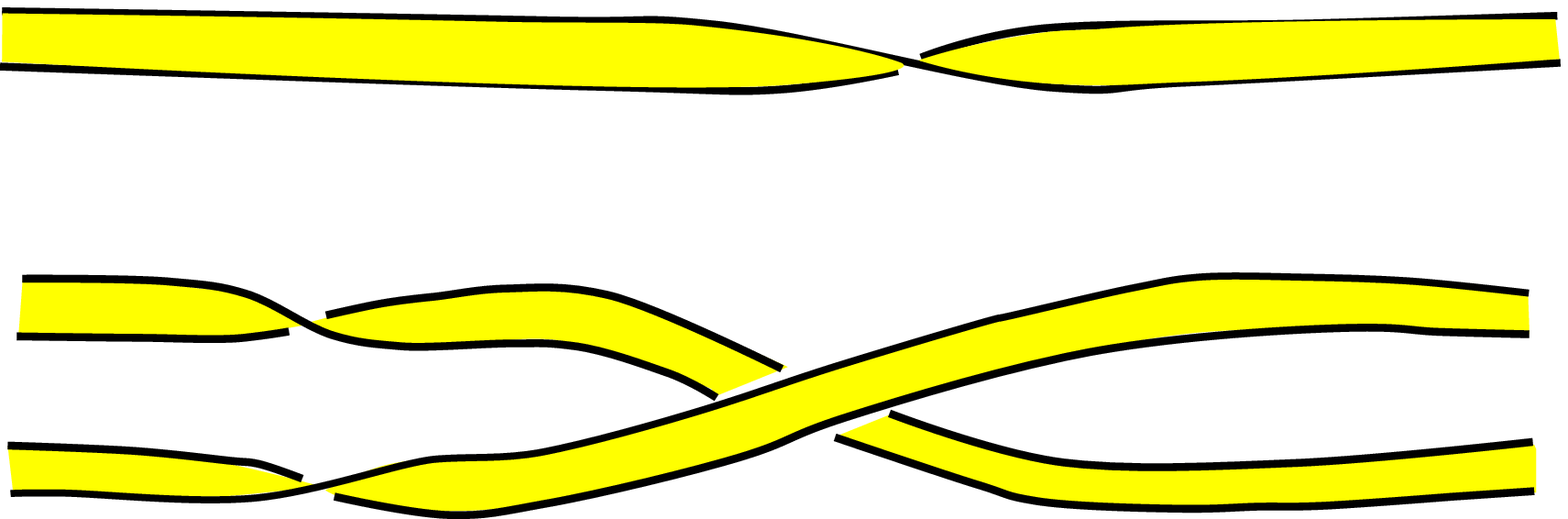}
\caption{Strips connecting the framed strands around the end points of a non closed edge of $\Gamma$.}
\label{figure:attaching_an_edge_b}
\end{center}
\end{figure}

Every closed edge $e$ of $\Gamma$ has a solid torus $V_e$ as regular neighborhood. We divide it in a 3-ball $B_p$ centered in a point $p\in e$ and in a 1-handle $H_e$, $V_e=B_p \cup H_e$. The graph $\Gamma$ intersects $B_p$ in a unknotted properly embedded arc. Put a $\teta$-graph in $\partial B_p$ whose vertices coincides with $\partial B_p \cap \Gamma$, and give it the framing of the 2-sphere $\partial B_p$. Now we have the same situation as before, and we apply the method above to connect the framed strands incident to $\partial B_p \cap \Gamma$ with strips running around the core of $H_e$.

Now we have obtained a framed link $L'\subset S^3$ (maybe with non orientable framing) lying in a regular neighborhood of $\Gamma$. Let $X^t$ be the simple polyhedron $X$ equipped with the further structure of a triangulation of each region. There is a natural bijectction between the components of $L'$ and the set of connected components of the boundary of the regions of $X$ (every component has the information of an ambient region, thus a component is taken twice if it is in the boundary of two different regions), or, alternatively, with the set of connected components of the boundary of the regular neighborhood of the 1-skeleton of $X$. Let $R$ be a region of $X$. Let $L'_R$ be the sub-link of $L'$ whose components are in bijection with the connected components of $R$ ($L' = \cup_R L'_R$, $R_1\neq R_2$ implies $L'_{R_1} \cap L'_{R_2} = \varnothing$). Let $\Gamma_R$ be an embedding of the 1-skeleton of the triangulation of $R$ whose restriction to $\partial R$ is $L'_R$ (without framing). With \emph{internal vertex} of $\Gamma_R$ we mean the image under the embedding of a vertex of the triangulation that lies in the interior of $R$. Select a small 3-ball neighborhood $B_v$ of every internal vertex $v$ of $\Gamma_R$. The set $B_v \cap \Gamma_R$ is a finite number of unknotted and unlinked strands with one end in $v$ and one in $\partial B_v$. Put a 0-framed unknot in $\partial B_v$ containing $\partial B_v \cap \Gamma_R$. Every internal edge $e$ of $\Gamma_R$ (an edge adjacent to an internal vertex) connects two distinct points, $p_{e,1}$ and $p_{e,2}$, lying in framed strands: either in $L'_R$ or in a 0-framed unknot around an internal vertex. Let $H_e$ be a regular neighborhood of $e$ (see Fig.~\ref{figure:bridge2}). As before we connect the local framed strands incident to $p_{e,1}$ and $p_{e,2}$ with two strips running through $H_e$. 

\begin{figure}[htbp]
\begin{center}
\includegraphics[width = 9 cm]{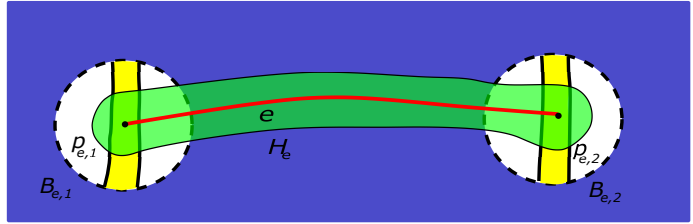}
\caption{The regular neighborhood $H_e$ of an edge $e$ of $\Gamma_R$. The tube $H_e$ is colored with green, $e$ with red, the framed tetrahedra with yellow, and the blue part covers the rest.}
\label{figure:bridge2}
\end{center}
\end{figure} 

Let $B_{e,1}$ and $B_{e,2}$ be two small ball neighborhoods of $p_{e,1}$ and $p_{e,2}$. We can positively parametrize $B_{e,1}$ and $B_{e,2}$ as in Fig.~\ref{figure:attaching_an_edge4}: the point is in the origin of $\mathbb{R}^3$, the framing lies on the plane $\{ (x,y,0) \ | \ x,y\in \mathbb{R} \} $, the first strand lies on the second positive semi-axis ($y\geq 0$, $x=z=0$) and the second one lines on the second negative semi-axis ($y\leq 0$, $x=z=0$).
\begin{figure}[htbp]
\begin{center}
\includegraphics[width = 9 cm]{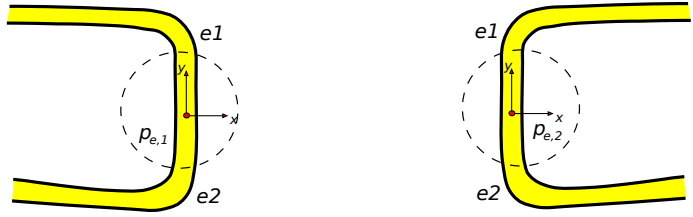}
\caption{Parametrized neighborhood of the non singular points.}
\label{figure:attaching_an_edge4}
\end{center}
\end{figure}

The regular neighborhood of $e$ in $X$ describes a bijection between the edges incident to $p_{e,1}$ and the ones incident to $p_{e,2}$. Up to enumerations of the incident edges, and up to isotopies, we have just one possible bijection: the enumerated edges are fixed. Thus we connect these strands with two strips passing through $H_e$ and running around $e$ as described in Fig.~\ref{figure:attaching_an_edge_c}: the strips lie in a rectangle whose intersection with $B_{e,1}$ and $B_{e,2}$ is the plane $\{z=0\}$ (see Remark~\ref{rem:bridge} and Fig.~\ref{figure:bridge2}).
\begin{figure}[htbp]
\begin{center}
\includegraphics[width = 4.5 cm]{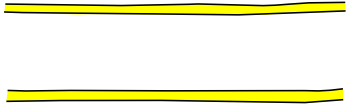}
\end{center}
\caption{Strips connecting the framed strands around the end points of an internal edge of $\Gamma_R$.}
\label{figure:attaching_an_edge_c}
\end{figure} 

Repeat this procedure for every region $R$ of $X$. Now we have obtained a framed link lying in a regular neighborhood of a connected graph $\Gamma^t \subset S^3$ containing $\Gamma$. The graph $\Gamma^t$ is an embedding of the 1-skeleton of $X^t$ extending the one of $X$ and the ones of the regions, $\Gamma, \Gamma_R \subset \Gamma^t$. If $\Gamma^t$ is a circle, add a 0-framed meridian unknot that encircles the framed link. By ``\emph{maximal tree}'' we mean a collapsible sub-graph of an ambient graph whose set of vertices is the same as the ambient one. If $\Gamma^t$ has a vertex (is not a circle), take a maximal tree $T$ of $\Gamma^t$. For each edge of $\Gamma^t$ not lying in $T$ add a 0-framed unknot that encircles the strips running around its regular neighborhood.

We get a framed link $(L,f_1)$ in $S^3$ (the framing $f_1$ may not be orientable) with some components, $\epsilon_1,\ldots, \epsilon_k$, corresponding to the boundary components of the tubular neighborhood of the 1-skeleton of $X^t$, and the other ones, $\delta_1,\ldots,\delta_g$, are the added 0-framed unknots. For each region $R$ of $X$ take the curves $\epsilon_{i(R,1)}, \ldots , \epsilon_{i(R,k_R)}$ that lie in the regular neighborhood of $\Gamma_R$. Add to their framings some half-twists so that the sum (with sign) of the added twists is equal to $-\gl(R)\in \frac 1 2 \Z$. So we get the framing $f$.

\begin{rem}
The framing $f$ clearly depends on the choice made when we modified $f_1$. If $X$ is standard we can take the trivial triangulation: just one triangle for each region. In that case there is only one choice to modify $f_1$ and we can easily see that it makes $f$ orientable:

The gleam$\gl(R)$ is an integer if and only if there is an even number of edges of $X$ that form $\partial R$ and define the second bijection in the list given when we described the connecting strips. Only in that case we applied a half-twist to each strip (see Fig.~\ref{figure:attaching_an_edge_b}-(right)). Hence $L'_R$ is a knot with an orientable framing plus an even number of half-twists. Namely $f$ is orientable.
\end{rem}

\begin{theo}
One of the choices above makes $(L,f)$ a surgery presentation of $M$ in $S^3$.
\end{theo}

\begin{rem}
Links in $\#_g(S^1\times S^2)$ like $\epsilon_1 \cup \ldots \cup \epsilon_k$, are called \emph{universal links}. The word ``universal'' is due precisely to the fact that we can get any orientable connected closed 3-manifold by surgering on them. See \cite{Costantino-Thurston} and \cite{Costantino2}, in particular \cite[Proposition 3.35]{Costantino-Thurston} and \cite[Proposition 3.36]{Costantino-Thurston}.
\end{rem}

\subsubsection{Step 2}

Let $(L,f)$ be the surgery presentation in $S^3$ of $\partial W$ described above. We remind that
$$
I_r(\partial W) = \eta \kappa^{-\sigma((L,f))} \Omega (L,f) ,
$$
where $\Omega (L,f)$ is the skein element got by coloring each component of $(L,f)$ with $\Omega= \eta \sum_{n=0}^{r-2}\cerchio_n \phi_n$. The skein $\Omega (L,f)$ is equal to the sum over $0 \leq n_1, \ldots n_k \leq r-2$ of $\eta^k \prod_i \cerchio_{n_i}$ times the skein element obtained by giving to every $\delta_j$ the color $\Omega$ and to each $\epsilon_i$ the $n_i^{\rm th}$ projector. Fix one of these colorings.

We change again the framing of the $\epsilon_{i(R)}$'s by adding $\gl(R)$ positive twists, namely we return to $f_1$. For Proposition~\ref{prop:half-framingchange} each of these framing changes produces a multiplication by $A_R$ ($ = (-1)^{\gl(R)n_{i(R)}} A^{-\gl(R) n_{i(R)}(n_{i(R)}+2)}$), the phase of the colored region $R$.

\begin{rem}
$f$ depends on a made choice and not all gives a surgery presentation. However we can easily check with this use of Proposition~\ref{prop:half-framingchange} that the skein elements given by all these $f$'s are the same.
\end{rem}

\subsubsection{Step 3}

Near each $\delta_j$ encircling two strands, we have the situation of the 2-strand fusion (see Fig.~\ref{figure:2fusion}). Apply the identity to each such $\delta_j$. The identity splits $L$ to a new link, and multiplies it by a coefficient. Furthermore the identity says that the colors of the $\epsilon_i$'s related to the same region must be equal, otherwise the summand is null. The new colored framed link consists of the curves $\delta_{j,1},\ldots , \delta_{j,g'}$ colored with $\Omega$ and the colored curves $\epsilon'_1,\ldots , \epsilon'_{k'}$ encircled by the $\delta_{j,l}$'s, all with the framing induced by $f_1$. Applying an isotopy we can see that the framed link $\epsilon'_1 \cup \ldots \cup \epsilon'_{k'}$ is equal to the framed link $L'$ we constructed before introducing the triangulations of the regions.

\subsubsection{Step 4}

Near each $\delta_{j,l}$ we have the situation of the 3-strand fusion (see Fig.~\ref{figure:3fusion}). Hence we apply it for each $l=1,\ldots,g'$. Thus if there is a non $q$-admissible triple $(n_{i1}, n_{i2}, n_{i3})$ of colors of $\epsilon'_i$'s encircled by a $\delta_{j,l}$, then the summand is null. After the application of all the 3-strand fusions it remains an unknotted 0-framed trivalent graph $G$ in the regular neighborhood a tree. 

\subsubsection{Step 5}

$G$ has two vertices for each edge of $X^t$ which does not lie in $T$ and three parallel edges for each edge of $T$. A tree graph has vertices connected by an edge with only another vertex, the \emph{leaves}, and the other ones are connected with two different vertices. Near a leaf of $T$, $G$ has the form of the left-hand side of the equality in the following Lemma~\ref{lem:a}. We apply the equality to the parts of $G$ near a fixed leaf of $T$. We get a multiple of another trivalent graph $G_1$ that is made in the same way as $G$ but encircling the embedding of the sub-graph $T_1$ of $T$ obtained removing the fixed leaf and the adjacent edge. We repeat this procedure until we finish the edges of $T$. The lemma says also that if there is an edge of $T$ with the three strands along it that are colored with a non $q$-admissible triple, then the summand is null. Therefore our summation is taken over all the $q$-admissible colorings of $X$.

\subsubsection{Step 6}

Fix one of these $q$-admissible colorings $\xi$. By the  applications of the 2-strand fusion identities and the framing change, we get for each region $R$ a contribute of $A_R \cdot \cerchio_{n_i}^{\chi(R)-1}
$ where $A_R$ is the phase of $R$ colored with $n_i$ (we added also the contribution $\cerchio_{n_i}$ times the considered skein element). 

By the applications of Lemma~\ref{lem:a} we get a contribute $\tetra_v$ for each colored vertex of $X$, and a $\teta_e$ for each colored edge of $X$ lying on the maximal tree of the 1-skeleton of $X^t$.

By the applications of the 3-strand fusion we get a contribute of $\teta_e$ for each colored edge of $X$ that does not lie on the maximal tree.

$\sigma((L,f))=\sigma(W_1)$, where $W_1$ is the 4-manifold obtained attaching to $D^4$ a 2-handle along each component of $(L,f)$. Let $W$ be the 4-manifold obtained giving a dot to the $\delta_j$'s. Let $W_g$ be the 4-dimensional orientable handlebody of genus $g$ (the compact 4-manifold with a handle-decomposition with just $k$ 0-handles and $k-g+1$ 1-handles for some $k>0$), and let $\#_{\partial,g}(D^2\times S^2)$ be the boundary connected sum of $g$ copies of $D^2\times S^2$. They have the same boundary: $\#_g(S^1\times S^2)$. The framed link $(L,f)$ has a corresponding framed link $L'$ in $\#_g(S^1\times S^2)$. Let $W'$ be the 4-manifold obtained attaching to $\#_g(S^1\times S^2)\times [-1,1]$ a 2-handle along each component of $L'\times\{1\}$. We have $W= W'\cup W_g$ and $W_1= W' \cup \#_{\partial,g}(D^2\times S^2)$. The signature is additive and $\sigma(W_g)=\sigma(\#_{\partial,g}(D^2\times S^2)) =0$. Therefore $\sigma((L,f))=\sigma(W_1)= \sigma(W') + \sigma(W_g) = \sigma(W') + \sigma(\#_{\partial,g}(D^2\times S^2)) = \sigma(W)$. 

Moreover the 2- and 3-strand fusions give a contribute of $\eta^{-g}$. Therefore we get that the summand of $\Omega (L,f)$ is equivalent to
$$
\eta^{k-g}|X|_\xi  .
$$
$k$ is the number of triagles of $X^t$ and $g$ is the genus of its 1-skeleton. Hence $1-g+k=\chi(X)=\chi(W)$. Therefore
$$
I_r(\partial W) = \eta^{\chi(W)} \kappa^{-\sigma(L)} |X|^r .
$$
\end{proof}

\begin{lem}\label{lem:a}
$$
\pic{1.7}{0.7}{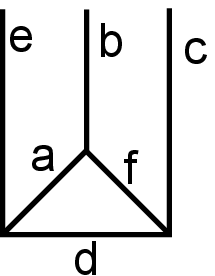} = \left\{\begin{array}{cl}
\frac{\pic{1.7}{0.5}{tetra_color.eps}}{\teta_{e,b,c}} \ \ \pic{1.7}{0.7}{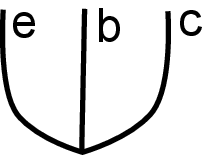} & \text{if }(e,b,c)\text{ is $q$-admissible}\\
0 & \text{if }(e,b,c)\text{ is not $q$-admissible}
\end{array}\right.
$$
\begin{proof}
\beq
\pic{1.7}{0.7}{lemma1.eps} & = & \sum_i \left\{\begin{matrix} a & b & i \\ c & d & f \end{matrix}\right\} \ \pic{1.7}{0.7}{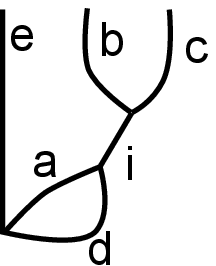} \\
 & = & \sum_{i,j} \left\{\begin{matrix} a & b & i \\ c & d & f \end{matrix}\right\} \left\{\begin{matrix} d & e & j \\ i & d & a \end{matrix}\right\} \ \pic{1.7}{0.7}{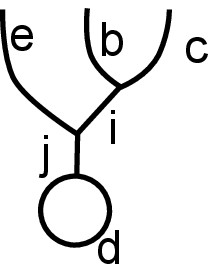}
\eeq
\beq
 & = & \left\{\begin{array}{cl}
 \left\{\begin{matrix} a & b & e \\ c & d & f \end{matrix}\right\} \left\{\begin{matrix} d & e & 0 \\ e & d & a \end{matrix}\right\} \ \pic{1.7}{0.7}{lemma4.eps} & \text{if }(e,b,c)\text{ is $q$-admissible} \\
 0 & \text{if }(e,b,c)\text{ is not $q$-admissible}
 \end{array}\right.\\
 & = & \left\{\begin{array}{cl}
 \frac{\pic{1.7}{0.5}{tetra_color.eps} }{\teta_{e,b,c}} & \text{if }(e,b,c)\text{ is $q$-admissible} \ \pic{1.7}{0.7}{lemma4.eps} \\
 0 & \text{if }(e,b,c)\text{ is not $q$-admissible}
 \end{array}\right.
\eeq
\end{proof}
\end{lem}

Now we discuss the general case.
\begin{proof}[Proof of Theorem~\ref{theorem:sh_for_RTW}, general case, \cite{Carrega_RTW}]
We proceed in the same way as the case without boundary:
\begin{enumerate}
\item{from a triangulation of the regions of $X$ we construct a surgery presentation of $M$ in $S^3$ (a union of copies of $S^3$) together with the colored trivalent graph $G'$ corresponding to $G$;}
\item{we change the framing of the surgery presentation by using the equality in Proposition~\ref{prop:half-framingchange};}
\item{we apply the 2-strand fusion identity Fig.~\ref{figure:2fusion} to remove the triangulation of the regions;}
\item{we apply the 3-strand fusion identity Fig.~\ref{figure:2fusion} to reduce ourselves to a trivalent graph in a tubular neighborhood of a tree (a union of trees);}
\item{we apply Lemma~\ref{lem:a} to reduce the trivalent graph to isolated tetrahedra;}
\item{we note that the contributions of all our moves give the shadow formula.}
\end{enumerate}
Here we give some clarifications in order to do the first step. Take as $\Gamma \subset S^3$ an embedding of the 1-skeleton of $X$ minus the edges adjacent to (or contained in) the boundary $\partial X$. After the procedure for the edges of $\Gamma$, not all the vertices of the isolated tetrahedra are going to be removed. In fact the ones corresponding to internal edges adjacent to $\partial X$ remain. This happens if and only if $G\subset \#_g(S^1\times S^2)$ has vetices. There is still a bijection between the edges (closed or not) of the resulting framed graph $L'\subset S^3$ and a set consisting of all the boundary components of the internal regions and all the components of $G$. Fix a triangulation $X^t$ for each internal region of $X$. As before we use the embeddings $\Gamma_R$'s, but this time $R$ runs just over all the internal regions. The graph $\Gamma^t$ is an embedding of the 1-skeleton of $X^t$ minus the edges (of $X$) adjacent to $\partial X$. The graph $T$ is still a maximal tree of such connected graph $\Gamma^t$. The framed graph $G'\subset S^3$ corresponding to $G\subset \#_g(S^1\times S^2)$ is the sub-graph of $L'$ whose components correspond to the components of $G$.

To end we note that in the shadow formula the contribute of the external regions together with the one of the external edges has no effects, and the same happens for the contribute of the edges adjacent to external vertices together with the external vertices.
\end{proof}

\chapter{Some topological applications}

Thirty years after its discovery, we know only a few relations between the Jones polynomial and its topological properties and most of them are just conjectures. In this chapter we explain some notable applications of quantum invariants and just recall some other ones. As we will see, several topological applications of quantum invariants concern their behavior near a fixed complex point. We will say more about: 
\begin{itemize}
\item{the \emph{Bullock's theorem} about the \emph{character variety} (Section~\ref{sec:Bul_char_var});}
\item{the \emph{volume conjecture} (Section~\ref{sec:vol_conj});}
\item{the \emph{Chen-Yang's volume conjecture} (Section~\ref{sec:CY_vol_conj});}
\item{the \emph{Tait conjecture} (Section~\ref{sec:Tait_conj0});}
\item{the \emph{Eisermann's theorem} (Section~\ref{sec:Eisermann});}
\item{the classification of \emph{rational 2-tangles} (Section~\ref{sec:rational_tangles});}
\item{a criterion for non sliceness of Montesinos links (Section~\ref{sec:Montesinos}).}
\end{itemize}
The character variety is another very studied object in 3-dimensional topology. The Bullock's theorem concerns the evaluation in $A=- 1$ of the entire $\mathbb{C}[A,A^{-1}]$-skein module. 

The volume conjecture and the Chen-Yang's volume conjecture are about a limit of evaluations respectively of the \emph{colored Jones polynomial}, and the Turaev-Viro invariants and the Reshetikhin-Turaev-Witten invariants, where the evaluation points converge to $1$. 

The Tait conjecture is a proved theorem about the \emph{crossing number} of \emph{alternating links}, it concerns the \emph{breadth} of the Jones polynomial that is something like the degree, it concerns the behavior near $\infty$ and near $0$. 

Eisermann's theorem connects the Jones polynomial to 4-dimensional smooth topology, in particular to \emph{ribbon surfaces}. This theorem, the classification of 2-tangles and the criterion for Montesinos links concern the behavior of the Kauffman bracket in the imaginary unit $q=A^2= i$.


Section~\ref{sec:slice-ribbon} introduces notions needed in Section~\ref{sec:Eisermann}, Section~\ref{sec:Montesinos} and Section~\ref{sec:sh_rib_handle}. Section~\ref{sec:sh_rib_handle} is about an application of shadows that goes in favour of the slice-ribbon conjecture.

The Tait conjecture (as a result, not just as a conjecture) and Eisermann's theorem have been extended by the author and Martelli \cite{Carrega_Tait1, Carrega_Taitg, Carrega-Martelli} in several directions by using the technology of Turaev's shadows (see Chapter~\ref{chapter:shadows}). We are going to discuss these generalizations in Chapter~\ref{chapter:Tait} and Chapter~\ref{chapter:Eisermann}.

A nice result of Frohman and Kania-Bartoszynska \cite{FK2} connects quantum invariants near $q=0$ to normal surfaces theory. In particular it studies the \emph{order} at $q=0$ (Definition~\ref{defn:order}) of the Turaev-Viro invariant (Definition~\ref{defn:Turaev-Viro}). This result was used extensively for instance in \cite{Costantino-Martelli}.

Another important conjecture is the \emph{slope conjecture} \cite{Garoufalidis}. This relates the degree of the colored Jones polynomial of a knot in $S^3$ with the slope of the incompressible surfaces of the complement. A little more precisely, it says that for every $n\geq 0$ the quantity $\frac{4}{n^2} {\rm deg} J_n(K)$ is the slope of a properly embedded, incompressible, $\partial$-incompressible surface in the complement of $K$, where ${\rm deg}J_n(K)$ is the degree of the $n^{\rm th}$ Jones polynomial of the knot $K \subset S^3$. This has been proved for a large class of knots.

The \emph{AJ-conjecture} (see for instance \cite{ThangLe-Tran, Marche2}) is another important conjecture. This relates the colored Jones polynomial to the $A$-\emph{polynomial}, an invariant related to the character variety of the complement of a knot. This concerns some more complex algebraic properties of the colored Jones polynomial, like generators of principal ideals related to it.

\section{Character variety and Bullock's theorem}\label{sec:Bul_char_var}

Let $M$ be a compact oriented 3-manifold. Let us consider the $\mathbb{C}[A,A^{-1}]$-skein module of $M$ and evaluate it in $A=-1$ (or $A=1$). With the notations of Subsection~\ref{subsec:sk_sp} we are considering $KM(M;\mathbb{C}, -1)$. We remind that the universal coefficient property (Theorem~\ref{theorem:sk_mod_prop}-(5.)) gives us an isomorphism of $\Z[A,A^{-1}]$-modules connecting $KM(M)$ and $KM(M;\mathbb{C}, -1)$, where $KM(M; \mathbb{C}, -1)$ has the structure of $\Z[A,A^{-1}]$-module given by the map $\Z[A,A^{-1}] \rightarrow \mathbb{C}$ defined by $A \mapsto -1$. 

In $KM(M; \mathbb{C},-1)$ the over/underpasses are not recognized:
$$
\pic{3}{0.6}{incrociop.eps} = \pic{3}{0.6}{incrociop2.eps} .
$$
The disjoint union of links gives a multiplication that makes $KM(M; \mathbb{C},-1)$ a commutative algebra with $\varnothing$ as identity.

\begin{theo}[Bullock]
There is a finite set of knots of $M$ that generate $KM(M; \mathbb{C},-1)$ as a $\mathbb{C}$-algebra.
\begin{proof}
See \cite{Bullock1}.
\end{proof}
\end{theo}

\begin{defn}
By a \emph{representation} we mean a homomorphism from the fundamental group of $M$ to the group of special complex matrices of rank 2:
$$
\rho : \pi_1(M) \rightarrow SL_2(\mathbb{C}) .
$$
The \emph{character} of the representation is the composition with the trace:
$$
\chi_\rho := \text{trace} \circ \rho : \pi_1(M) \rightarrow \mathbb{C} ,
$$
and $X(M)$ denotes the set of all characters.
\end{defn}

\begin{rem}
The trace is invariant under conjugation of matrices, namely if two representations, $\rho_1$ and $\rho_2$, are conjugate (there is a matrix $B\in SL_2(\mathbb{C})$ such that for every $g\in \pi_1(M)$, $\rho_1(\gamma) = B^{-1} \rho_2(\gamma) B$) their characters are the same $\chi_{\rho_1} = \chi_{\rho_2}$. It holds that two representations have the same character if and only if they are conjugated.

A \emph{faithful} (injective) representation of a 3-manifold $M$ with discrete image is equivalent to a complete hyperbolic structure on $M$.
\end{rem}

For every $\gamma \in \pi_1(M)$ there is a function $t_\gamma :
X(M) \rightarrow \mathbb{C}$ given by $\chi_\rho \mapsto \chi_\rho(\gamma)$. The following theorem has been discovered independently
by Vogt \cite{Vogt} and Fricke \cite{Fricke}, first proved by Horowitz
\cite{Hor}, and then rediscovered by Culler and Shalen \cite{CS}.

\begin{theo}[Vogt, Fricke, Horowitz, Culler-Shalen]
There exists a finite set of elements $\{\gamma_1,\ldots,\gamma_m\}$ in $\pi_1(M)$ such that every $t_\gamma$ is an element of the polynomial ring $\mathbb{C}[t_{\gamma_1},\ldots,t_{\gamma_m}]$. Moreover $X(M)$ is a closed algebraic sub-set of $\mathbb{C}^m$.
\end{theo}

\begin{defn}\label{defn:ring_char}
Recall that a \emph{closed algebraic set} $X$ in $\mathbb{C}^m$ is the common zero set of an ideal of polynomials in $\mathbb{C}[x_1,\ldots,x_m]$. The elements of  $\mathbb{C}[x_1,\ldots,x_m]$ are \emph{polynomial functions} on $X$, and the functions $x_i$ are \emph{coordinates} on $X$. The quotient of $\mathbb{C}[x_1,\ldots,x_m]$ by the ideal of polynomials vanishing on $X$ is called the \emph{coordinate ring} of $X$. Different choices of coordinates would clearly lead to different parameterizations of $X$, but it follows from \cite{CS} that any two parametrizations of
$X(M)$ are equivalent via polynomial maps. Hence their coordinate rings are isomorphic and we may identify them as one
object: the \emph{ring of characters} of  $\pi_1(M)$, which we denote
by $\mathcal{R}(M)$.
\end{defn}

Each knot $K$ in $M$ determines a unique $t_\gamma$ by taking $\gamma$ as the free homotopy class of $K$. We can define $\chi_\rho(K) := \chi_\rho(\gamma)$. Hence $K$ determines the map $t_\gamma$. Conversely, any $t_\gamma$ is
determined by some (non-unique) knot $K$.

\begin{theo}[Bullock]\label{theorem:Bullock}
Let $M$ be a compact orientable 3-manifold. The map
$$
\Phi : KM(M; \mathbb{C},-1) \rightarrow \mathcal{R}(M)
$$
given by
$$
\Phi(K)(\chi_\rho) = -\chi_\rho(K)
$$
is a well defined surjective map of $\mathbb{C}$-algebras. If $KM(M; \mathbb{C},-1)$ is generated
by the  knots $K_1,\ldots,K_m$ then $-\Phi(K_1),\ldots,-\Phi(K_m)$
are coordinates on $X(M)$. Furthermore the kernel of $\Phi$ is generated by the set of nilpotent elements of $KM(M; \mathbb{C},-1)$.
\begin{proof}
See \cite{Bullock2}.
\end{proof}
\end{theo}

In Theorem~\ref{theorem:known_sk_mod}-(9.), Theorem~\ref{theorem:known_sk_mod}-(10.) and Theorem~\ref{theorem:known_sk_mod}-(11.) (respectively results of \cite{Marche1}, \cite{PS1} and \cite{PS2}) we saw that if $M$ 
\begin{itemize}
\item{is the complement of a torus knot in $S^3$;} 
\item{has abelian fundamental group (\emph{e.g} $S^1\times S^2$, $T^3 = S^1\times S^1\times S^1$, a lens space $L(p,q)$);}
\item{is an orientable compact surface (maybe with non empty boundary);}
\end{itemize}
we have that the morphism $\Phi$ is an isomorphism.

\begin{conj}
For every compact 3-manifold $M$, the algebra $KM(M; \mathbb{C},-1)$ has no nilpotent elements. Hence $\Phi$ is an isomorphism.
\end{conj}

\section{The volume conjecture}\label{sec:vol_conj}

Since the work of W. Thurston and Mostow we know that there is a deep interaction between the topological study of 3-manifolds and low-dimensional hyperbolic geometry. One of the most important problems in modern knot theory is the \emph{volume conjecture}. This was introduced by H. Murakami and J. Murakami in 2001 \cite{Murakami} with the formalism of quantum groups and relates quantum invariants to the hyperbolic volume of the complement of a hyperbolic knot. Mostow's rigidity theorem states that if a 3-manifold admits a complete hyperbolic structure with finite volume, then every other hyperbolic structure is related to the previous one by an isometry. Therefore every metric information (volume, length spectrum, \ldots) is also a topological information. A link is said to be \emph{hyperbolic} if its complement admits a complete, finite-volume, hyperbolic structure. Almost all links are hyperbolic. For all the information about hyperbolic geometry we can see \cite{Benedetti-Petronio}. Originally the volume conjecture concerns knots in $S^3$ and a limit of evaluations of the colored Jones polynomial. This conjecture has been generalized in many ways. Now we state it in an extended form.

Let $L$ be a framed link in the connected sum $\#_g(S^1\times S^2)$ of $g\geq 0$ copies of $S^1\times S^2$ and let $\langle L, n \rangle$ be the Kauffman bracket of $L$ with each component colored with the $n^{\rm th}$ projector (normalized so that $\left\langle \pic{0.8}{0.3}{banp.eps} , 1 \right\rangle = -A^2-A^{-2}$). This depends on the framing of $L$, but by the identity in Fig.~\ref{figure:framingchange} the module of an evaluation of $\langle L, n \rangle$ in a root of unity does not.

\begin{conj}[(Extended) Volume conjecture]\label{conj:vol_conj}
Let $K\subset \#_g(S^1\times S^2)$ be a hyperbolic knot. Then
$$
{\rm Vol}(\#_g (S^1\times S^2) \setminus K) = \lim_{n \rightarrow \infty} \frac{2\pi}{n+1} \log \left.\left| \frac{\langle K, n \rangle}{\cerchio_n^{1-g}} \right|\right|_{q=A^2= \exp(\frac{\pi\sqrt{-1}}{n+1}) } ,
$$
where $\cerchio_n$ is the skein element of the unknot colored with the $n^{\rm th}$ projector ($\cerchio_n = (-1)^n [n+1]$, $[n]= (q^n -q^{-n})/(q-q^{-1})$), ${\rm Vol}(\#_g (S^1\times S^2) \setminus K)$ is the hyperbolic volume of the complement of $K$ and the limit is taken over all natural numbers $n$ such that $\frac{\langle K, n \rangle}{\cerchio_n^{1-g}}$ has neither a zero nor a pole in $q=A^2=\exp{\pi\sqrt{-1}}{n+1}$. For instance we know that if $L\subset \#_g(S^1\times S^2)$ is $\Z_2$-homologically non trivial, and $n$ is odd then $\langle L, n\rangle =0$ (Proposition~\ref{prop:0Kauff_gr}), hence in this case the limit is taken only over even numbers. If $K$ is not hyperbolic ${\rm Vol}(\#_g(S^2\times S^1) \setminus K)$ is the sum of the volume of the hyperbolic parts of the complement of $K$ in a JSJ decomposition, that is equal to $v_3 \cdot \| \#_g(S^1\times S^2) \setminus K\|$, where $v_3$ is the volume of the regular ideal hyperbolic tetrahedron and $\| M \|$ is the Gromov norm of $M$.
\end{conj}

We note that
$$
\cerchio_n|_{q=A^2= \exp(\frac{\pi\sqrt{-1}}{n+1}) } = 0 .
$$

Until 2008 the conjecture for knots in $S^3$ has been proved only for the \emph{figure-eight knot}, \emph{torus knots}, \emph{Whitehead doubles} of certain torus knots, and connected sums of these knots \cite{Murakami1, KasTir, Zheng}. It is well known that the volume conjecture does not hold for many split links in $S^3$. Hence it is not clear how to extend the conjecture to links. However the conjecture is true for the \emph{Whitehead link} \cite{Zheng} and \emph{Borromean rings} \cite{Garoufalidis-ThangLe}. More in general van der Veen defined the class of \emph{Whitehead chains}: a set of links in $S^3$ that comprehends the Whitehead link and Borromean rings, and proved the volume conjecture for these links \cite{vanderVeen}. Costantino proved the extended volume conjecture for an infinite family of hyperbolic links in $\#_g(S^1\times S^2)$: the \emph{universal hyperbolic links} \cite{Costantino2}. To do that he used the formalism of Turaev's shadows (Chapter~\ref{chapter:shadows}).

Usually the volume conjecture is expressed in terms of the \emph{colored Jones polynomial}. The $n^{\rm th}$ \emph{colored Jones polynomial} of a (framed or oriented) link $L\subset \#_g(S^1\times S^2)$ is defined in several (almost equivalent) ways:
\begin{itemize}
\item{$$
J_n(L, \#_g(S^1\times S^2)) := \langle L, n \rangle ;
$$}
\item{If $L\subset S^3$
$$
J_n(L, S^3) := ((-1)^nA^{n^2+2n})^{-w(L)} \langle L, n \rangle ,
$$
where $w(L)$ is the writhe number of $L$ with an arbitrary framing (the sum of the signs of all the crossings of an oriented diagram of $L$), hence $J_n(L, S^3) := \langle L, n\rangle$ if $L$ has the Seifert framing (in this way we get an invariant of oriented, non framed links and of non oriented, non framed knots in $S^3$);
}
\item{$$
J_n(L, \#_g(S^1\times S^2)) := \frac{\langle L, n \rangle}{(-1)^{gn} \cerchio_n^{1-g}} ,
$$
with this normalization $J_n(L,S^3) =1 $;}
\item{$$
J_n(L, \#_g(S^1\times S^2)) := \langle L, n-1 \rangle ;
$$}
\item{all the combinations of the previous variations of $\langle L, n \rangle$.}
\end{itemize}

The colored Jones polynomial comes out from the finite-dimensional indecomposable representations of the $\mathbb{C}[[h]]$-Hopf algebra $U_q(\mathfrak{sl}_2)$, also called $U_h(\mathfrak{sl}_2)$, where $q=e^h$ ($\mathbb{C}[[h]]$ is the ring of Laurent series over the complex numbers with variable $h$). $U_q(\mathfrak{sl}_2)$ is a non trivial (non co-commutative) quantization (or deformation) of the universal enveloping algebra $U(\mathfrak{sl}_2)$ of $\mathfrak{sl}_2(\mathbb{C})$ and has some good properties like a natural ribbon algebra structure. The indecomposable representations correspond to the irreducible representations of a field. As for $U(\mathfrak{sl}_2)$, the finite-dimensional indecomposable representations of $U_q(\mathfrak{sl}_2)$ are in bijection with the positive integers. The $n^{\rm th}$ representation $V_n$ is the $(n+1)$-dimensional free module over $\mathbb{C}[[h]]$ with a further algebraic structure. If we define the $n^{\rm th}$ Jones polynomial by giving to $L$ the color $n$, we get the one related to the $(n+1)$-dimensional representation, namely the $n^{\rm th}$ representation. Usually people using the capital letter $N$, instead of $n$, associate to the $N^{\rm th}$ Jones polynomial the $N$-dimensional representation.

\begin{itemize}
\item{F. Costantino uses $J_n(L, \#_g(S^1\times S^2)) := \frac{\langle L, n-1 \rangle}{(-1)^{n-1}[n]^{1-g}}$ with variable $t=A^{\frac{1}{2}}$. He uses half-integer colorings instead of integer colorings (\emph{e.g.} \cite{Costantino2}).}
\item{J. March\'e uses $J_n(L, S^3)) := \langle L, n \rangle$ with variable $t=A$ (\emph{e.g.} \cite{Marche2}).}
\item{T.T.Q. Le and T. Tran use $J_n(L, S^3)) := (-1)^{n-1}\langle L, n-1 \rangle$ with variable $t=A$ ($J_n\left( \cerchio \right) = [n]$). They define $J_{-n}(L,S^3) := - J_n(L, S^3)$ (\emph{e.g} \cite{ThangLe-Tran}).}
\item{People using the representation theory of $U_q(\mathfrak{sl}_2)$ usually use the variable $q=-A^2$.}
\end{itemize}

The volume conjecture for a hyperbolic knot $K$ in $S^3$ can be expressed in terms of its \emph{Kashaev invariant} $\langle K \rangle_n$:
$$
{\rm Vol}(S^3 \setminus K) = \lim_{n \rightarrow \infty} \frac{2\pi}{n} \log |\langle K \rangle_n| .
$$

\section{The Chen-Yang's volume conjecture}\label{sec:CY_vol_conj}

This is a recent version of the volume conjecture for the Turaev-Viro invariants $TV_r(M)$ (Definition~\ref{defn:Turaev-Viro}) and the Reshetikhin-Turaev-Witten invariants $I_r(M)$ (Definition~\ref{defn:RTW}). This has been stated by Chen and Yang \cite{Chen-Yang} and has been supported by many numerical evidences. 

To get a Turaev-Viro or a Reshetikhin-Turaev-Witten invariant we need to fix an integer $r\geq 3$ and a $4r^{\rm th}$ root of unity $A$. Usually the root of unity is $A= e^{\frac{\pi i}{2r}}$. It has been proved that with that choice of $A$ for each $r$ the growth of these invariants is polynomial. The main idea is to change the choice of the root of unity and to use the fact that we do not need that $A$ is primitive, but just $A^4$ is so, namely $A^{4n}\neq 1$ for $n<r$. The choice is 
$$
A_r:=e^{\frac{\pi i}{r}} .
$$
Clearly $A_r^{4r}=1$, if $r\in 2\Z+1$, $A_r^4$ is a primitive $r^{\rm th}$ root of unity.

\begin{conj}[Chen-Yang's volume conjecture for T-V]\label{conj:C-Y_vol_conj_T-V}\label{conj:vol_conj_TV}
Let $M$ be a 3-manifold with a complete hyperbolic structure (maybe with cusps or geodesic boundary). Using the root of unity $A_r$ we have
$$
\lim_{r\rightarrow \infty} \frac{2\pi}{r-2} \log( TV_r(M) ) = {\rm Vol}(M) ,
$$
where $r$ runs just over the odd integers and ${\rm Vol}(M)$ is the volume of $M$.
\end{conj}

The numerical computations of the limit of Conjecture~\ref{conj:C-Y_vol_conj_T-V} for complements of knots $K\subset S^3$ shows that these sequences converge to the volume faster than the sequences of the volume conjecture (Conjecture~\ref{conj:vol_conj}).

The following extends Conjecture~\ref{conj:C-Y_vol_conj_T-V} for closed orientable 3-manifolds:
\begin{conj}[Chen-Yang's volume conjecture for R-T-W]\label{conj:vol_conj_RTW}
Let $M$ be a closed oriented hyperbolic 3-manifold. Using the root of unity $A_r$ we have
$$
\lim_{r\rightarrow \infty} \frac{4\pi}{r-2} \log( I_r(M) ) = {\rm Vol}(M) + i\pi^2( 2CS(M) +k) 
$$
for some $k\in \Z$, where $r$ runs just over the odd integers, ${\rm Vol}(M)$ is the volume of $M$ and $CS(M)$ is the \emph{Chern-Simons} invariant of $M$.
\end{conj}

These conjectures can be easily generalized to non hyperbolic 3-manifolds where ${\rm Vol}(M)$ is the volume of the hyperbolic part of $M$ in the JSJ decomposition. It holds that ${\rm Vol}(M)= v_3 \|M\|$, where $v_3$ is the volume of the regular ideal hyperbolic tetrahedron and $\|M\|$ is the Gromov norm of $M$.

By Theorem~\ref{theorem:Turaev-Viro} we have $TV_r(M) = |I_r(M)|^2$. Therefore we can compute the Turaev-Viro invariants with the shadow formula for the Reshetikhin-Turaev-Witten invariants (Theorem~\ref{theorem:sh_for_RTW}). In particular the shadow formula can be used to study both Conjecture~\ref{conj:vol_conj_TV} and Conjecture~\ref{conj:vol_conj_RTW}.

An oriented compact 3-manifold $M$ is said to be a \emph{graph manifold} if it can be decomposed by cutting along tori in blocks homeomorphic to solid tori or the product of a pair of pants with $S^1$. The set of graph manifolds coincides with the  set of 3-manifolds whose JSJ decomposition has only Seifert-fibered or torus bundle pieces. Therefore they are the 3-manifolds such that 
$$
{\rm Vol}(M) = 0.
$$
Moreover graph manifolds can also be characterized as the 3-manifolds with zero \emph{shadow complexity}, namely the manifolds which admit a shadow without vertices (see \cite{Costantino-Thurston}). Therefore to get the invariants of graph manifolds we can use the shadow formula on such simple shadows that has a nicer form than usual. Thus graph manifolds form a good candidate to start proving the generalized Conjecture~\ref{conj:vol_conj_TV}.

\section{The Tait conjecture}\label{sec:Tait_conj0}

One of the first invariants of links in $S^3$ is the \emph{crossing number}: the minimal number of crossings that a diagram must have to present that link. In general it is hard to compute. During his attempt to tabulate all knots in $S^3$ in the $19^{\rm th}$ century \cite{Tait}, P.G. Tait stated three conjectures about crossing number, \emph{alternating links} and writhe number (Definition~\ref{defn:writhe}). An \emph{alternating link} is a link that admits an \emph{alternating diagram}: a diagram $D$ such that the parametrization of its components $S^1\rightarrow D\subset D^2$ meets overpasses and underpasses alternately. All the conjectures have been proved before 1991. 

\begin{defn}\label{defn:reduced_0}
A diagram $D$ of a link in $S^3$ is said to be \emph{reduced} if it has no crossings as the ones in Fig.~\ref{figure:reducedDS3} (the blue parts cover the rest of the diagram).
\end{defn} 

\begin{figure}[htbp]
\begin{center}
\includegraphics[scale=0.6]{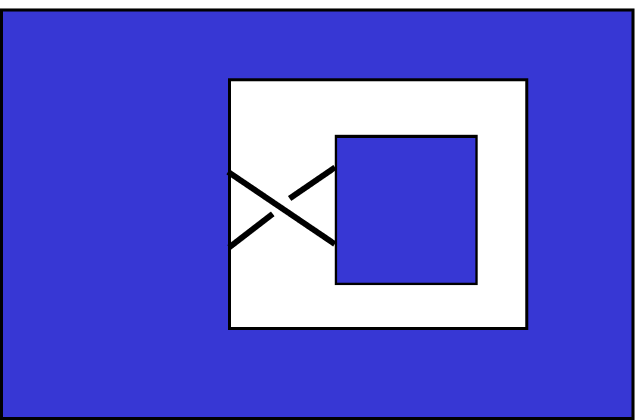}
\hspace{0.5cm}
\includegraphics[scale=0.6]{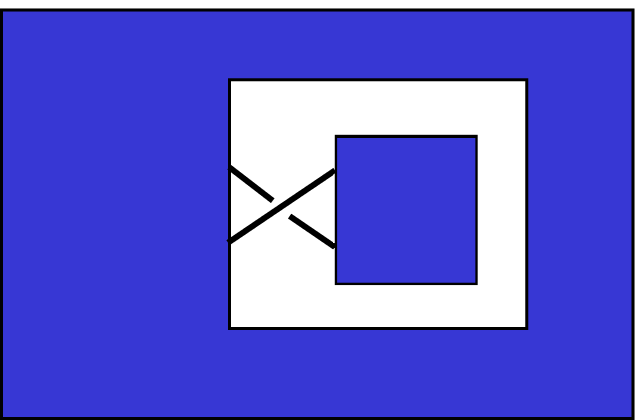}
\end{center}
\caption{Non reduced diagrams of links in $S^3$.}
\label{figure:reducedDS3}
\end{figure}

We focus just on one Tait conjecture, the following one:
\begin{theo}[The Tait conjecture in $S^3$]\label{theorem:Tait_conj_0}
Every reduced alternating diagram of links in $S^3$ has the minimal number of crossings.
\begin{proof}
See for instance \cite[Chapter 5]{Lickorish} or Chapter~\ref{chapter:Tait}.
\end{proof}
\end{theo}

This conjecture was proved in 1987 by M. Thistlethwaite \cite{Thistlethwaite}, L.H. Kauffman \cite{Kauffman_Tait} and K. Murasugi \cite{Murasugi1, Murasugi2}.

Theorem~\ref{theorem:Tait_conj_0} is one of the most notable applications of the Jones polynomial. In fact the proof of Thistelwaithe-Kauffman-Murasugi is based on the study of the \emph{breadth} of the Jones polynomial:

\begin{defn}\label{defn:breadth}
Let $f\in \Z[A,A^{-1}]$ be a non zero Laurent polynomial. The \emph{breadth} $B(f)\in \Z$ of $f$ is the difference between the biggest and the lowest degree of the non zero monom
ials of $f$. If $f=0$ we define $B(f):=0$.
\end{defn}

The breadth of the Jones polynomial, or of the Kauffman bracket, is independent of the chosen framings or orientations for the link, and we have $B(J(L))=B(\langle L\rangle)$. The theorem of Thistelwaithe-Kauffman-Murasugi shows that we can get information about the crossing number from the breadth of the Jones polynomial (or of the Kauffman bracket). In particular we can easily compute it if the link is alternating:
\begin{theo}
Let $D\subset D^2$ be a $n$-crossing alternating link diagram. Then 
$$
B(f(L)) = B(\langle D \rangle) = 4n+4 
$$
(with the variable $A$ and the normalization $\left\langle \pic{0.8}{0.3}{banp.eps} \right\rangle = -A^2-A^{-2} $).
\begin{proof}
See for instance \cite[Chapter 5]{Lickorish} or Chapter~\ref{chapter:Tait}.
\end{proof}
\end{theo} 

We have a natural representation of links in the connected sum $\#_g(S^1\times S^2)$ of $g\geq 0$ copies of $S^1\times S^2$ with diagrams in the disk with $g$ holes (see Subsection~\ref{subsec:diag}). This representation allows us to talk about crossing number and alternating links even in this general case of links in $\#_g(S^1\times S^2)$. Moreover we know that the Kauffman bracket is also defined on $\#_g(S^1\times S^2)$ (see Remark~\ref{rem:sk_sp_S1xS2} and Definition~\ref{defn:Kauf}). Therefore it is natural to ask if the Tait conjecture holds also in this general case, and if we can follow the same method of Thistelwaithe-Kauffman-Murasugi by improving our knowledge of the Jones polynomial. In Chapter~\ref{chapter:Tait} we are going to discuss this topic and we will prove that the proper version of the Tait conjecture in $\#_g(S^1\times S^2)$ is false for $\Z_2$-homologically non trivial links (Definition~\ref{defn:Z_2_tr}), and if $g=1$ or $g=2$ the proper version of the Tait conjecture is true for $\Z_2$-homologically trivial links. In Chapter~\ref{chapter:Tait} the method of Thistelwaithe-Kauffman-Murasugi will be followed and extended. Note that by Remark~\ref{rem:tensor} studying diagrams in the disk with $g$ holes $S_{(g)}$ that are contained in a disk with $g'\leq g$ holes is equivalent to studying links in $\#_{g'}(S^1\times S^2)$. Therefore we will get Theorem~\ref{theorem:Tait_conj_0} just by focusing on diagrams contained in a disk.

\section{Slice and ribbon}\label{sec:slice-ribbon}

This section introduces the concepts of ``\emph{ribbon surface}'', ``\emph{ribbon knot}'', ``\emph{slice knot}'', \ldots We will need these notions for Section~\ref{sec:Eisermann}, Section~\ref{sec:Montesinos} and Section~\ref{sec:sh_rib_handle}.

Every link in $S^3$ bounds an embedded orientable surface. Given a link we can consider all these surfaces.
\begin{defn}
The \emph{genus} of a knot in $S^3$ is the minimum genus of a connected compact orientable surface embedded in $S^3$ whose boundary coincides with the knot.
\end{defn}

We note that a knot has genus $0$ if and only if it is the \emph{unknot}, namely if it bounds a disk in the 3-dimensional ambient space.

\subsection{Slice}

There is an obvious compact 4-manifold whose boundary is $S^3$: the 4-disk $D^4$. Given a link in $S^3$ we can consider the properly embedded surfaces of $D^4$ whose boundary is $L$ and define another important invariant for knots:
\begin{defn}
The \emph{slice genus} of a knot is the minimum genus of a connected compact orientable (smooth or PL and locally flat) surface properly embedded in $D^4$ whose boundary coincides with the knot.
\end{defn}

\begin{rem}
A \emph{locally flat surface} is a topologically embedded surface $S\subset M$ such that for each point $p \in S$ there is a continuous map $\varphi: V \rightarrow \mathbb{R}^m$, such that $V$ is a neighborhood of $p$ in $M$, $m$ is the dimension of $M$, the restriction $V \rightarrow \varphi(V)$ is a homeomorphism, and $\varphi(V)\subset \{ (x,y,0, \ldots, 0) \in \mathbb{R}^m \}$. As said before we are in the smooth or PL category and we consider only locally flat surfaces. This observation is fundamental for the definition of the slice genus. In fact if we allowed all the topological surfaces, we would have that each knot is the boundary of a properly embedded disk that is just the cone in the 4-ball over the knot, hence the slice genus would be always $0$. If we look at the origin of $D^4$ we can see that this disk admits a smooth structure if and only if the knot is the unknot in $S^3$.
\end{rem}

Given a surface in $S^3$ we can push it in the interior of $D^4$ getting a properly embedded surface. Hence the slice genus is lower-equal than the genus. A knot has slice genus $0$ if and only if it bounds a properly embedded disk in $D^4$.
\begin{defn}
A knot in $S^3$ is said to be \emph{slice} if it has slice genus $0$. More generally a link in $S^3$ is said to be \emph{slice} if it bounds a properly embedded collection of disks in the 4-ball.
\end{defn}
There are non trivial slice knots as the one in Fig.~\ref{figure:exribbon}.

\begin{figure}[htbp]
$$
\pic{2}{0.7}{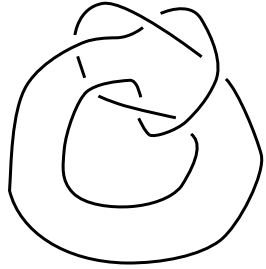}
$$
\caption{A ribbon knot.}
\label{figure:exribbon}
\end{figure}

We can extend the notion of ``slice knot'' to each compact orientable 4-manifold with boundary.

\subsection{Ribbon}

\begin{defn}
An immersed smooth surface $S\rightarrow M$ in a 3-manifold $M$ is \emph{ribbon} if one of the following equivalent conditions holds:
\begin{itemize}
\item{the surface $S$ may be isotoped to an immersed surface in $M$ having only \emph{ribbon singularities} as in Fig.~\ref{figure:ribbon_singularity};}
\item{the surface $S$ may be isotoped in $M\times [-1,1]$ into a properly embedded surface in Morse position, with only minima and saddle points (no maxima) with respect to the height function, $M\times [-1,1]\rightarrow [-1,1]$, $(x,t)\mapsto t$, as in Fig.~\ref{figure:Morse}.}
\end{itemize}
A ribbon surface may be not connected and non orientable.
\end{defn}

\begin{defn}\label{defn:ribbon_link}
The \emph{ribbon genus} of a knot is the minimal genus of a connected orientable ribbon surface bounded by the knot. Clearly the ribbon genus is not lower than the slice genus. A link is said to be \emph{ribbon} if it bounds a ribbon surface that is a collection of immersed disks (as the one in Fig.~\ref{figure:exribbon}).
\end{defn}

\begin{rem}\label{rem:ribbon_g}
In the 4-dimensional orientable handlebody $W_g$ of genus $g\geq 0$ (the compact 4-manifold with a handle-decomposition with $k$ 0-handles and $k-1+g$ 1-handles) there is a graph $\Gamma \subset W_g$ of genus $g$ (if $g=0$ $W_g$ is the 4-disk $D^4$ and $\Gamma$ is a point) whose complement is diffeomorphic to a collar of the boundary $W_g \setminus \Gamma \cong \#_g(S^1\times S^2) \times [0,1)$. By an isotopy we can put every properly embedded surface $S\subset W_g$ in a position such that $S\cap \Gamma = \varnothing$. Therefore an immersed surface $S\rightarrow \#_g(S^1\times S^2)$ is ribbon if and only if may be isotoped in $W_g$ into a properly embedded surface in Morse position, with only minima and saddle points (no maxima) with respect to the distant function from the graph $\Gamma$. 
\end{rem}

\begin{figure}[htbp]
\begin{center}
\includegraphics[width = 4 cm]{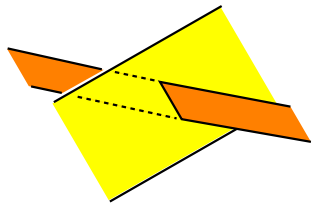}
\caption{A ribbon singularity}
\label{figure:ribbon_singularity}
\end{center}
\end{figure}

\begin{figure}[htbp]
\begin{center}
\includegraphics[width = 8 cm]{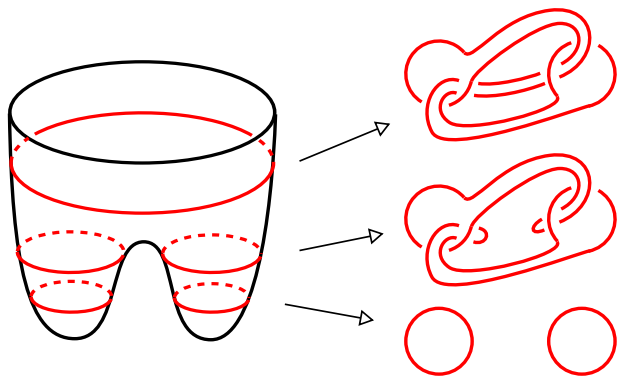}
\caption{A ribbon disk in $D^4$ in Morse position with two minima and one saddle. Each regular level gives a link in $S^3$.}
\label{figure:Morse}
\end{center}
\end{figure}

\begin{prop}
A link $L\subset M$ in a 3-manifold $M$ bounds a ribbon surface if and only if it is $\Z_2$-homologically trivial ($0=[L]\in H_1(M;\Z_2)$). Hence every link in $S^3$ bounds a ribbon surface.
\begin{proof}
It is well known in low-dimensional topology that a link $L\subset M$ is $\Z_2$-homologically trivial if and only if it bounds an embedded (maybe not orientable) surface $S\subset M$, hence a ribbon one. By definition if a link $L\subset M$ is ribbon it bounds a properly embedded surface in $M\times [-1,1]$, hence $0=[L]\in H_1(M\times [-1,1];\Z_2)$, hence $0=[L]\in H_1(M;\Z_2)$. 
\end{proof}
\end{prop}

\subsection{The slice-ribbon conjecture}

A very important conjecture in knot theory is the \emph{slice-ribbon conjecture}:
\begin{conj}[Slice-ribbon conjecture]
A knot in $S^3$ is ribbon if and only if it is slice.
\end{conj}
It is generalized in several ways:
\begin{conj}
$\ $
\begin{itemize}
\item{A link in $S^3$ is ribbon if and only if it is slice.}
\item{In $S^3$ the ribbon genus coincides with the slice genus.}
\item{The previous facts hold also for different ambient 4-manifolds, for instance 4-dimensional orientable handlebodies.}
\end{itemize}
\end{conj}

There are some results which suggest that the slice-ribbon conjecture is true. Lisca \cite{Lisca} in 2007 proved that the conjecture is true for \emph{2-bridge links} in $S^3$ by a gauge theoretic method. Further developments were done by Greene and Jubuka \cite{Greene-Jubuka}, Karimi \cite{Karimi} and Lecuona \cite{Lecuona1, Lecuona2}.

Gompf, Sharlemann and Thompson \cite{GST} created a potential counter-example both for the \emph{property $2R$ conjecture} and for the slice-ribbon conjecture. The construction is based on non trivial handle-decompositions of the 4-ball with only 0-, 1- and 2-handles. They took the boundary of the co-core of two 2-handles, $K_1$ and $K_2$. Hence $K_1$ and $K_2$ are clearly slice knots in $S^3$. They discovered that both $K_1$ and $K_2$ are ribbon. We can consider the link $L$ given by the union of $K_1$ and $K_2$. The link $L$ is clearly slice and for the moment we see no reason why it should be also ribbon. They connected the bounded slice disks with a knotted band. The result is clearly another slice disk and hence its boundary a slice knot that may not be ribbon. Following this idea Abe and Tange \cite{Abe-Tange} created a new family of slice knots that may not be ribbon. We can find an introduction to these topics in \cite{Abe}. Gompf, Sharlemann and Thompson found also a diagram representing this link in $S^3$.

\section{Eisermann's theorem}\label{sec:Eisermann}
Eisermann's theorem relates the Jones polynomial with ribbon surfaces, hence with 4-dimensional topology. 

\subsection{Statement}

We know that the Jones polynomial (and the Kauffman bracket) for links in $S^3$ is a Laurent polynomial, hence it can not have a pole in points different from $0$ and $\infty$. Moreover $A^4-1$ is the minimal polynomial of a $4^{\rm th}$ root of unity $\sqrt i$, hence for a link $L\subset S^3$ the multiplicity of $\langle L \rangle$ in $\sqrt i$ as a zero is at least $n$ if and only if $\langle L \rangle / (-A^2-A^{-2})^n$ is still a Laurent polynomial.

\begin{theo}[Eisermann]\label{theorem:Eisermann}
If $S\subset D^4$ is a ribbon surface bounded by a link $L=\partial S \subset S^3$, then the multiplicity of the Kauffman bracket (or the Jones polynomial) of $L$ (with variable $A$ and normalization $\left\langle \pic{0.8}{0.3}{banp.eps} \right\rangle = -A^2-A^{-2}$) in a $4^{\rm th}$ root of unity $A=\sqrt i$ as a zero is at least the Euler characteristic $\chi(S)$ of $S$:
$$
\frac{ \langle L \rangle }{ (-A^2-A^{-2})^{\chi(S)} } \in \Z[A,A^{-1}] .
$$
\end{theo}

Theorem~\ref{theorem:Eisermann} has been proved by Eisermann \cite{Eisermann} using an easy and combinatorial method. In \cite{Carrega-Martelli} the theorem was extended in several directions: it concerns colored graphs in the connected sum $\#_g(S^1\times S^2)$ of $g$ copies of $S^1\times S^2$. In that ambient the Kauffman bracket may not be a Laurent polynomial and we may have poles in $4^{\rm th}$ roots of unity, hence the notion of ``multiplicity'' has been extended to the one of \emph{order} (Definition~\ref{defn:order}). In that more general case the proof requires more complicated tools, in particular the shadow formula has been used (Section~\ref{sec:sh_for_br}). We are going to discuss this extension in Chapter~\ref{chapter:Eisermann}.

\subsection{Alexander polynomial and 2-component slice links}\label{subsec:Alex_slice}

\begin{quest}
Is there a $k$-component slice link $L\subset S^3$ such that the multiplicity of the Kauffman bracket of $L$ in a $4^{\rm th}$ root of unity (using the variable $A$) is lower than $k$? This wolud imply that $L$ is slice but not ribbon.
\end{quest}

Eisermann showed \cite{Eisermann} that the multiplicity of the Kauffman bracket of a $k$-component link in $S^3$ in a $4^{\rm th}$ root of unity $\sqrt i$ is at most $k$ (see Proposition~\ref{prop:upper_bound}). Since we use the normalization $\left\langle \pic{0.8}{0.3}{banp.eps} \right\rangle = -A^2-A^{-2}$, for every link $L\subset S^3$ we have $\langle L \rangle|_{A=\sqrt i} = 0$. Therefore the multiplicity of the Kauffman bracket of a knot in $S^3$ in $\sqrt i$ is always $1$. Thus this multiplicity can not be used to get information about the ribbon surfaces bounded by a knot in $S^3$. We end this section by proving a remark of
Eisermann \cite{Eisermann}, that says that the Kauffman bracket of a $k$-component slice link $L\subset S^3$ with $k\geq 2$, is divided by $(-A^2-A^{-2})^2$ (still getting a Laurent polynomial) ($\ord_{q=A^2=i} \langle L \rangle \geq 2$), hence if $k=2$ the multiplicity in a $4^{\rm th}$ root of unity is $2$ ($\ord_{q=A^2=i} \langle L \rangle =2$). Therefore it can not be used to say that one such link is not ribbon.

Let $L\subset S^3$ be an oriented link. We denote by $\Delta(L)\in \Z[q,q^{-1}]$ the \emph{Alexander-Conway polynomial} of $L$ (see \cite{Lickorish}[Chapter 6]). We use the variable $q=-t^\frac 1 2$ instead of the classical $t$. This is completely determined by the following skein relations on oriented diagrams:
$$
\begin{array}{rcl}
\Delta\left( \pic{1.2}{0.3}{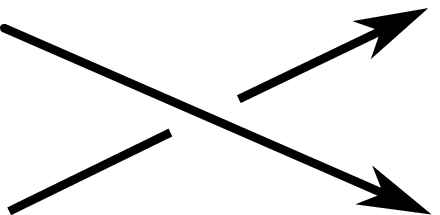} \right) - \Delta\left( \pic{1.2}{0.3}{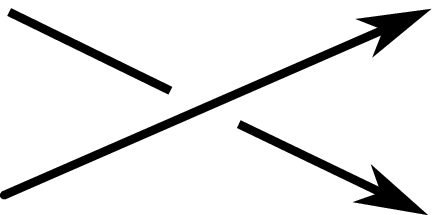} \right)  & = & (q-q^{-1})  \Delta\left( \pic{1.2}{0.3}{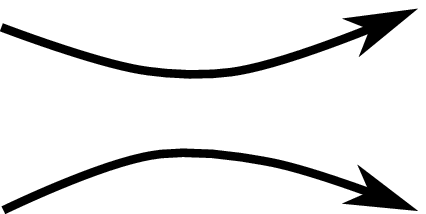} \right) \\
 \Delta\left(\pic{0.8}{0.3}{banp.eps} \right)  & = & 1
\end{array} .
$$
The skein relations imply that $\Delta\left(L \sqcup \pic{0.8}{0.3}{banp.eps} \right) = 0$.

As an invariant of oriented links in $S^3$ the Jones polynomial with variable $q=A^2$, $V(L)\in \Z[q,q^{-1}]$, is determined by the following skein relations:
$$
\begin{array}{rcl}
q^2 V \left( \pic{1.2}{0.3}{incrociopos.eps} \right) - q^{-2} V\left( \pic{1.2}{0.3}{incrocioneg.eps} \right)  & = & (q^{-1} -q)  V\left( \pic{1.2}{0.3}{Acanaleor.eps} \right) \\
 V\left(\pic{0.8}{0.3}{banp.eps} \right)  & = & (-q-q^{-1})
\end{array} .
$$

It follows that for every oriented link $L$
$$
\left.\frac{V(L)}{q+q^{-1}}\right|_{q=i} = \Delta(L)|_{q=i} .
$$

\begin{theo}[Kawauchi]
Let $L\subset S^3$ be a slice link with at least two components. Then
$$
\Delta(L) = 0.
$$
\begin{proof}
See \cite{Kawauchi}.
\end{proof}
\end{theo}

Hence for every oriented slice link $L\subset S^3$ with more than one component $\Delta(L)|_{q=i} =0$, that implies the following:

\begin{prop}\label{prop:Alex} 
Let $L\subset S^3$ be a $k$-component slice link with $k\geq 2$. Then
$$
\frac{\langle L \rangle}{ (-A^2-A^{-2})^2} \in \Z[A,A^{-1}] 
$$
($\ord_{q=A^2=i} \langle L \rangle \geq 2$). Therefore if $k=2$ the multiplicity of $\langle L \rangle$ in $q=A^2=i$ is $2$
$$
\ord_{q=A^2=i} \langle L \rangle = 2.
$$
\end{prop}

\section{The classification of rational tangles}\label{sec:rational_tangles}

We introduced the notion of ``$n$-tangle'' and ``diagram of a tangle'' (Definition~\ref{defn:tangle}).

\begin{defn}\label{defn:rational_tangle}
A tangle is said to be \emph{rational} if there is an isotopy of the ambient 3-cube $D^3$, maybe not fixing the boundary, that sends it to the trivial one: $n$ straight (unkontted and unlinked) strands.
\end{defn}

\begin{theo}[Conway]
Rational 2-tangles are in bijection with the extended rational numbers 
$$
T \mapsto C(T) \in \mathbb{Q}\cup \{\infty \} .
$$
\begin{proof}
See \cite{Conway} or \cite{Goldman-Kauffman}.
\end{proof}
\end{theo}

Given a diagram $D$ of a 2-tangle we can reduce it to an unique linear combination of the form
$$
\pic{1}{0.3}{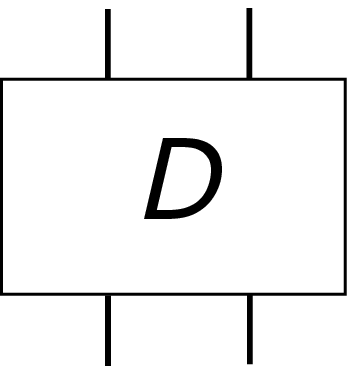} = a(D) \ \pic{0.7}{0.3}{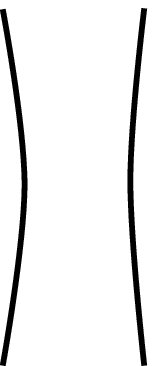} + b(D) \ \pic{0.7}{0.3}{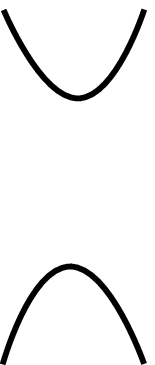} ,
$$
using the \emph{skein relations}:
$$
\begin{array}{rcl}
 \pic{1.2}{0.3}{incrociop.eps}  & = & A \pic{1.2}{0.3}{Acanalep.eps}  + A^{-1}  \pic{1.2}{0.3}{Bcanalep.eps}  \\
 D \sqcup \pic{0.8}{0.3}{banp.eps}  & = & (-A^2 - A^{-2})  D 
\end{array}
$$

\begin{theo}[Goldman-Kauffman]\label{theorem:GoKauf}
Let $T$ be a rational 2-tangle. Then for any diagram $D$ of $T$
$$
C(T) = -i \left. \frac{b(D)}{a(D)} \right|_{A=\sqrt{i}} ,
$$
where $i\in \mathbb{C}$ is the imaginary unit and $\sqrt{i}$ is a square root of $i$.
\begin{proof}
See \cite{Goldman-Kauffman}.
\end{proof}
\end{theo}

Using the skein relations every $n$-tangle diagram $D$ can be written in a unique way as a linear combination of the diagrams without crossings and without closed components, let $(a_j(D))_j$ be the coefficients of that linear combination. 

In \cite{Kwon2} Kwon ``categorified'' Theorem~\ref{theorem:GoKauf} showing that two diagrams, $D_1$ and $D_2$, represent the same rational 2-tangle if and only if there is an integer $n$ such that 
$$
a(D_1) = (-A^3)^n a(D_2) \ \ \text{and}\ \ b(D_1)=(-A^3)^n b(D_2) .
$$
The result seems to be useless since the invariant was already injective, but this formulation can be easily generalized to rational $n$-tangles using the $(a_j(D))_j$'s and it comes out to be true for alternating rational 3-tangles:
\begin{theo}[Kwon]
Let $D_1$ and $D_2$ be two diagrams of alternating rational 3-tangles. Then they represent the same tangle if and only if there is an integer $n$ such that for all $j$
$$
a_j(D_1) = (-A^3)^n a_j(D_2)  .
$$
\begin{proof}
See \cite{Kwon2}.
\end{proof}
\end{theo}
In \cite{Kwon1} Kwon described an algorithm to distinguish any two rational 3-tangles. It uses a modified version of Dehn's method for classifying simple closed curves on surfaces.

\section{An infinite family of non slice Montesinos links}\label{sec:Montesinos}

In this section we show a criterion to check if a Montesinos link is not slice that is based on Eisermann's theorem (Theorem~\ref{theorem:Eisermann}) and we show an infinite family of non slice 2-components rational links that come out from this criterion and have null linking number.

We introduced the notion of ``framed $n$-tangle'', ``closure'' of a framed tangle (Definition~\ref{defn:framed_tangle}) and ``product'' of tangles (Definition~\ref{defn:mult_tangle}).

\begin{defn}
The closure of a rational tangle is called \emph{rational link}.
\end{defn}

A well known class of links is the one of 2-\emph{bridge links} (see \cite{Bu-Zie}).

\begin{prop}
The closure of a rational 2-tangle is a 2-bridge link.
\begin{proof}
See \cite{Bu-Zie}.
\end{proof}
\end{prop}

In Section~\ref{sec:Eisermann} we saw some criteria to show that a link is not ribbon or not slice. The following is a more famous one:
\begin{theo}
The linking number (Definition~\ref{defn:linking_matrix}) of any two components of a slice link $L\subset S^3$ is zero.
\begin{proof}
See \cite{Bu-Zie}.
\end{proof}
\end{theo}

\begin{rem}\label{rem:closure_tangle}
Let $T$ be a framed 2-tangle, $D$ a diagram of $T$, and $L$ the closure of $T$. The Kauffman bracket of $L$ is equal to 
$$
\langle L \rangle = a(D) (-A^2-A^{-2})^2 + b(D) (-A^2-A^{-2}) .
$$
Hence if $b(D)|_{A=\sqrt{i}} \in \mathbb{C}\setminus \{0\}$, the multiplicity of $\langle L \rangle$ in $A=\sqrt{i}$ as a zero is $1$ ($\ord_{q=A^2=i} \langle L \rangle = 1$). Therefore by Proposition~\ref{prop:Alex} if $L$ has more than one component, it is not slice (and not ribbon).
\end{rem}

\begin{ex}
A link $L_n$ depending on an integer $n\in\Z$ is shown in Fig.~\ref{figure:ex_no_slice}. If $n\in 2\Z$ the link is a knot, otherwise it has two components. For $n\in 2\Z +1$, $L_n$ has linking number $0$. The link is the closure of a rational 2-tangle with number $C(L_n)$ different from $0$ or $\infty$ (see Section~\ref{sec:rational_tangles} and Theorem~\ref{theorem:GoKauf}). Hence $b(L_n)\neq 0$. Therefore by Remark~\ref{rem:closure_tangle} if $n\in 2\Z+1$, $L_n$ is not slice.
\end{ex}

\begin{figure}[htbp]
$$
\includegraphics[width=8 cm]{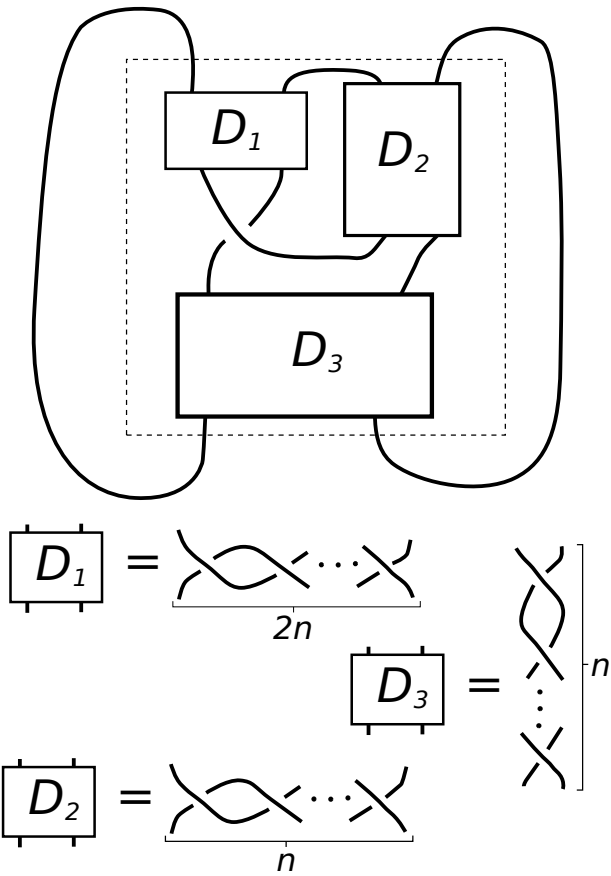}
$$
\caption{A link $L_n$ depending on $n\in\Z$. If $n\in 2\Z+1$ it is a non slice link with two components and linking number $0$.}
\label{figure:ex_no_slice}
\end{figure}

\emph{Montesinos links} form another famous class of links in $S^3$ (see \cite{Bu-Zie}). They comprehend 2-bridge links and \emph{pretzel links}. We can get all of them with diagrams described by integers $e$, $m$, and rational numbers $\frac{\alpha_1}{\beta_1}, \ldots, \frac{\alpha_m}{\beta_m} \in \mathbb{Q}\cup \{\infty\}$ ($\frac \alpha 0 = \infty$) as in Fig.~\ref{figure:Montesinos}: the box decorated with $\frac{\alpha_j}{\beta_j} $ is the rational 2-tangle $T$ such that $C(T)=\frac{\alpha_j}{\beta_j}$. Clearly they comprehend the closure of a rational tangles ($e=0$, $m=1$).

In \cite{Williams} is shown that no member of a five parameter family of Montsinos knots is slice. In \cite{Lecuona1} is shown that the slice-ribbon conjecture is true for a large family of Montesinos knots. In \cite{Lecuona2} a necessary, and in some cases sufficient, condition for sliceness inside the family of pretzel knots $P(p_1, \ldots,p_n)$ with one $p_i$ even is given.

\begin{figure}[htbp]
$$
\includegraphics[width=8 cm]{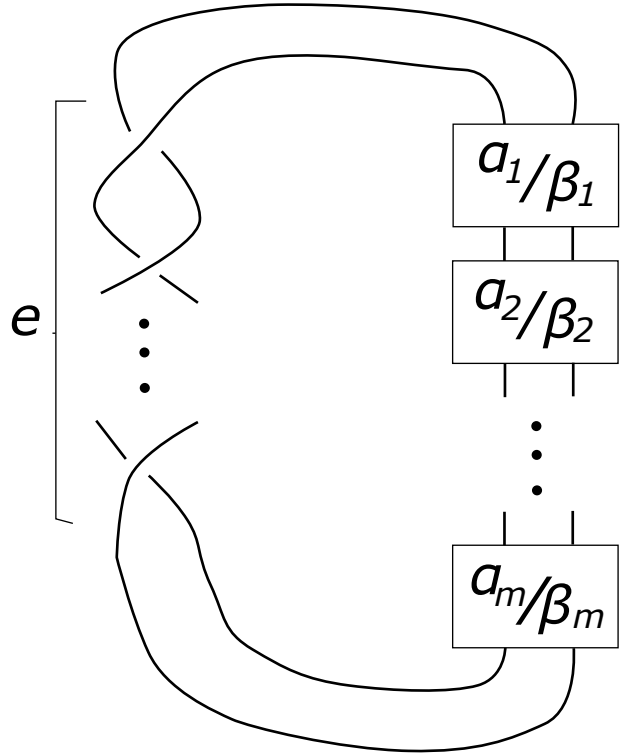}
$$
\caption{The Montesinos link $K(e, \frac{\alpha_1}{\beta_1}, \ldots, \frac{\alpha_m}{\beta_m})$.}
\label{figure:Montesinos}
\end{figure}

\begin{lem}\label{lem:sum_tangle}
Let $T$ be a 2-tangle that is the product of $T_1$ and $T_2$. Then $C(T)$ is the sum of $C(T_1)$ and $C(T_2)$.
\begin{proof}
Let $D_1$ and $D_2$ be diagrams of $T_1$ and $T_2$. A diagram $D$ of $T$ is obtained putting $D_2$ over $D_1$.
$$
\pic{1}{0.3}{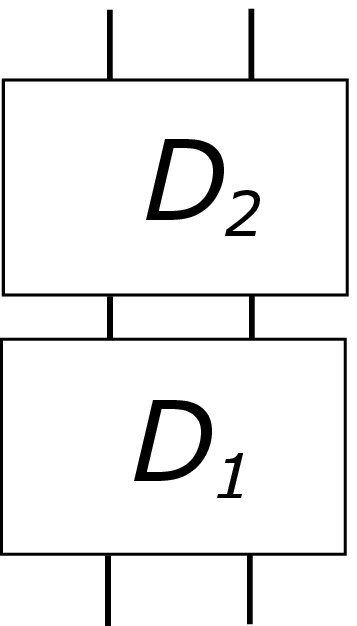} = a(D) \ \pic{0.7}{0.3}{D_tangle_0.eps} + b(D) \ \pic{0.7}{0.3}{D_tangle_infty.eps} ,
$$
\beq
\pic{1}{0.3}{D_diagr.eps}  & = & a(D_1) \ \pic{0.7}{0.3}{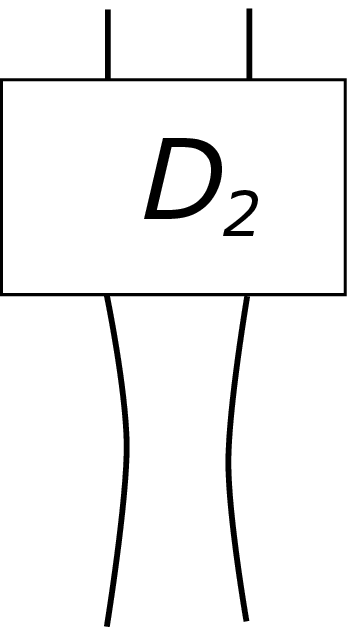} + b(D_1) \ \pic{0.7}{0.3}{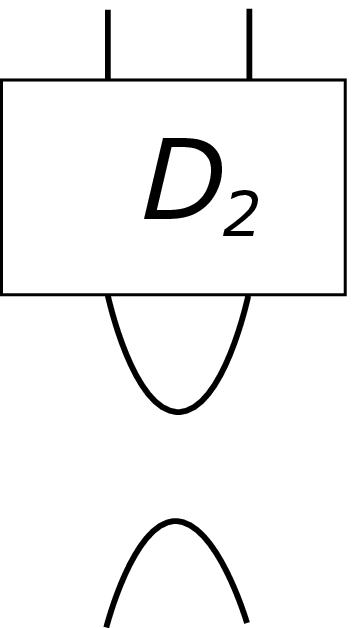} \\
& = & a(D_1)a(D_2) \ \pic{0.7}{0.3}{D_tangle_0.eps} + (b(D_1)a(D_2) +a(D_1)b(D_2) + (-A^2-A^{-2})b(D_1)b(D_2)) \ \pic{0.7}{0.3}{D_tangle_infty.eps} .
\eeq
Hence 
\beq
C(T) & = & -i \left. \frac{b(D)}{a(D)} \right|_{A=\sqrt{i}} \\
& = & -i \left. \frac{a(D_1)b(D_2) + a(D_2)b(D_1)}{a(D_1)a(D_2)} \right|_{A=\sqrt{i}} \\
& = & C(T_1) + C(T_2) .
\eeq
\end{proof} 
\end{lem}

\begin{prop}
Let $L\subset S^3$ be a Montesinos link with at least two components. If its diagram $K(e, \frac{\alpha_1}{\beta_1} , \ldots , \frac{\alpha_m}{\beta_m})$ satisfies 
$$
 \sum_{j=1}^m \frac{\alpha_j}{\beta_j} -e \neq 0 ,
$$
then $L$ is not slice.
\begin{proof}
The link $L$ is the closure of a 2-tangle $T$. The tangle $T$ is the product of $e$ rational 2-tangles $T'$ with number $C(T')$ equal to $-1$, and the rational 2-tangles with number $\frac{\alpha_j}{\beta_j}$ for $1\leq j\leq m$. Hence by Lemma~\ref{lem:sum_tangle}
$$
C(T)= \sum_{j=1}^m \frac{\alpha_j}{\beta_j} -e .
$$
By Remark~\ref{rem:closure_tangle}, if $b(D)|_{A=\sqrt{i}} \neq 0$, $L$ is not slice. The number $b(D)|_{A=\sqrt{i}} $ is non null if and only if $C(T)\neq 0$.
\end{proof}
\end{prop}

\section{An application of shadows to the slice-ribbon conjecture}\label{sec:sh_rib_handle}

Although the aim of the chapter is to present topological applications of quantum invariants, in this section we talk about something a little different. We present an application of Turaev's shadows that goes in favor of the slice-ribbon conjecture.

\subsection{Shadow links and handle links}

\begin{defn}
We say that a link $L$ in a closed 3-manifold $M$ is \emph{slice with respect to} $W$ if $W$ is a compact 4-manifold bounded by $M$ and $L$ bounds a set of properly embedded disjoint (smooth) disks in $W$.
\end{defn}

Let $X$ be a shadow of a closed orientable 3-manifold $M$. The manifold $M$ is the boundary of the 4-dimensional thickening $W$ of $X$. Every region $R$ is thickened to a $D^2$-bundle over $R$, $F(R)$. The \emph{core} of this thickening is $R\subset F(R)$, while the \emph{co-core} is a fiber-disk of $F(R)$ over a point $p\in R$. The boundary of a co-core of a region is a knot in $M$.

\begin{defn}
We say that a link $L$ in a closed orientable 3-manifold $M$ is a \emph{shadow link} with respect to the 4-manifold $W$ if $\partial W = M$ and there is a shadow $X$ of $W$ without closed regions such that each component $K$ of $L$ is the boundary of the co-core of a region $R_K$ of $X$. In order not to make confusion with the notion of ``shadow of a link'', we say that $X$ is a \emph{support} of $L$.
\end{defn}

\begin{rem}
By definition each shadow link is slice with respect to the 4-dimensional thickening of the shadow. In fact each component bounds a properly embedded disk in $W$ that is the co-core of a region, and all these disks are disjoint because either they lie in different regions or are parallel.
\end{rem}

\begin{rem}
By Proposition~\ref{prop:intersect_shadows} we can easily find a shadow of a shadow link $L$ (not a support). It suffices to take a support $X$ of $L$, for each component $K$ of $L$ attach an annulus to $X$ along a map that identifies a boundary component of the annulus to the boundary of an embedded closed disk $D_K$ in the corresponding region $R_K$, give to $D_K$ gleam $\pm 1$, to $R_K \setminus D_K$ gleam $\gl(R_K) \mp 1$, and leave the other gleams unchanged. In fact every component $K$ of $L$ bounds a properly embedded disk (the co-core of $R_K$) that intersects the region (the core of $R_K$) in just one point.
\end{rem}

\begin{defn}
A link $L$ in the connected sum $\#_g(S^1\times S^2)$ of $g\geq 0$ copies of $S^1\times S^2$ is said to be \emph{symmetric} if given a description of $\#_g(S^1\times S^2)$ as the double of a 3-dimensional orientable handlebody $H_g$ of genus $g$, we have that the following hold up to isotopies:
\begin{itemize}
\item{each component of $L$ is divided in two arcs by the boundary of $H_g$;}
\item{one of such arcs lies in $H_g$ and the other one lies in the other copy of $H_g$;}
\item{the link is symmetric with respect to this decomposition, namely doing the double of $H_g$ equipped with one of these arcs for each pair, we get the whole link $L\subset \#_g(S^1\times S^2)$.}
\end{itemize}
\end{defn}

We introduced the definition of ribbon link (Definition~\ref{defn:ribbon_link}) and Remark~\ref{rem:ribbon_g}.

\begin{lem}\label{lem:sym_and_ribbon}
Every symmetric link in $\#_g(S^1\times S^2)$ is ribbon.
\begin{proof}
Let $H_1$ and $H_2$ be two handlebodies of genus $g$ such that $H_1\cup H_2 = \#_g(S^1\times S^2)$. Put the link $L$ in a symmetric position with respect to $H_1$ and $H_2$. For $j=1,2$ there is a projection $\pi_j: L\cap H_j \rightarrow S$ such that the image $\pi_j(L)$ is a graph $D\subset S$ not depending on $j$, where $S:=\partial H_1 = \partial H_2$. The vertices of $D$ are either 4-valent or 1-valent (leaves). The 4-valent vertices of $D$ may be equipped with the further intormation of the over/underpasses. The 1-valent vertices of $D$ are the intersection points of $L$ with the surface $S$. The projection $\pi_j$ is injective everywhere except in the inverse image of the 4-valent vertices of $D$ where it is $2$ to $1$, moreover the inverse image of the 1-valent vertices are the fixed points of $\pi_j$. 

Let $N_j(S)\cong S  \times [0,1)$ be a collar of $S$ in $H_j$ that contains $L\cap H_j$. The projection $\pi_j$ extends to a projection $\pi_{j,N} : N_j(S) \rightarrow S$ such that the inverse image of the points of $D$ that satisfy: $\pi_{j,N}^{-1}(p)= p$ if $p$ is a 1-valent vertex, while $\pi_{j,N}^{-1}(p)$ is an arc if $p$ is not a 1-valent vertex. The union $\pi_{1,N}^{-1}(D) \cup \pi_{2,N}^{-1}(D) \subset \#_g(S^1\times S^2)$ is a ribbon surface consisting of disks and bounded by $L$. The ribbon singularities are contained in the union of the inverse images of the 4-valent vertices of $D$.
\end{proof}
\end{lem}

\begin{lem}\label{lem:sh_and_sym}
Every shadow link in $\#_g(S^1\times S^2)$ with support a shadow of the 4-dimensional handlebody with only $0$ gleams is symmetric, hence by Lemma~\ref{lem:sym_and_ribbon} it is ribbon.
\begin{proof}
Let $X$ be such a support. By Lemma~\ref{lem:spines} the polyhedron $X$ is a spine of the 3-dimensional handlebody of genus $g$, $H_g$, and the 4-dimensional thickening $W$ is $H_g\times[-1,1]$. The 4-dimensional thickening of a region is the product with an interval of its 3-dimensional one in $H_g$. Hence the co-core of a region $R$ is the product of a properly embedded arc $\alpha_R$ in $H_g$ with $[-1,1]$. The boundary of the co-core of a region $R$ is exactly the double of $\alpha_R$ in the double of $H_g$, $\#_g(S^1\times S^2)$. Therefore the link is symmetric and by Lemma~\ref{lem:sym_and_ribbon} the link is ribbon.
\end{proof}
\end{lem}

In Proposition~\ref{prop:eq_sh_and_ribbon} we show that if the shadow link given by all the regions of a shadow $X$ is ribbon, then also the one given by all the regions of any shadow $X'$ equivalent to $X$ is ribbon. To prove this, we apply the following lemma to embedded pants coming from the thickening of points of edges of the shadow.

\begin{lem}\label{lem:ribbon_pants}
Let $\iota_P: P\rightarrow M$ be an embedding of a pair of pants in a 3-manifold $M$ and let $K_1$, $K_2$ and $K_3$ be the boundary components of $\iota_P( P)$. If the link $K_1\cup K_2$ bounds a ribbon surface $\iota_{1,2}: D_1 \cup D_2 \rightarrow M$ that is the union of two disks and consists of two connected components of a bigger ribbon surface $\iota_S: S \rightarrow M$. Then the link $K_1\cup K_2 \cup K_3$ bounds a ribbon surface $\iota_{1,2,3}:D_1\cup D_2\cup D_3 \rightarrow M$ such that
\begin{itemize}
\item{$D_1$, $D_2$ and $D_3$ are disks;}
\item{the restriction of $\iota_{1,2,3}$ to the disks $D_1 \cup D_2$ bounded by $K_1\cup K_2$ is $\iota_{1,2}$, $\iota_{1,2,3}(D_1\cup D_2) = \iota_{1,2}(D_1\cup D_2)$;}
\item{$\iota_{1,2,3}$ extends to a ribbon surface $\iota'_S: S\cup D_3 \rightarrow D$;}
\item{the restriction to $S$ does not intersect the restriction to the disk $D_3$ bounded by $K_3$, $\iota'_S(S) \cap \iota_{1,2,3}(D_3) = \iota_{1,2,3}(D_1\cup D_2) \cap \iota_{1,2,3}(D_3) = \varnothing$.}
\end{itemize} 
\begin{proof}
Take a \emph{parallel copy} $\iota'_{1,2}:D_1\cup D_2 \rightarrow M$ of $\iota_{1,2}: D_1\cup D_2\rightarrow M$, and let $K'_1$ and $K'_2$ be the boundary components of $\iota'_{1,2}(D_1 \cup D_2)$. By ``parallel copy'' we mean a disjoint copy that locally follows the part of the ribbon disks as in Fig.~\ref{figure:parallel_copies}. Hence $K_j$ and $K'_j$ are isotopic. Let the disk $D_3$ be the pair of pants $P$ with $D_1$ and $D_2$ glued to two boundary components, and let $B$ be a band in $P$ connecting those boundary components. Since $B$ is the thickening of an arc and $P\setminus (D_1\cup B \cup D_2)$ is just an annulus, we can find an embedding $\iota'_P : P \rightarrow M$ such that $\iota'_P$ is isotopic to $\iota_P$, $\iota'_P(\partial P)= K'_1 \cup K'_2 \cup K_3$, and $\iota'_p(P)\cap \iota'_{1,2}(D_1\cup D_2) = \iota'_{1,2}(\partial D_1\cup \partial D_2) = K'_1\cup K'_2$. Thus we define the ribbon disk $\iota_3:D_3\rightarrow M$ as the union of $\iota'_P$ and $\iota'_{1,2}$. Then $\iota_{1,2,3}$ is the union of $\iota_{1,2}$ and $\iota_3$ and $\iota'_S$ is the union of $\iota_S$ and $\iota_3$. 
\end{proof}
\end{lem}

\begin{figure}[htbp]
\begin{center}
\includegraphics[scale=0.6]{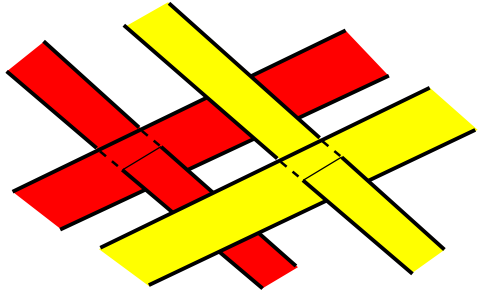}
\end{center}
\caption{Two parallel copies of a ribbon disk.}
\label{figure:parallel_copies}
\end{figure}

We introduced the notions of ``shadow moves'' and ``equivalent shadows'' (Definition~\ref{defn:sh_moves}).
\begin{prop}\label{prop:eq_sh_and_ribbon}
Let $X$ and $X'$ be two equivalent shadows. If the shadow link given by all the regions of $X$ is ribbon then it is so also the one given by all the regions of $X'$.
\begin{proof}
A \rm{Y}-\emph{graph} is a graph like a \rm{Y}: three edges all adjacent in a vertex and with a free end. A \rm{Y}-graph $Y$ has a natural 2-dimensional thickening $Y^{(2)}$ that is a 2-disk in which the graph $Y$ is properly embedded and is equipped with the natural retraction $\pi : D^2 \rightarrow Y$ (see Fig.~\ref{figure:Y(2)}). The 4-dimensional thickening of an edge $e$ of a shadow is the orientable $I$-bundle over its 3-dimensional thickening that is a $Y^{(2)}$-bundle over an open interval or a circle (see Remark~\ref{rem:reconstruction} and Theorem~\ref{theorem:reconstruction}). Hence locally the 4-dimensional thickening of $e$ is $Y^{(2)}\times (-1,1)\times [-1,1]$. This contributes to the boundary of the 4-manifold with the two copies of $Y^{(2)} \times (-1,1)$ glued together with the identity of $\partial Y^{(2)} \times (-1,1)$, where $\partial Y^{(2)}$ is the boundary of the 2-disk minus the open regular neighborhood of the intersection of the \rm{Y}-graph with the boundary of the disk (the black bold lines of Fig.~\ref{figure:Y(2)} form $\partial Y^{(2)}$). This is $S^2 \times (-1,1)$ minus a tubular neighborhood of three strands of the form $\{x\}\times (-1,1)$, namely the product of  a pair of pants with $(-1,1)$. For each (internal) point $p$ of an edge of the \rm{Y}-graph, the double of $\pi^{(-1)}(p) \subset Y^{(2)}$ (the red segment in Fig.~\ref{figure:Y(2)}) is the boundary of the co-core of the region where $p$ lies. In $0\in (-1,1)$ there are two copies of $Y^{(2)} \times \{0\}$ glued together along $\partial Y^{(2)} \times \{0\}$. This gluing form a pair of pants $P_e \subset \partial W$ in the boundary of the 4-manifold whose boundary is the union of the boundary of the co-cores of the regions adjacent to the edge $e$.

The shadows $X$ and $X'$ are related by a sequence of shadow moves. Every shadow move modifies the shadow just locally. They slightly modify some regions, introduce a new one or remove an old one. 

The boundary of the co-cores of the regions modified by a shadow move are the same curves in the boundary of the 4-manifold given by them before the application of the move. In fact we can take the point $p\in R$ that defines one such co-core outside the small part that is modified by the move. In particular if these were boundary components of a ribbon surface $D^2\cup \ldots \cup D^2 \rightarrow M$, they remain to be boundary components of one such ribbon surface. 

Now we consider a move that adds a new region $R$. Take an edge $e$ adjacent to $R$. The edge $e$ touches $R$ and two old regions, $R_1$ and $R_2$ (maybe $R_1=R_2$). Since by hypothesis the link given by all the regions of the starting shadow is ribbon, the pant $P_e$ has two ribbon boundary components that are the boundary of the co-cores of $R_1$ and $R_2$. Therefore by Lemma~\ref{lem:ribbon_pants} $P_e$ gives a ribbon disk bounded by the boundary of the co-core of $R$ that does not intersect the the previous ribbon surface. Therefore the link given by all the regions of the final shadow is ribbon. 

If the move removes an old region the link given by all the regions of the the second shadow is equal to the first link minus the component given by the removed region. Therefore it would remain ribbon if it was so.
\end{proof}
\end{prop}

\begin{figure}[htbp]
\begin{center}
\includegraphics[scale=0.7]{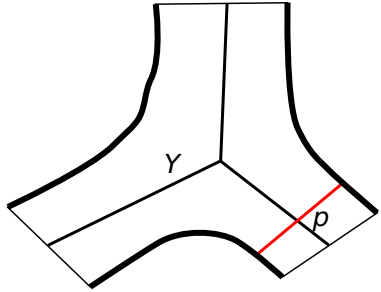}
\end{center}
\caption{The 2-dimensional thickening $Y^{(2)}$ of a \rm{Y}-graph $Y$. The black bold lines define the boundary $\partial Y^{(2)}$. The red arc is the inverse image of a (internal) point $p$ of an edge of the graph under the projection $\pi: D^2 \rightarrow Y$.}
\label{figure:Y(2)}
\end{figure}

\begin{rem}\label{rem:sh_eq1}
If the answer to Question~\ref{quest:eq_shadows} were true we would conclude that every shadow link with respect to a shadow of a 4-dimensional handlebody is ribbon just using Proposition~\ref{prop:eq_sh_and_ribbon}, Lemma~\ref{lem:sh_and_sym} and Lemma~\ref{lem:sym_and_ribbon}. Unfortunately we do not know if it is true or false.
\end{rem}

Let $D$ be a handle-decomposition of a compact 4-manifold $W$. Every handle $h^{(k)}$ is diffeomorphic to $D^k\times D^{4-k}$, where $k$ is the \emph{index} of $h^{(k)}$. The \emph{core} of $h^{(k)}$ is what corresponds to $D^k\times\{0\}$, while the \emph{co-core} is $\{0\}\times D^{4-k}$. If $D$ has no handle of index 3 and 4, up to isotopies each 2-handle $h^{(2)}$ intersects the boundary of $W$ with $D^2\times \partial D^2 = D^2\times S^1$. Hence in this case the boundary of a co-core of a 2-handle is a slice knot in $\partial W$ with respect to $W$.
\begin{defn}
We say that a link in a closed orientable 3-manifold $M$ is a \emph{handle link} with respect to $W$ if $W$ is a compact orientable 4-manifold bounded by $M$ and there is a handle-decomposition $D$ of $W$, composed only by 0-, 1- and 2-handles and such that each component $K$ of $L$ is the boundary of the co-core of a 2-handle $h^{(2)}_K$ of $D$. We say that $D$ is a \emph{support} of $L$.
\end{defn}

\begin{rem}
Almost all the known possible counter-examples of the slice-ribbon conjecture are known to be handle link.
\end{rem}

The following proposition shows that the notion of ``shadow link'' is equivalent to the one of ``handle link''.

\begin{prop}
$\ $
\begin{enumerate}
\item{Every shadow link is a handle link with respect to the same 4-manifold.}
\item{Every handle link is a shadow link with respect to the same 4-manifold.}
\end{enumerate}
\begin{proof}
$1.$ Each triangulation of a shadow gives rise to a handle-decomposition of the 4-manifold where each vertex of the triangulation is the core of a 0-handle, each edge is the core of a 1-handle and each triangle is the core of a 2-handle. The co-core of a region is isotopic to the co-core of each triangle that subdivides the region.

$2.$ Let $D$ be a support handle-decomposition of the handle link. In Theorem~\ref{theorem:admit_sh} we constructed a shadow of the same 4-manifold starting from $D$. This shadow is constructed by taking a shadow of the union of the 0- and 1-handles and then attaching disk regions to it. Each region of the final step corresponds to a 2-handle of $D$ and their co-cores coincide. Hence that shadow is a support for the link as a shadow link.
\end{proof}
\end{prop}

\subsection{Slice $\Rightarrow$ shadow ?}

\begin{quest}\label{quest:slice-shadow-handle}
Is every slice link a handle link (hence a shadow link) with respect to the same 4-manifold?
\end{quest}

\begin{rem}
If both the answers to Question~\ref{quest:eq_shadows} and Question~\ref{quest:slice-shadow-handle} were true we could conclude that every slice link is ribbon just using Remark~\ref{rem:sh_eq1}. Unfortunately we do not know if they are true or false. 
\end{rem}

By ``the contribute to the boundary'' of a 2-handle $h^{(2)} = D^2\times D^2$ attached along a map $\varphi: S^1\times D^2 \rightarrow \partial W$, we mean the part $D^2\times S^1 = \partial (h^{(2)} \cup_\varphi W ) \setminus \partial W$.

The following theorem partially answers positively to Question~\ref{quest:slice-shadow-handle}.

\begin{theo}\label{theorem:1}
Every slice link $L$ of a closed 3-manifold $M$ with respect to a compact orientable 4-manifold $W$ ($\partial W = M$) can be obtained in the following way:
\begin{enumerate}
\item{take a 4-manifold $W'\subset W$ with a handle-decomposition with just 0-, 1- and 2-handles;}
\item{construct a link $L'\subset \partial W'$ taking the boundary of the co-core of some 2-handles of $W'$;}
\item{add some 3-handles to $W'$ whose attaching spheres do not intersect the contribution to the boundary of the 2-handles of $W'$ giving the components of $L'$.}
\end{enumerate}
\begin{proof}
Thick $M$ to $M \times [-1,1]$. Attach a 4-dimensional 2-handle $h^{(2)}_i$ over each component of the slice link $L$ in $M\times \{1\}$. This forms a cobordism from $M$ to another compact orientable 3-manifold $M'$. Moreover this is a 4-sub-manifold of $W$. It is the regular neighborhood of the union of the boundary and the slice disks bounded by the components of the link. In fact the core of each $h^{(2)}_i$ is the slice disk in $W$ bounded by the attaching curve of $h^{(2)}_i$. Complete this cobordism to get $W$ (that is a cobordism from $M$ to the empty set $\varnothing$) by using only 1-, 2-, 3- and 4-handles. We can avoid using 0-handles because they would just add connected components of the 4-manifold or create complementary pairs of (0-1)-handles. Take the dual of this cobordism and handle-decomposition. Hence we get a cobordism starting from the empty set and arriving to $M$ that is $W$ with a handle-decomposition without 4-handles. The co-cores of the dual of the $h^{(2)}_i$'s are the cores of the $h^{(2)}_i$'s, hence their boundary are the components of the link.

On the boundary view point attaching a 4-dimensional 3-handle is selecting a 2-sided embedded 2-sphere $S$ (a sphere with a tubular neighborhood diffeomorphic to $S^2\times [-1,1]$), cutting the 3-manifold along $S$ and gluing a 3-disk along the two spheres created by the cut. Since a 3-handle intersects the boundary just in two disjoint 3-disks we can slide them with an isotopy so that they do not intersect the attaching tori of the dual of the $h^{(2)}_i$'s. In particular we can attach the dual of the $h^{(2)}_i$'s before attaching the 3-handles. 
\end{proof}
\end{theo}

\begin{prop}\label{theorem:slice-shadow_s_conn}
Let $L\subset M$ be a link. Then $L$ is a handle link (hence a shadow link) with respect to some simply connected 4-manifold $W$. Furthermore $W$ is homotopically equivalent to a bouquet of 2-spheres.
\begin{proof}
Attach to $M \times[-1,1]$ a 2-handle $h^{(2)}_K$ along each component $K$ of the link. We get a cobordism from $M$ to a closed 3-manifold $M'$ that is obtained by surgery on $L$. We follow Lemma~\ref{lem:shadow1} to get a cobordism $W'$ from the empty set to $M'$ with a handle-decomposition with one 0-handle and some 2-handles. Complete the cobordism $W':\varnothing \rightarrow M'$ adding the dual of the 2-handles $h^{(2)}_K$'s. The components of $L$ are the co-cores of these latest 2-handles.
\end{proof} 
\end{prop}

\chapter{The Tait conjecture in $\#_g(S^1\times S^2)$}\label{chapter:Tait}

We introduced the notions of ``diagram'' of a link in $\#_g(S^1\times S^2)$, ``H-decomposition'', ``e-shadow'', ``alternating diagram'', ``alternating link'', ``crossing number'', ``$\Z_2$-homologically trivial'' and ``Kauffman state'' (Section~\ref{sec:Kauf_g}). 

The Tait conjecture states that reduced alternating diagrams of links in $S^3$ have the minimal number of crossings (see Section~\ref{sec:Tait_conj0}). In this chapter we extend the result to alternating links in the connected sum $\#_g(S^1\times S^2)$ of $g\geq 0$ copies of $S^1\times S^2$. We follow \cite{Carrega_Tait1} and \cite{Carrega_Taitg}.

We find a dichotomy given by the $\Z_2$-homology class of links: the appropriate version of the statement is false for $\Z_2$-homologically non trivial links. On the other hand, in $S^1\times S^2$ and $\#_2(S^1\times S^2)$ the appropriate version of the statement is true for $\Z_2$-homologically trivial links, and the proof also uses the Jones polynomial. In the general case the method applied to $\Z_2$-homologically trivial links provides just a partial result and we are not able to say if the appropriate statement is true. As we saw in Proposition~\ref{prop:0Kauff_gr}, the Kauffman bracket is also very sensitive to the $\Z_2$-homology class.
 
The results for the case of links in $S^1\times S^2$ need just some basic notions of skein theory, while the general case of links in $\#_g(S^1\times S^2)$ needs more complicated tools. We prove the results for general $g$ using Turaev's shadows (see Chapter~\ref{chapter:shadows}, Section~\ref{sec:sh_for_br} and Remark~\ref{rem:sh_for_br}). 

The proof of the results is based on the study of the \emph{breadth} of the Kauffman bracket. The Kauffman bracket of links in $S^1\times S^2$ is an integral Laurent polynomial (Proposition~\ref{prop:Laurent_pol}) and we can use the classical notion of ``breadth'' (Definition~\ref{defn:breadth}). The Kauffman bracket of links in $\#_g(S^1\times S^2)$ for $g\geq 2$ may not be a Laurent polynomial, but just a rational function (see Example~\ref{ex:Kauff}), hence we need to extend the notion of ``breadth'' (Definition~\ref{defn:breadth_g}).

\section{The extended Tait conjecture}

In this section we introduce some further notions, give some results, state the main theorems (Theorem~\ref{theorem:Tait_conj_g} and Theorem~\ref{theorem:Tait_conj_Jones_g}), and we show that the natural extension of the conjecture is false if we remove the hypothesis of being $\Z_2$-homologically trivial or, in the case $g=1$, if we substitute ``\emph{simple}'' with ``\emph{quasi-simple}'' (Definition~\ref{defn:reduced_g} and Definition~\ref{defn:quasi-reduced}).

\subsection{Simple diagrams}

\begin{defn}\label{defn:reduced_g}
A link diagram $D \subset S_{(g)}$ is \emph{simple} if there is no disk $B$ embedded in $S_{(g)}$ whose boundary intersects $D$ exactly in one crossing (as the ones of Fig.~\ref{figure:reducedD_g}, the yellow square is the disk $B$), and $D$ has no crossings adjacent to two (possibly coinciding) external regions (as the ones of Fig.~\ref{figure:reducedD2_g}).
\end{defn}

\begin{figure}[htbp]
\begin{center}
\includegraphics[scale=0.55]{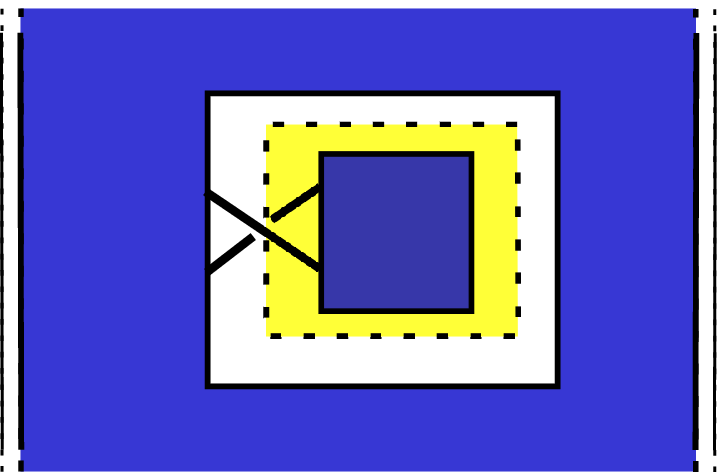} 
\hspace{0.5cm}
\includegraphics[scale=0.55]{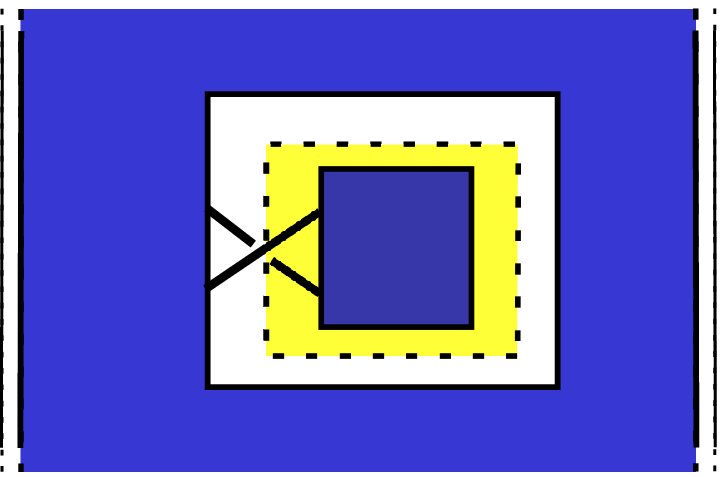}
\end{center}
\caption{Some non simple diagrams. We only show a
portion of $S_{(g)}$ that is a proper embedding of $[-1, 1]\times (-1,1)$. The blue parts cover the rest of the diagram and the yellow box is the embedded disk $B\subset S_{(g)}$ whose boundary intersects the diagram just in a crossing.}
\label{figure:reducedD_g}
\end{figure}

\begin{figure}[htbp]
\begin{center}
\includegraphics[scale=0.55]{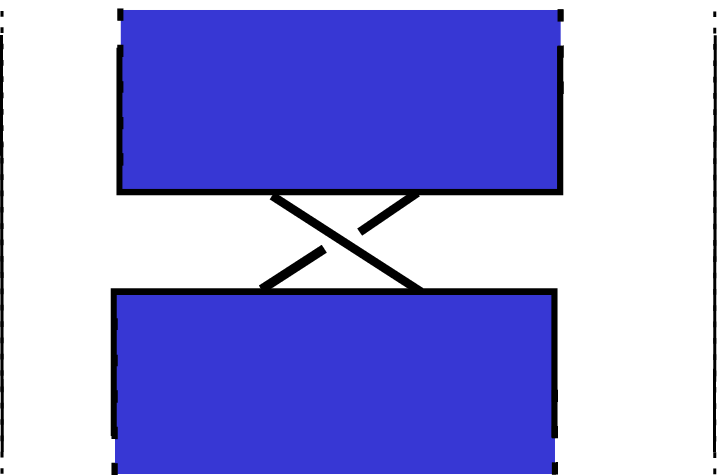} 
\hspace{0.5cm}
\includegraphics[scale=0.55]{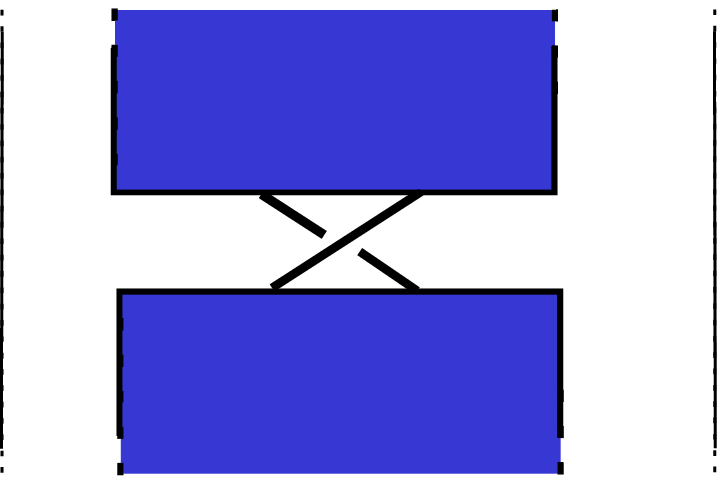}
\end{center}
\caption{More non simple diagrams. We only show a
portion of $S_{(g)}$ that is a proper embedding of $[-1, 1]\times (-1,1)$. The blue parts cover the rest of the diagram.}
\label{figure:reducedD2_g}
\end{figure}

There are no diagrams in the disk $D^2$ as the ones of Fig.~\ref{figure:reducedD2_g} that are not as the ones of Fig.~\ref{figure:reducedD_g}, hence clearly simple diagrams in $D^2$ are exactly the reduced ones (Definition~\ref{defn:reduced_0}: a diagram without crossings as the ones of Fig.~\ref{figure:reducedDS3}). In $S^3$ every link diagram with the minimal number of crossings is reduced. It would be nice if in $\#_g(S^1\times S^2)$ every link except some obvious selected cases like the knot $1_1$ in $S^1\times S^2$ (Subsection~\ref{subsec:1_1}), had a simple diagram with the minimal number of crossings. With the following proposition we can note that almost all links have one such diagram. The fact that we can not remove all the crossings as the ones of Fig.~\ref{figure:reducedD2_g} is the reason why we adopt the word ``simple'' instead of ``reduced''.

\begin{defn}\label{defn:quasi-reduced}
A \emph{quasi-simple} diagram is a link diagram $D\subset S^1\times [-1,1]=S_{(1)}$ without crossings as the ones in Fig.~\ref{figure:reducedD_g} and with at most one crossing as the ones in Fig.~\ref{figure:reducedD2_g}.
\end{defn}

\begin{prop}\label{prop:simple_diag}
$\ $
\begin{enumerate}
\item{Every diagram $D\subset S_{(g)}$ can be replaced by another diagram $D'$ without crossings as the ones of Fig.~\ref{figure:reducedD_g}, with no more crossings ($n(D') \leq n(D)$, where $n(\bar D)$ is the number of crossings of $\bar D$) and representing the same link in $\#_g(S^1\times S^2)$ by the same e-shadow (Definition~\ref{defn:e-shadow}).} 
\item{If a link in $\#_g(S^1\times S^2)$ does not intersect twice a non separating 2-sphere, it has a simple diagram with the minimal number of crossings.}
\item{For every e-shadow, every link in $S^1\times S^2$ is represented by a quasi-simple diagram with the minimal number of crossings for that e-shadow.}
\end{enumerate}
\begin{proof}
To get the first statement we simply apply an isotopy on the handlebody where the punctured disk lies. The second statement follows from the first one.

Consider the case $g=1$. Applying some Reidemeister moves we can put all the crossings as the ones in Fig.~\ref{figure:reducedD2_g} in the same band, one after the other (see the first move in Fig.~\ref{figure:screw}). We can suppose that these crossings are all of the same type: the band is represented by an alternating part of diagram (see the second move of Fig.~\ref{figure:screw}). If we apply the move described in Subsection~\ref{subsec:diag} we add or remove two crossings as the ones in Fig.~\ref{figure:reducedD2_g} (see the last three moves in Fig.~\ref{figure:screw}). Therefore, given a $n$-crossing diagram $D$ with $k$ crossings as the ones in Fig.~\ref{figure:reducedD2_g}, we can get another diagram $D'$ representing the same link which has $n- k + \bar k$ crossings, where $\bar k =0$ if $k\in 2\Z$ and $\bar k =1$ if $k \in 2\Z+1$, and has $\bar k$ crossings as the ones in Fig.~\ref{figure:reducedD2_g}.
\end{proof}
\end{prop}

Unfortunately we do not have an analogous version of Proposition~\ref{prop:simple_diag}-(3.) for $g\geq 2$.

\begin{quest}\label{quest:qred_no_red}
Are there links in $S^1\times S^2$ with crossing number bigger than $1$ and not admitting a simple diagram in the annulus with the minimal number of crossings? Is it true that links in $S^1\times S^2$ which are alternating, non \emph{H-split} (Definition~\ref{defn:split_homotopic_genus}), intersecting twice a non separating 2-sphere, $\Z_2$-homologically trivial and not bounding an orientable surface, do not have a simple diagram in the annulus with the minimal number of crossings? Are these links the only ones not admitting such diagrams?
\end{quest}

\begin{figure}[htbp]
$$
\picw{1.5}{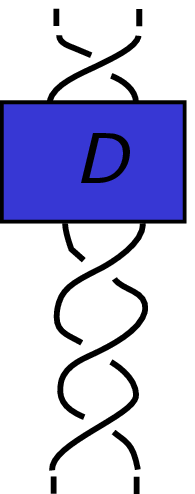} \ \ \leftrightarrow\ \  \picw{1.5}{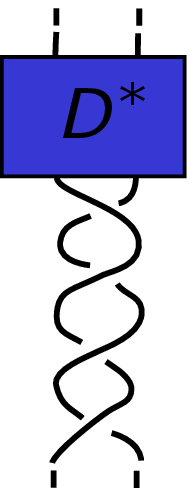} \ \ \leftrightarrow \ \ \picw{1.5}{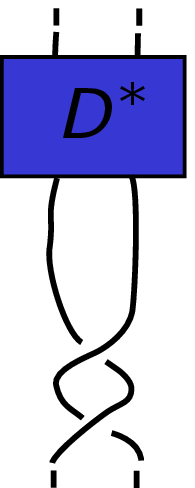}\ \  \leftrightarrow
$$
$$
\leftrightarrow\ \ \picw{1.5}{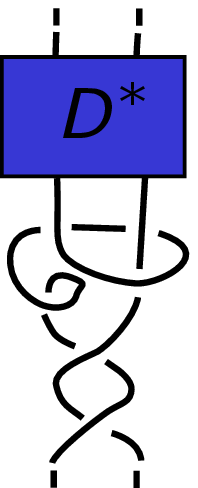}\ \  \leftrightarrow \ \ \picw{1.5}{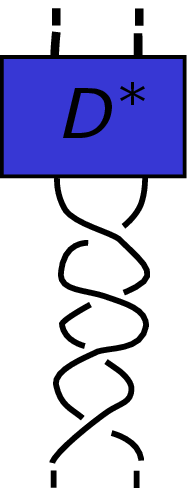} \ \ \leftrightarrow\ \  \picw{1.5}{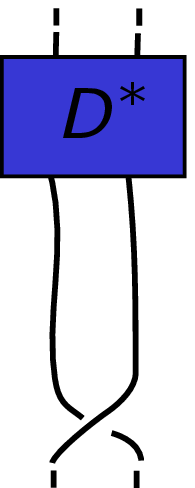}
$$
\caption{Moves for the crossings as the ones in Fig.~\ref{figure:reducedD2_g}. The part of diagram $D^*$ is the mirror image of $D$.}
\label{figure:screw}
\end{figure}

\begin{rem}\label{rem:quasi-simple}
Applying the moves seen in the proof of Proposition~\ref{prop:simple_diag}-(3.) we can substitute a crossing as the ones in Fig.~\ref{figure:reducedD2_g} with its mirror image and get another diagram representing the same link in $\#_g(S^1\times S^2)$ by the same embedding of the punctured disk. However the induced framings are different.
\end{rem}

The following proposition shows that Theorem~\ref{theorem:Tait_conj_g} becomes false if, in the case $g=1$, we simply replace ``simple'' with ``quasi-simple''.

\begin{prop}[\cite{Carrega_Tait1}]
The diagrams in Fig.~\ref{figure:no_Tait_qr} represent the same $\Z_2$-homologically trivial knot in $S^1\times S^2$.
\begin{proof}
By Remark~\ref{rem:quasi-simple} we can perform the moves in Fig.~\ref{figure:no_Tait_qr_moves} to get diagrams of the same link in $S^1\times S^2$.
\end{proof}
\end{prop}

\begin{figure}[htbp]
$$
\picw{2}{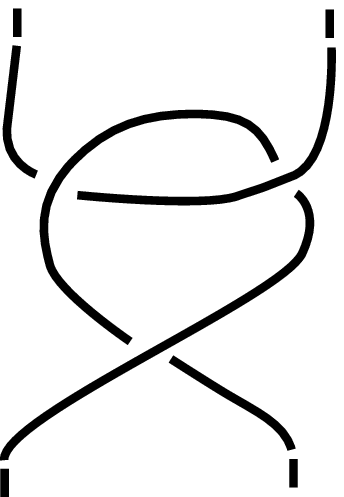} \ \ \leftrightarrow \ \  \picw{2}{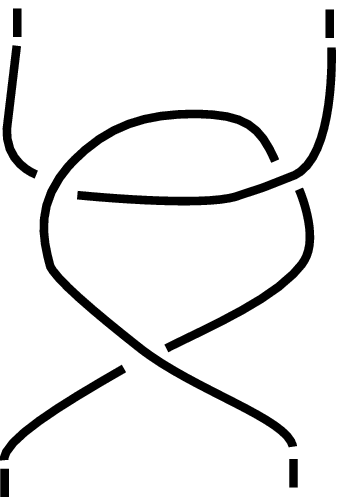} \ \ \leftrightarrow \ \ \picw{2}{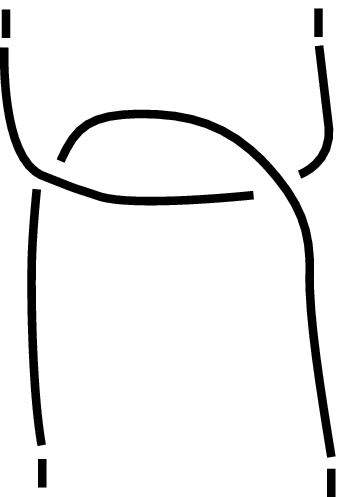}
$$
\caption{Moves in a part of $S_{(g)}$ that is a proper embedding of $[-1, 1]\times (-1,1)$ showing that, for $g=1$, the diagrams in Fig.~\ref{figure:no_Tait_qr} represent the same knot.}
\label{figure:no_Tait_qr_moves}
\end{figure}

\begin{figure}[htbp]
\begin{center}
\includegraphics[scale=0.55]{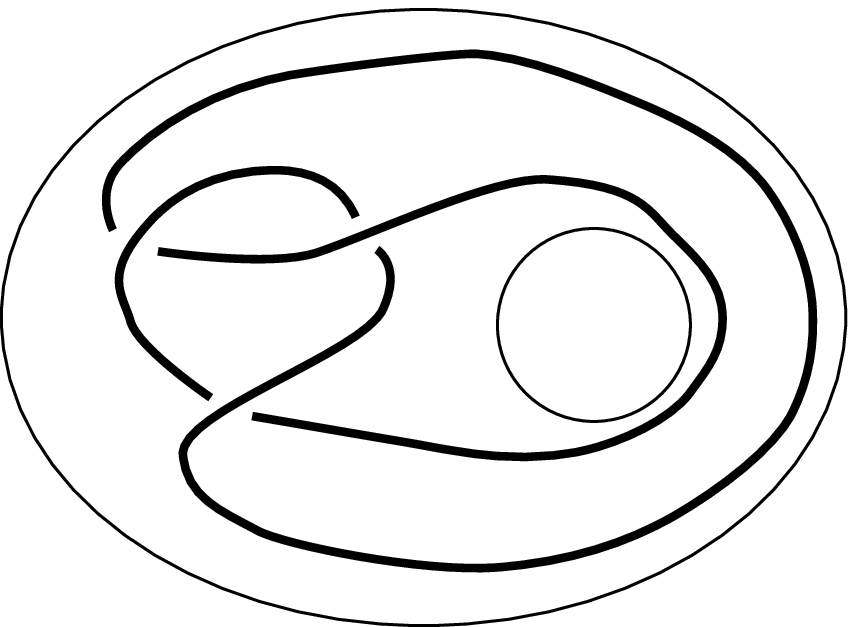} 
\hspace{0.5cm}
\includegraphics[scale=0.55]{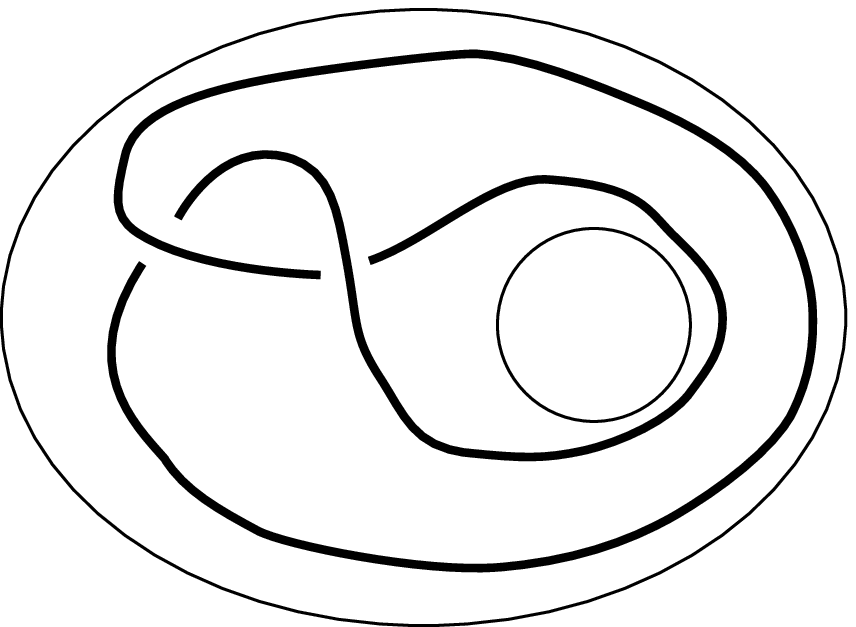}
\end{center}
\caption{Two quasi-simple alternating diagrams with different numbers of crossings and representing the same $\Z_2$-homologically trivial knot.}
\label{figure:no_Tait_qr}
\end{figure}

\subsection{More notions and results}

Let $D\subset S_{(g)}$ be a link diagram and $s$ a state of $D$. We introduced the following notations (Section~\ref{sec:Kauf_g}):
\begin{itemize}
\item{$sD$ is the number of homotopically trivial components of the splitting of $D$ with the state $s$;}
\item{$\sum s(i)$ is the sum over the crossings of the signs assigned by the state $s$;}
\item{$D_s$ is the diagram obtained removing the homotopically trivial components from the splitting of $D$ with $s$;}
\item{$X_s$ is the shadow collapsing onto a graph of genus $g$ that is associated to the state $s$ (Remark~\ref{rem:sh_for_br}).}
\end{itemize}

\begin{defn}\label{defn:adequate}
Let $D\subset S_{(g)}$ be a $n$-crossing link diagram. We denote by $g(D)$ the minimal number of holes of a punctured disk containing $D$ that is contained in $S_{(g)}$. We denote by $s_+$ and $s_-$ the constant states of $D$ assigning respectively always $+1$ and always $-1$. The diagram $D$ is said to be \emph{plus-adequate} if $s_+D > sD$ for every $s$ such that $\sum_i s(i)=n-2$, namely for every state $s$ differing from $s_+$ only at a crossing. It is said to be \emph{minus-adequate} if $s_-D > sD$ for every $s$ such that $\sum_i s(i) = 2-n$. The diagram $D$ is \emph{adequate} if it is both plus-adequate and minus-adequate. A link diagram is \emph{connected} if it is so as a 4-valent graph.
\end{defn}

\begin{rem}\label{rem:adeq}
If the diagram is contained in a 2-disk, being plus-adequate (resp. minus-adequate) is equivalent to the following fact: in the splitting of the diagram $D$ corresponding to $s_+$ (resp. $s_-$), for every crossing $j$ of $D$ the two strands replacing $j$ lie on two different components of that splitting of $D$ (remark after \cite[Definition 5.2]{Lickorish}). If the diagram is not contained in a 2-disk the two conditions are not related: Fig.~\ref{figure:rem_adeq} shows two diagrams in the annulus ($g=1$) satisfying the condition described above, but only the diagram on the right is plus-adequate.
\end{rem}

\begin{figure}
\begin{center}
\includegraphics[scale=0.55]{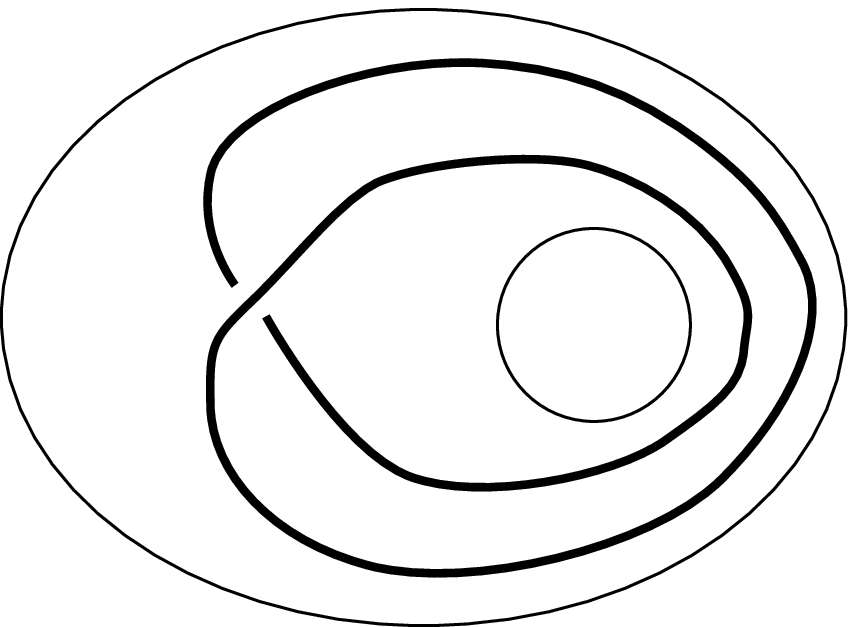} 
\hspace{0.5cm}
\includegraphics[scale=0.55]{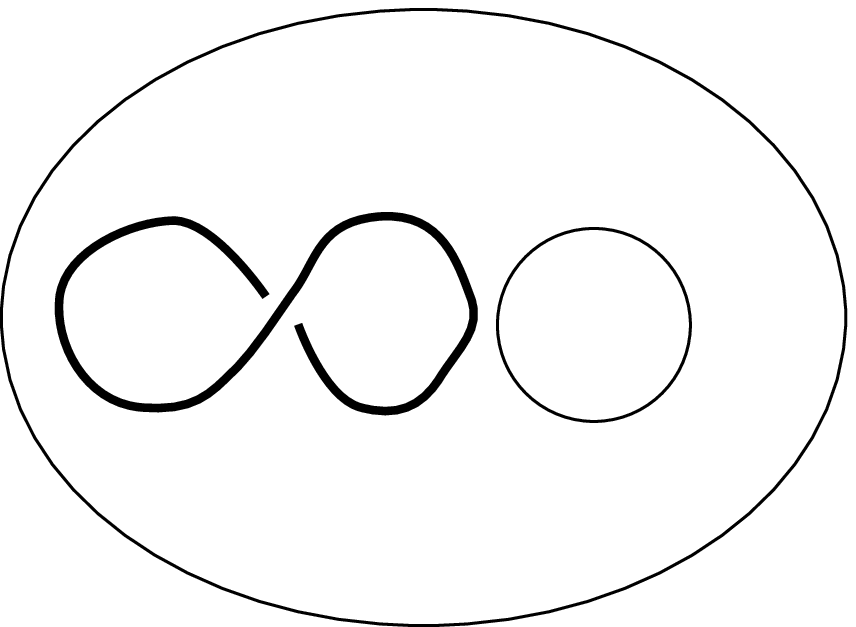}
\end{center}
\caption{Only the diagram on the right is plus-adequate.}
\label{figure:rem_adeq}
\end{figure}

\begin{prop}[\cite{Carrega_Taitg}]\label{prop:reduced_D_g}
Fix an e-shadow of $\#_g(S^1\times S^2)$. Every simple, alternating, connected link diagram $D\subset S_{(g)}$ that represents a $\Z_2$-homologically trivial link in $\#_g(S^1\times S^2)$ is adequate.
\begin{proof}
Every link diagram $D\subset S_{(g)}$ divides the punctured disk in connected \emph{regions} of dimension 2. We say that a region is \emph{external} if it touches the boundary of $S_{(g)}$, otherwise it is \emph{internal}. Since $D$ is connected we have that in the internal regions are 2-disks. We can give to the regions a black/white coloring as in a chessboard. The link represented by $D$ is $\Z_2$-homologically trivial, hence the external regions are colored in the same way. We say that a boundary component of a region is \emph{internal} if it is different from a boundary component of $S_{(g)}$. Since $D$ is alternating the splitting with $s_+$ is equal to the union of the internal boundary components of all the black regions or the white ones (see Fig.~\ref{figure:black-white}). We assume that the components of the splitting with $s_+$ bound the black regions, hence the splitting of $D$ with $s_-$ is equal to the union of the internal boundary components of all the white regions. Changing the splitting of a crossing either merges two different components of the splitting of $D$, or divides one component in two. The diagram $D$ is not plus-adequate if and only if there is a crossing $j$ in $D$ such that after splitting every crossing different from $j$ in the positive way we have a situation as in Fig.~\ref{figure:cases_lem}: if we change the splitting from $+1$ to $-1$ on $j$ we get one more homotopically trivial component by
\begin{enumerate}
\item{dividing in two a previous homotopically trivial component (Fig.~\ref{figure:cases_lem}-(left));}
\item{dividing in two a homotopically non trivial component (Fig.~\ref{figure:cases_lem}-(center));}
\item{fusing two homotopically non trivial components (Fig.~\ref{figure:cases_lem}-(right)).}
\end{enumerate}
In the various cases this implies that
\begin{enumerate}
\item{there is a black region that is adjacent twice to the same crossing;}
\item{the considered homotopically non trivial component must bound an external region that is a black annulus adjacent twice to the same crossing;}
\item{the two considered homotopically non trivial components must bound two external regions and these are black and adjacent to the same crossing.}
\end{enumerate}
The first two cases happen only if the crossing is as in Fig.~\ref{figure:reducedD_g}. The third case happens only if the crossing is as in Fig.~\ref{figure:reducedD2_g}. Therefore all these cases are avoided by our hypothesis. In the same way we prove minus-adequacy. 
\end{proof}
\end{prop}

\begin{figure}[htbp]
$$
\pic{1.6}{0.4}{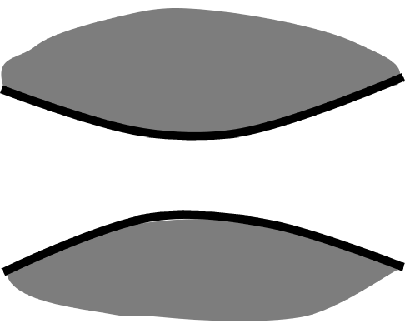} \ \stackrel{+1}{\longleftarrow} \ \pic{1.6}{0.4}{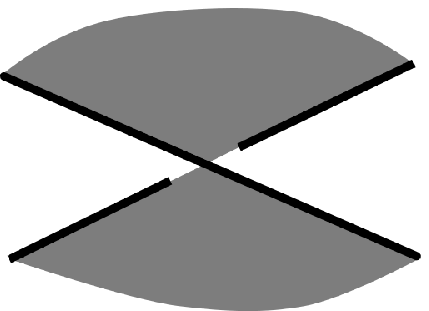} \ \stackrel{-1}{\longrightarrow} \ \pic{1.6}{0.4}{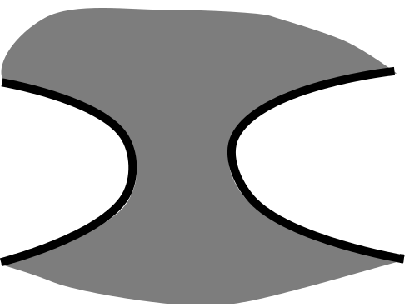}
$$
\caption{The colors of the
regions near a crossing in an alternating diagram are all of this
type. Therefore the splitting with $s_+$ (resp. $s_-$) gives the internal boundary components of all the black (white) regions.}
\label{figure:black-white}
\end{figure}

\begin{figure}[htbp]
\begin{center}
\parbox[c]{3cm}{ 
\includegraphics[scale=0.4]{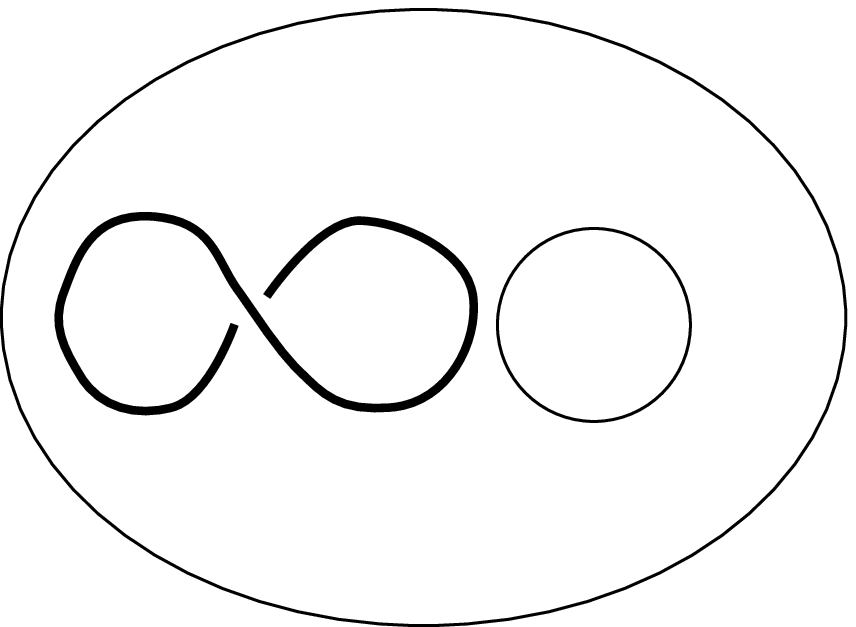} \\
\includegraphics[scale=0.4]{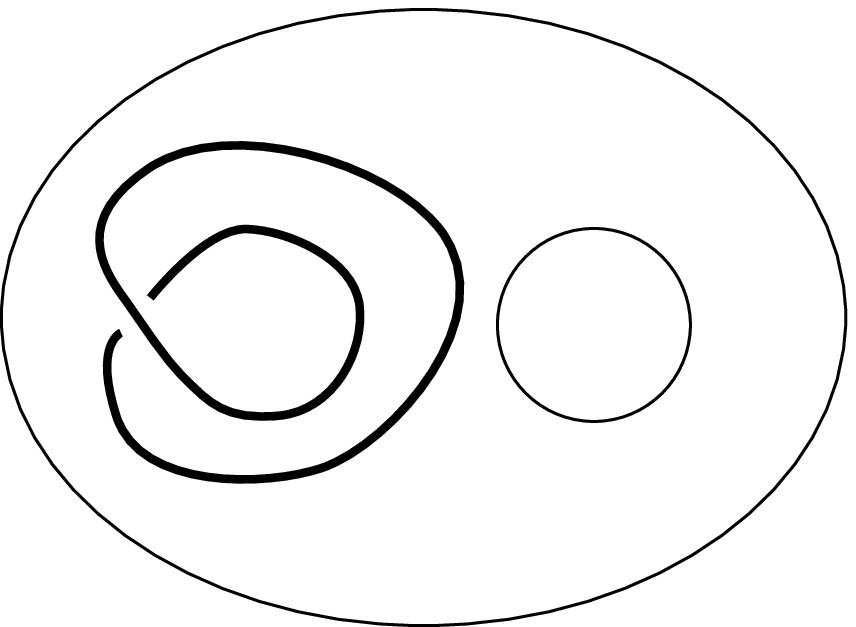} }
\hspace{0.35cm}
\parbox[c]{3cm}{  \includegraphics[scale=0.4]{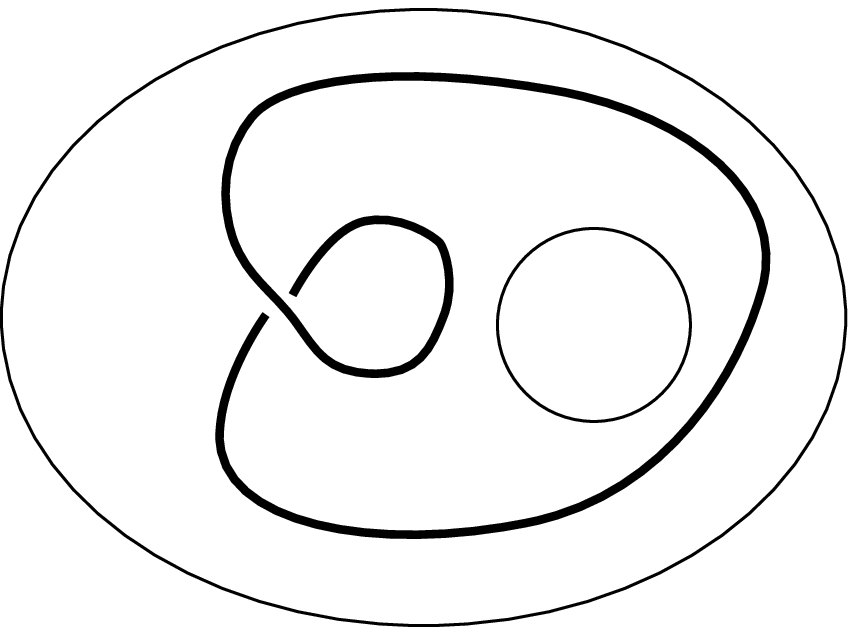} }
\hspace{0.35cm}
\parbox[c]{3cm}{ \includegraphics[scale=0.4]{rem_adeq1.eps} }
\end{center}
\caption{The cases described in the proof of Proposition~\ref{prop:reduced_D_g} and the first three type of cases in the proof of Lemma~\ref{lem:ineq1_g}. Only an homotopically non trivial annlus contained in $S_{(g)}$ is shown.}
\label{figure:cases_lem}
\end{figure}

\begin{defn}\label{defn:split_homotopic_genus}
A link $L\subset \#_g(S^1\times S^2)$ is \emph{H-split} if it has a non connected diagram in some e-shadow (non connected as a planar graph).

The \emph{homotopic genus} of $L$ is the minimum of $g(D)$ (Definition~\ref{defn:adequate}) as $D$ varies among all diagrams of $L$ in any e-shadow.
\end{defn}

\begin{prop}[\cite{Carrega_Tait1}]\label{prop:split}
A link $L\subset S^1\times S^2$ is H-split if and only if there are two non empty sub-links, $L_1$ and $L_2$, such that $L= L_1\cup L_2$ and they are separated by either a trivial sphere or a Heegaard torus. That means that either there is a 3-ball $B$ such that $L_1\subset B$ and $L_2\subset (S^1\times S^2)\setminus B$, or there are two disjoint solid tori, $V_1$ and $V_2$ such that $S^1\times S^2= V_1\cup V_2$, $L_1\subset V_1$ and $L_2\subset V_2$ ($\partial V_1= \partial V_2$ is the Heegaard torus).
\begin{proof}
Let $D,D_1,D_2\subset S^1\times [-1,1]$ be non empty link diagrams such that $D$ is the disjoint union of $D_1$ and $D_2$. Up to enumeration there are two cases: $D_1$ is contained in a 2-disk disjoint from $D_2$, and neither $D_1$ nor $D_2$ are contained in a 2-disk. Fix an e-shadow. In the first case we have a 3-ball in $S^1\times S^2$ that contains the components of the link that project on $D_1$ and is disjoint from the ones that project on $D_2$. In the second case we have a simple closed curve in the annulus that is isotopic to the core, this defines a separating Heegaard torus. Conversely, if we have a sphere or a Heegaard torus that separate the link components we can find one such diagram for some e-shadow.
\end{proof}
\end{prop}

\begin{prop}[\cite{Carrega_Taitg}]\label{prop:homot_genus}
The homotopic genus of a link $L\subset \#_g(S^1\times S^2)$ is equal to the minimum $g'$ such that the complement has a connected sum decomposition of the form 
$$
\#_g(S^1\times S^2)\setminus L = (\#_{g'}(S^1\times S^2)\setminus L') \# (\#_{g-g'}(S^1\times S^2))
$$
for some link $L'\subset \#_{g'}(S^1\times S^2) $.
\begin{proof}
If $L$ has homotopic genus $g'\leq g$ there is a diagram $D\subset S_{(g)}$ that represents $L$ via some e-shadow and is contained in a disk with $g'$ holes lying in $S_{(g)}$. Hence there is a factor $\#_{g-g'}(S^1\times S^2)$ in a connected sum decomposition of the complement of $L$. On the other hand if we have a decomposition $\#_g(S^1\times S^2)= (\#_{g'}(S^1\times S^2)\setminus L') \# (\#_{g-g'}(S^1\times S^2))$, we can get a diagram $D\subset S_{(g)}$ representing $L$ via some e-shadow by adding $g-g'$ 1-handles to $S_{g'}$ equipped with a diagram $D'\subset S_{(g')}$ of $L'$. This implies that the homotopic genus of $L$ is at most $g'$.
\end{proof}
\end{prop}

\begin{rem}
By Proposition~\ref{prop:homot_genus} if a link $L\subset \#_g(S^1\times S^2)$ has homotopic genus $g'<g$, its complement $\#_g(S^1\times S^2)\setminus L$ is reducible, in particular it is not hyperbolic. By Proposition~\ref{prop:homot_genus} and Remark~\ref{rem:tensor} the links with homotopic genus $g'\leq g$ can be seen as links in $\#_{g'}(S^1\times S^2)$, hence the really interesting links in $\#_g(S^1\times S^2)$ are the ones with homotopic genus $g$.
\end{rem}

We need to extend the notion of ``\emph{breadth}'' from Laurent polynomials to rational functions:
\begin{defn}\label{defn:breadth_g}
Let $f = g/h$ be a rational function and $g$ and $h$ be two polynomials. The \emph{degree}, or the \emph{order} at $\infty$, of $f$ is the difference of the maximal degree of the non zero monomials of $g$, $\ord_\infty g$, and the maximal degree of the non zero monomials of $h$, $\ord_\infty h$:
$$
\ord_\infty f := \ord_\infty g -\ord_\infty h .
$$ 
The \emph{order} at $0$ of $f$ is the difference of the minimal degree of the non zero monomials of $g$, $\ord_0 g$, and the minimal degree of the non zero monomials of $h$, $\ord_0 h$:
$$
\ord_0 f := \ord_0 g -\ord_0 h.
$$
The \emph{breadth} of $f$ is 
$$
B(f):= \ord_\infty f - \ord_0 f = B(g) -B(h).
$$
If $f=0$ we have $\ord_\infty f := -\infty$ and $\ord_0 f:= \infty$ and we define $B(f):=0$.
\end{defn}
Clearly the definitions above do not depend on the choice of $g$ and $h$. If $0\leq \ord_0 f < \infty$, $\ord_0 f$ is equal to the multiplicity of $f$ in $0$ as a zero. If $\ord_0 f \leq 0$, $\ord_0 f$ is the negative of the order of $f$ in $0$ as a pole. If $-\infty < \ord_\infty f \leq 0$, $\ord_\infty f$ is the negative of the multiplicity of $f$ in $\infty$ as a zero. If $0 \leq \ord_\infty f$, $\ord_\infty f$ is the order of $f$ in $\infty$ as a pole. Here are some easy and useful properties:
\begin{itemize}
\item{if $f,g\neq 0$ then $B(f/g) = B(f)-B(g)$;}
\item{$\ord_0 f(A)= - \ord_\infty f(A^{-1})$;}
\item{if $f$ is \emph{symmetric} (for all $A$ we have $f(A)=f(A^{-1})$), then $\ord_\infty f = -\ord_0 f$;}
\item{$\ord_\infty (f+g) \leq \max\{ \ord_\infty f , \ord_\infty g\}$;}
\item{$\ord_\infty (f+g) = \max\{ \ord_\infty f , \ord_\infty g\}$ if $\ord_\infty f \neq \ord_\infty g$;}
\item{$\ord_0 (f+g) \geq \min\{ \ord_0 f , \ord_0 g\}$;}
\item{$\ord_0 (f+g) = \min\{ \ord_0 f , \ord_0 g\}$ if $\ord_0 f \neq \ord_0 g$;}
\item{$\ord_\infty (f\cdot g) = \ord_\infty f +\ord_\infty g$, $\ord_0 (f\cdot g) = \ord_0 f +\ord_0 g$;}
\item{$\ord_\infty \frac 1 f = -\ord_\infty f$, $\ord_0 \frac 1 f = -\ord_0 f$.}
\end{itemize}
We recall that $\cerchio_n$ is the Kauffman bracket of a 0-framed homotopically trivial unknot colored with $n$. We have
$$
\cerchio_n = (-1)^n [n+1]|_{q=A^2} , \ \ [n+1]|_{q=A^2} = \frac{A^{2(n+1)} - A^{-2(n+1)} }{ A^2 -A^{-2} } .
$$
We note that
$$
\ord_\infty \cerchio_n = -\ord_0 \cerchio_n = 2n .
$$

Let $D\subset S_{(g)}$ be a link diagram. Using the skein relations we get
$$
\langle D \rangle = \sum_s  A^{\sum s(i)} (-A^2-A^{-2})^{sD} \langle D_s\rangle .
$$
Moreover we recall that the shadow formula applied on $X_s$ gives $\langle D_s \rangle$ (Remark~\ref{rem:sh_for_br}), hence
$$
\langle D_s \rangle = \sum_\xi \prod_R \cerchio^{\chi(R)}_{\xi(R)} ,
$$
where $\xi$ runs over all the admissible colorings of $X_s$ that extend the coloring of the boundary (that is $1$), $R$ runs over all the regions of $X_s$ (the external regions do not matter because either they are annuli or their color is $0$), $\chi(R)$ is the Euler characteristic of $R$, and $\xi(R)$ is the color of $R$ given by $\xi$. Therefore $\langle D_s \rangle$ is a symmetric function of $A^2$, namely there are two polynomials $f,h\in \Z[q]$ such that 
$$
\langle D_s \rangle = \frac{f|_{q=A^2}}{h|_{q=A^2}}  , \ \ \langle D_s \rangle|_A = \langle D_s \rangle|_{A^{-1}}  .
$$

\begin{defn}\label{defn:psi}
Let $s$ be a state of the link diagram $D\subset S_{(g)}$. We set:
$$
\psi(s):= \frac 1 2 \ord_\infty \langle D_s \rangle = -\frac 1 2 \ord_0 \langle D_s \rangle \ \ \in \Z .
$$
\end{defn}

Note that if $g\leq 1$ the quantity $\psi(s)$ is always null.

\subsection{Main theorems}

\begin{theo}[\cite{Carrega_Taitg}]\label{theorem:Tait_conj_g}
Let $D\subset S_{(g)}$ be an alternating, simple diagram that represents a link $L\subset \#_g(S^1\times S^2)$ by an e-shadow $e$. Suppose that $L$ is non H-split (\emph{e.g.} a knot), $\Z_2$-homologically trivial and with homotopic genus $g$. Then for any diagram $D'\subset S_{(g)}$ that represents $L$ by an e-shadow $e'$ we have
$$
n(D) \leq n(D') + \frac{g-1}{2} ,
$$
where $n(D)$ and $n(D')$ are the number of crossings of $D$ and $D'$. In particular if $g\leq 2$, we have that $n(D)$ is the crossing number of $L$.
\end{theo}

\begin{theo}[\cite{Carrega_Taitg}]\label{theorem:Tait_conj_Jones_g}
Let $D\subset S_{(g)}$ be a $n$-crossing, connected, alternating diagram of a $\Z_2$-homologically trivial link in $\#_g(S^1\times S^2)$ without crossings as the ones of Fig.~\ref{figure:reducedD_g} and not contained in a disk with less than $g$ holes ($g(D)=g$). Then
$$
B(\langle D \rangle) = 4n +4 -4g -4k ,
$$
where $k$ is the number of crossings adjacent to two external regions (as the ones of Fig.~\ref{figure:reducedD2_g}), neither adjacent twice to just one external region.
\end{theo}

Theorem~\ref{theorem:Tait_conj_Jones_g} gives us criteria to detect if a link in $\#_g(S^1\times S^2)$ is alternating:
\begin{cor}[\cite{Carrega_Taitg}]\label{cor:conj_Tait_Jones_g}
Let $L\subset \#_g(S^1\times S^2)$ be a non H-split $\Z_2$-homologically trivial link. 
\begin{enumerate}
\item{If $B(\langle L \rangle)$ is not a multiple of $4$, then $L$ is not alternating.}
\item{If $B(\langle L \rangle)$ is not a positive multiple of $4$, then $L$ is not alternating with a simple diagram.}
\item{If $L$ has homotopic genus $g$, $L$ does not intersect twice any non separating 2-sphere and $B(\langle L \rangle)< 4n+4-4g$, then either $L$ is not alternating, or $L$ has crossing number lower than $n$.}
\end{enumerate}
\begin{proof}
It follows from Theorem~\ref{theorem:Tait_conj_Jones_g}. To get the second statement we must note that every connected diagram in $S_{(g)}$ that is not contained in a disk with $g'<g$ holes has at least $g$ crossings.
\end{proof}
\end{cor}

\begin{ex}\label{ex:no_alt_g}
The knots represented by the diagrams in Fig.~\ref{figure:ex_no_alt} ($g=1$) are $\Z_2$-homologically trivial and have Kauffman bracket equal to $0$ (left) and $A-A^{-3}-A^{-5}$ (right). Therefore by Corollary~\ref{cor:conj_Tait_Jones_g}-(2.) they are not alternating. 

By Corollary~\ref{cor:conj_Tait_Jones_g}-(1.) the links in Example~\ref{ex:Kauff}-(6) and Example~\ref{ex:Kauff}-(7) ($g=2$) are non alternating. By Corollary~\ref{cor:conj_Tait_Jones_g}-(3.) the knot in Example~\ref{ex:Kauff}-(5) ($g=2$) either is non alternating or has crossing number lower than $4$. For all the links in $\#_2(S^1\times S^2)$ with crossing number $3$ there is a non separating sphere that intersects the link at most twice. The knot in Example~\ref{ex:Kauff}-(5) does not have one such sphere, hence it is not alternating. 

Unfortunately we are not able to say if the link in Example~\ref{ex:Kauff}-(11) ($g=2$) is alternating.
\end{ex}

\begin{figure}
\begin{center}
\includegraphics[scale=0.55]{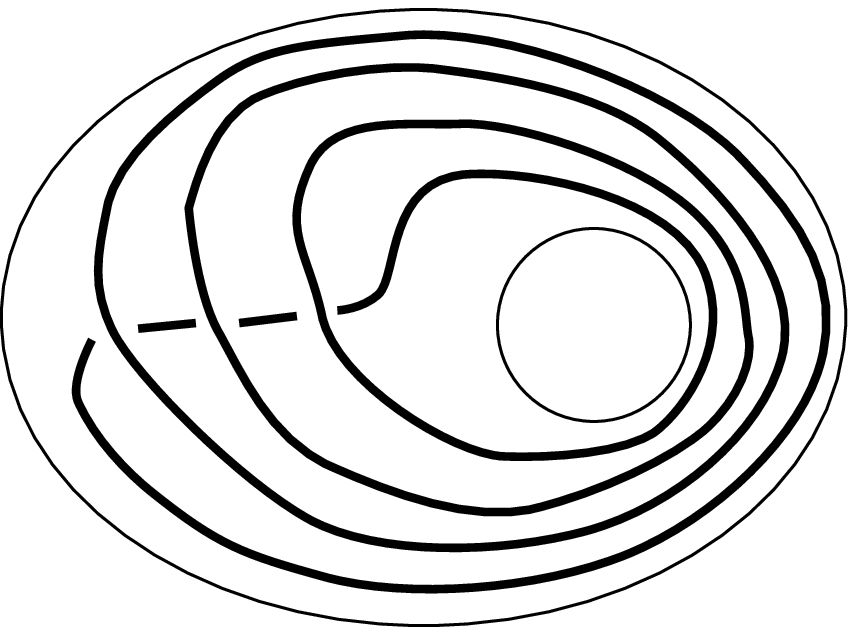}
\hspace{0.5cm}
\includegraphics[scale=0.55]{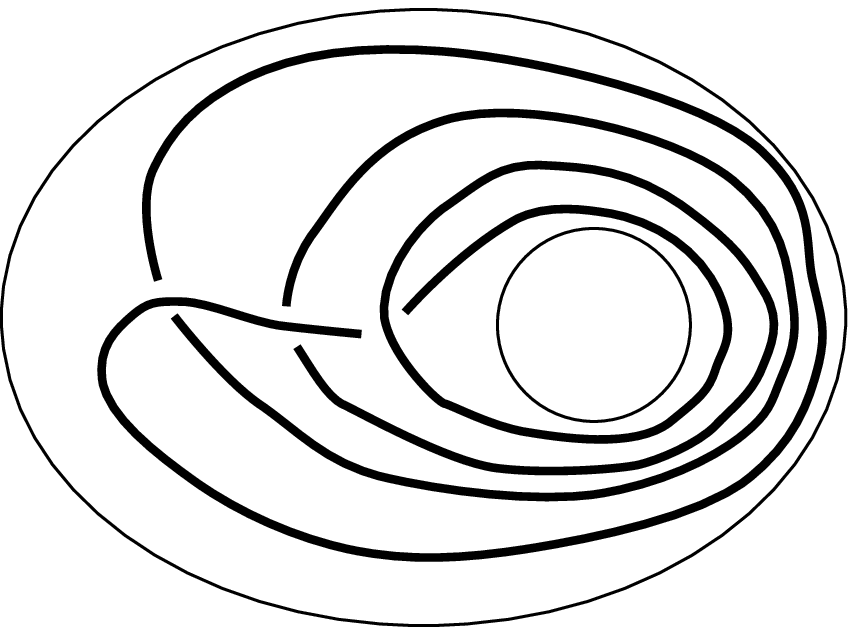}
\end{center}
\caption{Two $\Z_2$-homologically trivial knots in $S^1\times S^2$ whose Kauffman bracket are $0$ (left) and $A-A^{-3}-A^{-5}$ (right).}
\label{figure:ex_no_alt}
\end{figure}

\subsection{Alternating diagrams for $\Z_2$-homologically non trivial links}\label{subsec:alt_diag_Z_2_non_tr}

\begin{prop}[\cite{Carrega_Tait1}]\label{prop:no_Tait}
Once an e-shadow of $\#_g(S^1\times S^2)$ is fixed, two link diagrams in $S_{(g)}$ differing only in a part of the punctured disk that is a proper embedding of $[-1,1]\times (-1,1)$ where they are of the form
$$
\pic{1.8}{0.4}{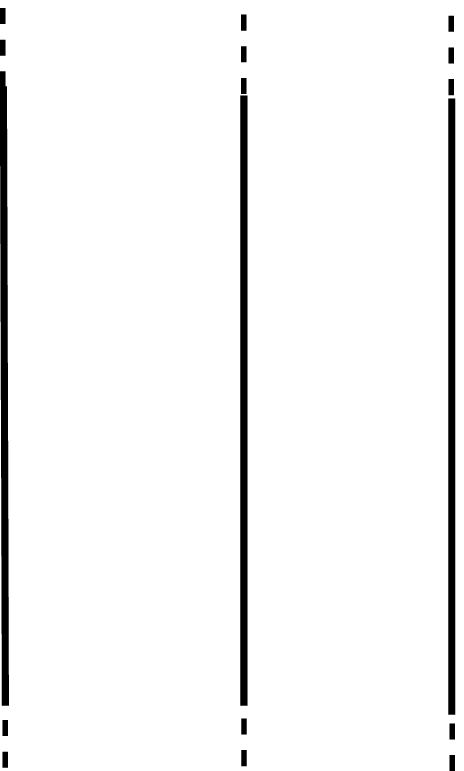} \ \ \ \text{and} \ \ \ \pic{1.8}{0.4}{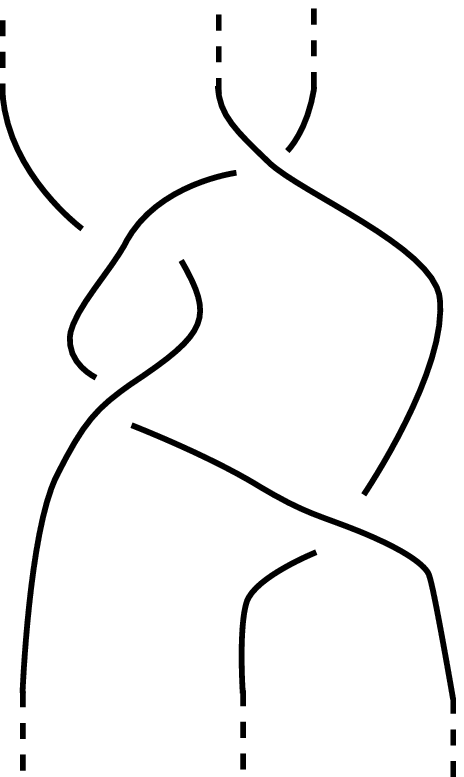},
$$
represent the same link in $\#_g(S^1\times S^2)$. Therefore the diagrams in Fig.~\ref{figure:no_Tait_g} represent the same knot $K$ in $\#_g(S^1\times S^2)$.
\begin{proof}
Let us consider the diagram with three parallel strands. If we apply the second Kirby move on the left strand of the initial diagram over a 0-framed unknot encircling the three strands (the move described in Subsection~\ref{subsec:diag}), we get the second diagram of the sequence in Fig.~\ref{figure:prop_no_Tait} or its mirror image. To get the third diagram of Fig.~\ref{figure:prop_no_Tait} we apply some Reidemeister moves. The fourth one is obtained applying again the previous moves to the third diagram but on the opposite sense. The other equivalences of Fig.~\ref{figure:prop_no_Tait} come applying Reidemeister moves.
\end{proof}
\end{prop}

\begin{figure}
\beq
 & \pic{2.2}{0.5}{diagrp1.eps}  \ \ \  \leftrightarrow\ \ \  \pic{2.6}{0.5}{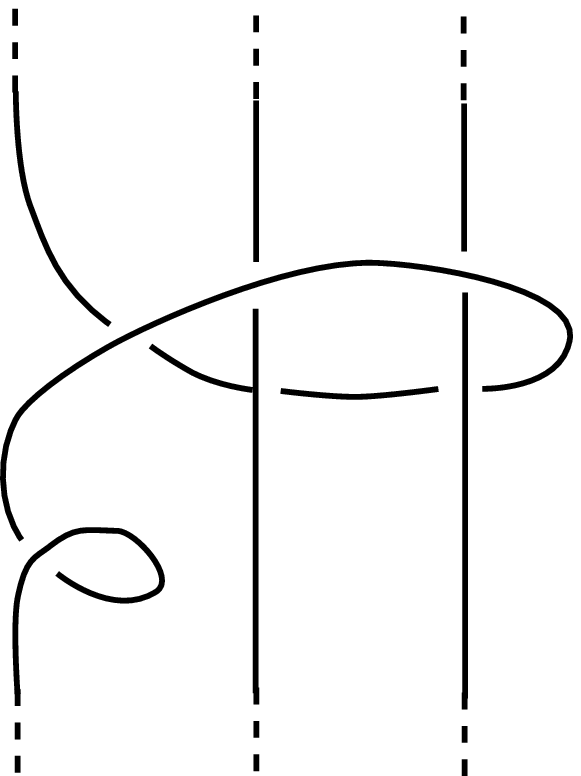} \ \ \  \leftrightarrow\ \ \  \pic{2.6}{0.5}{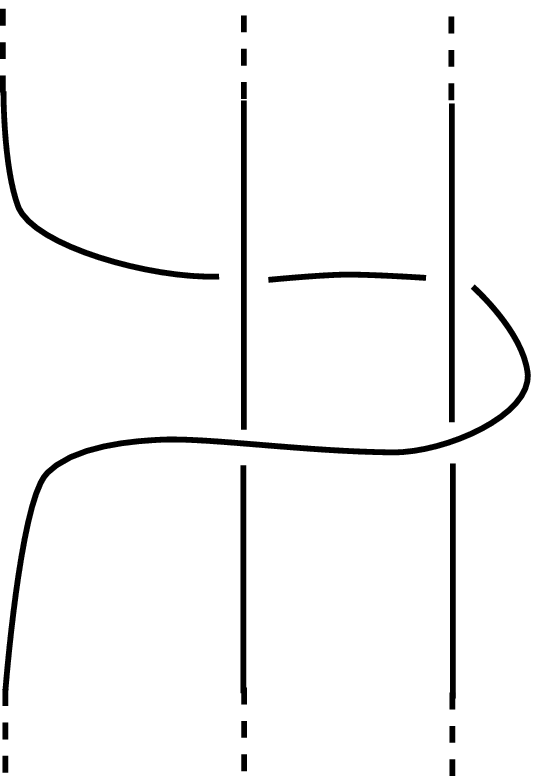} & \\
 & \leftrightarrow\ \ \ \  \pic{2.6}{0.5}{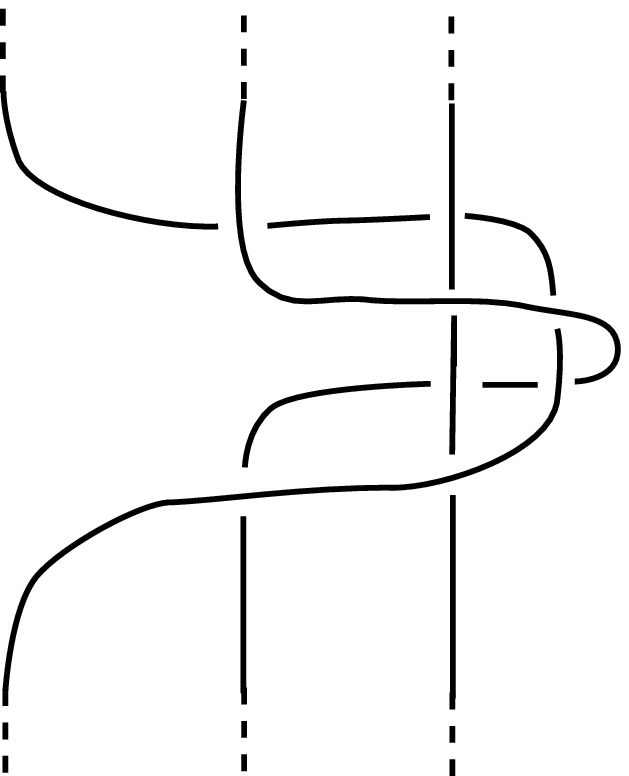}  \ \ \ \  \leftrightarrow\ \ \  \pic{2.6}{0.5}{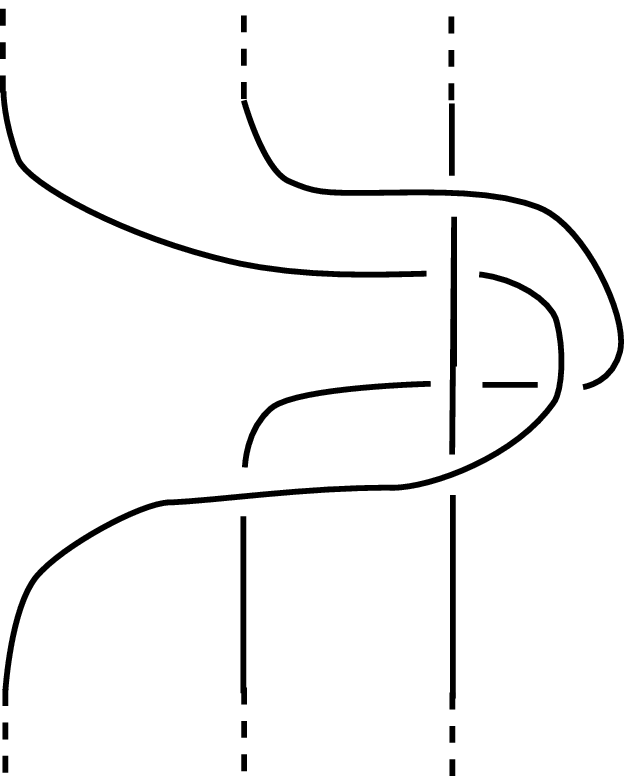}  \ \ \ \leftrightarrow \ \ \ \pic{2.6}{0.5}{diagrp2.eps}  . & 
\eeq
\caption{Moves on diagrams intersecting $(-1, 1)\times [-1,1]\subset S_{(g)}$ in three parallel strands.}
\label{figure:prop_no_Tait}
\end{figure}

The knot $K$ of Proposition~\ref{prop:no_Tait} is $\Z_2$-homologically non trivial. The diagrams are all alternating and simple but they have different numbers of crossings. Therefore the natural extension of the Tait conjecture for $\Z_2$-homologically non trivial links in $\#_g(S^1\times S^2)$ is false for any $g\geq 0$. We think that the homotopic genus of $K$ is $g$, of course if $g=1$ it is so. In fact $K$ is not contained in a 3-ball.

\begin{figure}
\begin{center}
\includegraphics[scale=0.45]{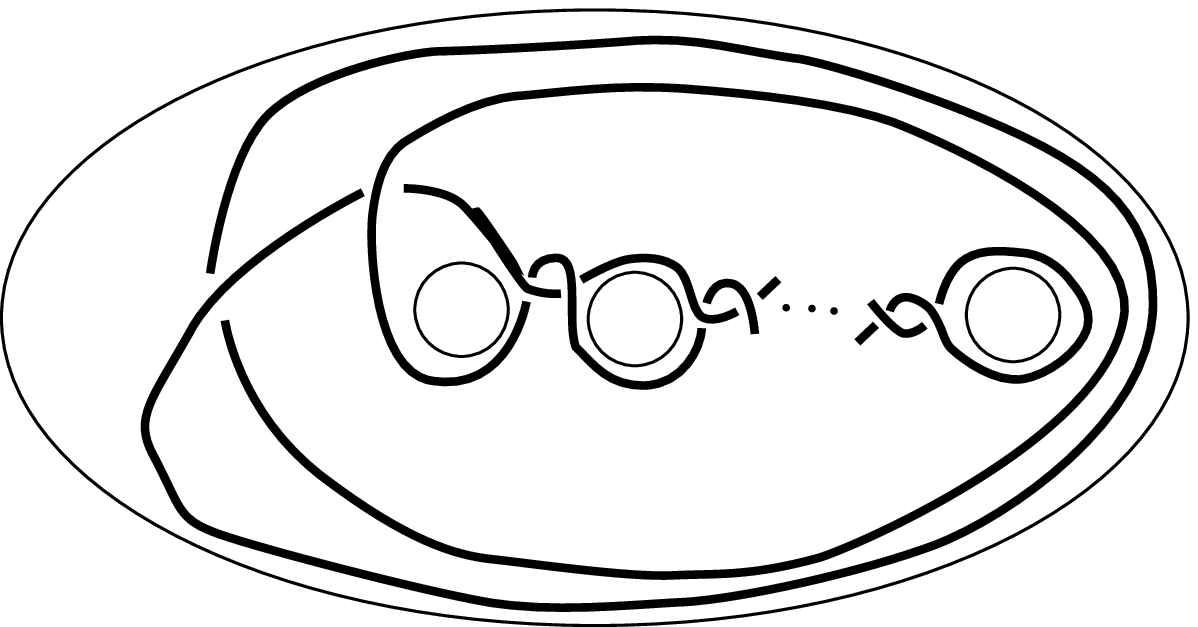} 
\hspace{0.5cm}
\includegraphics[scale=0.45]{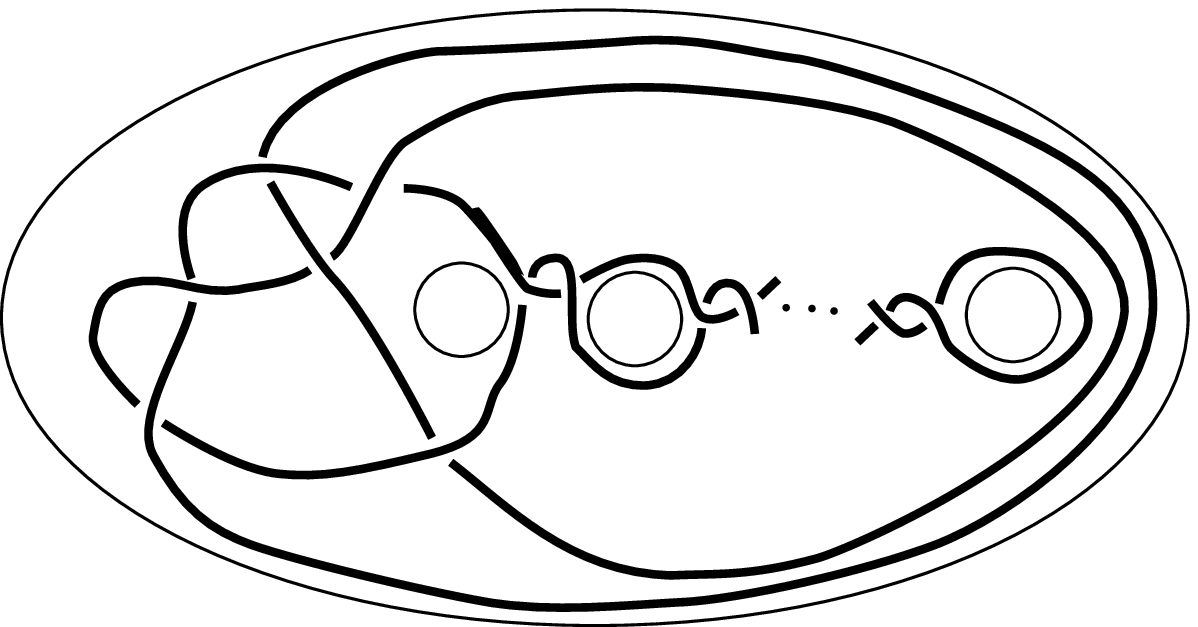}
\end{center}
\caption{Two alternating and simple diagrams of the same $\Z_2$-homologically non trivial knot which have different number of crossings.}
\label{figure:no_Tait_g}
\end{figure}

\section{Proof of the theorems}

In this section we prove Theorem~\ref{theorem:Tait_conj_g} and Theorem~\ref{theorem:Tait_conj_Jones_g}. We follow the classical proof of Kauffman, K. Murasugi and Thistlethwaite (for instance we can also see \cite[Chapter 5]{Lickorish}) applying some modifications for our case.

The following lemma is extremely useful to manage the quantity $\psi(s)$ and is no needed for the $g \leq 1$ case.

\begin{lem}[\cite{Carrega_Taitg}]\label{lem:psi}
Let $D\subset S_{(g)}$ be a diagram of a $\Z_2$-homologically trivial link in $\#_g(S^1\times S^2)$ and let $s$ be a state of $D$ such that $D_s$ is non empty. Then the following hold:
\begin{enumerate}
\item{there is a unique admissible coloring $\xi_0$ of the shadow $X_s$ that extends the boundary coloring and assigns only the colors $0$ and $1$;}
\item{$$
\psi(s) = \max_\xi \frac 1 2 \ord_\infty \prod_R \cerchio^{\chi(R)}_{\xi(R)} = \max_\xi \sum_R \chi(R) \cdot \xi(R) ;
$$}
\item{$$
\psi(s)= \sum_R \chi(R) \cdot \xi_0(R) .
$$}
\end{enumerate}
\begin{proof}
Every meromorphic function $f:\mathbb{C} \dashrightarrow \mathbb{C}$ is of the form $f(A) =\lambda_f A^{\ord_\infty f} + g(A)$, where $\lambda_f$ is a unique non zero complex number and $g:\mathbb{C} \dashrightarrow \mathbb{C}$ is a unique meromorphic function such that $\ord_\infty g < \ord_\infty f$. The equality $\ord_\infty (f_1+f_2) = \max\{ \ord_\infty f_1 , \ord_\infty f_2\}$ holds if and only if either $\ord_\infty f_1 \neq \ord_\infty f_2$ or $\lambda_{f_1} \neq - \lambda_{f_2}$. For an admissible coloring $\xi$ of the shadow $X_s$ let $f(\xi)$ be its contribution to the shadow formula:  
$$
f(\xi) := \prod_R \cerchio^{\chi(R)}_{\xi(R)} .
$$
We get the second statement by showing that $\lambda_{f(\xi)} = \pm 1$ and the sign does not depend on $\xi$ for each coloring $\xi$. We have 
$$
f(\xi) = (-1)^{\sum_R \chi(R)\xi(R)} \prod_R [\xi(R) +1]|_{q=A^2}^{\chi(R)} ,
$$
where 
$$
[n+1]|_{q=A^2} = \frac{A^{2(n+1)} - A^{-2(n +1)} }{ A^2 - A^{-2} } .
$$

The regions given by a diagram are colored odd or even like in a chessboard. Therefore the parity of the colors of the regions does not depend on the choice of the admissible coloring. Hence $(-1)^{\sum_R \chi(R)\xi(R)}$ does not depend on $\xi$. We have $\lambda_{[n+1]} = 1$. Therefore $(-1)^{\sum_R \chi(R)\xi(R)} \lambda_{f(\xi)} = 1$.

Since the link represented by $D$ is $\Z_2$-homologically trivial and the first skein relation does not change the $\Z_2$-homology class, we have that the link $L_s\subset \#_g(S^1\times S^2)$ represented by $D_s$ is $\Z_2$-homologically trivial. The link $L_s$ is contained in the punctured disk $S_{(g)}$ that is contained in the handlebody $H_g$ whose double is $\#_g(S^1\times S^2)$. Hence $0 = [L_s]= [D_s] \in H_1(S_{(g)};\Z_2)$. Therefore we can find a surface $S_s \subset S_{(g)}$ merging some internal regions of $S_{(g)}$ (given by $D_s$) and whose boundary is $D_s$. We color with $1$ the regions that compose $S_s$ and the annuli attached to the components of $D_s$, and we color with $0$ the remaining ones. We found an admissible coloring of $X_s$ that extends the boundary coloring and assigns just color $0$ and $1$. Such coloring is unique because, as we observed before, the parity of the color of a region does not depend on the admissible coloring. Let $\xi_0$ be such coloring. Since the Euler characteristics of the regions are all non positive and the colorings are all non negative, we have $\sum_R \chi(R)\xi(R) \leq \sum_R \chi(R)\xi_0(R)$ for every admissible coloring $\xi$. Therefore using this and the second statement we get the third statement.
\end{proof}
\end{lem}

Although we do not need it to prove the main theorems, it is interesting to get also a lower bound of the quantity $\psi(s)$:
\begin{prop}[\cite{Carrega_Taitg}]
Let $D\subset S_{(g)}$ be a link diagram and $s$ a state such that $D_s$ is non empty. For $1\leq h \leq g$ let $\varphi_h(D_s)$ be the number of internal regions given by $D_s$ that are disks with $h$ holes. Then
\begin{enumerate}
\item{
$$
\sum_{h=2}^{g(D)} (h-1)\varphi_h(D_s) \leq g(D_s) -1 ;
$$}
\item{
$$
1-g(D) \leq 1-g(D_s) \leq \psi(s) .
$$}
\end{enumerate}
\begin{proof}
$1.$ We proceed by induction on $g(D)$. If $D$ is contained in an annulus ($g(D)=1$) obviously the equality holds. Suppose that the inequality is true for every diagram $D'$ such that $g(D') < g(D)$. There are two cases:
\begin{enumerate}
\item{there is an internal region that is a disk with $g$ holes ($\varphi_g(D)=1$) and the other internal regions are annuli ($\varphi_h(D)=0$ for $1<h<g$);}
\item{there are no internal regions that are disks with $g$ holes ($\varphi_g(D)=0$).}
\end{enumerate}
In the first case the equality holds. Suppose we are in the second case. The diagram $D$ is obtained merging three disjoint diagrams, $D_1, D_2 , D_3 \subset S_{(g)}$, such that $g(D_1) + g(D_2) =g(D)$, $1\leq g(D_1), g(D_2)$, and $D_3$ is a (maybe empty) set of parallel curves encircling $D_1 \cup D_2$. If $D_3$ is non empty there is an $h$ such that $1<h<g$ and $\varphi_h(D) = \varphi_h(D_1)+\varphi_h(D_2) +1$ and $\varphi_{h'}(D) = \varphi_{h'}(D_1)+\varphi_{h'}(D_2)$ for $1<h'<g$, $h'\neq h$. The number $h$ is the number of holes of the region bounded by the most internal component of $D_3$. If $D_3$ is emty $\varphi_h(D) = \varphi_h(D_1)+\varphi_h(D_2)$ for all $h$. We conclude using the inductive hypothesis on $D_1$ and $D_2$.

$2.$ Clearly $g(D_s)\leq g(D) \leq g$, furthermore the Euler characteristic of the regions of the shadow $X_s$ is non positive. Therefore we have the inequalities on the left and right of the statement. For the the point $1.$ and Lemma~\ref{lem:psi}-(3.) we get
$$
1-g(D_s) \leq \sum_{h=2}^{g(D_s)} (1-h) \varphi_h(D_s) = \sum_{h=1}^{g(D_s)} (1-h) \varphi_h(D_s) \leq  \sum_R \chi(R)\xi_0(R) = \psi(s)  .
$$
\end{proof}
\end{prop}

\begin{defn}
Enumerate the boundary components of $S_{(g)}$. Let $D\subset S_{(g)}$ be a link diagram and $s$ a state of $D$. We denote by $p_i(s)$ the number of components of the splitting of $D$ with $s$ that are parallel to the $i^{\rm th}$ boundary component of $S_{(g)}$ ($i=1,\ldots, g+1$), and with $p(s)$ the total number of homotopically non trivial components of the splitting of $D$ with $s$. 
\end{defn}
For $g\geq 3$ it is not true that $\sum_{i=1}^{g+1} p_i(s) =p(s)$, but it is true if $g=2$. The links represented by a diagram $D\subset S_{(2)}$ are $\Z_2$-homologically trivial if and only if $p_i(s)+p_j(s)\in 2\Z$ for any state $s$ and any $i,j$. If this holds, since $p(s)=p_1(s)+p_2(s)+p_3(s)$, for every $i$ we have that $p_i(s)\in 2\Z$ if and only if $p(s)\in 2\Z$. These simple properties furnish some more information about the Kauffman bracket of links in $\#_2(S^1\times S^2)= (S^1\times S^2)\# (S^1\times S^2)$ (we do not need it to prove the main theorem):
\begin{prop}[\cite{Carrega_Taitg}]\label{prop:state_sum2}
Fix an e-shadow of $\#_2(S^1\times S^2)$. Let $D\subset S_{(2)}$ be a link diagram of a $\Z_2$-homologically trivial link in $\#_2(S^1\times S^2)$, and let $s$ be a state of $D$. Then
$$
\psi(s) = \begin{cases}
0 & \text{ if } p(s)\in 2\Z  \\
-1 & \text{ if } p(s) \in 2\Z+1
\end{cases}.
$$
\begin{proof}
If $p_i(s)=0$ for some $i=1,2,3$ ($p(s) \in 2\Z$) then $D_s$ is contained in the union of two disjoint annuli, $A_1$ and $A_2$, that encircle the boundary components different from the $i^{\rm th}$ one. Let $D_{s,j}$ be the diagram in the annulus $A_j$. Then by Remark~\ref{rem:tensor} $\langle D_s\rangle =\langle D_{s,1}\rangle \cdot \langle D_{s,2}\rangle $ where $D_{s,1}$ and $D_{s,2}$. Hence by \cite[Lemma 3.3]{Carrega_Tait1} $\langle D_s \rangle $ is a positive integer number and so $\psi(s) =0$.

Suppose $p_i(s) \neq 0$ for all $i=1,2,3$. The shadow $X_s$ has a region that is a 2-disk with two holes and the other ones are annuli. By Lemma~\ref{lem:psi} we can compute $\psi(s)$ using the admissible coloring $\xi_o$, $\psi(s) = \sum_R \chi(R) \cdot \xi_0(R)$. Hence $\psi(s)= -\xi_0(R')$, where $R'$ is the region that is a disk with two holes. For any admissible coloring $\xi$ of $X_s$ we have that the color $\xi(R')\in 2\Z+1$ if and only if $p(s)\in 2\Z+1$. We conclude since $\xi_0$ assigns just color $0$ and $1$.
\end{proof}
\end{prop}

Let $D\subset S_{(g)}$ be a $n$-crossing, connected diagram with $g(D)=g$ that represents a $\Z_2$-homologically trivial link in $\#_g(S^1\times S^2)$ by an e-shadow. Then for a state $s$ we get
\beq
\ord_\infty \langle D | s\rangle & = & \sum s(i) + 2( sD + \psi(s) ) \\  
\ord_0 \langle D | s\rangle & = & \sum s(i) - 2( sD + \psi(s) ) .
\eeq
Therefore
$$
\ord_\infty \langle D | s_+ \rangle - \ord_0 \langle D | s_- \rangle = 2(n + s_+D +s_-D +\psi(s_+) + \psi(s_-)) .
$$

The following is the part that needs most work to get the final result. If $g=1$ there are less cases to analyze and $\psi(s)$ is always null, hence the proof is much easier:
\begin{lem}[\cite{Carrega_Taitg}]\label{lem:ineq1_g}
Let $D\subset S_{(g)}$ be a $n$-crossing, connected diagram of a $\Z_2$-homologically trivial link in $\#_g(S^1\times S^2)$. Then
$$
B(\langle D \rangle) \leq \ord_\infty \langle D|s_+ \rangle - \ord_0 \langle D|s_- \rangle = 2(n + s_+D + s_-D + \psi(s_+) +\psi(s_-)) .
$$
Moreover if $D$ is adequate the equality holds:
$$
B(\langle D \rangle) = \ord_\infty \langle D | s_+ \rangle -\ord_0 \langle D|s_- \rangle .
$$
\begin{proof}
For every state $s$ we denote by $M(s)$ and $m(s)$ respectively the order at $\infty$ and in $0$ of $\langle D | s\rangle$. Hence
$$
B(\langle D \rangle) \leq \max_s M(s) - \min_s m(s) ,
$$
and the equality holds if there is a unique maximal $s$ for $M$ and a unique minimal $s$ for $m$. We have:
\beq
M(s) & = & \sum_i s(i) +2sD +2\psi(s) \\
m(s) & = & \sum_i s(i) -2sD -2\psi(s) .
\eeq
Let $s$ and $s'$ be two states differing only in a crossing $j$ where $s(j)=+1$ and $s'(j)=-1$. We have
\beq
M(s) - M(s') & = & 2(1+sD-sD' +\psi(s) - \psi(s')) \\
m(s) - m(s') & = & 2(1-sD+sD' -\psi(s) + \psi(s')) .
\eeq
After the splitting of every crossing different from $j$ we have some cases up to mirror image. Four of them are divided in three types and are shown in Fig.~\ref{figure:cases_lem} (we pictured just an homotopically non trivial annulus where the diagrams hold): the crossing $j$ lies on a component contained in an annulus, $s'D=sD \pm 1$. In the remaining cases $j$ lies on a component that is not contained in an annulus, $s'D = sD$, they are shown in Fig.~\ref{figure:cases_lem2} (we pictured just a disk with two holes that is not homotopically equivalent to an annulus where the diagrams hold). Now we better examine the four types of case in order to get information about $\psi(s)-\psi(s')$, and get $M(s)\geq M(s')$ and $m(s)\geq m(s')$.

In the first two types of case (Fig.~\ref{figure:cases_lem}-(left) and Fig.~\ref{figure:cases_lem}-(center)) $D_{s'}=D_s$, hence $\psi(s')=\psi(s)$. Therefore $M(s)-M(s')= 2\pm 2 \geq 0$, and $m(s)-m(s') = 2\pm 2\geq 0$. 

In the third type of case (Fig.~\ref{figure:cases_lem}-(right)) $D_{s'}$ and $D_s$ differ by the removal or the addition of an annulus region that produces the fusion or the division of previous regions. Suppose we are in the third case and the splitting with $s'$ is obtained fusing two components of the splitting with $s$ (the other case is analogous). Then $D_{s'}$ is obtained from $D_s$ removing two parallel components. The regions of $D_s$ that do not touch the selected parallel components remain unchanged in $D_{s'}$. The selected components of $D_s$ bound an annulus region $A$. The annulus $A$ touches two different regions, $R_1$ and $R_2$, that are disks respectively with $h_1>0$ and $h_2>0$ holes. In $D_{s'}$ the regions $A$, $R_1$ and $R_2$ are replaced with a region $R$ that is a disk with $h_1+h_2-1$ holes. Let $\xi_0$ and $\xi'_0$ be the admissible colorings of $X_s$ and $X_{s'}$ that extend the boundary coloring and assign only color $0$ and $1$ (see Lemma~\ref{lem:psi}). We have two cases:
\begin{enumerate}
\item{$\xi_0(R_1)=\xi_0 (R_2) =\xi'_0(R) =0$ and $\xi(A)=1$;}
\item{$\xi_0(R_1)=\xi_0 (R_2) =\xi'_0(R) =1$ and $\xi(A)=0$.}
\end{enumerate}
Since $\chi(A)=0$, $\chi(R) = \chi(R_1)+\chi(R_2)$, $\psi(s)= \sum_{\bar R} \chi(\bar R) \xi_0(\bar R)$ and $\psi(s')= \sum_{\bar R} \chi(\bar R) \xi'_0(\bar R)$, we get that $\psi(s)=\psi(s')$. Therefore $M(s)-M(s')= 2\pm 2\geq 0$ and $m(s)- m(s') =2\pm 2 \geq 0$.

Suppose we are in the fourth type of case (Fig.~\ref{figure:cases_lem2}) and the splitting with $s'$ is obtained dividing in two a component of the splitting with $s$ (the other case is analogous). In the same way we get $D_{s'}$ from $D_s$. We have that the selected component of $D_s$ bounds on a side a disk with $h_1+h_2$ holes, $R_1$, and on the other one a disk with $k\geq 2$ holes, $R_2$, while those two components of $D_{s'}$ together bound an internal region with $k+1$ holes, $R'_2$, and two different punctured disks with $h_1>0$ and $h_2$ holes, respectively $R'_{1,1}$ and $R'_{1,2}$. Again, let $\xi_0$ and $\xi'_0$ be the admissible colorings of $X_s$ and $X_{s'}$ that extend the boundary coloring and assign only the colors $0$ and $1$ (see Lemma~\ref{lem:psi}). We have two cases:
\begin{enumerate}
\item{$\xi_0(R_1)=\xi'_0(R'_{1,1})= \xi'_0(R'_{1,2}) =0$ and $\xi_0(R_2)=\xi'_0(R'_2)=1$;}
\item{$\xi_0(R_1)=\xi'_0(R'_{1,1})= \xi'_0(R'_{1,2}) =1$ and $\xi_0(R_2)=\xi'_0(R'_2)=0$.}
\end{enumerate}
Since $\psi(s)= \sum_{\bar R} \chi(\bar R)\xi_0(\bar R)$ and $\psi(s') = \sum_{\bar R} \chi(\bar R) \xi'_0(\bar R)$ we get $\psi(s')=\psi(s) -1$ in the first case, and $\psi(s')=\psi(s)$ in the second one. Therefore in each case $M(s)-M(s')\geq 0$ and $m(s)- m(s') \geq 0$.

Given a state $s$ we can connect it to $s_+$ finding a sequence of states $s_+=s_0 , s_1 \ldots , s_k =s$ such that $s_r$ differs from $s_{r+1}$ only in a crossing and $\sum_i s_r(i) = \sum_i s_{r+1}(i) +2$. Analogously for $s_-$. Hence $M(s) \leq M(s_+)$ and $m(s) \geq m(s_-)$. Thus we have the first statement. 

If the diagram is plus-adequate (resp. minus-adequate) the fourth type of case (Fig.~\ref{figure:cases_lem2}) is not allowed and it holds $M(s_+) - M(s)=4$ ($m(s) - m(s_-) =4$) for each $s$ such that $\sum_s s(i)= n-2$ ($= 2-n$). Namely we have a strict inequality already from the beginning of the sequence $s_0,\ldots, s_k$. This proves the second statement.
\end{proof}
\end{lem}

\begin{figure}[htbp]
\begin{center}
\includegraphics[scale=0.4]{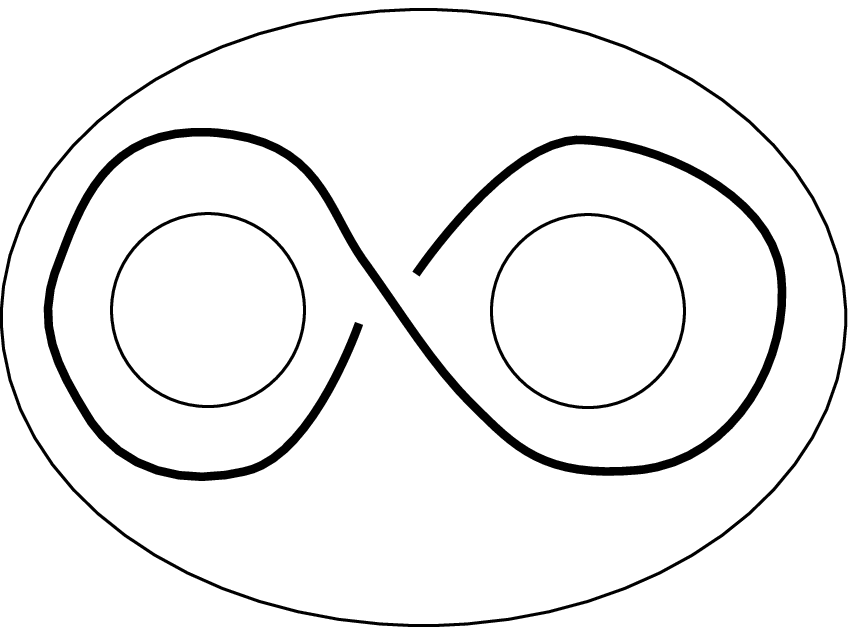}
\hspace{0.5cm}
\includegraphics[scale=0.4]{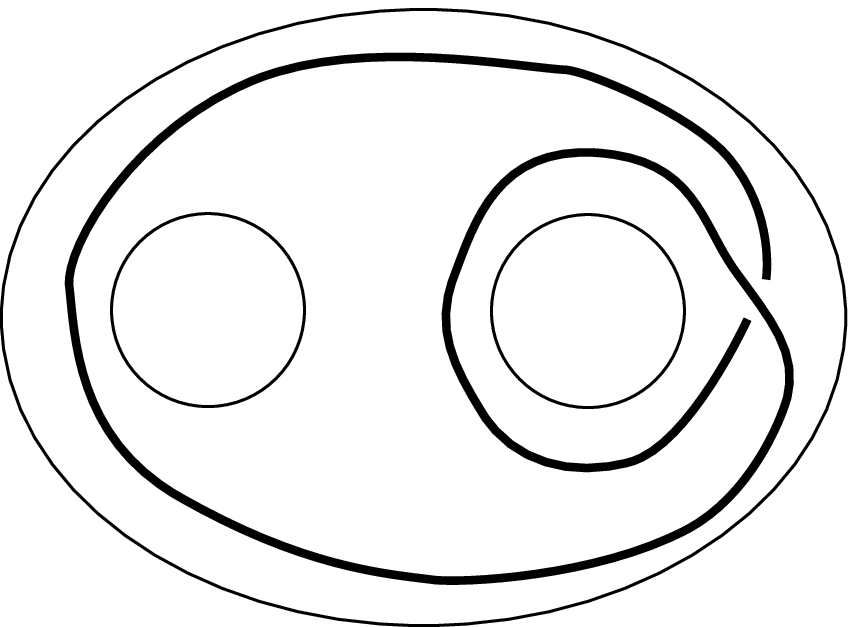}
\end{center}
\caption{The fourth type of case of the proof of Lemma~\ref{lem:ineq1_g}.}
\label{figure:cases_lem2}
\end{figure}

\begin{lem}[\cite{Carrega_Taitg}]\label{lem:ineq2_g}
Let $D\subset S_{(g)}$ be a $n$-crossing, connected diagram with $g(D)=g$ of a $\Z_2$-homologically trivial link in $\#_g(S^1\times S^2)$. Then
$$
s_+ D + s_- D \leq n + 1-g  .
$$
\begin{proof}
We proceed by induction on $n$. We have $g \leq n $, and, up to homeomorphism of $S_{(g)}$, there are exactly $2^g$ diagrams with these characteristics (connected, $g(D)=g$, and representing a $\Z_2$-homologically trivial link) and with $g$ crossings: the ones described in Fig.~\ref{figure:n=g}. Now we prove the statement for the diagrams with $g$ crossings (the base for the induction on $n$) by induction on $g$. 

For $g=0$ we have just a homotopically trivial circle and its Kauffman bracket is $-A^2-A^{-2}$, hence the statement holds. Suppose that the statement is true for such diagrams in $S_{(g-1)}$ with $g-1$ crossings. We split the crossing on the right of Fig.~\ref{figure:n=g}. The result is a diagram obtained taking a one of those diagrams $D'$ in $S_{(g-1)}$ and either adding an hole in the external region surrounding the hole on the right (according to the figure) of $S_{(g-1)}$, or adding a hole in the internal region and surrounding it with a circle. We did not add any homotopically trivial components, hence
$$
s_+D +s_-D =  s_+D' +s_-D' \leq 1 .
$$

Now suppose that the statement is true for all diagrams in $S_{(g)}$ with less than $n>g$ crossings. We claim, and prove later, that there is a crossing $j$ of $D$ and a splitting $D'$ of $j$ with the same characteristics of $D$ (connected and $g(D')=g$). We can suppose that the splitting of $j$ is done in the positive way (the negative way is analogous), hence $s_+D' = s_+D$ and $s_-D -1 \leq s_-D' \leq s_-D + 1$. Therefore
\beq
s_+D + s_-D & \leq  & s_+D' + s_- D' +1 \\
& \leq & (n-1)+1-g +1 \\
& = & n+1-g .
\eeq

Finally we prove the claim. Since $D$ is connected every crossing has a splitting that is still connected. If there is a crossing not adjacent to two external regions, every splitting $D'$ of the crossing satisfies $g(D')=g$. Since $n>g$, we get the previous case if $D$ has no more than $g$ crossings adjacent to two external regions. If there are at least $g+1$ crossings adjacent to two external regions, we take as $j$ one of them, its connected splitting $D'$  satisfies $g(D')=g$. 
\end{proof}
\end{lem}

\begin{figure}[htbp]
\begin{center}
\includegraphics[width=7cm]{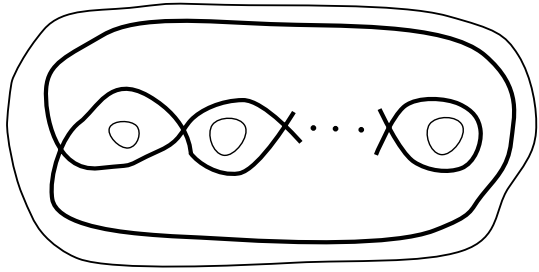}
\end{center}
\caption{Up to homeomorphism of $S_{(g)}$ all the link diagrams $D$ of $\Z_2$-homologically trivial links in $\#_g(S^1\times S^2)$ with $g$ crossings and $g(D)=g$ are obtained choosing the over/underpass for every crossing of the graph in figure.}
\label{figure:n=g}
\end{figure}

\begin{lem}[\cite{Carrega_Taitg}]\label{lem:alter_eq_g}
Let $D\subset S_{(g)}$ be a $n$-crossing alternating connected link diagram with $g(D)= g$ that represents a $\Z_2$-homologically trivial link in $\#_g(S^1\times S^2)$. Then 
$$
\ord_\infty \langle D|s_+ \rangle - \ord_0 \langle D|s_- \rangle = 4n +4 -4g .
$$
\begin{proof}
The number of edges of $D$ as a planar graph is $2n$, hence with a computation of Euler characteristic of $S_{(g)}$ we get that the sum of the Euler characteristics of the regions is $n+1-g$. All the internal regions are disks and there are $g+1$ external regions that are annuli. Thus there are $n+1-g$ internal regions. The diagram $D_{s_\pm}$ has a region that is a disk with $g$ holes, and the others are annuli. Thus
$$
\langle D|s_\pm \rangle = A^{\pm n} (-A^2 -A^{-2})^{s_\pm D + 1 - g} .
$$
The regions of $D_{s_\mp}$ are all disks. Thus
$$
\langle D|s_\mp \rangle = A^{\mp n} (-A^2 -A^{-2})^{s_\mp D}  .
$$
Since the link is alternating, the number of internal regions is equal to $s_+D+s_-D$. Therefore
\beq
\ord_\infty \langle D|s_+ \rangle - \ord_0 \langle D|s_- \rangle & = &  2(n + s_+ D + s_- D + 1 - g) \\
& = & 2(n+n+2-p(s_+ )- p(s_-) +1 -g) \\
& = & 4n + 4 -4g .
\eeq
\end{proof}
\end{lem}

\begin{proof}[Proof of Theorem~\ref{theorem:Tait_conj_g}]
Every diagram $\bar D\subset S_{(g)}$ of $L$ for any e-shadow must be connected and with $g(\bar D)=g$. By Proposition~\ref{prop:reduced_D_g}, Lemma~\ref{lem:ineq1_g}, Lemma~\ref{lem:ineq2_g} and Lemma~\ref{lem:alter_eq_g} we get
\beq
4n(D)+4 -4g & = & B(\langle D \rangle) \\
& = & B(\langle D' \rangle) \\
& \leq & 4n(D') +2 -2g +2(\psi(s_+,D') +\psi(s_-,D')) \\
& \leq & 4n(D') +2 -2g ,
\eeq
where $\psi(s_\pm,D')$ is the quantity $\psi(s_\pm)$ related to the diagram $D'$. Therefore $n(D)\leq n(D') + (g-1)/2$. If $g\leq 2$, $n(D)\leq n(D') +\frac 1 2$, but $n(D)$ and $n(D')$ are integers, hence $n(D)\leq n(D')$. 
\end{proof}

\begin{proof}[Proof of Theorem~\ref{theorem:Tait_conj_Jones_g}]
We have already proved the case $k=0$ in Lemma~\ref{lem:alter_eq_g}. Suppose $k>0$. We proceed by induction on $g$. We apply the second identity of Fig.~\ref{figure:sphere} (with $a=b=1$) near one such crossing. We get $\langle D \rangle= \langle D' \rangle /(-A^2-A^{-2})$, where $D'\subset S_{(g)}$ is a $n$-crossing, connected, alternating diagram with $g(D')=g-1$ and $0<h\leq k$ crossings adjacent twice to a region. We have that $\langle D' \rangle = \langle D''\rangle$, where $D''$ is a link diiaigram in $S_{(g-1)}$ with the same characteristics of $D'$. Moreover we can easily get a diagram $D'''\subset S_{(g-1)}$ representing the same link of $D''$ and is connected, alternating, with $g(D''')=g-1$ and has $n-h$ crossings. Therefore by the inductive hypothesis 
\beq
B(\langle D \rangle) & = & B(\langle D' \rangle) -4 \\
& = & B(\langle D''' \rangle) -4 \\
& = & 4(n-h) +4 - 4(g-1)-4 (k-h)-4 \\
& = & 4n +4 -4k .
\eeq
\end{proof}

\section{Open questions}
In this final section we ask some open questions whose solution in the positive would produce a sharper, or more complete, result than Theorem~\ref{theorem:Tait_conj_g}, and we provide some information that
suggests that the natural extension of the Tait conjecture in $\#_g(S^1\times S^2)$ could be false for $g\geq 3$. 

The first obvious question is the natural extension of the Tait conjecture:
\begin{quest}\label{quest:Tait_conj_g}
Fix an e-shadow of $\#_g(S^1\times S^2)$. Let $D\subset S_{(g)}$ be a connected, alternating, simple diagram of a $\Z_2$-homologically trivial link $L\subset \#_g(S^1\times S^2)$ that is non H-split and with homotopic genus $g$. Is the number of crossings of $D$ equal to the crossing number of $L$?
\end{quest}

The following two questions aim to remove from Theorem~\ref{theorem:Tait_conj_g} the hypothesis of ``homotopic genus $g$'' and ``non H-split''.

\begin{quest}
Let $D\subset S_{(g)}$ be an alternating connected link diagram. Once an e-shadow of $\#_g(S^1\times S^2)$ is fixed, is the link represented by $D$ non H-split?
\end{quest}

\begin{quest}
Let $D\subset S_{(g)}$ be a connected, alternating link diagram. Is the homotopic genus of the links represented by $D$ equal to $g(D)$?
\end{quest}

\begin{rem}\label{rem:ineq_psi}
If for every connected diagram $D\subset S_{(g)}$ with $g(D)=g$ and representing a $\Z_2$-homologically trivial link in $\#_g(S^1\times S^2)$, we had
$$
\psi(s_+) + \psi(s_-) \leq  2-g ,
$$
we could improve Theorem~\ref{theorem:Tait_conj_g} and answer positively to Question~\ref{quest:Tait_conj_g}. Unfortunately it is not true:  a diagram $D\subset S_{(3)}$ is shown in Example~\ref{ex:Kauff}-$(7)$, $D$ is simple, connected, with $g(D)=3$, represents a $\Z_2$-homologically trivial link, and it satisfies $\langle D | s_+ \rangle = A^7(-A^2-A^{-2})$, $\langle D|s_-\rangle = A^{-7}$, hence $\psi(s_+)=\psi(s_-)=0 >2-g$.
\end{rem}

In Remark~\ref{rem:ineq_psi} we did not require that the diagrams represent alternating links, and we know that the shown example represents a non alternating link (see Example~\ref{ex:no_alt_g}). So, it is natural to ask the following: 

\begin{quest}
Let $D\subset S_{(g)}$ be a connected (maybe not alternating) diagram with $g(D)=g$ that represents an alternating $\Z_2$-homologically trivial link in $\#_g(S^1\times S^2)$. Is it true that 
$$
\psi(s_+) + \psi(s_-) \leq 2-g \ \ ?
$$
\end{quest}

\chapter{A generalization of Eisermann's theorem}\label{chapter:Eisermann}

We introduced the notions of ``ribbon surface'', ``ribbon link'', ``slice link'', \ldots (Section~\ref{sec:slice-ribbon}). 

Eisermann showed that the Jones polynomial of a $n$-component ribbon link $L\subset S^3$ is divided by the Jones polynomial of the trivial $n$-component link (see Section~\ref{sec:Eisermann} and \cite{Eisermann}). In this chapter we improve this theorem by extending its range of application from links in $S^3$ to colored knotted trivalent graphs in $\#_g(S^1\times S^2)$, the connected sum of $g\geq 0$ copies of $S^1\times S^2$. We follow \cite{Carrega-Martelli}.

We show in particular that if the Kauffman bracket of a knot in $\#_g(S^1\times S^2)$ has a pole in $q=A^2=i$ of order $n$, the ribbon genus of the knot is at least $\frac {n+1}2$. 

We prove these estimates using Turaev's shadows (see Chapter~\ref{chapter:shadows} and Section~\ref{sec:sh_for_br}). We need to extend the notion of ``multiplicity'' in a point $q_0\in \mathbb{C}$ as a zero and we call it \emph{order} (Definition~\ref{defn:order}). In this chapter we also provide more lower and upper bounds of the order at $q=A^2=i$. 

Throughout this chapter we use the variable $q=A^2$ instead of the variable $A$.

\section{Statement}

The Kauffman bracket of a link in $S^3$ is just a Laurent polynomial. We now consider also links in $\#_g(S^1\times S^2)$ and framed trivalent colored graphs. In general the Kauffman bracket is a rational function and may not be a Laurent polynomial, in particular it may have poles in  $q=A^2=i$ (see Example~\ref{ex:Kauff}). Hence we can consider also the order of a pole and not just the multiplicity as a zero.

\begin{defn}\label{defn:order}
Given a meromorphic function $f$ defined in a neighborhood of $q_0\in \mathbb{C}$, we denote by
$$
\ord_{q_0}f\in \Z \cup\{+\infty\}
$$
the maximum integer $k$ such that $f(q)/(q-q_0)^{k-1}$ vanishes in $q_0$ and we call it the \emph{order} at $q_0$. If $\ord_{q_0}f = + \infty$ the function $f$ vanishes in a neighborhood of $q_0$, otherwise it has a Laurent expansion
$$
f(q) =\lambda (q-q_0)^{\ord_{q_0}f} + {\rm o}\big((q-q_0)^{\ord_{q_0}f}\big)
$$
for some $\lambda \neq 0$. This number is equal to the multiplicity of $f$ in $q_0$ as a zero, or to minus the order of $f$ in $q_0$ as a pole.
\end{defn}

We are interested in the order of the Kauffman bracket at $q=A^2=i$.

\begin{theo}[\cite{Carrega-Martelli}]\label{theorem:Eisermann_gen}
Let $L\subset \#_g(S^1\times S^2)$ be a link and $S\rightarrow \#_g(S^1\times S^2)$ a ribbon surface bounded by $L$. Then
$$
\ord_i \langle L \rangle \geq \chi(S).
$$
\begin{proof}
It follows directly from the more general Theorem~\ref{theorem:Eisermann_gen_gr}.
\end{proof}
\end{theo}

Clearly if we consider links in $S^3$ ($g=0$) we get Eiserman's theorem (Theorem~\ref{theorem:Eisermann}).

An interesting corollary is of course the following:
\begin{cor}[\cite{Carrega-Martelli}]\label{cor:lower_bound_ribbon_link}
If $L\subset \#_g(S^1\times S^2)$ is a $n$-component ribbon link then $\langle L \rangle$ vanishes at least $n$ times at $q=i$:
$$
\ord_i \langle L \rangle \geq n .
$$
\end{cor}

In Section~\ref{sec:Eisermann} we saw that the order at $q=A^2=i$ of the Kauffman bracket does not give any information about knots in $S^3$. Theorem~\ref{theorem:Eisermann_gen} is potentially stronger in $\#_g(S^1\times S^2)$ than in $S^3$ because now $\ord_i\langle L \rangle$ can be a negative number. In particular it provides non trivial information also for knots, as the following example shows:

\begin{ex}\label{ex:nodo}
The framed knot $K\subset \#_g(S^2\times S^1)$ drawn in Fig.~\ref{figure:knot_ribbon} has 
$$
\langle K \rangle =  \left. (-1)^{1-g} q^{-\frac{3g}2}\frac{(1+ q^2+q^4+q^6)^g}{(q+q^{-1})^{2g-1}} \right|_{q=A^2}
$$
and hence $\ord_i\langle K \rangle = g-(2g-1) = 1-g$. Therefore $K$ bounds no ribbon surface $S$ with $\chi(S) >1-g$. In particular, it is not a ribbon knot.
\end{ex}

\begin{figure}[htbp]
\begin{center}
\includegraphics[width = 9 cm]{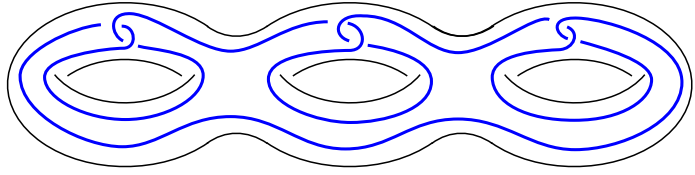}
\caption{A knot $K$ in $\#_g(S^1\times S^2)$. To get $\#_g(S^1\times S^2)$ simply double the handlebody in the picture. We draw here the case $g=3$, the general case is obvious from the picture. Note that the knot is null-homotopic.}
\label{figure:knot_ribbon}
\end{center}
\end{figure}

In order to be ribbon, a link $L\subset \#_g(S^1\times S^2)$ must be homotopically trivial, namely there is a continuous map $(S^1 \sqcup \ldots \sqcup S^1) \times [-1,1] \rightarrow \#_g(S^1\times S^2)$ such that the restriction $(S^1 \sqcup \ldots \sqcup S^1) \times \{1\} \rightarrow \#_g(S^1\times S^2)$ is the embedding of $L$ and the image of the restriction $(S^1 \sqcup \ldots \sqcup S^1) \times \{-1\} \rightarrow \#_g(S^1\times S^2)$ is a finite set of points. This condition can be easily checked looking at the diagrams in the punctured disk $S_{(g)}$ and the knot in Fig.~\ref{figure:knot_ribbon} satisfies it. In Example~\ref{ex:Kauff} and Chapter~\ref{chapter:table} we can find more examples of links and knots in $\#_g(S^1\times S^2)$ together with their Kauffman bracket. Throughout these examples there are links and knots that are homotopically trivial, and by using Theorem~\ref{theorem:Eisermann_gen} we can conclude that they are not ribbon. In particular looking at Chapter~\ref{chapter:table}, we can conclude that there are no ribbon (and non trivial) links in $S^1\times S^2$ with crossing number at most $3$.

\begin{quest}\label{quest:slice-ribbon_Eis_g}
Is there a $k$-component slice link $L \subset \#_g(S^1\times S^2)$ such that
$$
\ord_i \langle L \rangle < k \ \ ?
$$
\end{quest}

By Corollary~\ref{cor:lower_bound_ribbon_link} a positive answer to Question~\ref{quest:slice-ribbon_Eis_g} would imply that the slice-ribbon conjecture is false for $\#_g(S^1\times S^2)$. In Subsection~\ref{subsec:Alex_slice} we saw that the answer is ``no'' for the case $g=0$ and $k=1,2$. 

The \emph{ribbon genus} of a knot $K$ is the minimum genus of an orientable connected ribbon surface $S$ with $\partial S = K$. As a consequence of Theorem~\ref{theorem:Eisermann_gen}, the ribbon genus of the knot $K$ shown in Fig.~\ref{figure:knot_ribbon} is at least $\frac {g}2$. In this more general setting we get:

\begin{cor}[\cite{Carrega-Martelli}]
Let $K\subset \#_g(S^1\times S^2)$ be a knot. If $\langle K \rangle$ has a pole at $q=i$ of order $n>0$, the ribbon genus of $K$ is at least $\frac {n+1}2$.
\end{cor}

We consider also framed colored graphs getting Theorem~\ref{theorem:Eisermann_gen_gr}. Theorem~\ref{theorem:Eisermann_gen} follows immediately by Theorem~\ref{theorem:Eisermann_gen_gr}. We need the following new definition:

\begin{defn}
Let $(a,b,c)$ be an admissible triple. The following integers are called \emph{angles} of the triple:
$$
\frac{a+b-c}2, \ \frac{b+c-a}2, \ \frac{c+a-b}2.
$$

We say that the triple $(a,b,c)$ is \emph{red} if at least two of the three angles are odd numbers. 

Let $G$ be a colored framed knotted trivalent graph. A vertex of $G$ is \emph{red} if the colors of the three incident edges form a red triple.

The edges in $G$ having an odd color form a sub-link $L\subset G$ called the \emph{odd sub-link}.

We denote by $r(G)$ the number of red vertices of a graph $G$. If there is no ambiguity, we just use $r$.
\end{defn}

Clearly if the odd sub-link $L$ of a colored trivalent graph $G$ coincides with the whole graph, $L=G$, there are no vertices, in particular there are no red vertices.

The following is the main theorem:
\begin{theo}[\cite{Carrega-Martelli}]\label{theorem:Eisermann_gen_gr}
Let $G\subset \#_g(S^1\times S^2)$ be a framed trivalent colored graph and $L$ the odd sub-link of $G$. Let $S\rightarrow \#_g(S^1\times S^2)$ be a ribbon surface bounded by $L$. Then
$$
\ord_i \langle G \rangle \geq \chi(S) - \frac r 2
$$
where $\chi(S)$ is the Euler characteristic of $S$ and $r$ is the number of red vertices in $G$.
\begin{proof}
By Theorem~\ref{theorem:main_topol} there is a shadow $X$ of $(G,\#_g(S^1\times S^2))$ that contains $S$ and collapses onto a graph. Theorem~\ref{theorem:main_quantum} with the shadow $X$ implies that the order of $\langle G \rangle$ in $q=i$ is at least $\chi(S) - \frac r 2$.
\end{proof}
\end{theo}

The proof of Theorem~\ref{theorem:Eisermann_gen_gr} splits into two parts: the topological Theorem~\ref{theorem:main_topol}, and the more technical Theorem~\ref{theorem:main_quantum}, and we will see them in the following two sections. 

While the topological side of the story is a one-page proof, the technical part needs a long case-by-case analysis that we would have never pursued if we were not aware of Eisermann's Theorem.

\begin{quest}\label{quest:silce-ribbon_Eis_gr}
Is there a colored trivalent graph $G \subset \#_g(S^1\times S^2)$ such that the odd sub-link $L\subset G$ is a $k$-component slice link and
$$
\ord_i \langle G \rangle + \frac r 2 < k \ \ ?
$$
Here $r$ is the number of red vertices of $G$.
\end{quest}

A positive answer to Question~\ref{quest:silce-ribbon_Eis_gr} would imply that the slice ribbon conjecture is false for $\#_g(S^1\times S^2)$. This
question is open also in the $g=0$ case, that is in $S^3$.

\subsection{More manifolds and RTW}

In this subsection we just ask some questions.

The notion of ``ribbon surface'' applies to any kind of 3-manifold $M$, but the Jones polynomial does not. To define $\langle L \rangle$ as a rational function we need the skein space $K(M)$ to be 1-dimensional.

\begin{quest}
Let $L$ be a link in a closed 3-manifold $M$ with skein vector space $K(M)$ generated by the empty set. Let $S\rightarrow M$ be a ribbon surface bounded by $L$. Is it always true that
$$
\ord_i \langle L \rangle \geq \chi(S) \ \ ?
$$
\end{quest}

When $K(M)$ is not 1-dimensional, quantum invariants survive only at roots of unity: these are the well known Reshetikhin-Turaev-Witten invariants. As we saw these invariants can also be calculated using shadows, so it might be that some of the techniques used for the Jones polynomial extend to that context:

\begin{quest}
Can we relate the ribbon genus of a link to Reshetikhin-Turaev-Witten invariants?
\end{quest}

For instance we can define:
\begin{defn}
The \emph{order} $\ord_i I_r(M,G)$ at $q=A^2=i$ of the Reshetikhin-Turaev-Witten invariant of $(M,G)$ is the maximum integer $k$ such that
$$
\lim_{r \rightarrow \infty} \frac{I_r(M,G)}{(\zeta_r - i)^{k-1}} = 0 ,
$$
where $\zeta_r$ is a $2r^{\rm th}$ root of unity such that $(\zeta_r)^n\neq 1$ for $0<n<r$, $\zeta_r \rightarrow i$ for $r \rightarrow \infty$, and $I_r(M,G)$ is based on the evaluation at $q=A^2=\zeta_r$ (for instance $\zeta_r= \exp(\pi i((\lfloor \frac{r}{2}\rfloor +1)/r))$, where $\lfloor x \rfloor$ is the integer part of $x$).
\end{defn}

\begin{prop}
$$
\ord_i I_r(\#_g(S^1\times S^2),G) = \ord_i \langle G \rangle .
$$
\begin{proof}
It follows from the shadow formulas: Theorem~\ref{theorem:sh_for_br} and Theorem~\ref{theorem:sh_for_RTW}.
\end{proof}
\end{prop}

Therefore we can shift the problem of the order at $q=i$ of the Jones polynomial to the one of the Reshetikhin-Turaev-Witten invariants.
\begin{quest}
Can this approach lead somewhere?
\end{quest}

Furthermore we note that we get the same object even if we use the square of the module of the Reshetikhin-Turaev-Witten invariant $|I_r(M,G)|^2$ (and $|\zeta_r - i|^{2(k-1)}$). Hence maybe we can also shift the problem to the order at $q=i$ of the Turaev-Viro invariants of a certain pair $(M',G')$. We have that the Turaev-Viro invariant of a pair $(N,H)$, where $N$ is a compact orientable 3-manifold and $H$ is a colored trivalent graph on the boundary of $N$, is equal to the Reshetikhin-Turaev-Witten invariant of $(\tilde{N}, \tilde{H})$ where $\tilde{N}$ is the double of $N$ and $\tilde{H}$ is the image of $H$ under the standard inclusion of $N$ in $\tilde{N}$. If $N$ has no boundary we have that the Turaev-Viro invariant of $N$ is equal to the square of the absolute value of the Reshetikhin-Turaev-Witten invariant $|I_r(N)|^2$ (Theorem~\ref{theorem:Turaev-Viro}).

\begin{quest}
Can we shift the problem of the order at $q=i$ of the Jones polynomial of $(M,G)$ also to the one of the Turaev-Viro invariants of a certain pair $(M',G')$?
\end{quest}

\section{The topological part of the proof}

\subsection{The result}

Let $W_g$ be the 4-dimensional orientable handlebody of genus $g\geq 0$ (the orientable compact 4-manifold with a handle-decomposition with $k$ 0-handles and $k+g-1$ 1-handles). We recall that a ribbon surface $S\rightarrow \#_g(S^1\times S^2)= \partial W_g$ is equivalent to a properly embedded surface in $W_g$ that is in Morse position with respect to the distance function from a graph $\Gamma \subset W_g$ and has just points of minima and saddles (no maxima) (see Remark~\ref{rem:ribbon_g}). 

The following purely topological fact is one of the two results that lead to Theorem~\ref{theorem:Eisermann_gen_gr}:
\begin{theo}[\cite{Carrega-Martelli}] \label{theorem:main_topol}
Let $G$ be a knotted trivalent graph in $\#_g(S^1\times S^2)$ and $L$ be a sub-link of $G$. Every ribbon surface $S\subset W_g$ bounded by $L$ is contained in some shadow $X$ of $W_g$ such that $X$ collapses onto a graph and $\partial X = G$.
\end{theo}

We single out a couple of examples:

\begin{ex}
Consider the trivially embedded annulus $S\subset D^4$ as in Fig.~\ref{figure:shadow_annulus_0}. A shadow $X$ containing $S$ is constructed by attaching a disk $D$ to its core. Note that indeed $D^4$ collapses onto $X$ that collapses onto a point.
\end{ex}

\begin{figure}
\begin{center}
\includegraphics[width = 7 cm]{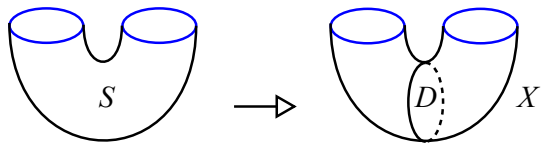}
\caption{If $S\subset D^4$ is the trivially embedded annulus bounding the unlink $\partial S$, a shadow $X= S \cup D$ is obtained by attaching a disk $D$ to its core.}
\label{figure:shadow_annulus_0}
\end{center}
\end{figure}

\begin{ex}
The 2-disk $D^2$ is properly embedded in the 3-disk and hence in $S^3$. The disk is itself a shadow of $D^4$. However, a non trivial ribbon disk $D\subset  D^4$ is not a shadow: $D^4$ does not collapse onto $D$ (see Proposition~\ref{prop:trivial}). The shadow containing $D$ may be rather complicate.
\end{ex}

\begin{ex}
The disk with $g\geq 0$ holes $S_{(g)}$ is a properly embedded surface in the 3-dimensional handlebody of genus $g$, hence it is an embedded surface in its double $\#_g(S^1\times S^2)$. The surface is itself a shadow of $W_g$.
\end{ex}

Every ribbon surface $S\rightarrow \#_g(S^1\times S^2)$ can be constructed from a trivalent graph $G_S\subset \#_g(S^1\times S^2)$ as in Fig.~\ref{figure:construct_ribbon}, consisting of some disjoint circles and some arcs connecting them in space.

\begin{figure}
\begin{center}
\includegraphics[width = 11.5 cm]{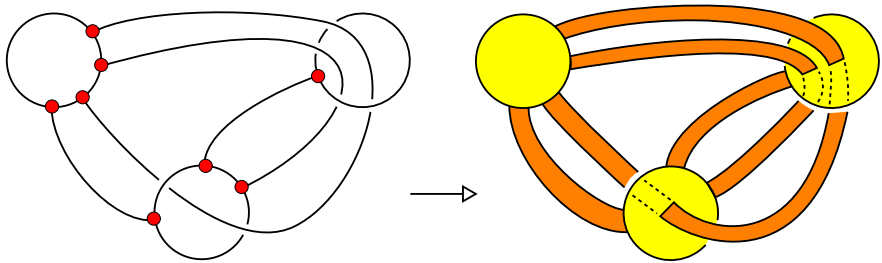}
\caption{Every ribbon surface can be constructed from a planar diagram with some disjoint circles representing the minima and some edges connecting them representing the saddles (left). The surface is obtained by filling the circles (yellow) and thickening the edges to (orange) bands.}
\label{figure:construct_ribbon}
\end{center}
\end{figure}

\begin{proof}[Proof of Theorem~\ref{theorem:main_topol}]
For simplicity, first we suppose $L=G$. Construct $S$ from a trivalent graph $G_S \subset \#_g(S^1\times S^2)$ as in Fig.~\ref{figure:construct_ribbon}-(left) and let $D_S \subset S_{(g)}$ be a diagram of $G_S$. Via Reidemeister moves we may suppose that there is a smallest surface in $S_{(g)}$ that is homeomorphic to $S_{(g)}$ and contains $D_S$. We construct a shadow $X$ for $G_S$ using the algorithm described in the proof of Proposition~\ref{prop:shadow2} using $D_S$.

Note that $X$ contains the yellow disks of Fig.~\ref{figure:construct_ribbon}-(right). To complete the construction, we simply add to $X$ the orange bands shown in Fig.~\ref{figure:construct_ribbon}-(right), and then push their interior a bit inside the 4-dimensional handlebody $W_g$. We end up with a shadow $X$ containing the whole of $S$ and with $\partial X = \partial S=L$.

Now we consider the general case $L \subset G$. Let $G_S$ be a graph defining $S$. Up to isotopies we may suppose that $G\setminus L$ is attached to $L$ only at the circles that form $G_S$. Hence the graph $G_S$ can be easily extended to a trivalent graph $G'$ such that $G_S\subset G' \subset \#_g(S^1\times S^2)$ and after the thickening of the orange bands of $G_S$ we get the graph $G$. We proceed like in the previous case to get a shadow of the trivalent graph $G'$ from a diagram in $S_{(g)}$. Then we add the orange bands and get a shadow of $G$.
\end{proof}

\begin{figure}[htbp]
\begin{center}
\includegraphics[width = 10 cm]{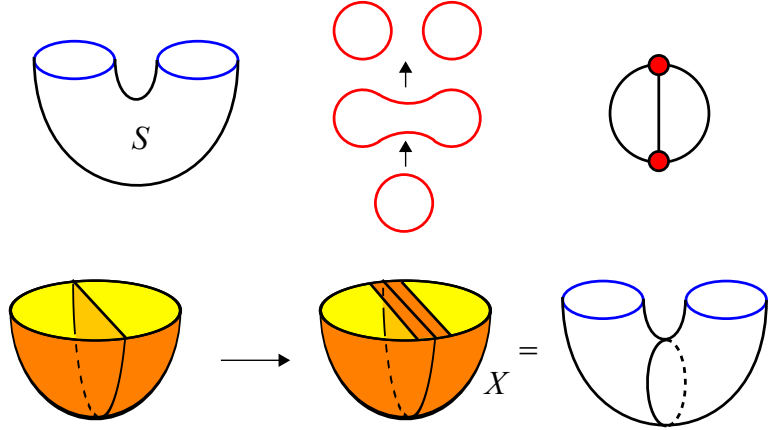}
\caption{How to build a shadow $X$ containing a given ribbon surface $S$. We show here the construction for the ribbon annulus $S$.}
\label{figure:shadow_annulus}
\end{center}
\end{figure}

\begin{ex}\label{ex:annulus}
Fig.~\ref{figure:shadow_annulus} illustrates the construction in a simple case. The ribbon surface $S\subset D^4$ is a trivially embedded annulus with boundary $L=\partial S$ the unlink with two components; the annulus $S$ in Morse position has one minimum and one saddle, and it is hence a ribbon surface constructed from the graph $G$ shown in Fig.~\ref{figure:shadow_annulus}-(top-right): a circle (the minimum) with a diameter (encoding the saddle). A shadow for $G$ is shown in Fig.~\ref{figure:shadow_annulus}-(bottom-left). By adding a band we obtain a shadow $X$ for $L$ containing $S$, and $X$ is just $S$ with a disk attached to its core. Note that indeed $D^4$ collapses onto $X$ that collapses to a point.
\end{ex}

\subsection{Non-ribbon surfaces}
One may wonder whether every surface $S$ is contained in a shadow. We now show that this is not true: indeed being contained in a shadow is quite restrictive. 
\begin{defn}
A properly embedded surface $S\subset W_g$ is \emph{homotopically ribbon} if the inclusion 
$$
(\#_g(S^1\times S^2) \setminus \partial S) \hookrightarrow (W_g \setminus S)
$$
induces an epimorphism on fundamental groups
$$
\pi_1(\#_g(S^1\times S^2) \setminus \partial S) \twoheadrightarrow \pi_1(W_g \setminus S) .
$$
\end{defn}

For a general properly embedded surface $S\subset W_g$, the following implications hold:
$$
S {\rm \ ribbon} \Longrightarrow S {\rm\ contained\ in\ a\ shadow} \Longrightarrow S {\rm\ homotopically\ ribbon}.
$$
We have already proved the first implication, so we now turn to the second. 
\begin{prop}[\cite{Carrega-Martelli}]\label{prop:sh_homot_rib}
If $S\subset W_g$ is contained in a shadow $X$ of the 4-dimensional handlebody $W_g$ then it is homotopically ribbon.
\begin{proof}
The shadow $X$ contains $S$ and is hence obtained from $S$ by adding cells of index $0$, $1$, or $2$. Therefore a regular neighborhood $N(X)$ of $X$ is obtained from a regular neighborhood $N(S)$ of $S$ by adding handles of index $0$, $1$, or $2$. Since $W_g$ collapses onto $X$, we can take $N(X)=W_g$.

By turning handles upside-down we get that $W_g\setminus N(S)$ is obtained from a collar of $\#_g(S^1\times S^2)\setminus N(\partial S)$ by adding handles of index $4$, $3$, or $2$. Since there are no 1-handles, the inclusion 
$$
\#_g(S^1\times S^2) \setminus N(\partial S) \hookrightarrow W_g \setminus N(S)
$$
induces a surjection on fundamental groups.
\end{proof}
\end{prop}

\begin{rem}
Note that in Proposition~\ref{prop:sh_homot_rib} we do not require that the shadow $X$ collapses onto a graph.
\end{rem}
We do not know if any of the two implications can be reversed. It is easy to construct some surface $S\subset D^4$ that is not homotopically ribbon, hence such an $S$ cannot be contained in a shadow. The following example is certainly known to experts.

\begin{prop}
The trivial knot bounds some disk that is not homotopically ribbon.
\begin{proof}
Pick a knotted 2-sphere $S\subset S^4$ whose complement has non cyclic fundamental group $G$, for instance a spun knot \cite[Chapter 3.J]{Rolfsen}.

By tubing one such knotted sphere with a trivial properly embedded disk we get a disk $D^2\subset D^4$ such that $\pi_1(D^4\setminus D^2) = G$. Since $\partial D^2$ is the trivial knot, the complement $S^3\setminus \partial D^2$ is a solid torus and has cyclic $\pi_1$. The map 
$$
\pi_1(S^3\setminus \partial D^2) \longrightarrow \pi_1(D^4 \setminus D^2)
$$
cannot be surjective since the left group is cyclic and the right one is not.
\end{proof}
\end{prop}

\begin{quest}\label{quest:homot_rib1}
Is every homotopically ribbon surface of $W_g$ ribbon?
\end{quest}

The requirement that $S$ is contained in some shadow lies between these two properties and breaks this question in two parts:

\begin{quest} \label{quest:homot_rib2}
$\ $
\begin{enumerate}
\item{Is every homotopically ribbon surface in $W_g$ contained in a shadow of $W_g$?}
\item{Is every surface contained in a shadow of $W_g$ ribbon?}
\end{enumerate}
\end{quest}

\section{The more technical part of the proof}

\subsection{The results}

We turn to quantum invariants. By analyzing carefully the shadow formula for the Kauffman bracket we prove all the needed estimates at $q=i$. The main result of this section is Theorem~\ref{theorem:order} from which follows Theorem~\ref{theorem:main_quantum}, which is one of the two ingredients of the proof of Theorem~\ref{theorem:Eisermann_gen_gr}.

\begin{theo}[\cite{Carrega-Martelli}]\label{theorem:main_quantum}
Let $G\subset \#_g(S^1\times S^2)$ be a framed trivalent colored graph and $L$ be the odd sub-link of $G$. Let $X$ be a shadow of the 4-dimensional handlebody $W_g$ that collapses onto a graph and $\partial X= G$. Let $S\subset X$ be a (unique) surface with $\partial S = L$. Then
$$
\ord_i\langle G \rangle \geq \chi(S) - \frac r 2
$$
where $r$ is the number of red vertices in $G$.
\end{theo}

We will need the following:

\begin{prop}[\cite{Carrega-Martelli}]\label{prop:correspondences}
Let $X$ be a shadow collapsing onto a graph of a trivalent knotted graph $G\subset \#_g(S^1\times S^2)$. There are natural 1-1 correspondences:
$$
\left\{ \begin{array}{c} {\rm properly\ embedded} \\ {\rm surfaces\ }S \subset X 
\end{array} \right\}
\longleftrightarrow
H_2(X,G;\Z_2) 
 \text{ and }
H_1(G;\Z_2) 
\longleftrightarrow 
\left\{ \begin{array}{c} {\rm links} \\ L \subset G
\end{array} \right\} ,
$$
and a natural injective homomorphism
$$
H_2(X,G;\Z_2) \longrightarrow H_1(G;\Z_2) .
$$
According to the correspondences the surface $S$ is sent to $L=\partial S$. The empty surface is included. Moreover if $g=0$ the homomorphism is surjective too.
\begin{proof}
The morphism $\partial: H_2(X,G;Z_2) \rightarrow H_1(G;\Z_2)$ is injective because
$X$ collapses onto a graph and hence $H_2(X;\Z_2)=\{e\}$. If $g=0$ we have $H_1(X;\Z_2)=\{e\}$ hence in that case the homomorphism is also surfective. Using cellular homology, every $\Z_2$-homology class in $(X,G)$ is realized by a unique cycle, and that cycle is a surface since $X$ has simple singularities.
\end{proof}
\end{prop}

Let now $\xi$ be an admissible coloring for $X$. Its reduction modulo $2$ is a cycle in $H_2(X, G;\Z_2)$ because the admissibility relation around every interior edge of $X$ reduces to $i+j+k \equiv 0$ (mod $2$). This cycle gives a surface $S_\xi \subset X$ that consists of all regions in $X$ having an odd color: we call $S_\xi$ the \emph{odd surface} of $\xi$.

Proposition~\ref{prop:correspondences} implies the following:

\begin{cor}[\cite{Carrega-Martelli}]\label{cor:unique_sur}
Let $G\subset \#_g(S^1\times S^2)$ be a colored framed knotted trivalent graph and $X$ be any shadow of $G$ that collapses onto a graph. The odd surface $S_\xi\subset X$ of a coloring $\xi$ that extends that of $G$ is the unique surface whose boundary $\partial S_\xi$ is the odd sub-link of $G$. In particular $S_\xi$ does not depend on $\xi$.
\end{cor}

\begin{theo}[\cite{Carrega-Martelli}]\label{theorem:order}
Let $X$ be a shadow colored by $\xi$. We have 
$$
\ord_i \langle X \rangle_\xi \geq \chi(S_\xi) - \frac{r}2 ,
$$
where $r$ is the number of red vertices in $\partial X$.
\end{theo}

In contrast with the topological Theorem~\ref{theorem:main_topol}, this theorem has a long technical proof, to which we devote the rest of this section.

Before starting with the proof we prove Theorem~\ref{theorem:main_quantum}:

\begin{proof}[Proof of Theorem~\ref{theorem:main_quantum}]
We have
$$
\ord_i \langle G \rangle = \ord_i \sum_\xi \langle X \rangle_\xi
 \geq \min_\xi \ord_i \langle X \rangle_\xi 
$$
where $\xi$ runs over all the admissible colorings of $X$ that extends the coloring of $G$. By Theorem~\ref{theorem:order}
$$
\ord_i \langle G \rangle \geq \min_\xi \chi(S_\xi) - \frac r 2 .
$$
By Corollary~\ref{cor:unique_sur} for each $\xi$ the surface $S_\xi$ coincides with $S$.
\end{proof}

\subsection{Localization of Theorem~\ref{theorem:order}}

We now localize the proof of Theorem~\ref{theorem:order}, by reducing it to the building blocks \cerchio, \teta, and \tetra. The following lemma will be proved in the next subsection:

\begin{lem}[\cite{Carrega-Martelli}]\label{lem:block}
Let $G$ be a colored $\cerchio, \teta$, or $\tetra$. 
We have 
$$
\ord_i \langle G \rangle \geq |L| - \frac r 2 
$$
where $L$ is the odd (un-)link $L\subset G$, $|L|$ is the number of components of $L$, and $r$  is the number of red vertices in $G$. If $G=\cerchio$ or $\teta$ then the equality holds.
\end{lem}

Note that for $G=\cerchio, \teta, \tetra$ we have:
\begin{itemize}
\item $|L|=1$ if $G$ contains some odd colored edges,
\item $|L|=0$ otherwise.
\end{itemize}
We postpone the proof of Lemma~\ref{lem:block} to the next subsections, and we now deduce Theorem~\ref{theorem:order} from it.

\begin{proof}[Proof of Theorem~\ref{theorem:order}]
We have
$$
\langle X \rangle_\xi = \frac{\prod_f \cerchio_f^{\chi(f)}q_f \prod_v \tetra_v \prod_{v_\partial} \teta_{v_\partial}}
{\prod_e \teta_e^{\chi(e)} \prod_{e_\partial} \cerchio_{e_\partial}^{\chi(e_\partial)}} .
$$
The phase $q_f$ is a monomial in $q$ and hence does not contribute to $\ord_i\langle X \rangle_\xi$. We get
\beq
\ord_i \langle X \rangle_\xi & = & \sum_f \chi(f)\cdot \ord_i\cerchio_f + \sum_v \ord_i\tetra_v +  \sum_{v_\partial} \ord_i \teta_{v_\partial} \\ 
&  & - \sum_e \chi(e)\cdot \ord_i \teta_e - \sum_{e_\partial} \chi(e_\partial)\cdot\ord_i\cerchio_{e_\partial} .
\eeq
We now use Lemma~\ref{lem:block}. Note that for every colored $\cerchio, \teta, \tetra$ involved, we have $|L|=1$ precisely when the corresponding stratum (vertex, edge, or region) is contained in $S_\xi$, otherwise we get $|L|=0$. We denote by $r(G)$ the number of red vertices in $G$ and we get:
\beq
\ord_i\langle X \rangle_\xi & \geq & \sum_{f\subset S_\xi} \chi(f) + \sum_{v\in S_\xi} 1 +  \sum_{v_\partial \in S_\xi} 1 
- \sum_{e\subset S_\xi} \chi(e) - \sum_{e_\partial\subset S_\xi} \chi(e_\partial) \\
&  & - \sum_v \frac {r(v)}2  
- \sum_{v_\partial}\frac {r(v_\partial)}2 
+ \sum_e \frac {r(e)}2 \\
& = & \ \chi(S_\xi) - \sum_v \frac {r(v)}2 
- \sum_{v_\partial}\frac {r(v_\partial)}2
+ \sum_e \frac {r(e)}2 .
\eeq
Let $e$ be an interior edge. The two vertices of $\teta_e$ are colored by the same triple $(a,b,c)$: hence $\teta_e$ has either zero or two red vertices. If an interior vertex $v$ of $X$ is adjacent to $e$, then $\tetra_v$ has a corresponding vertex colored by $(a,b,c)$. If an exterior vertex $v_\partial$ is adjacent to $e$, then both vertices of $\teta_{v_\partial}$ are colored as $(a,b,c)$. From this we get
$$
\sum_e r(e) = \sum_v r(v) + \sum_{v_\partial} \frac{r(v_\partial)}2
$$
and therefore
$$ 
\ord_i\langle X \rangle_\xi \geq \chi(S_\xi) - \sum_{v_\partial} \frac {r(v_\partial)}4 = \chi(S_\xi) - \frac r2
$$
because $r(v_\partial)$ equals $2$ when $v_\partial$ is red and $0$ otherwise.
\end{proof}

\subsection{Order of generalized multinomials}

It remains to prove Lemma~\ref{lem:block}, and to do so we will need the following:

\begin{prop}[\cite{Carrega-Martelli}]\label{prop:orders}
We have
\beq
\ord_i[n] & = & 
\begin{cases} 
0 & \text{ if } n\in 2\Z+1   \\
1 & \text{ if } n\in 2\Z 
\end{cases} , \\
\ord_i[n]! & = & \big\lfloor \frac n2 \big\rfloor, \\
\ord_i \begin{bmatrix} 
m_1, \ldots, m_h \\ 
n_1, \ldots, n_k 
\end{bmatrix} 
& = & 
\big\lfloor \frac{\#\big\{{\rm odd}\ n_i \big\}}2 \big\rfloor - \big\lfloor\frac{\#\big\{{\rm odd\ } m_j \big\}}2 \big\rfloor 
\eeq
where $\lfloor x \rfloor$ is the integer part of $x$.
\begin{proof}
The function
$$
[n] = \frac {q^n - q^{-n}}{q-q^{-1}} = \frac{q^{-n}}{q-q^{-1}} (q^{2n}-1)$$
has simple zeroes at the $(2n)^{\rm th}$ roots of unity (except $q=\pm 1$), hence at $q=i$ when $n$ is even. The equality $\ord_i[n]! = \lfloor \tfrac n2 \rfloor$ follows. On the multinomial, recall that $m_1+\ldots +m_h = n_1+\ldots+n_h = N$ by hypothesis. We get
\beq
\ord_i \begin{bmatrix} m_1, \ldots, m_h \\ n_1, \ldots n_k \end{bmatrix} 
& = & \sum_i \big\lfloor \frac{m_i}2 \big\rfloor - \sum_j \big\lfloor \frac{n_j}2 \big\rfloor
\\
& = & \big\lfloor\frac N2\big\rfloor - \big\lfloor\frac{\#\big\{{\rm odd\ } m_i\big\}}2\big\rfloor - 
\big\lfloor\frac N2\big\rfloor + \big\lfloor\frac{\#\big\{{\rm odd\ } n_j\big\}}2\big\rfloor \\
& = & \Big\lfloor \frac{\#\big\{{\rm odd}\ n_i \big\}}2\Big\rfloor - \Big\lfloor\frac{\#\big\{{\rm odd\ } m_j \big\}}2 \Big\rfloor.
\eeq
\end{proof}
\end{prop}

We can now evaluate $\cerchio, \teta$, and $\tetra$ at $q=i$. 

\subsection{Orders of the circle, theta, and tetrahedron}
It remains to prove Lemma~\ref{lem:block}.

\begin{proof}[Proof of Lemma~\ref{lem:block}]
If $G=\cerchio$ then
$\ord_i\cerchio_a = \ord_i[a+1]$ equals $1$ if $a$ is odd and $0$ if $a$ is even: the odd link $L$ is respectively $G$ and $\varnothing$, therefore $\ord_i\cerchio_a = |L|$ in any case.

If $G=\teta$ we have
\beq
\ord_i\teta_{a,b,c} & = & \ord_i \begin{bmatrix} \frac{a+b+c}2+1, \frac{a+b-c}2, \frac{b+c-a}2, \frac{c+a-b}2 \\
a, b, c, 1 \end{bmatrix} \\
 & = & \Big\lfloor \frac{\#\big\{{\rm odd}\ a,b,c,1 \big\}}2 \Big\rfloor - 
 \Big\lfloor\frac{\#\big\{{\rm odd\ } \frac{a+b+c}2+1, \frac{a+b-c}2, \frac{b+c-a}2, \frac{c+a-b}2 \big\}}2 \Big\rfloor \\
 & = & |L| - \frac r2.
\eeq
To prove the last equality, note that the first addendum is $0$ if $a,b,c$ are even and $1$ otherwise (there are either zero or two odd numbers in $a,b,c$ by admissibility), and $L\subset G$ is respectively empty or a circle. Concerning the second addendum, note that
$$
\frac{a+b+c}2 +1= \frac{a+b-c}2 + \frac{b+c-a}2 + \frac{c+a-b}2 +1 ,
$$
and hence one easily sees that the second addendum equals
$$
\Big\lfloor\frac{\#\big\{{\rm odd\ } \frac{a+b-c}2, \frac{b+c-a}2, \frac{c+a-b}2 \big\}}2 \Big\rfloor ,
$$
which is $1$ if the triple is red and $0$ otherwise, by definition.

For $G=\tetra$ we do a long case-by-case analysis. We recall the formula
$$
\pic{2}{0.8}{tetra_color.eps} =
\begin{bmatrix} \Box_i-\triangle_j \\
a, b, c, d, e, f \end{bmatrix} \cdot \sum_{z = \max \triangle_j }^{\min \Box_i}\!\!\! (-1)^z
\begin{bmatrix} z+1 \\
z-\triangle_j, \Box_i-z, 1 \end{bmatrix}.
$$
with
\beq
& \triangle_1 = \frac{a+b+c}{2},\ \triangle_2 = \frac{a+e+f}{2},\ \triangle_3 =\frac{ d+b+f}{2},\ \triangle_4 = \frac{d+e+c}{2}, & \\
& \Box_1 = \frac{a+b+d+e}{2},\ \Box_2 = \frac{a+c+d+f}{2},\ \Box_3 = \frac{b+c+e+f}{2}. &
\eeq
Note that
$$
a+b+c+d+e+f = \sum_i \Box_i = \sum_j \triangle_j .
$$
We now estimate the factor
\begin{equation} \label{eqn:factor2}
\sum_{z = \max \triangle_j }^{\min \Box_i} (-1)^z 
\begin{bmatrix} 
z+1 \\
z-\triangle_j, \Box_i-z, 1 
\end{bmatrix}
\end{equation}
in terms of the parity of the $\Box_j$'s and the $\triangle_i$'s.

We first consider the case $a+b+c+d+e+f$ is even. In that case the number of odd $\Box_i$'s is either $0$ or $2$, while the number of odd $\triangle_j$'s is either $0$, $2$, or $4$. Using Proposition~\ref{prop:orders} we easily see that
$$
\ord_i \begin{bmatrix} 
z+1 \\
z-\triangle_j, \Box_i-z, 1 
\end{bmatrix}
$$
is a number that depends on the parity of $z$, on the number of odd $\Box_i$'s ($0$ or $2$) and of odd $\triangle_j$'s ($0$, $2$ or $4$) according to the tables: 

\begin{center}
\begin{tabular}{|c|c|c|}
\hline
\multicolumn{3}{c}{$z$ even} \\
\hline \hline
 & $0\ \Box_i$ & $2\ \Box_i$  \\
\hline
 $0\ \triangle_j$ & $0$ & $1$ \\
 $2\ \triangle_j$ & $1$ & $2$ \\
 $4\ \triangle_j$ & $2$ & $3$ \\
\hline
\end{tabular}
\qquad
\begin{tabular}{|c|c|c|}
\hline
\multicolumn{3}{c}{$z$ odd} \\
\hline \hline
 & $0\ \Box_i$ & $2\ \Box_i$  \\
\hline
 $0\ \triangle_j$ & $4$ & $3$ \\
 $2\ \triangle_j$ & $3$ & $2$ \\
 $4\ \triangle_j$ & $2$ & $1$ \\
\hline
\end{tabular}
\end{center}

By taking the minimum we get that the order at $q=i$ of (\ref{eqn:factor2}) is at least:

\begin{equation} \label{eqn:even}
\begin{tabular}{|c|c|c|} 
\hline
 & $0\ \Box_i$ & $2\ \Box_i$  \\
\hline
 $0\ \triangle_j$ & $0$ & $1$ \\
 $2\ \triangle_j$ & $1$ & $2$ \\
 $4\ \triangle_j$ & $2$ & $1$ \\
\hline
\end{tabular}
\end{equation}
The case $a+b+c+d+e+f$ odd is treated analogously: now the number of odd $\Box_i$'s is either $1$ or $3$, and the number of odd $\triangle_i$'s is either $1$ or $3$. We get
\begin{center}
\begin{tabular}{|c|c|c|}
\hline
\multicolumn{3}{|c|}{$z$ even} \\
\hline \hline
 & $1\ \Box_i$ & $3\ \Box_i$  \\
\hline
 $1\ \triangle_j$ & $1$ & $2$ \\
 $3\ \triangle_j$ & $2$ & $3$ \\
\hline
\end{tabular}
\qquad
\begin{tabular}{|c|c|c|}
\hline
\multicolumn{3}{c}{$z$ odd} \\
\hline \hline
 & $1\ \Box_i$ & $3\ \Box_i$  \\
\hline
 $1\ \triangle_j$ & $3$ & $2$ \\
 $3\ \triangle_j$ & $2$ & $1$ \\
\hline
\end{tabular}
\end{center}
The order at $q=i$ of (\ref{eqn:factor2}) is hence at least:
\begin{equation} \label{eqn:odd}
\begin{tabular}{|c|c|c|} 
\hline
 & $1\ \Box_i$ & $3\ \Box_i$  \\
\hline
 $1\ \triangle_j$ & $1$ & $2$ \\
 $3\ \triangle_j$ & $2$ & $1$ \\
\hline
\end{tabular}
\end{equation}
We now turn to the factor
\begin{equation}\label{eqn:factor1}
\begin{bmatrix} \Box_i-\triangle_j \\
a, b, c, d, e, f \end{bmatrix}.
\end{equation}
The $12$ numbers $\Box_i - \triangle_j$ are the angles of the vertices of the colored \tetra, namely they are of type $\frac{m+n-p}2$ where $(m,n,p)$ are the colors of the edges incident to some vertex: there are $4$ vertices and $3$ such expressions at each vertex. The $12$ numbers correspond to the $12$ red arcs in the picture
\begin{center}
\includegraphics[width = 1.5 cm]{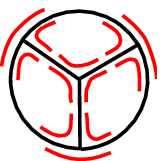}
\end{center}
where the red arc corresponding to $\frac{m+n-p}2$ is the one parallel to the edges $m,n$ and opposite to $p$. The parities of these $12$ numbers determine the parities of all the quantities $\Box_i, \triangle_j, a,b,c,d,e,f$, and hence also $|L|$ and $\frac r2$. The possible configurations (considered up to symmetries of the tetrahedron) are easily classified and are shown in Table~\ref{table:cases} and Table~\ref{table:cases2}.

\begin{table}
\begin{center}
\begin{tabular}{|c|c|c|c|c|c|c|c|}
\hline
\phantom{\Big|}\! odd $\Box_i$'s & odd $\triangle_j$'s & red arcs & $\ord_i\big((\ref{eqn:factor1})\big)$ & $\ord_i\big((\ref{eqn:factor2})\big)$ & $|L|$ & $\frac r2$ & works? \\
\hline \hline
 0 & 0 & \raisebox{-0.65 cm}{\includegraphics[width = 1.8 cm]{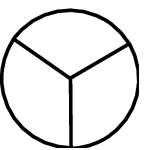}} & 0 & $\geq 0$ & $0$ & $0$ & yes \\
 0 & 2 & \raisebox{-0.65 cm}{\includegraphics[width = 1.8 cm]{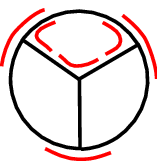}} & $-1$ & $\geq 1$ & $1$  & 1& yes \\
  0 & 4 & \raisebox{-0.65 cm}{\includegraphics[width = 1.8 cm]{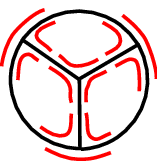}} & $-6$ & $\geq 2$ & $0$ & 2& no \\
  2 & 0 & \raisebox{-0.65 cm}{\includegraphics[width = 1.8 cm]{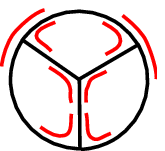}} & $-2$ & $\geq 1$ & $1$  & 2& yes \\
  2 & 2 & \raisebox{-0.65 cm}{\includegraphics[width = 1.8 cm]{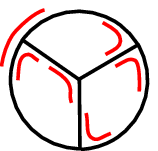}} & $-1$ & $\geq 2$ & $1$  & 1& yes \\
  2 & 2 & \raisebox{-0.65 cm}{\includegraphics[width = 1.8 cm]{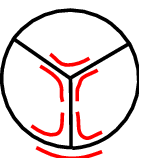}} & $-3$ & $\geq 2$ & $0$  & 1& yes \\
  2 & 4 & \raisebox{-0.65 cm}{\includegraphics[width = 1.8 cm]{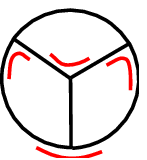}} & $0$ & $\geq 1$ & $1$  & 0 & yes \\
\hline
\end{tabular}
\caption{For each case: the number of odd $\Box_i$'s, of odd $\triangle_j$'s, the red arcs, the order of the first factor (\ref{eqn:factor1}), of the second (\ref{eqn:factor2}) estimated in (\ref{eqn:even}), the number of components of the odd link $|L|$, and $\frac r2$. If 
(\ref{eqn:factor1}) + (\ref{eqn:factor2}) $\geq |L|-\frac r2$ then the estimate works (last column).}
\label{table:cases}
\end{center}
\end{table}

\begin{table}
\begin{center}
\begin{tabular}{|c|c|c|c|c|c|c|c|}
\hline
\phantom{\Big|}\! odd $\Box_i$'s & odd $\triangle_j$'s & red arcs & $\ord_i\big((\ref{eqn:factor1})\big)$ & $\ord_i\big((\ref{eqn:factor2})\big)$ & $|L|$ & $\frac r2$ & works? \\
\hline \hline
 1 & 1 & \raisebox{-0.65 cm}{\includegraphics[width = 1.8 cm]{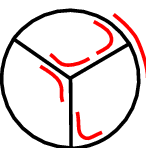}} & $-1$ & $\geq 1$ & $1$ & $1$ & yes \\
 1 & 3 & \raisebox{-0.65 cm}{\includegraphics[width = 1.8 cm]{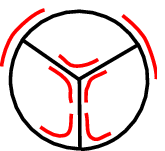}} & $-2$ & $\geq 2$ & $1$ & $1$ & yes \\
 3 & 1 & \raisebox{-0.65 cm}{\includegraphics[width = 1.8 cm]{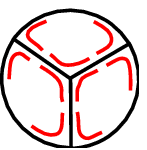}} & $-3$ & $\geq 2$ & $1$ & $2$ & yes \\
 3 & 3 & \raisebox{-0.65 cm}{\includegraphics[width = 1.8 cm]{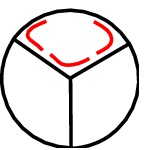}} & 0 & $\geq 1$ & $1$ & $0$ & yes \\
\hline
\end{tabular}
\caption{For each case: the number of odd $\Box_i$'s, of odd $\triangle_j$'s, the red arcs, the order of the first factor (\ref{eqn:factor1}), of the second (\ref{eqn:factor2}) estimated in (\ref{eqn:odd}), the number of components of the odd link $|L|$, and $\frac r2$. If 
(\ref{eqn:factor1}) + (\ref{eqn:factor2}) $\geq |L|-\frac r2$ then the estimate works (last column).}
\label{table:cases2}
\end{center}
\end{table}

As the tables show, the needed inequality 
$$
\ord_i\big((\ref{eqn:factor2})\big) + \ord_i\big((\ref{eqn:factor1})\big) \geq |L| + \frac r2
$$
is verified for all the configurations, except one bad case: when the $\Box_i$'s are all even and the $\triangle_j$'s are all odd. For that case we need to prove that 
$$
\ord_i\big((\ref{eqn:factor2})\big) + \ord_i\big((\ref{eqn:factor1})\big) \geq -2
$$
but we only get $\geq -4$. This bad case holds for instance when $a=b=c=d=e=f=2$ and hence $\Box_i = 4$ and $\triangle_j = 3$. If we look more carefully at this example we find

\beq
\pic{2}{0.8}{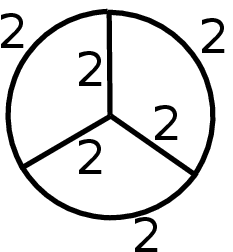} & = &
\begin{bmatrix} 1 \cdots 1 \\
2, 2, 2, 2, 2, 2 \end{bmatrix} \cdot \sum_{z = 3}^{4} (-1)^z 
\begin{bmatrix} z+1 \\
z-3, 4-z, 1 \end{bmatrix} \\
& = & \frac{1}{[2]^6} \cdot \left(-[4]!+[5]!\right) \\
& = & \frac{[4]!}{[2]^6} \cdot ([5]-1).
\eeq
Now it turns out that the difference
$$
[5]-1 = q^4+q^2+q^{-2}+q^4 = (q+q^{-1})(q^3+q^{-3}) = [2]\big([4]-[2]\big)
$$
has order $2$ in $q=i$: this difference produces a cancellation that increases the order of (\ref{eqn:factor2}) at $q=i$ by two, giving overall the desired $-4$ instead of the $\geq -2$ expected by the tables.

We now prove that this kind of cancellation holds in general, provided that the $\Box_i$'s are all even and the $\triangle_j$'s are all odd. The sum
$$
\sum_{z = \max \triangle_j }^{\min \Box_i} (-1)^z 
\begin{bmatrix} z+1 \\
z-\triangle_j, \Box_i-z, 1 
\end{bmatrix}
$$
goes from the odd $z=\max \triangle_j$ to the even $z=\min \Box_i$ and so contains an even number of terms. Two subsequent terms $z=2k-1$ and $z=2k$ give
$$
- \begin{bmatrix} 2k \\
2k-1-\triangle_j, \Box_i-2k+1, 1 
\end{bmatrix}
+ 
\begin{bmatrix} 
2k+1 \\
2k-\triangle_j, \Box_i-2k, 1 
\end{bmatrix}
$$
that may be rewritten as
$$
- \begin{bmatrix} 
2k \\ 
2k-1-\triangle_j, \Box_i-2k,1,1,1,1 
\end{bmatrix} 
\cdot \left( \frac 1{\prod_i[\Box_i-2k+1]} - \frac{[2k+1]}{\prod_j[2k-\triangle_j]}\right) .
$$
The left factor has order $2$ as prescribed by Table~\ref{table:cases}. Quite surprisingly, the second factor 
$$
\frac{\prod_j[2k-\triangle_j] - {[2k+1]}\cdot \prod_i[\Box_i-2k+1]}
{\prod_i[\Box_i-2k+1]\cdot \prod_j[2k-\triangle_j]} .
$$
has order at least $2$: all the quantum integers in the formula are quantum odd numbers; the denominator is a non-zero constant at $q=i$, while the numerator has order $\geq 2$ thanks to the following lemma. 

\begin{lem}[\cite{Carrega-Martelli}]\label{lem:bad_case}
Let $x_1,\ldots,x_n,y_1,\ldots,y_m$ be odd non negative integers with
$$
\sum_j (y_j-1) \equiv \sum_i (x_i-1) \ ({\rm mod}\ 4) .
$$
Then
$$
\ord_i\left(\prod_i[x_i] - \prod_j[y_j]\right) \geq 2 .
$$
\begin{proof}
We set $f(q) = \prod_i[x_i] - \prod_j[y_j]$ and write $\sqrt{-1}$ instead of $i$ to avoid confusion. Now
$$
[2k+1](\sqrt{-1}) = (-1)^k
$$
gives
$$ 
f(\sqrt{-1})  = (-1)^{\frac 12\sum_i (x_i-1)} - (-1)^{\frac 1 2 \sum_j (y_j -1)} = 0
$$
since $\frac 1 2 \sum_i (x_i-1)$ and $\frac 1 2 \sum_j (y_j -1)$ have the same parity by hypothesis. This gives $\ord_i f \geq 1$. We now calculate the derivative $f'$ of $f$. Note that
$$
[n]' = \frac{n(q^{n-1} + q^{-n-1})(q-q^{-1}) - (1+q^{-2})(q^n-q^{-n})}{(q-q^{-1})^2} .
$$
vanishes when $q=\sqrt{-1}$ and $n$ is odd, since both $q^{n-1}+q^{-n-1}$ and $1+q^{-2}$ do. Therefore the derivatives of $\prod [x_i]$ and $\prod [y_j]$ both vanish at $q=\sqrt{-1}$ and hence $f'(\sqrt{-1})=0$. Therefore $\ord_i f \geq 2$.
\end{proof}
\end{lem}

To conclude the proof of Lemma~\ref{lem:block} we must verify that 
$$
\sum_j (2k-\triangle_j-1) \equiv 2k + \sum_i (\Box_i-2k) \ ({\rm mod}\ 4)
$$
and apply Lemma~\ref{lem:bad_case}. This is equivalent to $\sum_j \triangle_j \equiv \sum_i \Box_i$ which is true since actually $\sum_j \triangle_j = \sum_i \Box_i$.
\end{proof}

\section{More lower bounds for the order at $q=A^2=i$}\label{sec:lower_bounds}

In this section we investigate more lower bounds for the order at $q=A^2=i$ of the Kauffman bracket.

\begin{prop}\label{prop:1-g}
Let $L$ be link in $\#_g(S^1\times S^2)$. Then
$$
\ord_i \langle L \rangle \geq 1-g ,
$$
and this estimation can be sharp.
\begin{proof}
The Kauffman bracket $\langle L \rangle $ is a linear combination of diagrams in the punctured disk $S_{(g)}$ without crossings and homotopically trivial components
$$
\langle L \rangle = \sum_j \lambda_j \langle D_j \rangle
$$
(see Proposition~\ref{prop:state_sum}). The order at $q=i$ of $\langle L \rangle$ is bigger equal than the one of the minimal order of the summands. The coefficients $\lambda_j$'s are integral Laurent polynomials, hence they can have only non negative order at $q=i$. We get a natural shadow $X_j$ of the link given by $D_j$ with $1+g$ that has a number of boundary components equal to $1+g+k_j$, where $k_j$ is the number of components of $D_j$ (see Remark~\ref{rem:sh_for_br}). The shadow $X_j$ is obtained by attaching to the punctured 2-disk an annulus to each component of the diagram and giving to all the regions gleam $0$. We have that $\langle D_j \rangle$ is equal to the Kauffman bracket of the boundary of $X_j$ with all the components corresponding to the ones of $D_j$ colored with $1$ and the other ones with $0$. 

The admissible colorings of $X_j$ that extend the one of the boundary define a unique (probably disconnected) surface $S_j$ bounded by the boundary components corresponding to the ones of $D_j$ (see Proposition~\ref{prop:Z_2_tr} and Corollary~\ref{cor:unique_sur}). We have that $\ord_i \langle D_j\rangle$ is bigger equal than the Euler characteristic of $S_j$. The surface $S_j$ is diffeomorphic to a sub-surface of $S_{(s)}$ that is bounded by $D_j$. Since $D_j$ has no homotopically trivial components, the lowest Euler characteristic of one such sub-surface is exactly $1-g$. It is reached by a smaller 2-disk with $g$ holes surrounded by some annuli. Hence $\langle L \rangle \geq 1-g$.
\end{proof}
\end{prop}

\begin{quest}
Let $L$ be a link in $\#_g(S^1\times S^2)$. Let $g'\leq g$ be the minimal number of holes of a punctured disk in $S_{(g)}$ that contains a diagram of $L$ ($g'$ is the minimal number such that $L$ can be seen as a link in $\#_{g'}(S^1\times S^2)$, see Remark~\ref{rem:tensor}, Definition~\ref{defn:split_homotopic_genus} and Proposition~\ref{prop:homot_genus}). Let $S\rightarrow \#_g(S^1\times S^2)$ be a ribbon surface without components without boundary and with biggest Euler characteristic that is bounded by $L$. Is the order at $q=i$ of $\langle L \rangle$ equal to the maximum between $1-g'$ and the Euler characteristic of $S$:
$$
\ord_i \langle L \rangle = \max\{ 1-g', \chi(S)\} \ \ ?
$$
\end{quest}

More in general we have the following:
\begin{prop}
Let $L\subset \#_g(S^1\times S^2)$ be a colored framed link. If at least one component of $L$ has an odd color then
$$
\ord_i \langle L \rangle \geq 1-g .
$$
\begin{proof}
We prove it by induction on the maximum color $c$ on $L$. If $c=1$, this is the case of Proposition~\ref{prop:1-g}. If some component $K$ of $L$ has a color $c>1$, we modify $K$ via the well known skein move shown in Fig.~\ref{figure:Cheb} that takes place in a solid torus neighborhood of $K$. Each of the new two addenda is a colored link with at least one odd-colored component. We perform this move on all components with maximum color $c$ and we conclude by induction since the order at a point is at least the lowest order of the summands.
\end{proof}
\end{prop}

\begin{rem}
Proposition~\ref{prop:1-g} does not hold for any knotted framed colored trivalent graph with at least one odd color in the obvious sense. In fact the graph in Fig.~\ref{figure:graph_tetr} lies in $S^3$ and has order at $q=i$ of the Kauffman bracket equal to $-1$.
\end{rem}

\begin{figure}[htbp]
\begin{center}
\includegraphics[scale=0.5]{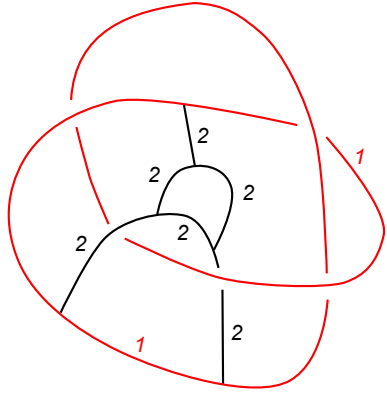}
\end{center}
\caption{A knotted colored framed trivalent graph in $S^3$ with order at $q=A^2=i$ of the Kauffman bracket equal to $-1$ and $r=6$, $\langle D \rangle = (-A^{24} +3A^{18} -A^{16} -A^{15} -2A^{14} +A^{13} -3A^{11} +3A^{10} +2A^9 +4A^8 -3A^7 -A^6 +A^5 -3A^4 -A^3 -3A^2 -2) /(A^{11} +A^7)$}
\label{figure:graph_tetr}
\end{figure}

\begin{quest}
Does Proposition~\ref{prop:1-g} hold also for any knotted framed trivalent colored graph $G\subset \#_g(S^1\times S^2)$ with at least one odd color in the following sense:
$$
\ord_i \langle G \rangle + \frac r 2 \geq 1-g ?
$$
Here $r$ is the number of red vertices of $G$.
\end{quest}

\begin{quest}\label{quest:no_graphs}
Let $G$ be a knotted colored framed trivalent graph in $\#_g(S^1\times S^2)$ and let $L$ be its odd sub-link. Is it always true that $\ord_i \langle G \rangle + \frac r 2$  is bigger equal than the order at $q=i$ of $\langle L \rangle$ where the components of $L$ are colored with $1$:
$$
\ord_i \langle G \rangle + \frac r 2 \geq \ord_i \langle L \rangle \ \ ?
$$
\end{quest}

\begin{rem}
We have to put ``bigger equal'' instead of ``equal'' in Question~\ref{quest:no_graphs} because the graph in Fig.~\ref{figure:graph_tetr} is an example where the equality does not hold. It is a graph in $S^3$ whose odd sub-link is the trefoil knot colored with $1$, the number of red vertices is $6$ and the order at $q=i$ of the Kauffman bracket is $-1$ ($-1+\frac 6 2 > 1$).
\end{rem}

\section{Upper bounds for the order at $q=A^2=i$}\label{sec:upper_bounds}

In this section we provide some upper bounds of the order at $q=A^2=i$ of the Kauffman bracket.

The following has been proved by Eisermann \cite{Eisermann} for links in $S^3$, the same proof applies to the case of links in $S^1\times S^2$.
\begin{prop}\label{prop:upper_bound}
Let $L$ be a $k$-component link in $S^1\times S^2$ or in $S^3$. If $L$ is homotopically trivial then
$$
\ord_i \langle L \rangle \leq k .
$$
\begin{proof}
 Evaluate the Kauffman bracket in $A=-1$. By Proposition~\ref{prop:no_0Kauff} $\langle L \rangle|_{A=-1} = (-2)^k$. We know that $\langle L \rangle = (-A^2 -A^{-2})^{\ord} \cdot f$, where $\ord$ is the order at $q=A^2=i$ of $\langle L \rangle$ and $f$ is a rational function with null order at $q=A^2=i$. Therefore
$$
(-2)^k = \langle L \rangle|_{A=- 1}  =  (-2)^\ord f(- 1) .
$$
By Proposition~\ref{prop:state_sum} $\langle L \rangle$ is an integral Laurent polynomial and hence $f$ is so. Thus $f(-1)$ is an integer number and $(-2)^\ord$ divides $(-2)^k$ in $\Z$, namely $\ord \leq k$.
\end{proof}
\end{prop}

\begin{quest}
Does Proposition~\ref{prop:upper_bound} hold also for links in $\#_g(S^1\times S^2)$ with $g\geq 2$?
\end{quest}

\begin{cor}
The order at $q=i$ of the Kauffman bracket of every $k$-components ribbon link $L$ in $S^1\times S^2$ or in $S^3$ is $k$
$$
\ord_i \langle L \rangle = k .
$$
\begin{proof}
It follows from Proposition~\ref{prop:upper_bound} and Corollary~\ref{cor:lower_bound_ribbon_link}.
\end{proof}
\end{cor}

The following slightly generalizes the previous result:
\begin{prop}\label{prop:upper_bound2}
Let $L$ be a homotopically trivial, colored, framed link in $S^1\times S^2$ or in $S^3$ and let $c_1,\ldots,c_k$ be the odd colors of $L$. If $c_j=2^{d_j}m_j -1$ with $m_j \in 2\Z+1$ then
$$
\ord_i \langle L \rangle \leq \sum_{j=1}^k d_j .
$$
\begin{proof}
With some little complications that we are going to explain, the topics of the proof of Proposition~\ref{prop:upper_bound} work here too.
The evaluation of the Kauffman bracket in $-1$ is still an invariant not distinguishing the over/underpasses. In fact we can put the projectors outside the crossings getting local situations consisting of some parallel straight strands passing over another set of parallel straight strands. Here we apply the first skein relation to show the equivalence of the two skeins once evalued in $A=-1$. Therefore $\langle L \rangle|_{A=-1}$ is equal to the evaluation in $A=-1$ of the Kauffman bracket of the colored unlink colored with $c_1, c_2, \ldots , c_k$. We used the fact that $\cerchio_1|_{A=-1} = -2$. Now we have $\cerchio_n|_{A=-1} = (-1)^n(n+1)$. Hence 
$$
\langle L \rangle|_{A=-1} =  \prod_{j=1}^k (-1)^{c_j}(c_j +1) .
$$ 
Hence
$$
\langle L \rangle|_{A=-1} = 2^{\sum_{j=1}^k d_j} n ,
$$
where $n$ is an odd integer.
\end{proof}
\end{prop}

\begin{quest}
Can we improve Proposition~\ref{prop:upper_bound2} saying that
$$
\ord_i \langle L \rangle \leq k \ \ ?
$$
\end{quest}

\begin{quest}
Does Proposition~\ref{prop:upper_bound2} hold also for homotopically trivial, colored, framed trivalent graphs in the following sense:
$$
\ord_i \langle G \rangle + \frac r 2 \leq \sum_j d_j \ \ ?
$$
Here $r$ is the number of red vertices and the $d_j$'j are the numbers related to the odd colors that are described in the proposition.
\end{quest}

\begin{rem}
It seems to be impossible to have a sharper upper bound of the order at $q=i$ of the Kauffman bracket of a link that is related just to easy properties of the link (for instance number of components and homotopy class). In fact we can find a knot in $S^1\times S^2$ with homotopy class $\pm 4 \in \pi_1(S^1\times S^2) \cong \Z$ whose order at $q=i$ is $1$. This is pictured in Fig.~\ref{figure:ex_knot1} and its Kauffman bracket is $2A^3-A^{-1}$. Moreover there are $\Z_2$-homologically trivial links in $S^1\times S^2$ whose Kauffman bracket is $0$ (see Fig.~\ref{figure:ex_0Kauf}), hence $\ord_i \langle L \rangle = \infty$.
\end{rem}

\begin{figure}[htbp]
\begin{center}
\includegraphics[scale=0.55]{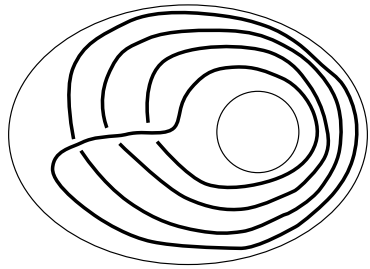}
\end{center}
\caption{A homotopically non trivial knot in $S^1\times S^2$ whose order at $q=i$ is $1$.}
\label{figure:ex_knot1}
\end{figure}

\chapter{Table of knots and links in $S^1\times S^2$}\label{chapter:table}

In this chapter we classify the non H-split links in $S^1\times S^2$ (Definition~\ref{defn:split_homotopic_genus} and Proposition~\ref{prop:split}) that are not contained in a 3-ball and have crossing number at most $3$ (Definition~\ref{defn:alt_cr_num}), and we compute some invariants. 

Links are seen up to reflections with respect to a Heegaard torus and reflections with respect to the $S^1$ factor, namely the homeomorphisms $S^1\times S^2\rightarrow S^1\times S^2$ $(e^{i\theta},y)\mapsto (e^{-i\theta}, y)$. This means that two diagrams $D_1,D_2\subset S^1\times [-1,1]$ represent the same object if $D_2$ is equal to $D_1$ after changing all the over/underpasses, and if $D_2$ is the image of $D_1$ under the maps $S^1\times [-1,1]\rightarrow S^1\times [-1,1]$ $(e^{i\theta},t) \mapsto (e^{i\theta},-t)$ and $(e^{i\theta},t) \mapsto (e^{-i\theta}, t)$. 

The computed invariants are: the number of components, the $\Z_2$-homology class $[L]\in H_1(S^1\times S^2;\Z_2)$, if the link is homotopically trivial (there is a continuous map $(S^1\cup \ldots \cup S^1) \times [0,1]\rightarrow S^1\times S^2$ whose restriction to $(S^1\cup \ldots \cup S^1)\times \{0\}$ is the embedding of $L$, and the image of restriction to $(S^1\cup \ldots \cup S^1)\times \{1\}$ is a finite set of points), if the link is alternating (Definition~\ref{defn:alt_cr_num}), the Kauffman bracket (or the Jones polynomial) $\langle D \rangle$, the order at $q=A^2=i$ $\ord_{q=A^2=i} \langle D\rangle$ of the Kauffman bracket, the first homology group of the complement $H_1(S^1\times S^2 \setminus L;\Z)$, a presentation of the fundamental group of the complement $\pi_1(S^1\times S^2\setminus L)$, if the link is hyperbolic (the complement has a complete finite-volume hyperbolic structure), the hyperbolic volume (if it is hyperbolic), (whether there are) some separating surfaces in the complement (incompressible, not parallel to each other and not parallel to the boundary). 

The $\Z_2$-homology class and if the link is homotopically trivial can be easily checked looking at the diagrams (see Proposition~\ref{prop:Z_2_tr}). 

The Kauffman bracket is the one computed using the pictured diagrams, we look at this Laurent polynomial up to multiplications by an integer power of $-A^3$ so that it becomes an invariant of unframed and unoriented links. The Kauffman bracket is sufficient to distinguish the $\Z_2$-homologically trivial links. 

We can say that all the $\Z_2$-homologically trivial links in the list whose pictured diagram is non alternating are actually non alternating by applying Theorem~\ref{theorem:Tait_conj_Jones_g}.

The other (non trivial) invariants are computed using SnapPy. We say that a link is hyperbolic if and only if SnapPy found a geometric triangulation of the complement that has just positive tetrahedra, otherwise we say that it is not hyperbolic. We remind that if SnapPy finds one such triangulation it is extremely probable that the result is correct. For the non separating surfaces we report the result of the command ``.splitting\_surfaces()''.

Although we just consider links with crossing number at most $3$, interesting examples come out. For instance we find a $\Z_2$-homologically trivial knot with null Kauffman bracket ($3_4$), knots and links whose complement has first homology group with torsion ($1_1$, $2_2$, $2_3$, $3_1$, $3_2$, $3_3$, $3_4$, $L3_2$) or with a free part with rank bigger than the number of components ($2_1$). 

In order to be ribbon, a link must be homotopically trivial (and so $\Z_2$-homologically trivial). There are knots that are not homotopically trivial, computing the order at $q=A^2=i$ of the Kauffman bracket of those knots and using Theorem~\ref{theorem:Eisermann_gen} we can conclude that there are no ribbon links in the list. 

Sometimes there are more than one diagram to represent the same link with the same embedding of the annulus. We can show the equivalence of the represented links by Reidemeister moves and the move described in Subsection~\ref{subsec:diag} as shown in Subsection~\ref{subsec:alt_diag_Z_2_non_tr}.

There are three pairs of links in the list whose elements we are not able to say if they are equivalent, in the sense described above, or not: $(2_2, 2_3)$, $(3_5, 3_6)$ and $(L3_3, L3_4)$. They are all $\Z_2$-homologically non trivial. The invariants we computed give the same result on the elements of the pairs (even the same presentation of the fundamental group). Hence probably their complements are diffeomorphic. Unfortunately all the known invariants do not distinguish two links $L,L' \subset S^1\times S^2$ that are related by an orientation preserving diffeomorphim of the ambient manifold $\varphi:S^1\times S^2 \rightarrow S^1\times S^2$, $\varphi(L)=L'$. In particular the complements of $3_5$ and $3_6$ are isometric as hyperbolic manifolds, and it is conjectured that such links are related by a diffeomorphism of $S^1\times S^2$. We count the elements of these pairs separately.

We can find more examples of links in $\#_g(S^1\times S^2)$ and their Kauffman bracket in Section~\ref{sec:Kauf_g}.

Throughout this chapter we present some tables and we need some abbreviations. ``N. comp.'' is for the number of components of the link. On the fields ``$\Z_2$-h.l. tr.'', ``H.t.. tr.'' and ``Alt'', we say ``yes'' if the link is respectively $\Z_2$-homologically trivial, homotopically trivial and alternating, otherwise we say ``no''. ``$\ord_{q=A^2=i}$'' is for the order in $q=A^2=i$ of the Kauffman bracket of the link. We write ``yes'' on the fields ``H.l. tor.'' and ``Big h.l.'', if the first homology group of the complement of the link has respectively torsion and rank bigger than the number of components, otherwise we write ``no''. We write ``yes'' or ``no'' on the field ``Hyp.'' according to the convention about hyperbolicity and SnapPy mentioned above.

\section{Knots}

\subsection{$0_1$}
$\ $\\
\begin{minipage}{0.5\textwidth}
\begin{itemize}

\item{$\Z_2$-homologically trivial: no.}

\item{Homotopically trivial: no.}

\item{Alternating: yes.}

\item{Kauffman bracket:\\
$ \langle D \rangle = 0 $.}

\item{$\ord_{q=A^2=i} \langle D \rangle = \infty $.}

\item{Hyperbolic: no.}

\item{Homology group:\\
$H_1(S^1\times S^2 \setminus K; \Z) = \Z $.}

\item{Fundamental group,\\
$\Z $.}

\item{Separating surfaces: no.}

\end{itemize}
\end{minipage}
\qquad
\begin{minipage}{0.3\textwidth}
\begin{center}
\includegraphics[scale=0.45]{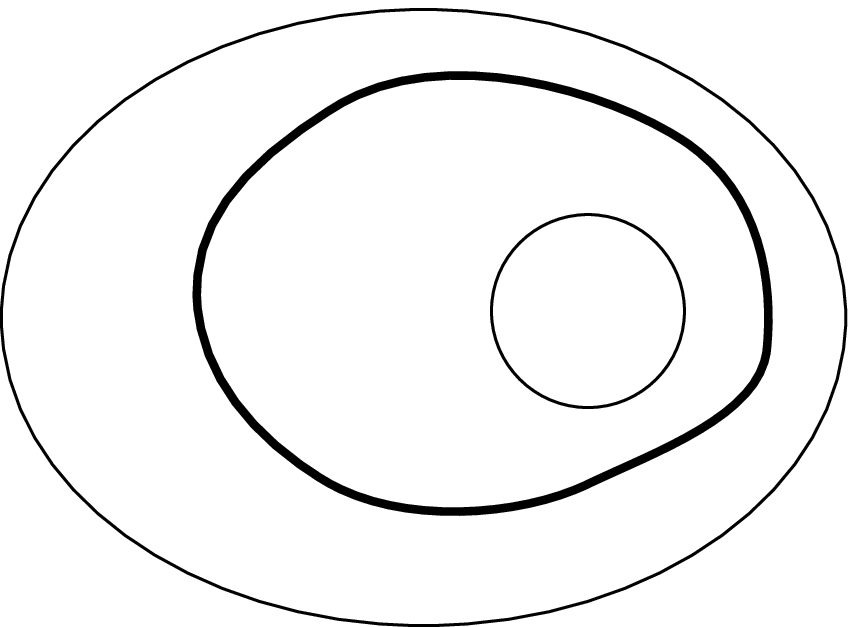} ($0_1$)
\\
\vspace{0.5cm}
\includegraphics[scale=0.45]{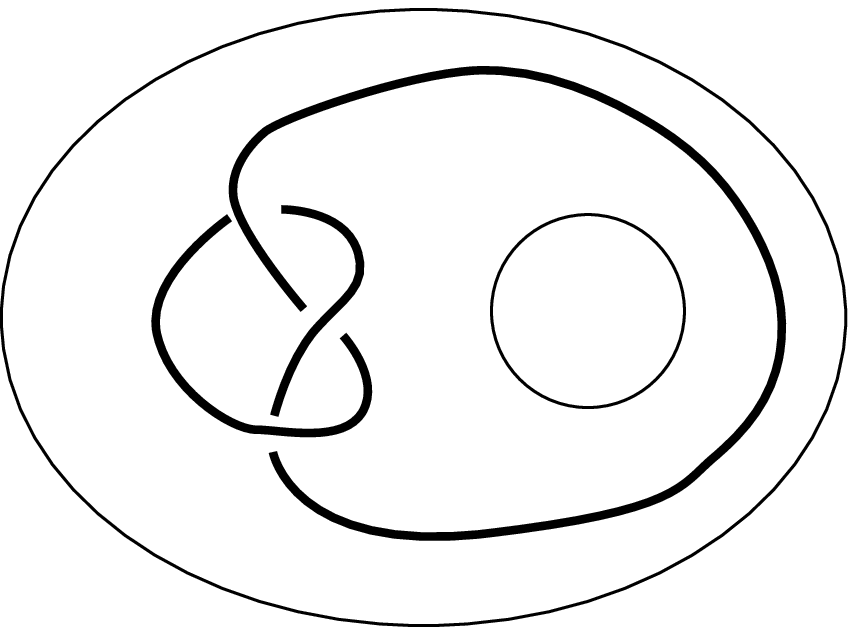} ($0_1$)
\end{center}
\end{minipage}

\subsection{$1_1$}\label{subsec:1_1}
$\ $\\
\begin{minipage}{0.5\textwidth}
\begin{itemize}

\item{$\Z_2$-homologically trivial: yes.}

\item{Homotopically trivial: no.}

\item{Alternating: yes.}

\item{Kauffman bracket:\\
$\langle D \rangle = -A^3$.}

\item{$\ord_{q=A^2=i} \langle D \rangle = 0$.}

\item{Hyperbolic: no.}

\item{Homology group:\\
$H_1(S^1\times S^2 \setminus K; \Z) = \Z_2 \oplus \Z$.}

\item{Fundamental group:\\
generators:\\
$a,b$,\\
relators:\\
$aabb$.}

\item{Separating surfaces: no.}

\end{itemize}
\end{minipage}
\qquad
\begin{minipage}{0.3\textwidth}
\begin{center}
\includegraphics[scale=0.45]{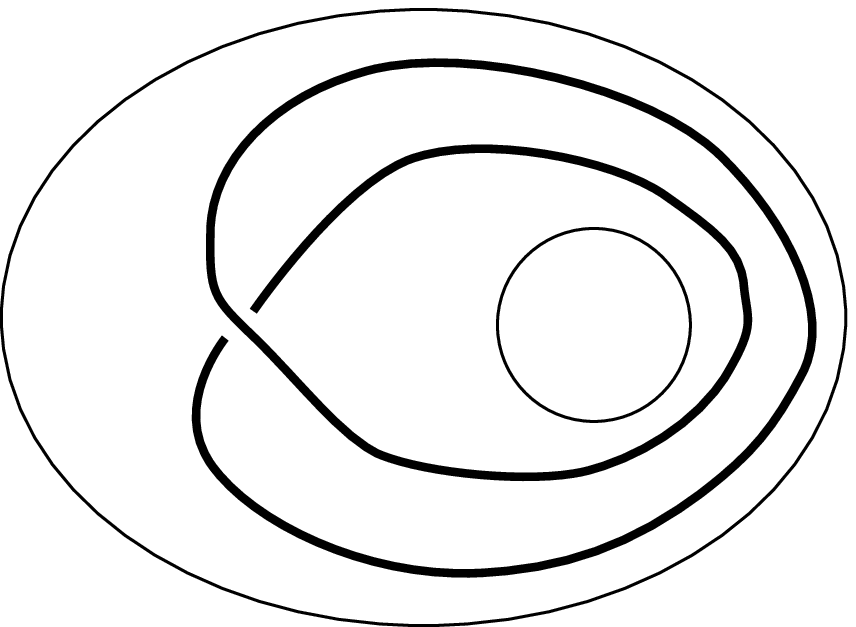} ($1_1$)
\end{center}
\end{minipage}

\subsection{$2_1$}
$\ $\\
\begin{minipage}{0.5\textwidth}
\begin{itemize}

\item{$\Z_2$-homologically trivial: yes.}

\item{Homotopically trivial: yes.}

\item{Alternating: yes.}

\item{Kauffman bracket:\\
$\langle D \rangle = -A^4-A^{-4}$.}

\item{$\ord_{q=A^2=i} \langle D \rangle = 0$.}

\item{Hyperbolic: no.}

\item{Homology group:\\
$H_1(S^1\times S^2 \setminus K; \Z) = \Z^2$.}

\item{Fundamental group,\\
generators:\\
$a,b$,\\
relators:\\
$abbab^{-1}a^{-1}a^{-1}b^{-1}$.}

\item{Separating surfaces: no.}

\end{itemize}
\end{minipage}
\qquad
\begin{minipage}{0.3\textwidth}
\begin{center}
\includegraphics[scale=0.45]{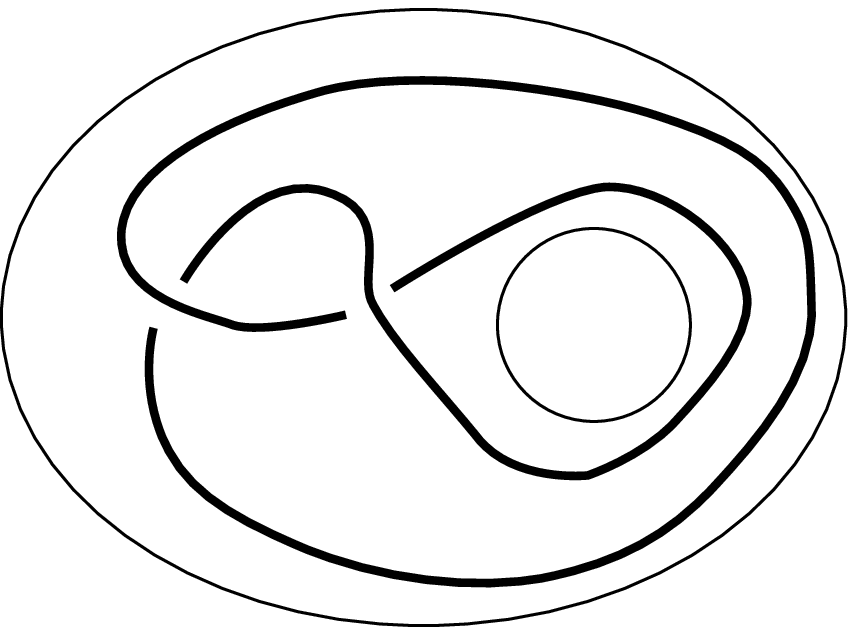} ($2_1$)
\end{center}
\end{minipage}

\subsection{$2_2$ and $2_3$}
We do not know if they are equivalent or not.\\
\begin{minipage}{0.5\textwidth}
\begin{itemize}

\item{$\Z_2$-homologically trivial: no.}

\item{Homotopically trivial: no.}

\item{Alternating: yes for $2_2$, we do not know for $2_3$.}

\item{Kauffman bracket:\\
$\langle D \rangle = 0$.}

\item{$\ord_{q=A^2=i} \langle D \rangle = \infty$.}

\item{Hyperbolic: no.}

\item{Homology group:\\
$H_1(S^1\times S^2 \setminus K; \Z) = \Z_3 \oplus \Z$.}

\item{Fundamental group,\\
generators:\\
$a,b$,\\
relators:\\
$aaabbb$.}

\item{Separating surfaces: no.}

\end{itemize}
\end{minipage}
\qquad
\begin{minipage}{0.3\textwidth}
\begin{center}
\includegraphics[scale=0.45]{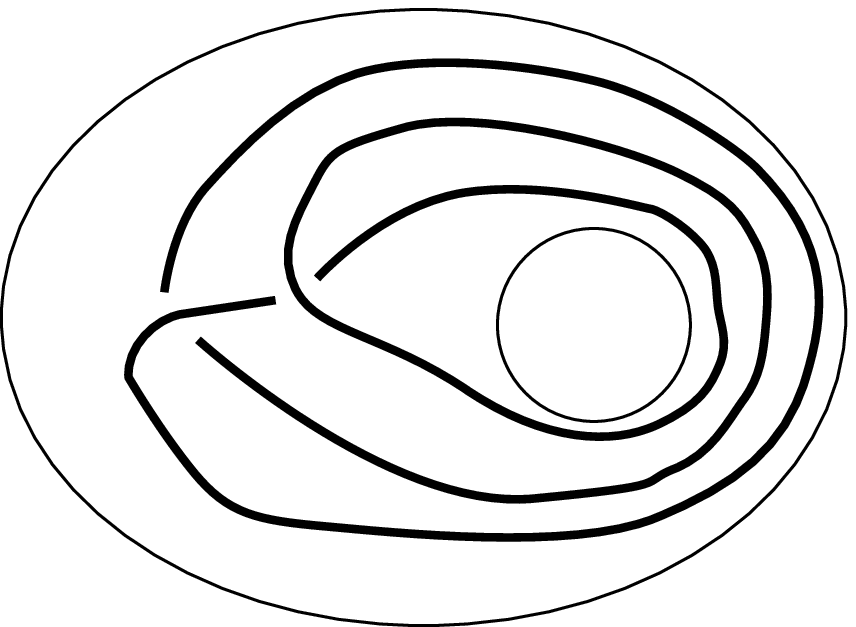} ($2_2$)
\\
\vspace{0.5cm}
\includegraphics[scale=0.45]{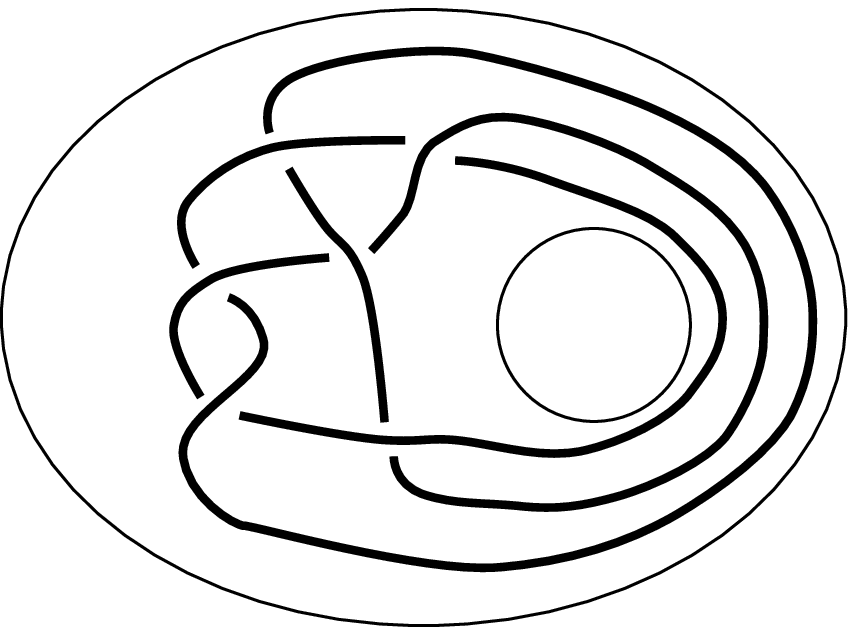} ($2_2$)
\\
\vspace{0.5cm}
\includegraphics[scale=0.45]{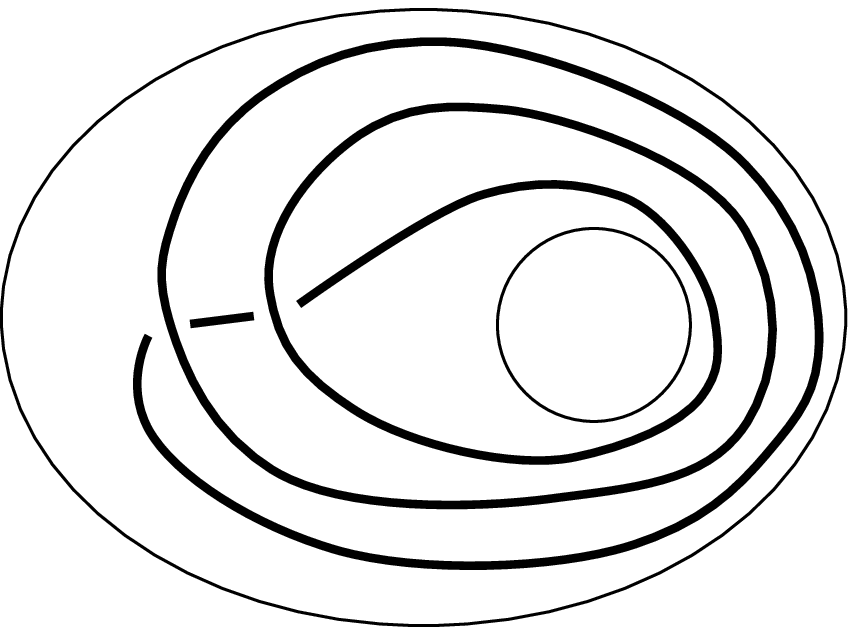} ($2_3$)
\end{center}
\end{minipage}


\subsection{$3_1$}
$\ $\\
\begin{minipage}{0.5\textwidth}
\begin{itemize}

\item{$\Z_2$-homologically trivial: yes.}

\item{Homotopically trivial: yes.}

\item{Alternating: yes.}

\item{Kauffman bracket:\\
$\langle D \rangle = -A^5 -A^{-3} +A^{-7}$.}

\item{$\ord_{q=A^2=i} \langle D \rangle = 0$.}

\item{Hyperbolic: no.}

\item{Homology group:\\
$H_1(S^1\times S^2 \setminus K; \Z) = \Z_2 \oplus \Z$.}

\item{Fundamental group,\\
generators:\\
$a,b$,\\
relators:\\
$aaba^{-1}b^{-1}abbab^{-1}a^{-1}b$.}

\item{Separating surfaces:\\
a two-sided torus.}

\end{itemize}
\end{minipage}
\qquad
\begin{minipage}{0.3\textwidth}
\begin{center}
\includegraphics[scale=0.45]{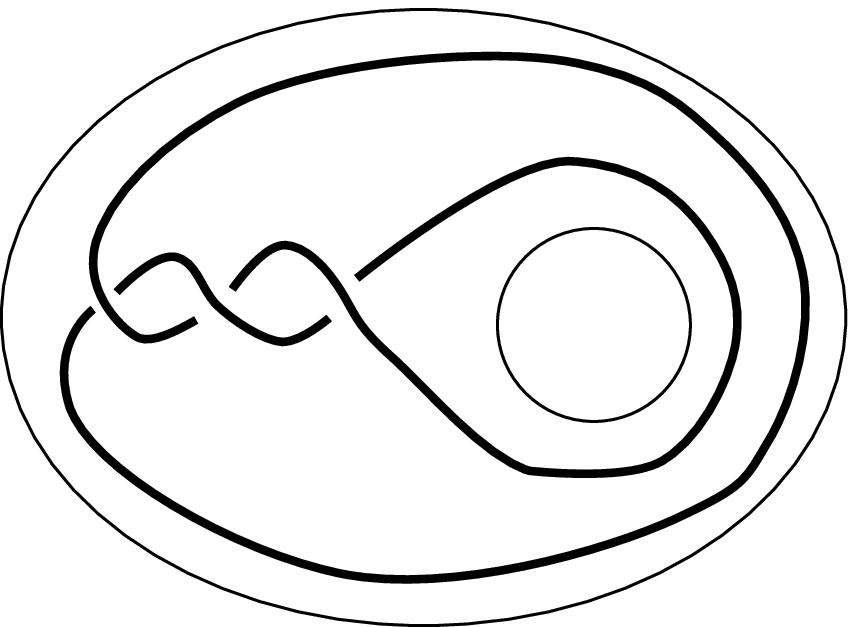} ($3_1$)
\end{center}
\end{minipage}

\subsection{$3_2$}
$\ $\\
\begin{minipage}{0.5\textwidth}
\begin{itemize}

\item{$\Z_2$-homologically trivial: yes.}

\item{Homotopically trivial: no.}

\item{Alternating: yes.}

\item{Kauffman bracket:\\
$\langle D \rangle = A^7 -A^3 +A^{-1} + A^{-5}$.}

\item{$\ord_{q=A^2=i} \langle D \rangle = 0$.}

\item{Hyperbolic: yes.}

\item{Hyperbolic volume:\\
$5.33348956690$}

\item{Homology group:\\
$H_1(S^1\times S^2 \setminus K; \Z) = \Z_4 \oplus \Z$.}

\item{Fundamental group,\\
generators:\\
$a,b$,\\
relators:\\
$aaba^{-1}baabbaaba^{-1}baab^{-1}a^{-1}a^{-1}b^{-1}$.}

\item{Separating surfaces: no.}

\end{itemize}
\end{minipage}
\qquad
\begin{minipage}{0.3\textwidth}
\begin{center}
\includegraphics[scale=0.45]{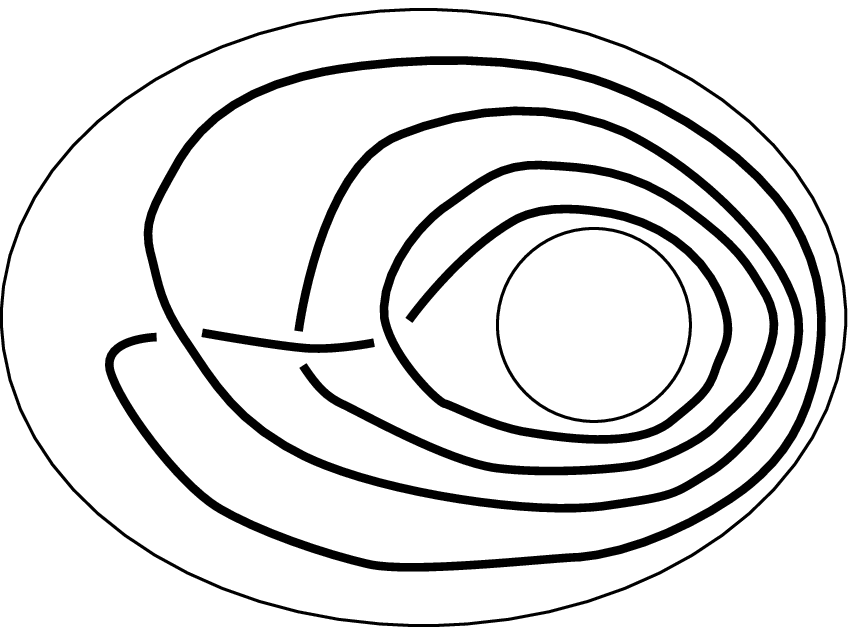} ($3_2$)
\end{center}
\end{minipage}

\subsection{$3_3$}
$\ $\\
\begin{minipage}{0.5\textwidth}
\begin{itemize}

\item{$\Z_2$-homologically trivial: yes.}

\item{Homotopically trivial: no.}

\item{Alternating: no.}

\item{Kauffman bracket:\\
$\langle D \rangle = A -A^{-3} -A^{-5}$.}

\item{$\ord_{q=A^2=i} \langle D \rangle = 0$.}

\item{Hyperbolic: no.}

\item{Homology group:\\
$H_1(S^1\times S^2 \setminus K; \Z) = \Z_4 \oplus \Z$.}

\item{Fundamental group,\\
generators:\\
$a,b$,\\
relators:\\
$aabbbbbaabbb$.}

\item{Separating surfaces:\\
a two-sided torus,\\
and a one-sided non orientable surface with null Euler characteristic.}

\end{itemize}
\end{minipage}
\qquad
\begin{minipage}{0.3\textwidth}
\begin{center}
\includegraphics[scale=0.45]{3_3.eps} ($3_3$)
\end{center}
\end{minipage}

\subsection{$3_4$}
$\ $\\
\begin{minipage}{0.5\textwidth}
\begin{itemize}

\item{$\Z_2$-homologically trivial: yes.}

\item{Homotopically trivial: no.}

\item{Alternating: no.}

\item{Kauffman bracket:\\
$\langle D \rangle = 0$.}

\item{$\ord_{q=A^2=i} \langle D \rangle = \infty$.}

\item{Hyperbolic: no.}

\item{Homology group:\\
$H_1(S^1\times S^2 \setminus K; \Z) = \Z_4 \oplus \Z$.}

\item{Fundamental group,\\
generators:\\
$a,b$,\\
relators:\\
$aaaabbbb$.}

\item{Separating surfaces: no.}

\end{itemize}
\end{minipage}
\qquad
\begin{minipage}{0.3\textwidth}
\begin{center}
\includegraphics[scale=0.45]{3_4.eps} ($3_4$)
\end{center}
\end{minipage}

\subsection{$3_5$ and $3_6$}
We know that the complements of these knots are isometric as hyperbolic manifolds. We do not know if they are equivalent or not. \\
\begin{minipage}{0.5\textwidth}
\begin{itemize}

\item{$\Z_2$-homologically trivial: no.}

\item{Homotopically trivial: no.}

\item{Alternating: yes for $3_5$, we do not know for $3_6$.}

\item{Kauffman bracket:\\
$\langle D \rangle = 0$.}

\item{$\ord_{q=A^2=i} \langle D \rangle = \infty$.}

\item{Hyperbolic: yes.}

\item{Hyperbolic volume:\\
$3.66386237671$}

\item{Homology group:\\
$H_1(S^1\times S^2 \setminus K; \Z) = \Z$.}

\item{Fundamental group,\\
generators:\\
$a,b$,\\
relators:\\
$aaabba^{-1}b^{-1}b^{-1}b^{-1}a^{-1}bb$.}

\item{Separating surfaces: no.}

\end{itemize}
\end{minipage}
\qquad
\begin{minipage}{0.3\textwidth}
\begin{center}
\includegraphics[scale=0.45]{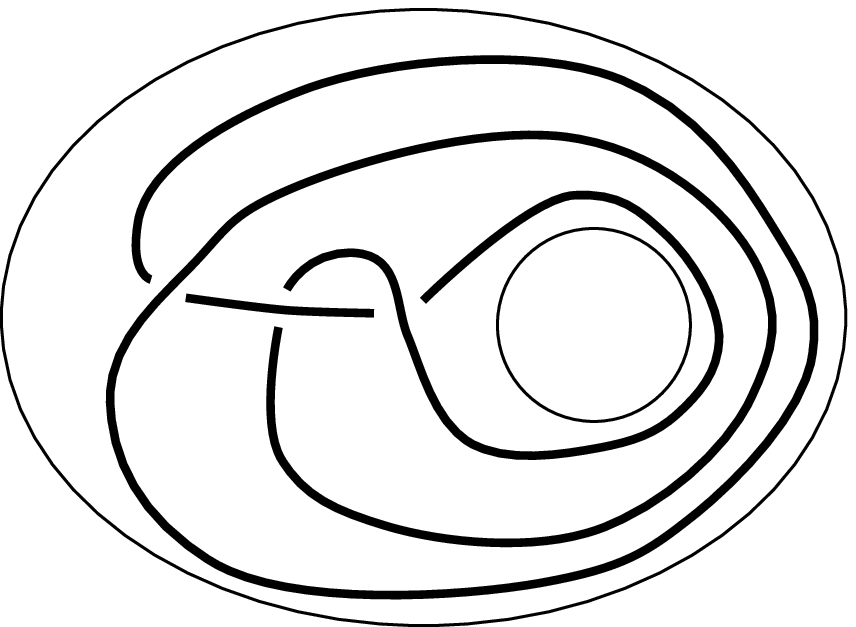} ($3_5$)
\\
\vspace{0.5cm}
\includegraphics[scale=0.45]{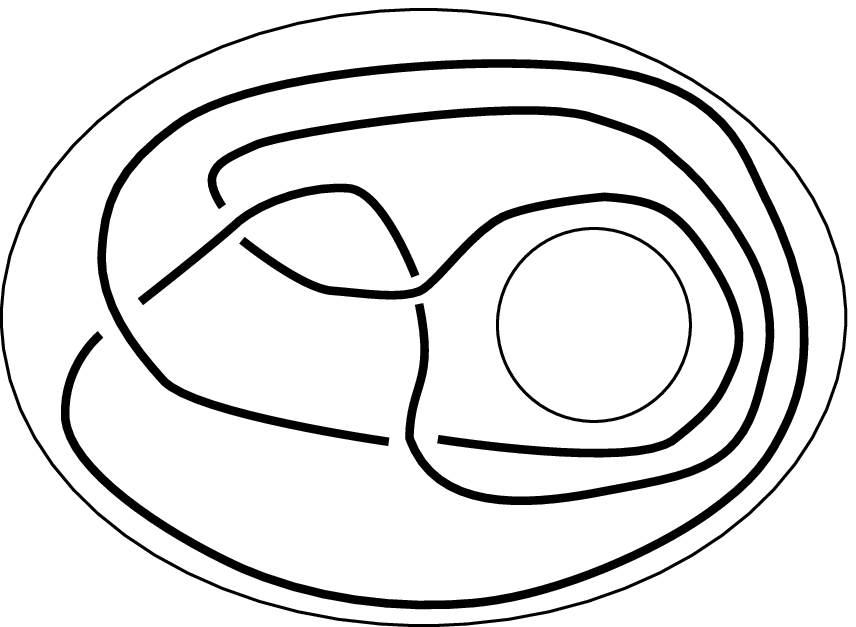} ($3_5$)
\\
\vspace{0.5cm}
\includegraphics[scale=0.45]{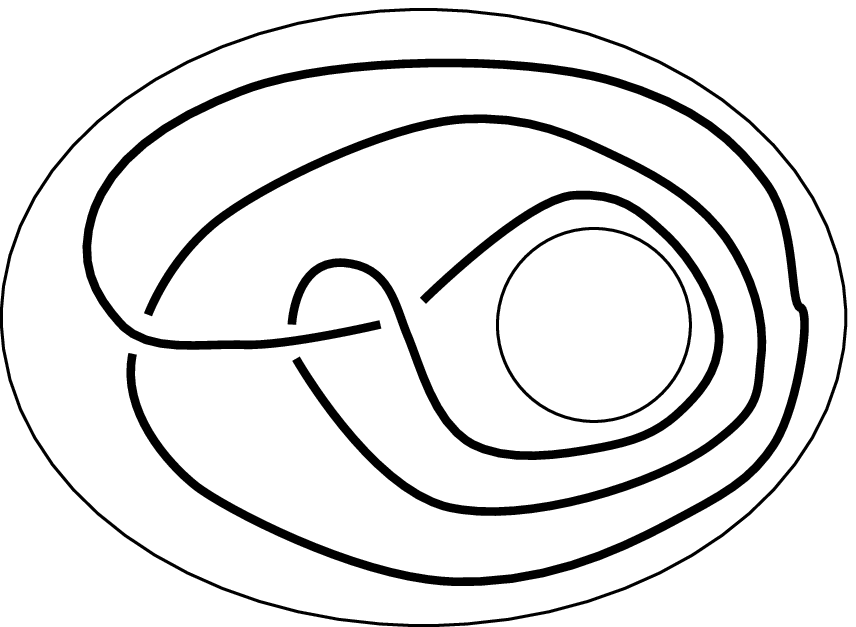} ($3_6$)
\end{center}
\end{minipage}

\subsection{Summary for knots}

$$
\begin{array}{|c|c|c|c|c|c|c|c|}
\hline 
\text{Name} & \text{$\Z_2$-h.l. tr.} & \text{H.t. tr.} & \text{Alt.} & \ord_{q=A^2=i} & \text{H.l. tor} & \text{Big h.l.} & \text{Hyp.} \\
\hline \hline
0_1 & \text{no} & \text{no} & \text{yes} & \infty & \text{no} & \text{no} & \text{no} \\
\hline \hline
1_1 & \text{yes} &  \text{no} & \text{yes} & 0 & \text{yes} & \text{no} & \text{no} \\
\hline \hline
2_1 & \text{yes} & \text{yes} & \text{yes} & 0  & \text{no} & \text{yes} & \text{no} \\
\hline
2_2 & \text{no} & \text{no} & \text{yes} & \infty & \text{yes} & \text{no} & \text{no} \\
\hline
2_3 & \text{no} & \text{no} & \text{no} & \infty & \text{yes} & \text{no} & \text{no} \\
\hline \hline
3_1 &\text{yes} &\text{yes} &\text{yes} & 0 & \text{yes} & \text{no} &\text{no} \\
\hline
3_2 &\text{yes} &\text{no} &\text{yes} & 0 &\text{yes} &\text{no} &\text{yes} \\
\hline
3_3 & \text{yes} &\text{no} &\text{no} & 0 &\text{yes} &\text{no} &\text{no} \\
\hline
3_4 &\text{yes} &\text{no} &\text{no} & \infty &\text{yes} &\text{no} &\text{no} \\
\hline
3_5 &\text{no} &\text{no} &\text{yes} & \infty &\text{no} &\text{no} &\text{yes} \\
\hline
3_6 &\text{no} &\text{no} &\text{no} & \infty &\text{no} &\text{no} &\text{yes} \\
\hline \hline
\end{array}
$$

\section{Links}
	
\subsection{$L3_1$}
$\ $\\
\begin{minipage}{0.5\textwidth}
\begin{itemize}

\item{No. of components: $2$.}

\item{$\Z_2$-homologically trivial: yes.}

\item{Homotopically trivial: no.}

\item{Alternating: yes.}

\item{Kauffman bracket:\\
$\langle D \rangle = A^7 -A^3 -A^{-5}$.}

\item{$\ord_{q=A^2=i} \langle D \rangle = 0$.}

\item{Hyperbolic: no.}

\item{Homology group:\\
$H_1(S^1\times S^2 \setminus K; \Z) = \Z^2$.}

\item{Fundamental group,\\
generators:\\
$a,b$,\\
relators:\\
$aab^{-1}a^{-1}bba^{-1}a^{-1}bab^{-1}b^{-1}$.}

\item{Separating surfaces:\\
a two-sided torus.}

\end{itemize}
\end{minipage}
\qquad
\begin{minipage}{0.3\textwidth}
\begin{center}
\includegraphics[scale=0.45]{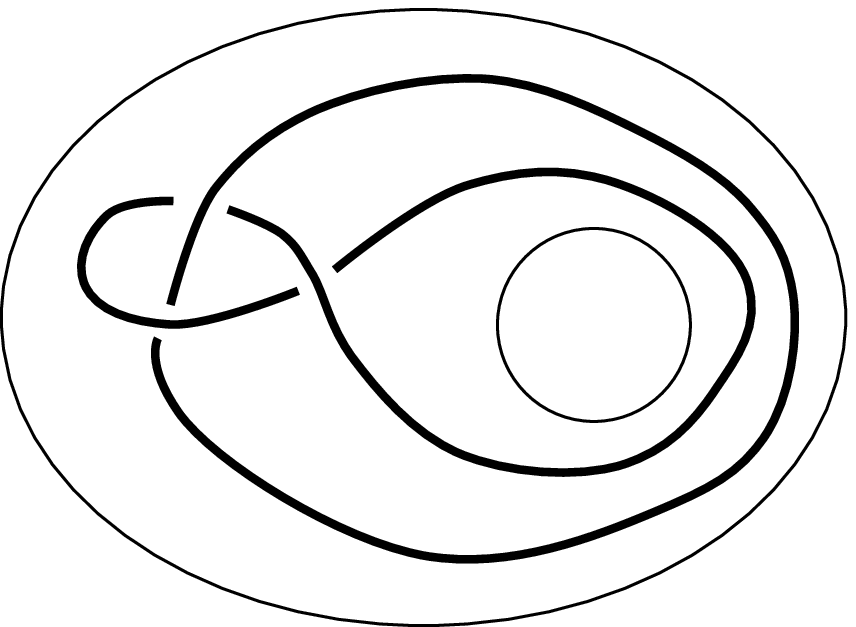} ($L3_1$)
\end{center}
\end{minipage}

\subsection{$L3_2$}
$\ $\\
\begin{minipage}{0.5\textwidth}
\begin{itemize}

\item{No. of components: $2$.}

\item{$\Z_2$-homologically trivial: yes.}

\item{Homotopically trivial: no.}

\item{Alternating: yes.}

\item{Kauffman bracket:\\
$\langle D \rangle = A^7 +A^{-1}$.}

\item{$\ord_{q=A^2=i} \langle D \rangle = 0$.}

\item{Hyperbolic: no.}

\item{Homology group:\\
$H_1(S^1\times S^2 \setminus K; \Z) = \Z_2 \oplus \Z^2$.}

\item{Fundamental group,\\
generators:\\
$a,b,c$,\\
relators:\\
$ac^{-1}a^{-1}c$,\\
$abab^{-1}$.}

\item{Separating surfaces:\\
a two-sided torus,\\
and two one-sided non-orientable surfaces with null Euler characteristic.}

\end{itemize}
\end{minipage}
\qquad
\begin{minipage}{0.3\textwidth}
\begin{center}
\includegraphics[scale=0.45]{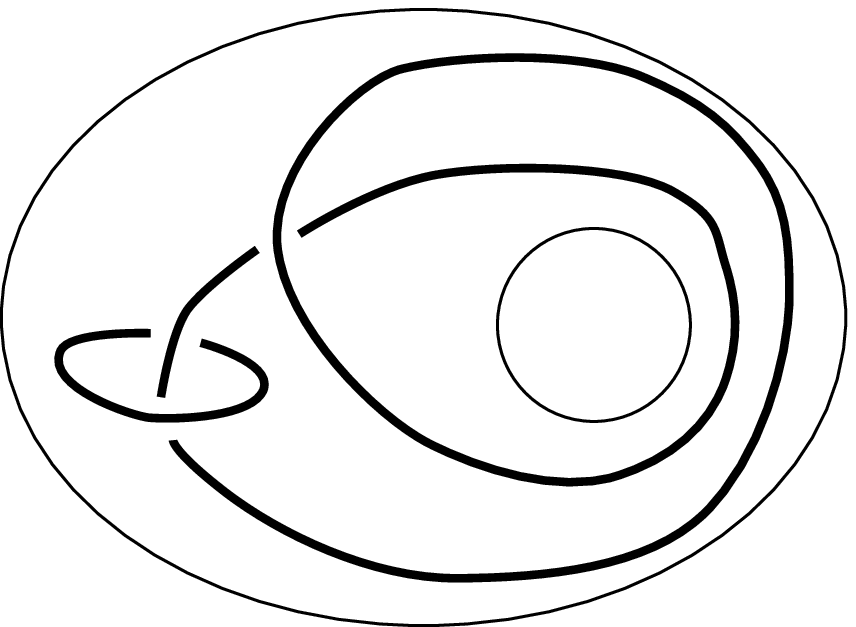} ($L3_2$)
\end{center}
\end{minipage}

\subsection{$L3_3$ and $L3_4$}
We do not know if they are equivalent or not.\\
\begin{minipage}{0.5\textwidth}
\begin{itemize}

\item{No. of components: $2$.}

\item{$\Z_2$-homologically trivial: no.}

\item{Homotopically trivial: no.}

\item{Alternating: yes for $L3_3$, we do not know for $L3_4$.}

\item{Kauffman bracket:\\
$\langle D \rangle = 0$.}

\item{$\ord_{q=A^2=i} \langle D \rangle = \infty$.}

\item{Hyperbolic: no.}

\item{Homology group:\\
$H_1(S^1\times S^2 \setminus K; \Z) = \Z^2$.}

\item{Fundamental group,\\ 
generators:\\
$a,b,c$,\\
relators:\\
$ab^{-1}b^{-1}a^{-1}bb$.}

\item{Separating surfaces:\\
a two-sided torus.}

\end{itemize}
\end{minipage}
\qquad
\begin{minipage}{0.3\textwidth}
\begin{center}
\includegraphics[scale=0.45]{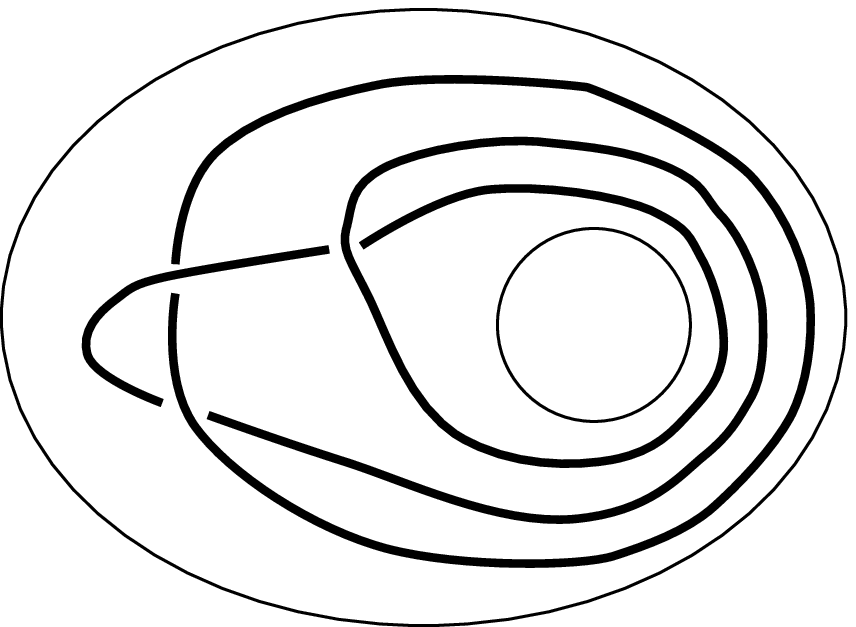} ($L3_3$)
\\
\vspace{0.5cm}
\includegraphics[scale=0.45]{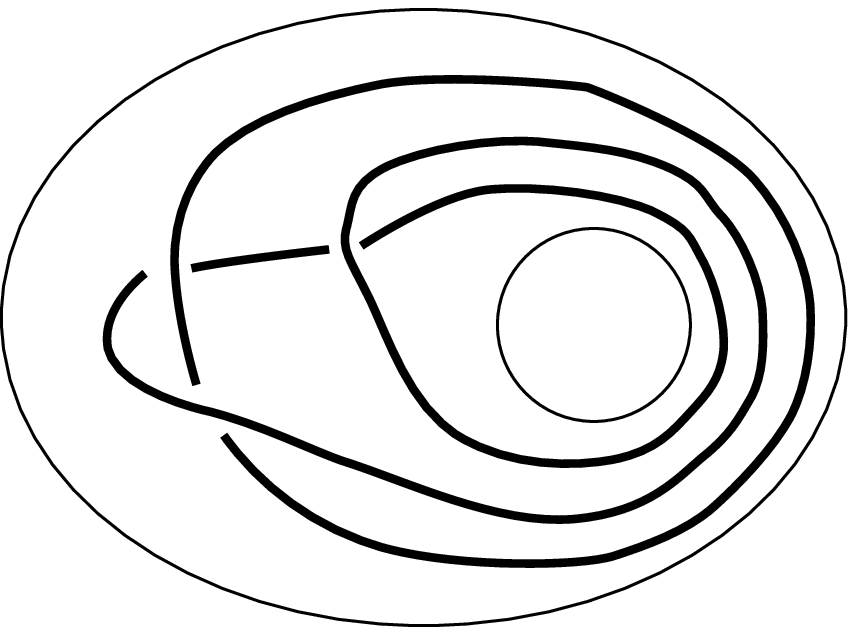} ($L3_4$)
\end{center}
\end{minipage}

\subsection{Summary for links non knots}
	
$$
\begin{array}{|c|c|c|c|c|c|c|c|c|}
\hline 
\text{Name} & \text{N. comp.} & \text{$\Z_2$-h.l. tr.} & \text{H.t. tr.} & \text{Alt.} & \ord_{q=A^2=i} & \text{H.l. tor.} & \text{Big h.l.} & \text{Hyp.} \\
\hline \hline
L3_1 & 2  & \text{yes}  & \text{no} & \text{yes} & 0 & \text{no} & \text{no} & \text{no}  \\
\hline
L3_2 & 2  & \text{yes}  & \text{no} & \text{yes} & 0 & \text{yes} & \text{no} & \text{no}  \\
\hline
L3_3 & 2  & \text{no}  & \text{no} & \text{yes} & \infty & \text{no} & \text{no} & \text{no}  \\
\hline
L3_4 & 2  & \text{no}  & \text{no} & \text{yes} & \infty & \text{no} & \text{no} & \text{no}  \\
\hline \hline
\end{array}
$$

\section{Summary}

\begin{center}

\begin{tabular}{|l||c|c|c|c|}
\hline
 & $0$ & $1$ & $2$ & $3$ \\
\hline \hline
Max. n. of components & $1$ & $1$ & $2$ & $2$ \\
\hline
N. of knots  & $1$ & $1$ & $3$ & $7$ \\
\hline
N. of 2-components links  & $0$ & $0$ & $0$ & $4$ \\
\hline
N. of $Z_2$-h.l. non tr. knots  & $1$ & $0$ & $2$ & $2$ \\
\hline
N. of $Z_2$-h.l. non tr., non knots, links & $0$ & $0$ & $0$ & $2$ \\
\hline
N. of h.t. tr. knots & $0$ & $0$ & $1$ & $2$ \\
\hline
N. of h.t. tr., non knots, links & $0$ & $0$ & $0$ & $0$ \\
\hline
N. of alternating knots & $1$ & $1$ & $2$ & $4$ \\
\hline
N. of alternating,non knots, links & $0$ & $0$ & $0$ & $4$ \\
\hline
N. of $\Z_2$-h.l. tr. alternating knots & $0$ & $1$ & $1$ & $2$ \\
\hline
N. of $\Z_2$-h.l. tr. alternating, non knots, links & $0$ & $0$ & $0$ & $2$ \\
\hline
N. of $Z_2$-trivial links with null Jones pol. & $0$ & $0$ & $0$ & $1$ \\
\hline
N. of $Z_2$-trivial knots with symmetric Kauf. br. & $0$ & $1$ & $1$ & $0$ \\
\hline
N. of $Z_2$-trivial, non knots, links with symmetric Kauf. br. & $0$ & $0$ & $0$ & $1$ \\
\hline
N. of hyp. knots & $0$ & $0$ & $0$ & $3$ \\
\hline
N. of, non knots, hyp. links & $0$ & $0$ & $0$ & $0$ \\
\hline
N. of knots with h.l. with tor. & $0$ & $1$ & $1$ & $4$ \\
\hline
N. of, non knots, links with h.l. with tor. & $0$ & $0$ & $0$ & $1$ \\
\hline
N. of knots with big h.l. & $0$ & $0$ & $1$ & $0$ \\
\hline
N. of, non knots, links with big h.l. & $0$ & $0$ & $0$ & $0$ \\
\hline
N. of links with $\ord_{q=A^2=i}= $ n. comp. & $0$ & $0$ & $0$ & $0$ \\
\hline
N. of ribbon links & $0$ & $0$ & $0$ & $0$ \\
\hline
\end{tabular}
\end{center}

\end{document}